\documentclass[11pt]{article}
\usepackage{geometry}                % See geometry.pdf to learn the layout options. There are lots.
\usepackage{fullpage}

\usepackage{Adrians_standard_stylesheet}
\usepackage{quiver}
\usepackage{epigraph}

\usepackage[nottoc]{tocbibind}

\usepackage{lipsum}       % for sample text
\usepackage{changepage}	% for indentation
\usepackage{scalerel}		% for scaling symbols

\tikzcdset{													%tikzcd displaystyle
  cells={font=\everymath\expandafter{\the\everymath\displaystyle}},
}

\theoremstyle{plain}
\newtheorem{theorem}{Theorem}[subsection]
\newtheorem{corollary}[theorem]{Corollary}
\newtheorem{lemma}[theorem]{Lemma}
\newtheorem{proposition}[theorem]{Proposition}

\newtheorem{atheorem}{Theorem}

\theoremstyle{definition}
\newtheorem{definition}[theorem]{Definition}
\newtheorem{notation}[theorem]{Notation}
\newtheorem{convention}[theorem]{Convention}

\newtheorem{example}[theorem]{Example}  

\theoremstyle{remark}
\newtheorem{remark}[theorem]{Remark}
\newtheorem{warning}[theorem]{Warning}

\author
{
Adrian Clough\thanks{
The author acknowledges support by \emph{Tamkeen} under \emph{NYUAD Research Institute grant} \texttt{CG008}.
}
}

\title{The homotopy theory of differentiable sheaves} 

\begin{document}

\maketitle

\begin{abstract}  
Many important theorems in differential topology relate properties of manifolds to properties of their underlying homotopy types -- defined e.g.\ using the total singular complex or the \v{C}ech nerve of a good open cover. Upon embedding the category of manifolds into the $\infty$-topos $\Diff^\infty$ of differentiable sheaves one gains a further notion of underlying homotopy type: the \emph{shape} of the corresponding differentiable sheaf. \par
In a first series of results we prove using simple cofinality and descent arguments that the shape of any manifold coincides with many other notions of underlying homotopy types such as the ones mentioned above. Our techniques moreover allow for computations, such as the homotopy type of the Haefliger stack, following Carchedi. \par
This leads to more refined questions, such as what it means for a mapping differential sheaf to have the correct shape. To answer these we construct model structures as well as more general homotopical calculi on the $\infty$-category $\Diff^\infty$ (which restrict to its full subcategory of $0$-truncated objects, $\Diff_{\leq 0}^\infty$) with shape equivalences as the weak equivalences. These tools are moreover developed in such a way so as to be highly customisable, with a view towards future applications, e.g.\ in geometric topology. \par
Finally, working with the $\infty$-topos $\Diff^0$ of sheaves on topological manifolds, we give new and conceptual proofs of some classical statements in algebraic topology. These include Dugger and Isaksen's hypercovering theorem, and the fact that the Quillen adjunction between simplicial sets and topological spaces is a Quillen equivalence. 	
\end{abstract} 

\tableofcontents

\section{Introduction} \label{Introduction} 

\subsection{Overview}	\label{Overview}

Many important results about smooth manifolds such as the classification of compact surfaces or the Poincar\'{e}-Hopf theorem express differential topological properties in terms of suitably defined underlying homotopy types of smooth manifolds. Similarly, important invariants of smooth manifolds such as their de Rham cohomology only depend on their underlying homotopy type. Let $M$ be a smooth manifold, then there are many ways in which to define its underlying homotopy type, e.g., one may take
\begin{enumerate} 
\item	\label{s1}	its smooth total singular complex; 
\item	its underlying topological space; 
\item	\label{s3} or to a hypercover % https://q.uiver.app/#q=WzAsNCxbMywwLCJNIl0sWzIsMCwiXFxjb3Byb2QgXFxtYXRoYmZ7Un1eZCJdLFsxLDAsIlxcY29wcm9kIFxcbWF0aGJme1J9XmQiXSxbMCwwLCJcXGNkb3RzIl0sWzEsMF0sWzIsMSwiIiwwLHsib2Zmc2V0IjotMX1dLFsyLDEsIiIsMCx7Im9mZnNldCI6MX1dLFszLDJdLFszLDIsIiIsMCx7Im9mZnNldCI6LTJ9XSxbMywyLCIiLDAseyJvZmZzZXQiOjJ9XV0=
\begin{tikzcd}
	\cdots & {\coprod \mathbf{R}^d} & {\coprod \mathbf{R}^d} & M
	\arrow[from=1-3, to=1-4]
	\arrow[shift left, from=1-2, to=1-3]
	\arrow[shift right, from=1-2, to=1-3]
	\arrow[from=1-1, to=1-2]
	\arrow[shift left=2, from=1-1, to=1-2]
	\arrow[shift right=2, from=1-1, to=1-2]
\end{tikzcd} (e.g., the \v{C}ech complex of a good open cover of  $M$), one may associate the corresponding simplicial set obtaining by replacing every copy of $\mathbf{R}^d$ by $\mathbf{1}$,% https://q.uiver.app/#q=WzAsMyxbMiwwLCJcXGNvcHJvZCBcXG1hdGhiZnsxfSJdLFsxLDAsIlxcY29wcm9kIFxcbWF0aGJmezF9Il0sWzAsMCwiXFxjZG90cyJdLFsxLDAsIiIsMCx7Im9mZnNldCI6LTF9XSxbMSwwLCIiLDAseyJvZmZzZXQiOjF9XSxbMiwxXSxbMiwxLCIiLDAseyJvZmZzZXQiOi0yfV0sWzIsMSwiIiwwLHsib2Zmc2V0IjoyfV1d
\begin{tikzcd}
	\cdots & {\coprod \mathbf{1}} & {\coprod \mathbf{1}}.
	\arrow[shift left, from=1-2, to=1-3]
	\arrow[shift right, from=1-2, to=1-3]
	\arrow[from=1-1, to=1-2]
	\arrow[shift left=2, from=1-1, to=1-2]
	\arrow[shift right=2, from=1-1, to=1-2]
\end{tikzcd}\end{enumerate} 
\par Unfortunately, these constructions suffer from at least two defects: 1.\ They all rely on specific models of homotopy types (i.e., simplicial sets and topological spaces). 2. Non of these constructions are expressed in terms of a universal property. 

These two defects may be remedied by thinking of underlying homotopy types in terms of covering spaces: Let $\mathcal{E}$ be an $\infty$-topos and denote by $\pi: \mathcal{E} \to \mathcal{S}$ the unique geometric morphism to $\mathcal{S}$. Let $X$ be an object in $\mathcal{E}$,  then for any map $A \to B$ of homotopy types and any map $X \to \pi^*B$, the pullback square 
\begin{equation}	\label{covering space}
% https://q.uiver.app/#q=WzAsNCxbMCwwLCJFIl0sWzEsMCwiXFxwaV4qQSJdLFsxLDEsIlxccGleKkIiXSxbMCwxLCJYIl0sWzAsMV0sWzEsMl0sWzMsMl0sWzAsM10sWzAsMiwiIiwxLHsic3R5bGUiOnsibmFtZSI6ImNvcm5lciJ9fV1d
\begin{tikzcd}
	E & {\pi^*A} \\
	X & {\pi^*B}
	\arrow[from=1-1, to=1-2]
	\arrow[from=1-2, to=2-2]
	\arrow[from=2-1, to=2-2]
	\arrow[from=1-1, to=2-1]
	\arrow["\lrcorner"{anchor=center, pos=0.125}, draw=none, from=1-1, to=2-2]
\end{tikzcd}
\end{equation} 
produces a covering space over $X$ (see \cite[Prop.~3.4]{mH2018}). Moreover, given a further pair consisting of a homotopy type $B'$ and a morphism $X \to \pi^*B'$, as well as a map $B' \to B$ for which the diagram 
% https://q.uiver.app/#q=WzAsMyxbMCwxLCJYIl0sWzEsMCwiXFxwaV4qQiciXSxbMSwyLCJcXHBpXipCIl0sWzEsMl0sWzAsMV0sWzAsMl1d
\[\begin{tikzcd}
	& {\pi^*B'} \\
	X \\
	& {\pi^*B}
	\arrow[from=1-2, to=3-2]
	\arrow[from=2-1, to=1-2]
	\arrow[from=2-1, to=3-2]
\end{tikzcd}\] 
commutes, the covering space $E \to X$ may be obtained via the same construction from the morphism $A \times_B B' \to B'$. Thus, any covering space over $X$ obtained from $B$ as in (\ref{covering space}) may be obtained from $B'$ in the same way, and if there exists a universal morphism $X \to \pi^*C$ as above, then all covering spaces over $X$ constructed as in (\ref{covering space}) may be obtained from homotopy types over $C$. While $\mathcal{E}(X, \pi^*(\emptyinput))$ is not in general representable, it \emph{is} pro-representable, so that $\mathcal{E} \leftarrow \mathcal{S}: \pi^*$ admits a formal left adjoint $\pi_!: \mathcal{E} \to \Pro(\mathcal{S})$, a colimit preserving functor which associates to any object $X$ a pro-homotopy type called its \emph{shape}. In many cases, large classes of covering spaces may be recovered from its shape (see \cite[Thms.~3.13~\&~4.3]{mH2018} and Remark \ref{loc top cov}).  For example, when $\mathcal{E}$ is the $\infty$-topos of sheaves on the $\infty$-category of schemes w.r.t.\ the \'{e}tale topology, then the shape coincides with the \'{e}tale homotopy type, and the category of $0$-truncated covering spaces over any scheme, can be recovered from its \'{e}tale homotopy type (see \cite[\S 5]{mH2018}). 

For a simpler example, consider the  $\infty$-topos $[A^{\op}, \mathcal{S}]$ of presheaves on a small $\infty$-category $A$. In this case, the shape functor forms a true left adjoint, i.e., it factors through $\mathcal{S} \hookrightarrow \Pro(\mathcal{S})$, and is given by $\colim: [A^{\op}, \mathcal{S}] \to \mathcal{S}$. A salient property of this example is that the shape of any representable object is contractible, and that $[A^{\op}, \mathcal{S}]$ is generated under colimits by objects of contractible shape. Moreover, for a second small $\infty$-category $B$ together with a functor $u: A \to B$, we obtain a triple adjunction 
\begin{equation}	\label{presheaf triple adjunction}
% https://q.uiver.app/#q=WzAsMixbMCwwLCJbQV57XFxvcH0sIFxcbWF0aGNhbHtTfV0iXSxbMiwwLCJbQl57XFxvcH0sIFxcbWF0aGNhbHtTfV0iXSxbMCwxLCJ1XyoiLDEseyJvZmZzZXQiOjV9XSxbMCwxLCJ1XyEiLDEseyJvZmZzZXQiOi01fV0sWzEsMCwidV4qIiwxXSxbMyw0LCIiLDIseyJsZXZlbCI6MSwic3R5bGUiOnsibmFtZSI6ImFkanVuY3Rpb24ifX1dLFs0LDIsIiIsMix7ImxldmVsIjoxLCJzdHlsZSI6eyJuYW1lIjoiYWRqdW5jdGlvbiJ9fV1d
\begin{tikzcd}
	{[A^{\op}, \mathcal{S}]} && {[B^{\op}, \mathcal{S}]}
	\arrow[""{name=0, anchor=center, inner sep=0}, "{u_*}"{description}, shift right=5, from=1-1, to=1-3]
	\arrow[""{name=1, anchor=center, inner sep=0}, "{u_!}"{description}, shift left=5, from=1-1, to=1-3]
	\arrow[""{name=2, anchor=center, inner sep=0}, "{u^*}"{description}, from=1-3, to=1-1]
	\arrow["\dashv"{anchor=center, rotate=-90}, draw=none, from=1, to=2]
	\arrow["\dashv"{anchor=center, rotate=-90}, draw=none, from=2, to=0]
\end{tikzcd}\end{equation} 
where $u_!: [A^{\op}, \mathcal{S}] \to [B^{\op}, \mathcal{S}]$ always preserves shapes, and $[A^{\op}, \mathcal{S}] \leftarrow [B^{\op}, \mathcal{S}] :u^*$ preserves shapes precisely when $u: A \to B$ is initial (a.k.a.\ cofinal, a.k.a.\ coinitial, a.k.a.\ \dots). 
\par  
We now return to the constructions described in points \ref{s1}.\ - \ref{s3}.\ above. Denote by $\Diff^r$ the $\infty$-topos of \Emph{$r$-times differentiable sheaves} --- $\mathcal{S}$-valued sheaves w.r.t.\ the usual Grothendieck topology on the category of Cartesian spaces $\mathbf{R}^d$ ($d \geq 0$) and $r$-times differentiable maps between them. Observe that the category of $r$-times differentiable manifolds forms a full subcategory of $\Diff^r$. Using the technology of fractured $\infty$-toposes we are able to show that the shape of $\mathbf{R}^d$ coincides with the shape of its underlying topological space, which is seen to be contractible via a simple Galois theoretic proof (see Lemma \ref{diff loc contr}). Thus, $\Diff^r$ is generated by a set of objects of contractible shape, and has many of the pleasant properties of presheaf $\infty$-categories. For instance, as for presheaf categories, it follows that the shape functor factors through $\mathcal{S}$. 
\par More importantly, we are able to make similar cofinality arguments for $\Diff^r$ as for presheaf $\infty$-toposes, bringing us back to the question of calculating shapes. For example, (any number of variants of) the functor $u: \Delta \to \Diff^r$ sending $[n]$ to the standard simplex can be regarded as initial in a certain sense, and one obtains an adjunction 
\begin{equation}	\label{intro nerve}
 \adjunction{u_!: [\Delta^{\op}, \mathcal{S}]}	{\Diff^r: u^*}	
 \end{equation}
in which both adjoints preserve shapes. Moreover, if $r \geq s$, the forgetful functor $\Cart^r \to \Cart^s$ induces a triple adjunction 
% https://q.uiver.app/?q=WzAsMixbMCwwLCJcXERpZmZeciJdLFsyLDAsIlxcRGlmZl5zIl0sWzAsMSwidV8qIiwxLHsib2Zmc2V0Ijo1fV0sWzAsMSwidV8hIiwxLHsib2Zmc2V0IjotNX1dLFsxLDAsInVeKiIsMV0sWzMsNCwiIiwyLHsibGV2ZWwiOjEsInN0eWxlIjp7Im5hbWUiOiJhZGp1bmN0aW9uIn19XSxbNCwyLCIiLDIseyJsZXZlbCI6MSwic3R5bGUiOnsibmFtZSI6ImFkanVuY3Rpb24ifX1dXQ==
\begin{equation}	\label{intro smoothing}
\begin{tikzcd}
	{\Diff^r} && {\Diff^s}
	\arrow[""{name=0, anchor=center, inner sep=0}, "{u_*}"{description}, shift right=5, from=1-1, to=1-3]
	\arrow[""{name=1, anchor=center, inner sep=0}, "{u_!}"{description}, shift left=5, from=1-1, to=1-3]
	\arrow[""{name=2, anchor=center, inner sep=0}, "{u^*}"{description}, from=1-3, to=1-1]
	\arrow["\dashv"{anchor=center, rotate=-90}, draw=none, from=1, to=2]
	\arrow["\dashv"{anchor=center, rotate=-90}, draw=none, from=2, to=0]
\end{tikzcd}
\end{equation}
analogous to (\ref{presheaf triple adjunction}), where again $u_!$ and $u^*$ preserve shapes. If $s = 0$, then $u_!$ sends any manifold to its underlying topological space. 
\par Finally, taking a hypercover $U_\bullet$ of $M$ such that $U_n = \coprod \mathbf{R}^d$ for all $n \geq 0$, we observe that
\begin{equation}	\label{intro descent} 
\pi_! M \simeq \pi_! \colim_{[n] \in \Delta} (U_n) \simeq \colim_{[n] \in \Delta} \pi_! (U_n)	\simeq	\colim_{[n] \in \Delta} \pi_! (\coprod \mathbf{R}^d)	\simeq	\colim_{[n] \in \Delta} \coprod \pi_! (\mathbf{R}^d)\simeq	\colim_{[n] \in \Delta} \coprod \pi_! \mathbf{1}_\mathcal{S}	
\end{equation}
by descent and the fact that $\pi_!: \Diff^r \to \mathcal{S}$ preserves colimits, showing that the simplicial set associated to $U_\bullet$ indeed calculates the correct homotopy type.  Applying (\ref{intro nerve}) to point 1.\ above, (\ref{intro smoothing}) to 2., and (\ref{intro descent}) to 3., we obtain the following theorem (see \S \ref{underlying homotopy type}): 
\begin{atheorem}\label{homotopy types agree}
The homotopy types described in points 1.\ - 3.\ above are all equivalent to the shape of $M$. \qed
\end{atheorem} 
%! add proof
As we have seen, the shape is defined for all differentiable sheaves, not just manifolds, and one important example for which we might want to calculate the shape is given by the internal mapping sheaf $\underline{\Diff}^r(A,X)$ for differentiable sheaves $A$ and $X$. To illustrate why this might be useful, consider the analogous situation in the context of compactly generated topological spaces, where we assume that $A$ is a CW-complex. The internal mapping space $\underline{\TSpc}(A,X)$ (consisting of the set of continuous maps equipped with the compact-open topology) is then a model for the mapping homotopy type of the homotopy types modelled by $A$ and $X$. In the differentiable setting, when $A$ and $X$ are manifolds, with $A$ closed, the set ${\Diff}^r(A,X)$ may be endowed with the structure of an infinite dimensional Fr\'{e}chet manifold \cite[Th.~1.11]{mGvG1973}, and it is a folk theorem that its underlying homotopy type is again equivalent to $\mathcal{S}(\pi_!A,\pi_!X)$. By \cite[Lm~A.1.7]{kW2012} the Fr\'{e}chet manifold of smooth maps from $M$ to $N$ is canonically equivalent to $\underline{\Diff}^r(M,N)$. Moreover, the shape functor $\pi_!: \Diff^r \to \mathcal{S}$ commutes with products, so that we obtain a comparison morphism $\pi_!\underline{\Diff}^r(A,X) \to \mathcal{S}(\pi_! A, \pi_!X)$. A differentiable sheaf $A$ is then said to satisfy the \emph{differentiable Oka principle} if the map $\pi_!\underline{\Diff}^r(A,X) \to \mathcal{S}(\pi_! A, \pi_!X)$ is an isomorphism for all differentiable sheaves $X$ (see \cite{hSuS2021}), and it is natural to ask for which differentiable sheaves the differentiable Oka principle holds. We obtain the following generalisation of the main statement of \cite{dBEpBDBdP2019} (see Theorem \ref{manfib}). %! in a future version illustrate useful mapping spaces, of which we would like to know the shape.
\begin{atheorem}	\label{oka introduction}
Any paracompact Hausdorff $C^\infty$-manifold locally modelled on Hilbert spaces, nuclear Fr\'{e}chet spaces, or nuclear Silva spaces satisfies the differentiable Oka principle. %! define!
\qed
\end{atheorem} 

The theory leading up to the proof of Theorem \ref{homotopy types agree} above could be viewed as a study of the interaction of shapes with colimits---which is quite simple, as shape functors commute with all colimits. The proof of Theorem \ref{oka introduction} on the other hand boils down to showing that the shape functor $\pi_!: \Diff^r \to \mathcal{S}$ commutes with certain pullbacks---which is more difficult. Specifically, one needs a method for identifying morphisms $X \to Y$ in $\Diff^r$ such that any pullback along $X \to Y$ commutes with $\pi_!$. It turns out that the $\infty$-toposes considered in this article are such that if they admit homotopical calculi (such as model structures) then $X \to Y$ has the desired property whenever it is a fibration in any of these calculi. Thus, we are led to developing flexible tools for constructing such homotopical calculi, which we do using the theory of test categories. 
%\par In the following two sections we discuss, respectively, applications of our results to geometric topology and the homotopy theory of topological spaces. After discussing the relation ship of our results to other work in \S \ref{Relation to other work} we give a detailed description of the organisation of this article in \S \ref{Organisation}. 

\subsection{Applications to geometric topology}	\label{geometric topology}  

Here we discuss some of the good properties of $\Diff_{\leq 0}^r$, the topos of set valued sheaves on manifolds, and illustrate how these might be relevant to problems in geometric topology, and in particular to Gromov's sheaf theoretic \textit{h}-principle (these applications will not be further discussed in the body of this article; for more details see \cite{dA2009}, \cite{oRW2011}, \cite{eD2014}, \cite{aK2019}).  
\par Let $\mathbf{Emb}_d^\infty$ denote the topological category whose objects are the $d$-dimensional smooth manifolds, and where $\underline{\mathbf{Emb}}_d^\infty(M,N)$ is the set of smooth embeddings of $M$ in $N$, equipped with, equivalently, the underlying topology of the Fréchet manifold $\mathbf{Emb}_d^\infty(M,N)$ or the $C^\infty$-compact-open topology. Recall that a sheaf $F$ on $\mathbf{Emb}_d^\infty$ valued in topological spaces is \Emph{invariant} if the map $\underline{\mathbf{Emb}}_d^\infty(M,N) \times F(M) \to F(N)$ is continuous. \par
Fixing a smooth manifold $N$, the following are examples of invariant sheaves: 
\begin{enumerate}[label = \normalfont \arabic*.]
	\item	\label{ex1}		The sheaf $\underline{\Imm}(\emptyinput, N)$ sending each manifold $M$ to the space of immersions of $M$ in $N$. 
	\item	\label{ex2}		The sheaf $\underline{\mathrm{Subm}}(\emptyinput, N)$ sending each manifold $M$ to the space of submersion of $M$ to $N$.
	\item	\label{ex3}		The sheaf $\mathrm{Conf}$ of configurations sending any manifold $M$ to the space of finite subsets of $M$, topologised in such a way that points may ``disappear off to infinity'' when $M$ is open (See \cite[\S 3]{oRW2011}).
%	\item	\label{ex4}		The sheaf sending any manifold $M$ to the space of symplectic forms on $M$ (for $d$ even). 
\end{enumerate} 
An invariant sheaf $F$ is \Emph{microflexible} (\cite[Def.~5.1]{oRW2011}) if for
	\begin{enumerate}[label = (\roman*)]
	\item	any polyhedron $K$, 
	\item	any manifold $M$, 
	\item	compact subsets $A \subseteq B \subseteq M$, and 
	\item	subsets $U \subseteq V \subseteq M$ containing $A$ and $B$, respectively,
	\end{enumerate} 
the lifting problem
% https://q.uiver.app/?q=WzAsNSxbMCwxLCJbMCxcXHZhcmVwc2lsb25dXFx0aW1lcyBLIl0sWzEsMSwiWzAsMV1cXHRpbWVzIEsiXSxbMSwwLCJcXHswXFx9IFxcdGltZXMgSyJdLFsyLDEsIkYoVSkiXSxbMiwwLCJGKFYpIl0sWzAsMV0sWzIsMF0sWzIsMV0sWzEsM10sWzIsNF0sWzQsM10sWzAsNCwiIiwwLHsic3R5bGUiOnsiYm9keSI6eyJuYW1lIjoiZGFzaGVkIn19fV1d
\begin{equation}	\label{microflexibility}
\begin{tikzcd}
	& {\{0\} \times K} & {F(V)} \\
	{[0,\varepsilon]\times K} & {[0,1]\times K} & {F(U)}
	\arrow[from=2-1, to=2-2]
	\arrow[from=1-2, to=2-1]
	\arrow[from=1-2, to=2-2]
	\arrow[from=2-2, to=2-3]
	\arrow[from=1-2, to=1-3]
	\arrow[from=1-3, to=2-3]
	\arrow[dashed, from=2-1, to=1-3]
\end{tikzcd}
\end{equation} 
admits a solution for some $0 < \varepsilon < 1$, possibly after passing to a smaller pair $U \subseteq V$ containing $A$ and $B$, respectively. Examples \ref{ex1}\ - \ref{ex3}\ listed above are microflexible.   \par
For any invariant sheaf $F$ and any manifold $M$ one may construct the \emph{scanning map}  (see \cite[Lect.~17]{jF2011})
\begin{equation}	\label{associated} 
\mathrm{scan}: F(M)	\to \Gamma \big(\mathrm{Fr}(TM) \times_{\mathrm{O}_n} F(\mathbf{R}^n)\to M \big), 
\end{equation}
and $F$ is said to \emph{satisfy the \textit{h}-principle on $M$} if the scanning map is an equivalence. 

\begin{theorem}[{\cite[Lect.~20]{jF2011}}]	\label{microflexible h principle}
Every microflexible invariant sheaf satisfies the \textit{h}-principle on any open manifold.  \qed
\end{theorem} 

This is a very powerful theorem, as the study of $ \Gamma \big(\mathrm{Fr}(TM) \times_{\mathrm{O}_n} F(\mathbf{R}^n)\to M \big)$ is often easier than that of $F(M)$. 

\begin{example} 
For $F =  \underline{\Imm}(\emptyinput, N)$ (as in \ref{ex1}\ above), the space $ \Gamma \big(\mathrm{Fr}(TM) \times_{\mathrm{O}_n} F(\mathbf{R}^n)\to M \big)$ can with little effort be shown to be equivalent to the space of formal immersions of $M$ into $N$, that is, the set of bundle maps 
% https://q.uiver.app/?q=WzAsNCxbMCwwLCJUTSJdLFsxLDAsIlROIl0sWzAsMSwiTSJdLFsxLDEsIk4iXSxbMCwxXSxbMCwyXSxbMSwzXSxbMiwzLCJmIl1d
\[\begin{tikzcd}
	TM & TN \\
	M & N
	\arrow[from=1-1, to=1-2]
	\arrow[from=1-1, to=2-1]
	\arrow[from=1-2, to=2-2]
	\arrow["f", from=2-1, to=2-2]
\end{tikzcd}\]
which restrict to monomorphisms $T_xM \to T_{fx}N$ for all $x \in M$. The \textit{h}-principle can thus be used to prove the famed Smale-Hirsch theorem (see \cite{sS1959c} \& \cite{mH1959} for details).   \qede
\end{example}

Theorem \ref{microflexible h principle} may be viewed as a statement that any microflexible invariant sheaf $F: (\mathbf{Emb}_n^\infty)^{\op} \to \TSpc$ retains many of its exactness properties when composed with the functor $\TSpc \to \mathcal{S}$, sending any topological space to its (singular) homotopy type. The geometry of the constituent spaces of $F$ is frequently crucial for proving microflexibility. However, 
	\begin{enumerate}
	\item	\label{hard} 		it is often difficult to construct suitable topologies on these spaces which exhibit this geometry, and
	\item	\label{smooth} 		these topologies then fail to account for natural smooth structures which one would expect these spaces to admit. 
	\end{enumerate}
In fact, the constituent spaces of $F$ are oftentimes more naturally viewed as objects of $\Diff^\infty$ (as already observed in \cite{sGiMuTmW2009} and \cite{aK2019}), so that one is lead to consider sheaves of the form $F: (\mathbf{Emb}_n^\infty)^{\op} \to \Diff^\infty$. At a first glance, it may look as if we are introducing a new complication by considering sheaves valued in an $\infty$-category rather than an ordinary category. However, in most cases, such as in the examples \ref{ex1}\ - \ref{ex3}\ considered above, we obtain sheaves valued in $\Diff_{\leq 0}^\infty$. The following theorem provides a first justification for replacing $\TSpc$ with  $\Diff_{\leq 0}^\infty$ (see Proposition \ref{ttm} and Theorem \ref{cst}).  
\begin{atheorem}[{\cite[\S 6.1]{dcC2003}}]
The topos $\Diff_{\leq 0}^\infty$ admits a model structure such that the restriction of the shape functor $\pi_!: \Diff_{\leq 0}^r \to \mathcal{S}$ exhibits $\mathcal{S}$ as a localisation of $\Diff_{\leq 0}^r$ along the weak equivalences. \qed
\end{atheorem}
Thus, many of the techniques developed in this article may be used without knowledge of $\infty$-categories. Moreover, $\Diff_{\leq 0}^\infty$ has excellent formal properties, which are directly relevant to the microflexibility condition (Theorem \ref{compact compact} \& Corollary \ref{truncated colimits}):

\begin{atheorem}	\label{compact intro}
Closed manifolds are categorically compact in $\Diff^\infty$ (and thus in $\Diff_{\leq 0}^\infty$). \qed
\end{atheorem} 

\begin{atheorem}	\label{filtered homotopy colimits} 
Filtered colimits in $\Diff_{\leq 0}^\infty$ are homotopy colimits. \qed
\end{atheorem} 

To give a simple illustration of how these properties are relevant to the sheaf theoretic \emph{h}-principle, we see that the lifting condition (\ref{microflexibility}) may now be replaced with 
% https://q.uiver.app/?q=WzAsNSxbMCwxLCJbMCxcXHZhcmVwc2lsb25dXFx0aW1lcyBLIl0sWzEsMSwiWzAsMV1cXHRpbWVzIEsiXSxbMSwwLCJcXHswXFx9IFxcdGltZXMgSyJdLFsyLDEsIlxcY29saW1fe1UgXFxzdXBzZXRlcSBBfUYoVSkiXSxbMiwwLCJcXGNvbGltX3tWIFxcc3Vwc2V0ZXEgQn1GKFYpIl0sWzAsMV0sWzIsMF0sWzIsMV0sWzEsM10sWzIsNF0sWzQsM10sWzAsNCwiIiwwLHsic3R5bGUiOnsiYm9keSI6eyJuYW1lIjoiZGFzaGVkIn19fV1d
\begin{equation}	\label{nice microflexibility}
\begin{tikzcd}
	& {\{0\} \times K} & {\colim_{V \supseteq B}F(V)} \\
	{[0,\varepsilon]\times K} & {[0,1]\times K} & {\colim_{U \supseteq A}F(U)}
	\arrow[from=2-1, to=2-2]
	\arrow[from=1-2, to=2-1]
	\arrow[from=1-2, to=2-2]
	\arrow[from=2-2, to=2-3]
	\arrow[from=1-2, to=1-3]
	\arrow[from=1-3, to=2-3]
	\arrow[dashed, from=2-1, to=1-3]
\end{tikzcd}
\end{equation} 
eliminating the necessity to gradually choose smaller and smaller open neighbourhoods  $V \supseteq U$ of $B \supseteq A$. Indeed, this is close to how Gromov originally formulated the microflexibility condition (see \cite[\S 1.4.2]{mG1986}) but instead using quasi-topological spaces (introduced by Spanier; \cite{eS1963}) as a replacement for topological spaces, with the intention of obtaining well-behaved colimits (as explained in \cite[\S 1.4.1]{mG1986}). Unfortunately, both theorems \ref{filtered homotopy colimits} and \ref{compact intro} fail for quasi-topological spaces, as shown in Example \ref{conf} below, so that (\ref{nice microflexibility}) does not give the correct formulation of microflexbility in this setting.

A further use of the good formal properties of $\Diff_{\leq 0}^\infty$ is suggested by Ayala in \cite[p.~19]{dA2009}: A key step in the construction of the scanning map (\ref{associated}) involves carefully choosing a connection on $M$ and then reparametrising the resulting exponential map $\exp: TM \to M$ (see \cite[\S 6]{oRW2011}). In order to formulate an \emph{h}-principle which works for any exponential function, Ayala constructs the following variant of the scanning map given by 
\begin{equation}	\label{Ayala scanning} 
\mathrm{scan}: F(M)	\to \Gamma \big(\mathrm{Fr}(TM) \times_{\mathrm{O}_n} \colim_{\delta > 0} F\big(\mathring{B^n}_\delta(0)\big) \big). 
\end{equation}
The colimit $\colim_{\delta > 0} F\big(\mathring{B^n}_\delta(0)\big)$ is again taken in the category of quasi-topological spaces in \cite{dA2009} with the expectation that it has the same homotopy type as  $F(\mathbf{R}^n)$, but this once more fails by Example \ref{conf}. Fortunately, by Theorem \ref{filtered homotopy colimits} the colimit does have the correct homotopy type when taken in $\Diff_{\leq 0}^r$. More generally, we believe that working with differentiable sheaves throughout in \cite{dA2009} would fix issues which arise from working with quasi-topological spaces. \par

\begin{example}	\label{conf}
For each $\delta > 0$ the space $\mathrm{Conf}\big(\mathring{B}^n_\delta(0)\big)$ is weakly equivalent to $S^n$ (see Theorem \ref{conf ball} below). In Ayala's variant of quasi-topological spaces (see {\normalfont \cite[Def.~2.7]{dA2009}}) the colimit is equivalent to the Sierpinski space, which is contractible. In other variants of quasi-topological spaces (e.g., {\normalfont \cite[\S 3]{eSjhcW1957}, \cite[\S 1.4.1]{mG1986}}) one still obtains a contractible two-point space.  \qede	%! Explain in detail. 
\end{example}  

\paragraph{Configuration spaces} 

We conclude this subsection with a proof of the following fact, already used in Example \ref{conf}. 

\begin{theorem} \label{conf ball} 
The space $\mathrm{Conf}(\mathbf{R}^n)$ is weakly equivalent to $S^n$ for any $n \geq 0$. 
\end{theorem} 

For any smooth manifold $M$ we first redefine $\mathrm{Conf}(M)$ to be the differentiable sheaf which associates to any Cartesian space $\mathbf{R}^d$ the set of embeddings $C \hookrightarrow M \times \mathbf{R}^d$ such that the map $C \to M$ is a submersion with $0$-dimensional fibres. Using the smoothing argument in \cite[Lm.~2.17]{sGoRW2010} one can show that the singular homotopy type of $\mathrm{Conf}(M)$ as a topological space coincides with its shape as a differentiable sheaf. (Note that the definition of $\mathrm{Conf}(M)$ as a differential sheaf is much simpler than the definition of $\mathrm{Conf}(M)$ as a topological space.)

We can now prove Theorem \ref{conf ball} by making precise an idea originally due to Segal (\cite[Prop.~3.1]{gS1979}):

\begin{proof}[Sketch of proof of Theorem \ref{conf ball}]
 For every $\varepsilon > 0$ denote by $\mathrm{Conf}_{\varepsilon}(\mathbf{R}^n)$ (resp.~$\mathrm{Conf}_{\leq 1}(\mathbf{R}^n)$) the subspace of $\mathrm{Conf}(\mathbf{R}^n)$ consisting of those configurations containing at most one point in $\mathring{B}_{\varepsilon}(0)$ (resp., all of $\mathbf{R}^n$), then $\mathrm{Conf}_{\leq 1}(\mathbf{R}^n)$ may be exhibited as a retract of $\mathrm{Conf}_{\varepsilon}(\mathbf{R}^n)$ by pushing all points outside of $\mathring{B}_{\varepsilon}(0)$ in any configuration in $\mathrm{Conf}_{\varepsilon}(\mathbf{R}^n)$ off to infinity. Moreover, $\mathrm{Conf}_{\leq 1}(\mathbf{R}^n)$ is $\mathbf{R}$-homotopy equivalent to $S^n$, as $\mathrm{Conf}_{\leq 1}(\mathbf{R}^n)$ is essentially the one-point-compactification of $\mathbf{R}^n$. Finally, we have $\colim_{\varepsilon > 0} \mathrm{Conf}_{\varepsilon}(\mathbf{R}^n) = \mathrm{Conf}(\mathbf{R}^n)$, so that 
$$ \pi_! \,  \mathrm{Conf}(\mathbf{R}^n) = \pi_! \colim_{\varepsilon > 0} \mathrm{Conf}_{\varepsilon}(\mathbf{R}^n) = \colim_{\varepsilon > 0} \pi_!  \mathrm{Conf}_{\varepsilon}(\mathbf{R}^n) = \colim_{\varepsilon > 0} \pi_! S^n = \pi_!S^n,$$
where the second equivalence follows from Theorem \ref{filtered homotopy colimits}. 
\end{proof} 

\subsection{Applications to the homotopy theory of topological spaces}	\label{Applications to the homotopy theory of topological spaces}

In this section we set $r = 0$; in other words, we are now considering the $\infty$-topos $\Diff^0$ of sheaves on topological manifolds. Now, the topological realisation - total singular complex  adjunction
\begin{equation}	\label{trtsca} 
\cadjunction{{|} \emptyinput {|}: \widehat{\Delta}}	{\TSpc: s}
\end{equation} 
factors as
$$	\adjunctionchain{\widehat{\Delta}}	{\Diff^0_{\leq 0}}	{\TSpc}$$
where the first map is obtained from the cosimplicial diagram consisting of the standard topological simplices, and the second adjunction is obtained from the inclusion $v: \Cart^0 \hookrightarrow \TSpc$. Thus, by Theorem \ref{homotopy types agree} the singular homotopy type of any topological space $X$ is given by the shape of $v^*X$. This observation allows us to give simple proofs of well-known theorems relating the descent and homotopy theory for topological spaces: 
	\begin{enumerate} 
	\item	Lurie's Seifert-Van Kampen theorem (see \cite[Th.~A.3.1]{jL2012}, Theorem \ref{lsvkt}). 
	\item	Dugger and Isaksen’s hypercovering theorem (see \cite[Th.~1.1]{dDdI2004}, Theorem \ref{dihc}). 
	\item	the fact that for any principal $G$-bundle $P \to B$, the topological space $B$ is a homotopy quotient by the action of $G$ on $P$ (see Theorem \ref{homotopy quotient}). 
	\end{enumerate} 
As an illustration of these techniques we consider a topological space $X$ covered by open subsets $U$ and $V$. A modern interpretation of the Seifert Van-Kampen theorem is that the square
% https://q.uiver.app/#q=WzAsNCxbMSwwLCJYIl0sWzAsMSwiVSJdLFsyLDEsIlYiXSxbMSwyLCJVIFxcY2FwIFYiXSxbMywxLCIiLDAseyJzdHlsZSI6eyJ0YWlsIjp7Im5hbWUiOiJob29rIiwic2lkZSI6ImJvdHRvbSJ9fX1dLFszLDIsIiIsMix7InN0eWxlIjp7InRhaWwiOnsibmFtZSI6Imhvb2siLCJzaWRlIjoidG9wIn19fV0sWzEsMCwiIiwwLHsic3R5bGUiOnsidGFpbCI6eyJuYW1lIjoiaG9vayIsInNpZGUiOiJ0b3AifX19XSxbMiwwLCIiLDIseyJzdHlsZSI6eyJ0YWlsIjp7Im5hbWUiOiJob29rIiwic2lkZSI6ImJvdHRvbSJ9fX1dXQ==
\begin{equation}	\label{SVK pushout} 
\begin{tikzcd}[column sep = small]
	& X \\
	U && V \\
	& {U \cap V}
	\arrow[hook', from=3-2, to=2-1]
	\arrow[hook, from=3-2, to=2-3]
	\arrow[hook, from=2-1, to=1-2]
	\arrow[hook', from=2-3, to=1-2]
\end{tikzcd}
\end{equation} 
induces a pushout square in $\mathcal{S}$. To prove this, all we need to do is to observe that (\ref{SVK pushout}) is carried to a pushout square in $\Diff^0$, and then the theorem follows from the fact that the shape functor $\pi_!: \Diff^0 \to \mathcal{S}$ preserves colimits. 

Our techniques also allow for a conceptual proof that the Quillen adjunction (\ref{trtsca}) is a Quillen equivalence. As $\widehat{\Delta} \leftarrow \TSpc: s$ creates weak equivalences, it is enough to show that the unit $\id_{\widehat{\Delta}} \to s \circ | \emptyinput |$ is a natural weak equivalence. Now, the adjunction (\ref{intro nerve}) restricts to an adjunction 
\begin{equation}	\label{QE top intro}
\adjunction{u_!: \widehat{\Delta}:}	{:\Diff_{\leq 0}^0: u^*}
\end{equation}
by Proposition \ref{simp mono} and thus both constituent functors in (\ref{QE top intro}) preserve shape equivalences, and moreover both the unit and counit are natural equivalences. One might hope that (\ref{QE top intro}) further restricts to the adjunction (\ref{trtsca}), so that by the same argument $\id_{\widehat{\Delta}} \to s \circ | \emptyinput |$ is a natural weak equivalence. Unfortunately, $u_!: \widehat{\Delta} \to \Diff^r_{\leq 0}$ does not factor through $\TSpc$, but it does turn out that for any simplicial set $X$ the comparison map $u_!X \to |X|$ is very close to being an isomorphism, and can be deformed into one using a sequence of homotopies $H^n: \Delta^n \times [0,1] \to \Delta^n$ (see \S \ref{The Quillen model structure on topological spaces}), and this suffices to show that $\id_{\widehat{\Delta}} \to s \circ | \emptyinput |$ is a natural weak equivalence. 

Similar techniques may be used to show that for any small ordinary category $A$, and any diagram $X: A \to \TSpc$, the geometric realisation of   
% https://q.uiver.app/#q=WzAsNSxbMywwLCJcXGNvcHJvZF97YV8wIFxcaW4gQV57XFxEZWx0YV4wfX1YX3thXzB9Il0sWzIsMCwiXFxjb3Byb2Rfe2FfMCBcXHRvIGFfMSBcXGluIEFee1xcRGVsdGFeMX19WF97YV8wfSJdLFsxLDAsIlxcY29wcm9kX3thXzAgXFx0byBhXzEgXFx0byBhXzIgXFxpbiBBXntcXERlbHRhXjJ9fVhfe2FfMH0iXSxbMCwwLCJcXGNkb3RzIl0sWzQsMCwiPSBCXlxcYnVsbGV0KFgsQSwqKSJdLFsxLDAsIiIsMCx7Im9mZnNldCI6LTF9XSxbMSwwLCIiLDIseyJvZmZzZXQiOjF9XSxbMiwxLCIiLDIseyJvZmZzZXQiOi0yfV0sWzIsMSwiIiwyLHsib2Zmc2V0IjoyfV0sWzIsMV0sWzMsMiwiIiwwLHsib2Zmc2V0IjozfV0sWzMsMiwiIiwwLHsib2Zmc2V0IjotM31dLFszLDIsIiIsMCx7Im9mZnNldCI6LTF9XSxbMywyLCIiLDIseyJvZmZzZXQiOjF9XV0=
\[\begin{tikzcd}
	\cdots & {\coprod_{a_0 \to a_1 \to a_2 \in A^{\Delta^2}}X_{a_0}} & {\coprod_{a_0 \to a_1 \in A^{\Delta^1}}X_{a_0}} & {\coprod_{a_0 \in A^{\Delta^0}}X_{a_0}} & {= B^\bullet(X,A,*)}
	\arrow[shift left, from=1-3, to=1-4]
	\arrow[shift right, from=1-3, to=1-4]
	\arrow[shift left=2, from=1-2, to=1-3]
	\arrow[shift right=2, from=1-2, to=1-3]
	\arrow[from=1-2, to=1-3]
	\arrow[shift right=3, from=1-1, to=1-2]
	\arrow[shift left=3, from=1-1, to=1-2]
	\arrow[shift left, from=1-1, to=1-2]
	\arrow[shift right, from=1-1, to=1-2]
\end{tikzcd}\]
yields a model of the homotopy colimit of $X$, by employing the same sequence of homotopies $H^n: \Delta^n \times [0,1] \to \Delta^n$ as above to deform the comparison map between the realisations of $B^\bullet(X,A,*)$ in $\Diff^r$ and $\TSpc$ into an isomorphism. 

We should like to emphasise that our results showing that $\Diff^r$ and $\Diff_{\leq 0}^r$ model homotopy types are completely independent of the homotopy theory of topological spaces, so that our new proofs of these classical algebro-topological acts are not circular.

\subsection{Organisation}	\label{Organisation}

This article divides into two parts:
	\begin{itemize}
	\item[]	Part \ref{Foundations}: Foundations	
	\item[]	Part \ref{Differentiable sheaves}: Differentiable sheaves. 
	\end{itemize}
Part \ref{Foundations} treats topos theoretic foundations and their interactions with localisations. Part \ref{Differentiable sheaves} then discusses applications of the technology of Part \ref{Foundations} to the theory of differentiable sheaves. More specifically, both parts consist of three sections, and each section in Part \ref{Differentiable sheaves} discusses applications of the corresponding section in Part \ref{Foundations}, yielding the following dependencies of the sections in this article: 

% https://q.uiver.app/#q=WzAsOCxbMCwxLCJcXFMgXFx0ZXh0e1xccmVme0ZyYWN0dXJlZCB0b3Bvc2VzfX0gXFx0ZXh0eyBGcmFjdHVyZWQgJFxcaW5mdHkkLXRvcG9zZXN9Il0sWzAsMiwiXFxTIFxcdGV4dHtcXHJlZntTaGFwZXMgYW5kIGNvZmluYWxpdHl9fSBcXHRleHR7IFNoYXBlcyBhbmQgY29maW5hbGl0eX0iXSxbMCwzLCJcXGJlZ2lue3RhYnVsYXJ9e2x9IFxcUyBcXHRleHR7XFxyZWZ7SG9tb3RvcHkgdGhlb3J5IGluIGxvY2FsbHkgY29udHJhY3RpYmxlfX0gXFx0ZXh0eyBIb21vdG9weSB0aGVvcnkgaW59IFxcXFwgXFx0ZXh0e2xvY2FsbHkgY29udHJhY3RpYmxlICgkXFxpbmZ0eSQtKXRvcG9zZXN9IFxcZW5ke3RhYnVsYXJ9Il0sWzEsMSwiXFxTIFxcdGV4dHtcXHJlZntiYXNpY319IFxcdGV4dHsgQmFzaWMgZGVmaW5pdGlvbnMgYW5kIHByb3BlcnRpZXMgb2YgZGlmZmVyZW50aWFibGUgc2hlYXZlc30iXSxbMSwyLCJcXFMgXFx0ZXh0e1xccmVme1NoYXBlcywgY29maW5hbGl0eSBhbmQgZGlmZmVyZW50aWFibGUgc2hlYXZlc319IFxcdGV4dHsgU2hhcGVzLCBjb2ZpbmFsaXR5LCBhbmQgZGlmZmVyZW50aWFibGUgc2hlYXZlc30iXSxbMSwzLCJcXFMgXFx0ZXh0e1xccmVme0hvbW90b3BpY2FsIGNhbGN1bGkgb24gZGlmZmVyZW50aWFibGUgc2hlYXZlc319IFxcdGV4dHsgSG9tb3RvcGljYWwgY2FsY3VsaSBvbiBkaWZmZXJlbnRpYWJsZSBzaGVhdmVzfSJdLFswLDAsIlxcdGV4dHtcXHRleHRiZntQYXJ0IFxccmVme0ZvdW5kYXRpb25zfX19Il0sWzEsMCwiXFx0ZXh0e1xcdGV4dGJme1BhcnQgXFxyZWZ7RGlmZmVyZW50aWFibGUgc2hlYXZlc319fSJdLFswLDFdLFsxLDJdLFswLDNdLFszLDRdLFsxLDRdLFs0LDVdLFsyLDVdXQ==
\[\begin{tikzcd}
	{\text{\textbf{Part \ref{Foundations}}}} & {\text{\textbf{Part \ref{Differentiable sheaves}}}} \\
	{\S \text{\ref{Fractured toposes}} \text{ Fractured $\infty$-toposes}} & {\S \text{\ref{basic}} \text{ Basic definitions and properties of differentiable sheaves}} \\
	{\S \text{\ref{Shapes and cofinality}} \text{ Shapes and cofinality}} & {\S \text{\ref{Shapes, cofinality and differentiable sheaves}} \text{ Shapes, cofinality, and differentiable sheaves}} \\
	{\begin{tabular}{l} \S \text{\ref{Homotopy theory in locally contractible}} \text{ Homotopy theory in} \\ \text{locally contractible ($\infty$-)toposes} \end{tabular}} & {\S \text{\ref{Homotopical calculi on differentiable sheaves}} \text{ Homotopical calculi on differentiable sheaves}}
	\arrow[from=2-1, to=3-1]
	\arrow[from=3-1, to=4-1]
	\arrow[from=2-1, to=2-2]
	\arrow[from=2-2, to=3-2]
	\arrow[from=3-1, to=3-2]
	\arrow[from=3-2, to=4-2]
	\arrow[from=4-1, to=4-2]
\end{tikzcd}\]

We will thus give an overview of the individual sections in the order
% https://q.uiver.app/#q=WzAsNixbMCwwLCJcXFMgXFx0ZXh0e1xccmVme0ZyYWN0dXJlZCB0b3Bvc2VzfX0iXSxbMCwxLCJcXFMgXFx0ZXh0e1xccmVme1NoYXBlcyBhbmQgY29maW5hbGl0eX19Il0sWzAsMiwiXFxTIFxcdGV4dHtcXHJlZntIb21vdG9weSB0aGVvcnkgaW4gbG9jYWxseSBjb250cmFjdGlibGV9fSJdLFsxLDAsIlxcUyBcXHRleHR7XFxyZWZ7YmFzaWN9fSJdLFsxLDEsIlxcUyBcXHRleHR7XFxyZWZ7U2hhcGVzLCBjb2ZpbmFsaXR5IGFuZCBkaWZmZXJlbnRpYWJsZSBzaGVhdmVzfX0iXSxbMSwyLCJcXFMgXFx0ZXh0e1xccmVme0hvbW90b3BpY2FsIGNhbGN1bGkgb24gZGlmZmVyZW50aWFibGUgc2hlYXZlc319Il0sWzAsM10sWzEsNF0sWzIsNV0sWzMsMV0sWzQsMl1d
\[\begin{tikzcd}[column sep=10em,row sep=scriptsize]
	{\S \text{\ref{Fractured toposes}}} & {\S \text{\ref{basic}}} \\
	{\S \text{\ref{Shapes and cofinality}}} & {\S \text{\ref{Shapes, cofinality and differentiable sheaves}}} \\
	{\S \text{\ref{Homotopy theory in locally contractible}}} & {\S \text{\ref{Homotopical calculi on differentiable sheaves}},}
	\arrow[from=1-1, to=1-2]
	\arrow[from=2-1, to=2-2]
	\arrow[from=3-1, to=3-2]
	\arrow[from=1-2, to=2-1]
	\arrow[from=2-2, to=3-1]
\end{tikzcd}\]
so that the discussion of each section in Part \ref{Foundations} is immediately followed by a discussion of the corresponding section in Part \ref{Differentiable sheaves} explaining its application to differentiable sheaves.  \\

\noindent\underline{\S \ref{Fractured toposes} Fractured $\infty$-toposes:} A fractured $\infty$-topos is an $\infty$-topos $\mathcal{E}$ together with extra structure making it possible to associate to its objects gros and petit $\infty$-toposes, as well as determine the relationship between them. In \S \ref{bf}, after discussing some basic definitions, we explain how this structure may be employed to prove useful properties about $\mathcal{E}$ such as hypercompleteness. Moreover, we explain how to construct fractured $\infty$-toposes from geometric sites in \S \ref{Geometric sites}. \\

\noindent\underline{\S \ref{basic} Basic definitions and properties of differentiable sheaves:} In \S \ref{basic} we endow $\Mfd^r$ with the structure of a geometric site, so that we may equip $\Diff^r$ with the structure of a fractured $\infty$-topos, and immediately use this structure to show, among other things, that $\Diff^r$ is local and has enough points. Then, in \S \ref{Diffeological spaces} we very briefly discuss diffeological spaces with a focus on manifolds with corners such as the simplices. The petit $\infty$-topos associated to any smooth manifolds (without boundary or corners) is the $\infty$-topos of sheaves on its underlying topological space, allowing us to relate properties of manifolds viewed as objects in $\Diff^r$ with their underlying topological spaces. In \S \ref{compact manifolds} we use this to deduce that any closed manifold is categorically compact in $\Diff^r$ from the fact that the $\infty$-topos of sheaves on its underlying topological space is proper. Surprisingly, the manifolds with non-empty boundary / corners from the preceding subsection are \emph{not} categorically compact in $\Diff^r$. \\

\noindent\underline{\S \ref{Shapes and cofinality} Shapes and cofinality:} As explained in \S \ref{Overview} the \emph{shape} of an object in an $\infty$-topos gives a Galoisic notion of its underlying pro-homotopy type. In \S \ref{Basic definitions and properties} we first define shapes of \emph{$\infty$-toposes} (rather than their objects), however, afterwards, while discussing the functoriality of shapes we will arrive at the notion of shapes of objects in an $\infty$-topos, and reconcile the two notions by noting that the shape of the final object of an $\infty$-topos is the same as the shape of the $\infty$-topos itself. Moreover, we give a cohomological criterion for when the shape of an $\infty$-topos is contractible. In \S \ref{Locally contractible toposes and test toposes} we specialise to locally contractible $\infty$-toposes -- those $\infty$-toposes which are generated under colimits by a set of objects of contractible shape (so that the shape of any object is a homotopy type) -- and in \S \ref{snerves} provide recognition principles, for when nerve diagrams (such as $\Delta \to \Diff^r$, giving rise to (\ref{intro nerve}) in \S \ref{Overview}) may be used to calculate these homotopy types. Finally, in \S \ref{Fractured toposes and local contractibility} we see that the shape of the petit $\infty$-topos of any object in a fractured $\infty$-topos is equivalent to the shape of its gros $\infty$-topos, giving rise to a technique for exhibiting a fractured $\infty$-topos as locally connected. \\

\noindent\underline{\S \ref{Shapes, cofinality and differentiable sheaves} Shapes, cofinality and differentiable sheaves:} In \S \ref{dialct} we use the techniques from \S \ref{Fractured toposes and local contractibility} to show that $\Diff^r$ is locally contractible: $\Diff^r$ is generated by the Cartesian spaces $\mathbf{R}^d$, whose shape may be calculated using its associated petit $\infty$-topos (which is just the $\infty$-topos of sheaves on its underlying topological space or $\mathbf{R}^d$) which may be shown to be contractible by combining the cohomological criterion from \S \ref{Basic definitions and properties} with the fact that covering spaces on $\mathbf{R}^d$ are trivial. In \S \ref{underlying homotopy type} we use the techniques of \S \ref{Locally contractible toposes and test toposes} to prove Theorem \ref{homotopy types agree} from \S \ref{Overview}. Then in \S \ref{locally contractible applications}  we give two more applications of the technology developed so far: In \S  \ref{The shape of the Haefliger stack} we show how Carchedi's calculation of the shape of the Haefliger stack seems almost inevitable using the calculus of shapes on $\Diff^r$, and in \S \ref{Colimits of hypercoverings of topological spaces} we give elementary new proofs of several Seifert - Van Kampen like theorems, such as Dugger and Isaksen's hypercovering theorem. \\

\noindent\underline{\S \ref{Homotopy theory in locally contractible} Homotopy theory in locally contractible ($\infty$-)toposes} Given an $\infty$-category $C$ together with a subcategory $W$ of weak equivalences, we discuss which (co)limits in $C$ are preserved by the localisation functor $\gamma: C \to W^{-1}C$, i.e, which (co)limits are \emph{homotopy (co)limits}.  In all of our applications $C$ will always be a subcategory of an $\infty$-topos $\mathcal{E}$, and $\gamma$ will always be the restriction of the shape functor $\pi_!: \mathcal{E} \to \mathcal{S}$ to $C$, so the relationship between homotopy limits and colimits is not symmetric, as $\pi_!: \mathcal{E} \to \mathcal{S}$ preserves \emph{all} colimits. Thus, any colimit in $C$ which commutes with the inclusion $C \hookrightarrow \mathcal{E}$ is a homotopy colimit. In \S \ref{truncated} we study which colimits are preserved by  $C \hookrightarrow \mathcal{E}$ when $C$ is the subcategory of $n$-truncated objects for some $0 \leq n \leq \infty$, and in \S \ref{Concrete objects} we study the situation when $\mathcal{E}_{\leq 0}$ is a local $\infty$-category, and $C$ is the ordinary category of concrete $0$-truncated sheaves on $\mathcal{E}$. In \S \ref{Basic theory of homotopical calculi} we discuss how to recognise homotopy limits for arbitrary localisations using homotopical calculi, such as model structures (on $\infty$-categories). In particular, we show how to recognise \emph{sharp morphisms} -- morphisms along which all pullbacks are homotopy pullbacks. Finally, in \S \ref{Homotopical calculi in locally contractible toposes} we combine the theory of test categories with the technology of \S \ref{snerves} to construct model structures on locally contractible $\infty$-toposes as well as ordinary toposes generated by objects of contractible shape. \\

\noindent\underline{\S \ref{Homotopical calculi on differentiable sheaves}	Homotopical calculi on differentiable sheaves:} In \S \ref{model structures} we use the technology of \S \ref{Homotopical calculi in locally contractible toposes} to construct a plethora of model structures on $\Diff^r$ and $\Diff^r_{\leq 0}$, all modelling $\mathcal{S}$, and use \S \ref{truncated} to describe several families of homotopy colimits on $\Diff^r_{\leq 0}$. In \S \ref{The Kihara model structure on diffeological spaces} we show that one of these model structures on $\Diff^r_{\leq 0}$ -- the Kihara model structure --  restricts to the subcategory of diffeological spaces $\Diff_{\concr}^r$ (which again models $\mathcal{S}$), and use \S \ref{Concrete objects} to describe several classes of homotopy colimits. With a bit of extra work, we also show that it is possible to recover the Quillen equivalence $\cadjunction{\widehat{\Delta}}	{\TSpc}$ in a similar fashion. 

In \S \ref{The differentiable Oka principle} we finally prove Theorem \ref{oka introduction}: The main idea is to show that the class of objects satisfying the differentiable Oka principle is closed under various (co)limits and under $\Delta^1$-homotopy equivalence. Then, one may show inductively that simplicial complexes built using Kihara's simplices satisfy the differentiable Oka principle, and that the manifolds in Theorem \ref{oka introduction} are $\Delta^1$-homotopy equivalent to such simplicial complexes. The above induction step relies on showing that for each differentiable sheaf $X$ and each Kihara boundary inclusion $\partial \Delta^n \hookrightarrow \Delta^n$ the map $X^{\partial \Delta^n} \leftarrow X^{\Delta^n}$ is sharp, which we do by exhibiting it as a \emph{squishy fibration}; a notion which we introduce in \S \ref{The differentiable Oka principle} for precisely this purpose.  We conclude \S \ref{The differentiable Oka principle} by providing examples of manifolds which do \emph{not} satisfy the differentiable Oka principle. 	\\

The article also includes three appendices collecting some necessary background material: 
	\begin{itemize}
	\item[	\S	\ref{cubical}]		recalls some basic facts about the cube category and cubical diagrams.
	\item[	\S	\ref{model structure}]	exhibits how the definition of model structures may be implemented in the $\infty$-categorical setting. 
	\item[	\S	\ref{pro}]			provides some (mostly new) results on pro-objects in $\infty$-categories (which are however already well-known in the ordinary categorical setting). 
	\end{itemize}

We also summarise some of the most important conventions adopted in this article. For a detailed account of our conventions, see \emph{\nameref{Conventions}}. 
\begin{itemize}
\item	Throughout the whole article, $r$ denotes some element of $\mathbf{N} \cup \{\infty\}$. 
\item	Following the widely adopted precedent set by Lurie we will refer to quasi-categories as \emph{$\infty$-categories}. (An $\infty$-category is however still, strictly speaking, a simplicial set, and thus we will often speak of maps from a simplicial set to an $\infty$-category, etc.)
\item	We adopt the ``French'' tradition of denoting the ordinary category of presheaves on any small ordinary category $A$ by $\widehat{A}$. E.g., the category of simplicial sets is denoted by $\widehat{\Delta}$. 
\item	Canonical isomorphisms are often denoted by equality signs. (An isomorphism is canonical if it originates from a universal property. More precisely, let $u: X \to C$ a right fibration, and $x,x'$ two final objects in $X$, then for any morphism $x \to x'$ the morphism $ux \to ux'$ is a canonical isomorphism, and we may write $x = x'$.)  
\end{itemize}

\subsection{Relation to other work}	\label{Relation to other work}

The main actor in this article is arguably the shape functor $\pi_!: \Diff^r \to \mathcal{S}$. A sketch of the existence of this functor was first provided by Dugger in \cite[Prop.~8.3]{dD2001}, and a complete construction was first given in \cite[Prop.~4.4.6]{uS2013}. Other constructions are given in \cite[\S 3]{dC2016o},  \cite[Prop.~1.3]{dBEpBDBdP2019}, \cite{sB2023}, \cite[\S 4.3]{aAaDpH2021}, \cite{dP2022}. 

A model structure on $\Diff_{\leq 0}^r$ in which the weak equivalence are given by the shape equivalences is first provided in \cite[\S 6.1]{dcC2003}. We should like to point out that many results in this article (in particular on locally contractible $\infty$-toposes and cofinality) are ultimately the product of us trying to understand \cite{dcC2003} and \cite{dcC2006} in $\infty$-categorical terms. 

We owe a similar intellectual dept to Carchedi who has written extensively on differentiable stacks, and whose work in \cite{dC2020} paved the way to the theory of fractured $\infty$-toposes developed in \cite{jL2017}. 

Early inspiration for this work also came from the article \cite{mS2018} on the relationship between $\Diff^r$ and $\mathcal{S}$ in a type theoretic setting. 

Theorem \ref{oka introduction} is a generalisation of the main theorem of \cite{dBEpBDBdP2019}, but relies on a careful analysis of the shape functor and its relationship to homotopical calculi rather than the combinatorics of simplicial sets. An important inspiration for adopting a more flexible attitude towards homotopical calculi is given in \cite[\S 7]{dcC2019}, and we should like to point out that a proof of Theorem \ref{oka introduction} in the vein of this article would be significantly harder without Kihara's simplices (\cite{hK2020}). 

Moreover, we give new and simpler proofs of classical results such as \cite[Lm.~5]{jM1957} and \cite[Th.~1.3]{dDdI2004}. See, in particular, the discussion of (Borel) equivariant homotopy theory in \S \ref{Colimits of hypercoverings of topological spaces}. 

\subsection{Acknowledgments} 

First and foremost I would like to thank my advisor Andrew Blumberg  for his patience and for giving me complete freedom in choosing the direction of my research. Moreover, the thesis on which this article is based benefited greatly from Andrew's insistence on clear exposition.   

Hotel Mama provided excellent conditions under which to finish my thesis during the pandemic. 

I would like to thank Rok Gregoric for taking a great interest in this project from start to finish and for being such a wonderful friend, Mitchell Riley for proofreading and greatly clarifying several parts of this article, Dmitri Pavlov for his detailed feedback and ensuing discussions on my thesis, Varkor (and collaborators) for creating \url{http://q.uiver.app} which surely led to countless hours saved typesetting the many diagrams in this article, Eva Hernandez for putting a smile on my face every time I went to the front office at UT, and finally, Katie Stanworth for securing the whiteboards on which so much of the work was carried out that transformed my thesis into the present article.

\part{Foundations}	\label{Foundations}

\section{Fractured $\infty$-toposes}	\label{Fractured toposes}

In an $\infty$-topos of sheaves on a site of geometric objects such as 
	\begin{enumerate} 
	\item manifolds (see \cite[\S 6.1]{dC2020} and \S \ref{dss} of this article), 
	\item \label{fte2}	(ordinary, derived, spectral) schemes (with either the Zariski or \'{e}tale topology) (see \cite[\S 4.2 \& \S 4.3]{DAGV}, \cite{DAGVII}, \cite[\S 6.2]{dC2020}, \cite[\S 2.6.4]{jL2017})
	\item	derived complex analytic spaces (see \cite[\S 11 \& \S 12]{DAGIX} \& \cite{mP2019}), 
	\item \label{fte4} derived manifolds (see \cite{dCpS2019}), 
	\end{enumerate} 
it is often possible to identify a suitable class of ``\'{e}tale morphisms'', giving rise to the attendant notion of Deligne-Mumford stack. In \cite{dC2020} Carchedi shows that, surprisingly, in many such cases the $\infty$-category of Deligne-Mumford stacks (with relaxed finiteness and separatedness conditions) and \'{e}tale morphisms between them form an $\infty$-topos. The relationship between the two resulting $\infty$-toposes is axiomatised by Lurie in terms of the notion of \emph{fractured $\infty$-topos} in \cite[Def.~20.1.2.1]{jL2017}. 

In \S \ref{bf} we begin by providing a definition of fractured $\infty$-toposes equivalent to Lurie's  which highlights the salient properties necessary for us (in particular in \S \ref{Fractured toposes and local contractibility}). After this, we discuss some useful properties of fractured $\infty$-toposes. Then, in \S \ref{Geometric sites} we discuss the notion of \emph{geometric sites} which allow us to construct fractured $\infty$-toposes (such as $\Diff^r$). Then, finally, in \S \ref{equivalence} we prove the equivalence between our definition of fractured $\infty$-topos and Lurie's. 

\subsection{Basic definitions and properties}	\label{bf}

\begin{definition}%[{\cite[Def.~20.1.2.1]{jL2017}}]	
\label{ft} 
A \Emph{fractured $\infty$-topos} is an adjunction 
$$	\adjunction{j_!: \mathcal{E}^{\corp}}	{\mathcal{E}: j^*}	$$
between $\infty$-toposes $\mathcal{E}^{\corp}$ and $\mathcal{E}$ satisfying properties \ref{generated} -\ref{admissibility} below: 
	\begin{enumerate}[label = (\alph*)]
	\item	\label{generated}	The topos $\mathcal{E}$ is generated under colimits by the objects in the image of $j_!$. 
	\item	\label{ft2}	For every object $U$ in $\mathcal{E}^{\corp}$, the left adjoint in 
$$	\adjunction{(j_!)_{/U}: \mathcal{E}_{/U}^{\corp}}	{\mathcal{E}_{/U}: (j^*)_{/U}}	$$
	is fully faithful. 
	\item	\label{cocontinuous}	The functor $\mathcal{E}^{\corp} \leftarrow \mathcal{E}: j^*$ preserves colimits.
	\item \label{admissibility}	For any pullback square  
% https://q.uiver.app/?q=WzAsNCxbMCwwLCJVJyJdLFsxLDAsIlUiXSxbMCwxLCJWJyJdLFsxLDEsIlYiXSxbMCwxXSxbMSwzXSxbMCwyXSxbMiwzXV0=
\[\begin{tikzcd}
	{U'} & U \\
	{V'} & V
	\arrow[from=1-1, to=1-2]
	\arrow[from=1-2, to=2-2]
	\arrow[from=1-1, to=2-1]
	\arrow[from=2-1, to=2-2]
\end{tikzcd}\]
in which $U \to V$ and $V'$ are in the image of $j_!$, the map $U' \to V'$ is in the image of $j_!$. 
\end{enumerate}	
An object in $\mathcal{E}$ is referred to as \Emph{corporeal} if it is in the image of $j_!$. \qede
\end{definition} 

For the rest of this section $j_!:  \mathcal{E}^{\corp} \hookrightarrow \mathcal{E}$ refers to a fixed fractured $\infty$-topos. Observe that $j_!$ is faithful (but never full, unless it is an equivalence). The $\infty$-topos $\mathcal{E}^{\corp}$ will then often be identified with its image under $j_!$. 

A morphism $U \to X$ in a fractured $\infty$-topos $\mathcal{E}$ is called \Emph{admissible} if for every pullback diagram  
% https://q.uiver.app/?q=WzAsNCxbMCwwLCJVJyJdLFsxLDAsIlUiXSxbMCwxLCJYJyJdLFsxLDEsIlgiXSxbMCwxXSxbMCwyXSxbMiwzXSxbMSwzXV0=
\[\begin{tikzcd}
	{U'} & U \\
	{X'} & X
	\arrow[from=1-1, to=1-2]
	\arrow[from=1-1, to=2-1]
	\arrow[from=2-1, to=2-2]
	\arrow[from=1-2, to=2-2]
\end{tikzcd}\]
in which $X'$ is in $\mathcal{E}^{\corp}$, the morphism $U' \to X'$ is in $\mathcal{E}^{\corp}$. Thus, $\mathcal{E}^{\corp}$ may be identified with the $\infty$-category of corporeal objects in $\mathcal{E}$ together with the admissible morphisms. Under mild conditions the structure of a fractured $\infty$-topos may be recovered from its class of admissible morphisms (see \cite[Rmk.~20.3.4.6]{jL2017}). Axiom \ref{admissibility} in Definition \ref{ft} can then be thought of as consisting of two parts: 
	\begin{enumerate}
	\item	Admissible morphisms are closed under pullbacks. 
	\item	If $U \to V$ is an admissible morphism, and $V$ is a corporeal object, then $U$ is a corporeal object. 
	\end{enumerate}
For any corporeal object $U$ the subcategory of $\mathcal{E}_{/U}$ spanned by the admissible morphisms is then equivalent to the $\infty$-topos $\mathcal{E}^{\corp}_{/U}$, so that we obtain \emph{gros} and \emph{petit} $\infty$-toposes $\mathcal{E}_{/U}$ and $\mathcal{E}^{\corp}_{/U}$ of $U$, respectively.

\begin{remark} 
In \cite{DAGV} Lurie introduces the notion \emph{geometry} (see Remark \ref{geometry}), which is a site with extra structure. Lurie presents a further interpretation of the notion of fractured $\infty$-topos in \cite[\S 21]{jL2017}, namely as a ``coordinate free'' version of geometries (in the sense that a site may be viewed as providing ``coordinates'' or generators and relations for an $\infty$-topos.). In the examples \ref{fte2} - \ref{fte4}  listed in the beginning of this section the $\infty$-toposes are all classifying $\infty$-toposes for various flavours of locally ringed $\infty$-toposes, such as strictly Henselian or locally $C^\infty$-ringed $\infty$-toposes. The structure of a fractured $\infty$-topos then makes it possible to define locally ringed morphisms for the various flavours of locally ringed $\infty$-toposes. From this perspective $\infty$-topos $\Diff^r$ is unusual in that we don't know of any insightful way of thinking of it as a classifying $\infty$-topos. \qede
\end{remark} 

We now conclude this subsection with some basic properties of fractured $\infty$-toposes, which exhibit some ways in which $\mathcal{E}^{\corp}$ controls certain properties of $\mathcal{E}$.  

\begin{proposition}	\label{conservative} 
The functor $\mathcal{E}^{\corp} \leftarrow \mathcal{E}: j^*$ is conservative. 
\end{proposition} 

\begin{proof} 
Let $X \to Y$ be a morphism in $\mathcal{E}$ such that $j^*X \to j^*Y$ is an isomorphism, then for every object $U$ in $\mathcal{E}^{\corp}$ the map $\mathcal{E}^{\corp}(U, j^*X) \to \mathcal{E}^{\corp}(U, j^*Y)$ is an isomorphism, so that for every object $U$ in $\mathcal{E}^{\corp}$ the map  $\mathcal{E}(j_!U, X) \to \mathcal{E}(j_!U, Y)$ is an isomorphism, but as any object $Z$ can be written as a colimit of objects in the image of $j_!$ it follows that $\mathcal{E}(Z, X) \to \mathcal{E}(Z, Y)$ is an isomorphism, so that $X \to Y$ is an isomorphism by the Yoneda lemma.
\end{proof} 

\begin{corollary} 
If $\mathcal{E}^{\corp}$ is hypercomplete, then so is $\mathcal{E}$. 
\end{corollary} 

\begin{proof} 
Let $X \to Y$ be an $\infty$-connected morphism in $\mathcal{E}$, then  $j^* (X_{\leq n} \to Y_{\leq n}) = (j^*X \to j^*Y)_{\leq n}$  as $j^*$ is both cocontinuous and preserves finite limits. Thus, $j^*X \to j^*Y$ is an isomorphism in $\mathcal{E}^{\corp}$, so that by Proposition \ref{conservative} $X \to Y$ is an isomorphism. 
\end{proof} 

\begin{corollary}	\label{fep}
If $\mathcal{E}^{\corp}$ has enough points, then so does $\mathcal{E}$. 
\end{corollary} 

\begin{proof} 
Denote by $j_*$ the right adjoint to $j^*$ (which exists by the adjoint functor theorem), then any point $p_*: \mathcal{S} \to \mathcal{E}^{\corp}$ yields a point $j_* p_*: \mathcal{S} \to \mathcal{E}$. Let $X \to Y$ be a morphism in $\mathcal{E}$ such that $p^*j^*X \to p^*j^*Y$ is an isomorphism for every point $p$ of $\mathcal{E}^{\corp}$, then $j^*X \to j^*Y$ is an isomorphism by assumption, and thus also $X \to Y$ by Proposition \ref{conservative}. 
\end{proof} 

Recall that an $\infty$-topos $\mathcal{F}$ is \Emph{local} if the global sections functor $\pi_*: \mathcal{F} \to \mathcal{S}$ admits a right adjoint, which we denote by $\pi^!$. 

\begin{proposition}	\label{fractured local} 
If there exists an object $\star$ in $\mathcal{E}^{\corp}$ such that $j_! \star = \mathbf{1}_\mathcal{E}$, and moreover $\mathcal{E}^{\corp}_{/\star} = \mathcal{S}$, then $\mathcal{E}$ is local. 
\end{proposition} 

\begin{proof} 
As $j_{/\star}^* = (\pi_\mathcal{E})_*$ commutes with colimits, we obtain a triple adjunction 
$$	\tripleadjunction{\mathcal{S} = \mathcal{E}^{\corp}_{/\star}}	{\mathcal{E} = \mathcal{E}_{/j_!\star}}	{(j_!)_{/\star}}	{(j^*)_{/\star}}	{(j_*)_{/\star}}	$$
where ${(j_!)_{/\star}} \dashv {(j^*)_{/\star}}$ is the unique geometric morphism $\mathcal{E} \to \mathcal{S}$. 
\end{proof} 

\subsection{Geometric sites}	\label{Geometric sites}

We now discuss the notion of \emph{geometric site}, culminating in Theorem \ref{fractured sheaves}, which will allow us to exhibit $\Diff^r$ as a fractured $\infty$-topos. 

\begin{definition}[{\cite[Def.~20.2.1.1]{jL2017}}]
Let $G$ be an $\infty$-category, then an \Emph{admissibility structure} on $G$ is a subcategory $G^{\ad}$, whose morphisms are referred to as \emph{admissible morphisms}, such that: 
	\begin{enumerate}[label = (\alph*)]
	\item	\label{as1}	Every equivalence in $G$ is an admissible morphism. 
	\item	\label{as2}		For any admissible morphism $U \to X$, and any morphism $X' \to X$ there exists a pullback square 
% https://q.uiver.app/?q=WzAsNCxbMCwwLCJVJyJdLFsxLDAsIlUiXSxbMSwxLCJYLCJdLFswLDEsIlgnIl0sWzAsMV0sWzEsMl0sWzAsM10sWzMsMl1d
\[\begin{tikzcd}
	{U'} & U \\
	{X'} & {X,}
	\arrow[from=1-1, to=1-2]
	\arrow[from=1-2, to=2-2]
	\arrow[from=1-1, to=2-1]
	\arrow[from=2-1, to=2-2]
\end{tikzcd}\]
in which $U' \to X'$ is admissible. 
	\item	\label{as3}	For any commutative triangle
% https://q.uiver.app/?q=WzAsMyxbMCwwLCJYIl0sWzIsMCwiWSJdLFsxLDEsIloiXSxbMCwxLCJmIl0sWzEsMiwiZyJdLFswLDIsImgiLDJdXQ==
\[\begin{tikzcd}
	X && Y \\
	& Z
	\arrow["f", from=1-1, to=1-3]
	\arrow["g", from=1-3, to=2-2]
	\arrow["h"', from=1-1, to=2-2]
\end{tikzcd}\]
in which $g: Y \to Z$ is admissible, the morphism $f: X \to Y$ is admissible iff $h: X \to Z$ is. 
\item	\label{as4}	Admissible morphisms are closed under retracts. 
	\end{enumerate}
	\qede
\end{definition} 

\begin{example}
The admissible morphisms in a fractured $\infty$-topos form an admissibility structure. \qede
\end{example}

\begin{definition}[{\cite[Def.~20.6.2.1]{jL2017}}]
	A \Emph{geometric site} is a triple $(G,G^{\ad},\tau)$ consisting of 
	\begin{enumerate}[label = (\roman*)]
	\item	a small $\infty$-category $G$, 
	\item	an admissibility structure $G^{\ad}$ on $G$, and
	\item	a Grothendieck topology $\tau$ on $G$, 
	\end{enumerate} 
such that every covering sieve in $\tau$ contains a covering sieve generated by admissible morphisms. \phantom{m}\ \qede
\end{definition} 

\begin{remark}	\label{geometry} 
A geometric site $(G,G^{\ad},\tau)$ for which $G$ is finitely complete is called a \emph{geometry} in \cite{DAGV}. \qede
\end{remark} 

\begin{lemma}[{\cite[Props.~20.6.1.1~\&~20.6.1.3]{jL2017}}]	\label{adtop}
Let  $(G,G^{\ad},\tau)$ be a geometric site, then there exists a Grothendieck topology on $G^{\ad}$ in which a sieve $R$ in $G^{\ad}$ is a covering sieve iff the sieve generated by $R$ in $G$ is a covering sieve. Any sheaf on $G$ restricts to a sheaf on $G^{\ad}$.  \qed
\end{lemma} 

\begin{theorem}[{\cite[Th.~20.6.3.4]{jL2017}}]	\label{fractured sheaves}
Let $(G,G^{\ad},\tau)$ be a geometric site, and denote by $\mathcal{E}$ the $\infty$-topos of sheaves on $G$, and, by $\mathcal{E}^{\corp}$ the $\infty$-topos of sheaves on $G^{\ad}$, then the restriction functor $ \mathcal{E}^{\corp} \leftarrow \mathcal{E}: j^*$ admits a left adjoint, and the resulting adjunction is a fractured $\infty$-topos. \qed
\end{theorem}  

\begin{remark} 
In the same way that not every $\infty$-topos is the category of sheaves on a site, not every fractured $\infty$-topos is given as in the preceding theorem. However, it \emph{is} true that every fractured $\infty$-topos may be realised as the localisation of a fractured presheaf $\infty$-topos, and that this presheaf $\infty$-topos may be obtained as in the preceding theorem with $\tau = \varnothing$.  See \cite[Th.~20.5.3.4]{jL2017}. \phantom{m}\ \qede
\end{remark} 

\subsection{Equivalence with Lurie's definition of fractured $\infty$-toposes}	\label{equivalence} 

\begin{proposition} 
Definitions \ref{ft} and {\normalfont \cite[Def.~20.1.2.1]{jL2017}} are equivalent. 
\end{proposition} 

\begin{proof} 
Lurie defines a fractured $\infty$-topos to be an $\infty$-topos $\mathcal{E}$ together with a subcategory $\mathcal{E}^{\corp}$ (which by \cite[Prop.~20.1.3.3]{jL2017} is an $\infty$-topos) satisfying conditions (0) - (3), which we do not repeat here. \\\\
First we prove {(a) - (d) $\implies$ (1) - (2):} 
\par\noindent \underline{(d) $\implies$ (1):} Let 
% https://q.uiver.app/?q=WzAsNCxbMCwwLCJVJyJdLFsxLDAsIlUiXSxbMCwxLCJWJyJdLFsxLDEsIlYiXSxbMCwxXSxbMSwzXSxbMCwyXSxbMiwzXV0=
\[\begin{tikzcd}
	{U'} & U \\
	{V'} & V
	\arrow[from=1-1, to=1-2]
	\arrow[from=1-2, to=2-2]
	\arrow[from=1-1, to=2-1]
	\arrow[from=2-1, to=2-2]
\end{tikzcd}\]
be a pullback square in which $U \to V$ and $V' \to V$ are in $\mathcal{E}^{\corp}$, then (1) follows from applying (d) first to the above pullback square, and then to the pullback square obtained by switching $U \to V$ and $V' \to V$. 
\par\noindent \underline{(a) \& (c) $\implies$ (2):} Follows from Proposition \ref{conservative}.   \\\\
We now prove (1) - (3) $\implies$ (a) - (c): Axioms (a) and (b) follow from \cite[Cor.~20.1.3.4]{jL2017} and \cite[Prop.~20.1.3.1]{jL2017}, respectively, and axiom (c) is contained in axiom (2). \\\\
We conclude the proof by showing that (d) $\Longleftrightarrow$ (0) \& (3) under the assumption of (a) - (c). Recall that we refer to the image under $j_!$ of any corporeal object $U$ again by $U$. 
\par\noindent \underline{(d) $\implies$ (3):} Observe that by (b) the map $j^*U \to j^*V \times_V U$ is corporeal, so that for every corporeal object $W$ we obtain a commutative diagram: 
% https://q.uiver.app/?q=WzAsNyxbMCwwLCJcXG1hdGhjYWx7RX1ee1xcY29ycH0oVyxqXipVKSJdLFswLDEsIlxcbWF0aGNhbHtFfV57XFxjb3JwfShXLGpeKlYgXFx0aW1lc19WIFUpIl0sWzAsMiwiXFxtYXRoY2Fse0V9XntcXGNvcnB9KFcsal4qVikiXSxbMSwxLCJcXG1hdGhjYWx7RX0oVyxqXipWIFxcdGltZXNfViBVKSJdLFsyLDEsIlxcbWF0aGNhbHtFfShXLFUpIl0sWzEsMiwiXFxtYXRoY2Fse0V9KFcsal4qVikiXSxbMiwyLCJcXG1hdGhjYWx7RX0oVyxWKSJdLFswLDFdLFsxLDNdLFsxLDJdLFszLDVdLFsyLDVdLFszLDRdLFs1LDZdLFs0LDZdLFswLDRdXQ==
\[\begin{tikzcd}
	{\mathcal{E}^{\corp}(W,j^*U)} \\
	{\mathcal{E}^{\corp}(W,j^*V \times_V U)} & {\mathcal{E}(W,j^*V \times_V U)} & {\mathcal{E}(W,U)} \\
	{\mathcal{E}^{\corp}(W,j^*V)} & {\mathcal{E}(W,j^*V)} & {\mathcal{E}(W,V).}
	\arrow[from=1-1, to=2-1]
	\arrow[from=2-1, to=2-2]
	\arrow[from=2-1, to=3-1]
	\arrow[from=2-2, to=3-2]
	\arrow[from=3-1, to=3-2]
	\arrow[from=2-2, to=2-3]
	\arrow[from=3-2, to=3-3]
	\arrow[from=2-3, to=3-3]
	\arrow[from=1-1, to=2-3]
\end{tikzcd}\]
The rightmost square is a pullback by the definition of $U \times_Vj^*V$, and the leftmost square is a pullback square by (b). But $\mathcal{E}^{\corp}(W, j^*V) \to \mathcal{E}(W,V)$ is an isomorphism by the universal property of $j^*V$, and thus $\mathcal{E}^{\corp}(W, j^*V \times_V U) \to \mathcal{E}(W,U)$ is an isomorphism, so that $\mathcal{E}^{\corp}(W,j^*U) \to \mathcal{E}^{\corp}(W,j^*V \times_V U)$ is an isomorphism. 
\par\noindent \underline{(0) \& (3) $\implies$ (d):} The pullback square in (d) factors as 
% https://q.uiver.app/?q=WzAsNixbMCwwLCJVJyJdLFsxLDAsImpeKlUiXSxbMiwwLCJVIl0sWzIsMSwiViJdLFsxLDEsImpeKlYiXSxbMCwxLCJWJyJdLFswLDVdLFswLDFdLFs1LDRdLFsxLDRdLFsxLDJdLFs0LDNdLFsyLDNdXQ==
\[\begin{tikzcd}
	{U'} & {j^*U} & U \\
	{V'} & {j^*V} & V.
	\arrow[from=1-1, to=2-1]
	\arrow[from=1-1, to=1-2]
	\arrow[from=2-1, to=2-2]
	\arrow[from=1-2, to=2-2]
	\arrow[from=1-2, to=1-3]
	\arrow[from=2-2, to=2-3]
	\arrow[from=1-3, to=2-3]
\end{tikzcd}\]
The rightmost square is a pullback by (3), and the outer square is a pullback by assumption, so that the leftmost square is also a pullback. The morphism $U' \to V'$ is then in the image of $j_!$ by (0) \& (3). 
\end{proof} 

\section{Shapes and cofinality}	\label{Shapes and cofinality}

As explained in \S \ref{Introduction} the shape of any object in an $\infty$-topos provides a Galois theoretic notion of its underlying pro-homotopy type. In \S \ref{Basic definitions and properties} we give a definition of the shape of an $\infty$-topos, and give a cohomological criterion for when a geometric morphism induces an equivalence of shapes. Then, we discuss local shape equivalences -- geometric morphisms satisfying an analogous property to initial functors. In \S \ref{Locally contractible toposes and test toposes} we specialise to locally contractible $\infty$-toposes -- those $\infty$-toposes for which the shape of all objects are homotopy types. We then show how certain nerve diagrams in locally contractible $\infty$-toposes satisfying a cofinality condition may be used to calculate shapes. Finally, in \S \ref{Fractured toposes and local contractibility} we discuss how the structure of a fractured $\infty$-topos interacts with the property of being locally contractible. 

Throughout this section $\mathcal{E}, \mathcal{F}, \mathcal{X}, \mathcal{Y}$ denote $\infty$-toposes. 

\subsection{Basic definitions and properties}	\label{Basic definitions and properties} 

%We begin by giving a definition of the shape of an $\infty$-topos, before moving on to a discussion of the functoriality of the shape construction. 

Observe that by \cite[Prop.~3.1.6]{DAGXIII} and \cite[Prop.~5.4.7.7]{jL2009} the copresheaf $(\pi_\mathcal{X})_* \circ \pi_\mathcal{X}^*$ may be identified with an object in $\Pro(\mathcal{S})$, called the \Emph{shape} of $\mathcal{X}$, and is denoted by $\Pi_\infty(\mathcal{X})$. The $\infty$-topos is said to have \Emph{trivial shape} if $\Pi_\infty(\mathcal{X}) = \mathbf{1}$. Observe that $\mathcal{X}$ has trivial shape iff $\mathcal{X} \leftarrow \mathcal{S}: \pi^*$ is fully faithful. 

Any geometric morphism $f: \mathcal{X} \to \mathcal{Y}$ gives to a morphism of shapes $\Pi_\infty(\mathcal{X}) \to \Pi_\infty(\mathcal{Y})$ by composing 
\begin{equation}	\label{shape morphism} 
	(\pi_\mathcal{X})_* \circ \pi_\mathcal{X}^* = \mathcal{X}(\mathbf{1}_\mathcal{X}, \pi_\mathcal{X}^*(\emptyinput)) = \mathcal{X}(f^*\mathbf{1}_\mathcal{Y}, \pi_\mathcal{Y}^*\circ f^*(\emptyinput))	\leftarrow \mathcal{Y}(\mathbf{1}_\mathcal{Y}, \pi_\mathcal{Y}^*(\emptyinput)) = (\pi_\mathcal{Y})_* \circ \pi_\mathcal{Y}^*.
\end{equation}
It turns out to be surprisingly difficult to coherently extend $\Pi_\infty$ to a functor $\Top \to \Pro(\mathcal{S})$. As $\Top$ admits all filtered limits (see \cite[Th.~6.3.3.1]{jL2009}),  the functor $\Top \leftarrow \mathcal{S}, \; \mathcal{S}_{/A} \mapsfrom A$ extends to a functor $\Top \leftarrow \Pro(\mathcal{S})$. In the upcoming \cite{lM3000} the shape $\Pi_\infty$ will be exhibited as the left adjoint of $\Top \leftarrow \Pro(\mathcal{S})$, thus not only showing that $\Pi_\infty$ can be made functorial, but also exhibiting a universal property of $\Pi_\infty$, and moreover providing a version of the Seifert - Van Kampen theorem, as $\Pi_\infty$ preserves colimits. 

Fortunately, we will only require ``local functoriality'' of $\Pi_\infty$: The functor $\mathcal{E} \leftarrow \mathcal{S}: \pi^*$ extends to a functor $\mathcal{E} \leftarrow \Pro(\mathcal{S})$. Tracing through the proof of \cite[Prop.~6.3.9]{dcC2019} and again applying \cite[Prop.~3.1.6]{DAGXIII} and \cite[Prop.~5.4.7.7]{jL2009} one sees that this functor admits a left adjoint given by $X \mapsto \mathcal{E}(X, \pi^*(\emptyinput))$, which we denote by $(\pi_\mathcal{E}) _!: \mathcal{E} \to \Pro(\mathcal{S})$ (or $\pi_!$, when $\mathcal{E}$ is clear from context). Like $\Pi_\infty$, the functor $(\pi_\mathcal{E}) _!$ is also a left adjoint, yielding a second version of the Seifert-Van Kampen theorem. The two shape functors are compatible in that we recover $(\pi_\mathcal{E})_!$ from $\Pi_\infty$ as the composition of $\mathcal{E} \xrightarrow{E \mapsto \mathcal{E}_{/E}} \Top \xrightarrow{\Pi_\infty} \Pro(\mathcal{S})$.  

The shape of an $\infty$-topos is a powerful invariant, motivating the following definition: 

\begin{definition} 
A geometric morphism $f: \mathcal{X} \to \mathcal{Y}$ is called a \Emph{shape equivalence}  if $\Pi_\infty f$ is an isomorphism. \qede
\end{definition} 

\begin{example} 
A functor $A \to B$ between small $\infty$-categories induces an equivalence between homotopy types $A_\simeq \xrightarrow{\simeq} B_\simeq$ iff the induced geometric morphism $\opadjunction{[A^{\op}, \mathcal{S}]}	{[B^{\op},\mathcal{S}]}$ is a shape equivalence. \qede
\end{example} 

We have the following cohomological Whitehead theorem for \emph{hypercomplete} $\infty$-toposes: 

\begin{proposition} 
If $\mathcal{X}, \mathcal{Y}$ are hypercomplete, then a geometric morphism $f: \mathcal{X} \to \mathcal{Y}$ is a shape equivalence iff the induced morphism 
$$	H^i(\mathcal{X}; E)	\leftarrow 	H^i(\mathcal{Y}; E)$$
is an isomorphism for all $i \geq 0$ and all $E$, where $E$ is a set for $i = 0$, a group for $i = 1$, and an Abelian group for $i \geq 2$. 
\end{proposition} 

\begin{proof} 
For the if statement we want to prove that for any homotopy type $K$ the induced map $\mathcal{X}\big(\mathbf{1}_\mathcal{X}, (\pi_\mathcal{X})^*(K)\big) \leftarrow \mathcal{Y}\big(\mathbf{1}_\mathcal{Y}, (\pi_\mathcal{X})^*(K)\big)$ is an equivalence. First, we observe that it is enough to show this for the special case when $K$ is $n$-truncated for some $n \in \mathbf{N}$, because for general $K$ we then have 
$$	
{\arraycolsep=1.5pt\def\arraystretch{1.3}
	\begin{array}{rc >{\displaystyle}l}
	\mathcal{X}\big(\mathbf{1}_\mathcal{X}, (\pi_\mathcal{X})^*(K)\big)	&	=	&	\mathcal{X}\big(\mathbf{1}_\mathcal{X}, \lim_i \,  (\pi_\mathcal{X})^*(K)_{\leq i} \big)	\\
											{}	&	=	&	\mathcal{X}\big(\mathbf{1}_\mathcal{X}, \lim_i \,  (\pi_\mathcal{X})^*(K_{\leq i}) \big)	\\
											{}	&	=	&	\lim_i \mathcal{X}\big(\mathbf{1}_\mathcal{X}, (\pi_\mathcal{X})^*(K_{\leq i}) \big)		\\
											{}	&	=	&	\lim_i \mathcal{Y}\big(\mathbf{1}_\mathcal{X}, (\pi_\mathcal{Y})^*(K_{\leq i}) \big)		\\
											{}	&	=	&	\mathcal{Y}\big(\mathbf{1}_\mathcal{Y}, \lim_i \,  (\pi_\mathcal{Y})^*(K_{\leq i}) \big)	\\
											{}	&	=	&	\mathcal{Y}\big(\mathbf{1}_\mathcal{Y}, \lim_i \,  (\pi_\mathcal{Y})^*(K)_{\leq i} \big)	\\
											{}	&	=	&	\mathcal{Y}\big(\mathbf{1}_\mathcal{Y}, (\pi_\mathcal{Y})^*(K)\big),
	\end{array}
}
$$
where the first and last isomorphisms follow from the hypercompleteness assumption, and the second and penultimate isomorphisms follow from \cite[5.5.6.28]{jL2009}.

We prove the statement for $i$-truncated $K$ via induction on $i$:  The base case $i = 0$ holds by assumption. Let $i > 0$, and assume the statement holds for all $k$-truncated objects, for $0 \leq k < i$. Let $K$ be an $i$-truncated homotopy type, then we obtain the commutative square 
% https://q.uiver.app/#q=WzAsNCxbMCwwLCJcXG1hdGhjYWx7WH1cXGJpZygxX1xcbWF0aGNhbHtYfSwgKFxccGlfXFxtYXRoY2Fse1h9KV4qKEspXFxiaWcpIl0sWzEsMCwiXFxtYXRoY2Fse1l9XFxiaWcoMV9cXG1hdGhjYWx7WX0sIChcXHBpX1xcbWF0aGNhbHtZfSleKihLKVxcYmlnKSJdLFswLDEsIlxcbWF0aGNhbHtYfVxcYmlnKDFfXFxtYXRoY2Fse1h9LCAoXFxwaV9cXG1hdGhjYWx7WH0pXiooS197XFxsZXEge2kgLTF9fSlcXGJpZykiXSxbMSwxLCJcXG1hdGhjYWx7WX1cXGJpZygxX1xcbWF0aGNhbHtZfSwgKFxccGlfXFxtYXRoY2Fse1l9KV4qKEtfe1xcbGVxIHtpLTF9fSlcXGJpZykiXSxbMSwwXSxbMywyXSxbMSwzXSxbMCwyXV0=
\[\begin{tikzcd}
	{\mathcal{X}\big(1_\mathcal{X}, (\pi_\mathcal{X})^*(K)\big)} & {\mathcal{Y}\big(1_\mathcal{Y}, (\pi_\mathcal{Y})^*(K)\big)} \\
	{\mathcal{X}\big(1_\mathcal{X}, (\pi_\mathcal{X})^*(K_{\leq {i -1}})\big)} & {\mathcal{Y}\big(1_\mathcal{Y}, (\pi_\mathcal{Y})^*(K_{\leq {i-1}})\big)}
	\arrow[from=1-2, to=1-1]
	\arrow[from=2-2, to=2-1]
	\arrow[from=1-2, to=2-2]
	\arrow[from=1-1, to=2-1]
\end{tikzcd}\]
in which the bottom arrow is an isomorphism by the induction hypothesis. To show that the top horizontal morphism is an equivalence it is thus enough to show that for every fibre $L$ of $K \to K_{\leq i - 1}$ the map  $\mathcal{X}\big(\mathbf{1}_\mathcal{X}, (\pi_\mathcal{X})^*(L)\big)	\leftarrow \mathcal{Y}\big(\mathbf{1}_\mathcal{Y}, (\pi_\mathcal{Y})^*(L)\big)$ is an equivalence , as $1 = \mathcal{X}\big(\mathbf{1}_\mathcal{X}, (\pi_\mathcal{X})^*(1)\big) \xleftarrow{=} \mathcal{Y}\big(\mathbf{1}_\mathcal{Y}, (\pi_\mathcal{Y})^*(1)\big) = 1$, and both $\mathcal{X}\big(\mathbf{1}_\mathcal{X}, (\pi_\mathcal{X})^*(\emptyinput)\big)$ and $\mathcal{Y}\big(\mathbf{1}_\mathcal{Y}, (\pi_\mathcal{Y})^*(\emptyinput)\big)$ preserve finite limits. We check the equivalence on connected components and on loop spaces. For every point in $L$ we have 
$$
{\arraycolsep=1.5pt\def\arraystretch{1.3}
	\begin{array}{rc >{\displaystyle}l}
	\Omega \, \mathcal{X}\big(\mathbf{1}_\mathcal{X}, (\pi_\mathcal{X})^*(L)\big)	&	=	&	\mathcal{X}\big(\mathbf{1}_\mathcal{X}, (\pi_\mathcal{X})^*(\Omega L)\big)	\\
													{}	&	=	&	\mathcal{Y}\big(\mathbf{1}_\mathcal{Y}, (\pi_\mathcal{Y})^*(\Omega L)\big)	\\
													{}	&	=	&	\Omega \, \mathcal{Y}\big(\mathbf{1}_\mathcal{Y}, (\pi_\mathcal{X})^*(L)\big),
	\end{array}
}
$$
where the second isomorphism follows from the induction hypothesis. On connected components we have 
$$
{\arraycolsep=1.5pt\def\arraystretch{1.3}
	\begin{array}{rc >{\displaystyle}l}
	\pi_0 \, \mathcal{X}\big(\mathbf{1}_\mathcal{X}, (\pi_\mathcal{X})^*(L)\big)	&	=	&	H^i(\mathcal{X}, L)	\\
													{}	&	=	&	H^i(\mathcal{Y}, L)	\\
													{}	&	=	&	\pi_0 \, \mathcal{Y}\big(\mathbf{1}_\mathcal{Y}, (\pi_\mathcal{Y})^*(L)\big)
	\end{array}
}
$$
where the second isomorphism follows by assumption.

The only if statement is obvious. 
\end{proof} 

\begin{corollary}	\label{cc}
Let $\mathcal{X}$ be hypercomplete $\infty$-topos, then the shape fo $\mathcal{X}$ is contractible iff the canonical map 
$$	E \to H^0(\mathcal{X},E)	$$ 
is an equivalence for all sets $E$, and  
$$	H^i(\mathcal{X}, G) = 0	$$
for all $i$ and all $G$, where $G$ is a group for $i = 1$, and an Abelian group for all $i \geq 2$. \qed
\end{corollary} 

We now discuss how geometric morphisms satisfying extra conditions interact with shapes: 

\begin{definition} 
A geometric morphism $u: \mathcal{E} \to \mathcal{F}$ is called \Emph{essential} if $u^*$ admits an extra left adjoint, which we denote by $u_!$.  \qede
\end{definition} 

\begin{example} 
Any \'{e}tale geometric morphism in essential. \qede
\end{example}

\begin{proposition}	\label{essential shape} 
Let $u: \mathcal{E} \to \mathcal{F}$ be an essential geometric morphism, then $u_!$ preserves shapes. 
\end{proposition} 

\begin{proof} 
The functors $(\pi_\mathcal{F})_! \circ u_!$ and $(\pi_\mathcal{E})_!$ are both left adjoint to the extension of the functor  $\pi_\mathcal{E}^*$ to $\Pro(\mathcal{S}) \to \mathcal{E}$. 
\end{proof} 

\begin{example}
Let $u: A \to B$ be a functor between small $\infty$-categories, then the functor $u_!: [A^{\op}, \mathcal{S}] \to [B^{\op}, \mathcal{S}]$ preserves shapes. \qede
\end{example} 

We now turn to a notion of cofinality in the toposic context. Let $f: \mathcal{E} \to \mathcal{F}$ be a geometric morphism, then, by \cite[6.4.2]{dcC2019} the functor (\ref{shape morphism}) may be extended to a base change map 
\begin{equation}	\label{base change} 
% https://q.uiver.app/?q=WzAsMyxbMCwwLCJcXG1hdGhjYWx7RX0iXSxbMiwwLCJcXG1hdGhjYWx7Rn0iXSxbMSwxLCJcXFBybyhcXG1hdGhjYWx7U30pIl0sWzEsMCwiZl4qIiwyXSxbMCwyLCIoXFxwaV9cXG1hdGhjYWx7RX0pXyEiLDJdLFsxLDIsIihcXHBpX1xcbWF0aGNhbHtGfSlfISJdLFswLDUsIiIsMCx7InNob3J0ZW4iOnsic291cmNlIjozMCwidGFyZ2V0IjoyMH19XV0=
\begin{tikzcd}
	{\mathcal{E}} && {\mathcal{F}} \\
	& {\Pro(\mathcal{S})}
	\arrow["{f^*}"', from=1-3, to=1-1]
	\arrow["{(\pi_\mathcal{E})_!}"', from=1-1, to=2-2]
	\arrow[""{name=0, anchor=center, inner sep=0}, "{(\pi_\mathcal{F})_!}", from=1-3, to=2-2]
	\arrow[shorten <=15pt, shorten >=10pt, Rightarrow, from=1-1, to=0]
\end{tikzcd}
\end{equation} 
given by 
%\begin{equation} \label{shape map}
$$
	(\pi_\mathcal{E})_! f^*Y	=		\mathcal{E}(f^*Y, \big(\pi_\mathcal{E})^*(\emptyinput)\big)	
						=		\mathcal{E}(f^*Y, f^* \circ \big(\pi_\mathcal{F})^*(\emptyinput)\big)	
						\leftarrow	\mathcal{F}(Y, \big(\pi_\mathcal{F})^*(\emptyinput)\big)	
						=		(\pi_\mathcal{F})_! F
$$
%\end{equation}
or equivalently by 
$$	\Pi_\infty(f_{/Y}): \Pi_\infty(\mathcal{E}_{/f^*Y})	\to	\Pi_\infty(\mathcal{F}_{/Y}).	$$

\begin{definition}	\label{local shape equivalence}
The geometric morphism $f: \mathcal{E} \to \mathcal{F}$ is a \Emph{local shape equivalence} iff the base change map $(\pi_\mathcal{E})_! \circ f^* \Rightarrow (\pi_\mathcal{F})_!$ from (\ref{base change})  is an equivalence. \qede
\end{definition} 

\begin{example} 
A functor $A \to B$ between small $\infty$-categories is initial iff the induced geometric morphism $\opadjunction{[A^{\op}, \mathcal{S}]}	{[B^{\op},\mathcal{S}]}$ is a local shape equivalence. \qede
\end{example} 

We conclude this subsection with some useful properties of local shape equivalences: 

\begin{proposition}	\label{colim extend}
Let $f: \mathcal{E} \to \mathcal{F}$ be a geometric morphism. Assume that $\mathcal{F}$ is generated under small colimits by a subcategory $C$, and that the base change map $(\pi_\mathcal{E})_!(f^* F) \leftarrow (\pi_\mathcal{F})_! F$ is an isomorphism for every object $F$ in $C$, then $f$ is a local shape equivalence. \qed
\end{proposition} 

\begin{proposition}	\label{aspherical embedding} 
Let $a: \mathcal{E} \hookrightarrow \mathcal{F}$ be a geometric embedding which is also a local shape equivalence, then $(\pi_\mathcal{E})_! = (\pi_\mathcal{F})_! \circ a_*$. 
\end{proposition} 

\begin{proof}
By assumption $(\pi_\mathcal{E})_! \circ a^* = (\pi_\mathcal{F})_!$, so the corollary follows from precomposing with $a_*$. 
\end{proof}

\begin{proposition}	\label{ff shape} 
Any geometric morphism $f: \mathcal{E} \to \mathcal{F}$ such that $f^*$ is fully faithful is a local shape equivalence. 
\end{proposition} 

\begin{proof} 
For every $Y$ in $\mathcal{F}$ we have $\mathcal{E}\left(f^*Y, \pi_\mathcal{E}^*(\emptyinput)\right) = \mathcal{E}\left(f^*Y, f^*\pi_\mathcal{F}^*(\emptyinput)\right) = \mathcal{F}\left(Y, \pi_\mathcal{E}^* (\emptyinput)\right)$.	
\end{proof} 

\begin{corollary} 
If $\mathcal{E}$ has trivial shape, then $\pi_\mathcal{E}: \mathcal{E} \to \mathcal{S}$ is a local shape equivalence. \qed
\end{corollary} 

\begin{proposition} 
Any local $\infty$-topos (see \S \ref{bf}) has trivial shape.
\end{proposition} 

\begin{proof} 
The adjunction $\pi^! \vdash \pi_*$ is a geometric morphism, so that $\pi^!\pi_*$ is the direct image component of a geometric morphism $\mathcal{S} \to \mathcal{S}$ and thus equivalent to the identity. By \cite[Lm.~1.3]{pJiM1989} the counit of the induced adjunction $\cadjunction{\Ho( \pi_*): \Ho ( \mathcal{X})}	{\Ho (  \mathcal{S}): \Ho ( \pi^!)}$ is an isomorphism, and therefore the counit of $\pi^! \vdash \pi_*$ is an isomorphism. 
\end{proof} 

\subsection{Locally contractible toposes}	\label{Locally contractible toposes and test toposes}

We now specialise to a class of $\infty$-toposes, for which the theory of shapes is particularly nice. 

\begin{definition} 
An object in $\mathcal{E}$ is called \Emph{contractible} if its shape is trivial. \qede
\end{definition} 

\begin{proposition}[{\cite[Prop.~5.2.3]{lMsW2023}}]	\label{locally contractible} 
The following are equivalent:
	\begin{enumerate}[label = \normalfont{(\Roman*)}]
	\item	The shape functor $\pi_!: \mathcal{E} \to \Pro(\mathcal{S})$ factors through $\mathcal{S}$. 
	\item	The $\infty$-topos $\mathcal{E}$ is generated under colimits by its subcategory of contractible objects. 
	\end{enumerate}
\end{proposition} 

\begin{proof} 
The implication (II) $\implies$ (I) follows from the fact that the inclusion $\mathcal{S} \hookrightarrow \Pro(\mathcal{S})$ commutes with colimits. To show (I) $\implies$ (II), let $E$ be an object or $\mathcal{E}$, then $\pi_!E$ is the colimit of the constant diagram $\mathbf{1}$ indexed by $\pi_!E$.  Thus, $\pi^* \pi_!E$ is the colimit of the constant diagram $\mathbf{1}$ indexed by $\pi_!E$ in $\mathcal{E}$, so that $E$ is the colimit of the diagram $E \times_{\pi^*\pi_!E}\mathbf{1}$ indexed by $\pi_!E$, and $E \times_{\pi^*\pi_!E}1$ has contractible shape by \cite[Prop.~A.1.9]{jL2012}. 
\end{proof} 

\begin{definition} 
An $\infty$-topos is called \Emph{locally contractible} if it satisfies the equivalent conditions of Proposition \ref{locally contractible}. 	\qede
\end{definition} 

\begin{remark} 
Locally contractible $\infty$-toposes are called \emph{locally of constant shape} in \cite{jL2012}, and \emph{locally $\infty$-connected} in \cite{mH2018}.	\qede
\end{remark} 

\begin{proposition}[{\cite[Prop.~A.1.11]{jL2012}}]	\label{topos preloc} 
Assume that $\mathcal{E}$ is locally contractible, then the functor $(\pi_!)_{/\mathbf{1}_\mathcal{E}}: \mathcal{E}_{/\mathbf{1}_\mathcal{E}} = \mathcal{E} \to \mathcal{S}_{/\pi_!\mathbf{1}_\mathcal{E}}$ admits a fully faithful right adjoint. \qed
\end{proposition} 

\begin{remark}	\label{loc top cov} 
In Proposition \ref{topos preloc} the image of the left adjoint of $(\pi_!)_{/\mathbf{1}_\mathcal{E}}: \mathcal{E}_{/\mathbf{1}_\mathcal{E}} = \mathcal{E} \to \mathcal{S}_{/\pi_!\mathbf{1}_\mathcal{E}}$ is given by the subcategory of $\mathcal{E}$ spanned by the covering spaces of $\mathbf{1}_\mathcal{E}$. 
\end{remark} 

By \cite[Prop.~7.11.2]{dcC2019} we obtain the following corollary (with notation as in Proposition \ref{topos loc}): 

\begin{corollary}	\label{topos loc}
The functor $(\pi_!)_{/\mathbf{1}_\mathcal{E}}: \mathcal{E}_{/\mathbf{1}_\mathcal{E}} = \mathcal{E} \to \mathcal{S}_{/\pi_!\mathbf{1}_\mathcal{E}}$ exhibits $\mathcal{S}_{/ \pi_! \mathbf{1}_\mathcal{E}}$ as the localisation of $\mathcal{E}$ along its weak equivalences. \qed
\end{corollary} 

\begin{example}	\label{plc}
Let $A$ be a small $\infty$-category, then the constant presheaf functor $[A^{\op}, \mathcal{S}] \leftarrow \mathcal{S}$ admits both a left and a right adjoint, given by the colimit and limit functors, respectively, so that $[A^{\op}, \mathcal{S}]$ is a locally contractible $\infty$-topos. We have $\colim \mathbf{1} = A_\simeq$, and the image of the fully faithful right adjoint to the functor $\colim_{/\mathbf{1}}: [A^{\op}, \mathcal{S}] \to \mathcal{S}_{/A_\simeq}$ is spanned by those presheaves on $A$ carrying all morphisms in $A$ to isomorphisms.   \qede
\end{example} 

Informally, this right adjoint functor is given by sending any map $A \to \pi_!\mathbf{1}$ in $\mathcal{S}$ to the map $\mathbf{1} \times_{\pi^* \pi_! \mathbf{1}} \pi^*A \to \mathbf{1}$. For locally contractible $\infty$-toposes, this makes precise the idea explained in \S \ref{Overview} that $\pi_! \mathbf{1}$ is characterised by universally controlling the theory of covering spaces on $\mathbf{1}$. For similar statements for non-locally contractible $\infty$-toposes, see \cite{mH2018}. 

Before moving on to nerves, we briefly discuss equivariant homotopy theory in locally contractible $\infty$-toposes. Assume that $\mathcal{E}$ is locally contractible,  then for any group object $G$ in $\mathcal{E}$ the $\infty$-category $\mathcal{E}_G$ is an $\infty$-topos, as it is equivalent to $\mathcal{E}_{/BG}$. If, moreover, the shape functor $(\pi_\mathcal{E})_!: \mathcal{E} \to \mathcal{S}$ preserves finite products, then we see from the left square of 
% https://q.uiver.app/#q=WzAsNixbMCwwLCJcXG1hdGhjYWx7RX1fRyJdLFsxLDAsIlxcbWF0aGNhbHtFfV97L0JHfSJdLFsyLDAsIlxcbWF0aGNhbHtFfSJdLFswLDEsIlxcbWF0aGNhbHtTfV97XFxwaV8hR30iXSxbMSwxLCJcXG1hdGhjYWx7U31fey9CXFxwaV8hR30iXSxbMiwxLCJcXG1hdGhjYWx7U30iXSxbMCwzXSxbMCwxLCJcXHNpbWVxIl0sWzEsMl0sWzMsNCwiXFxzaW1lcSJdLFsxLDRdLFs0LDVdLFsyLDVdXQ==
\[\begin{tikzcd}
	{\mathcal{E}_G} & {\mathcal{E}_{/BG}} & {\mathcal{E}} \\
	{\mathcal{S}_{\pi_!G}} & {\mathcal{S}_{/B\pi_!G}} & {\mathcal{S}}
	\arrow[from=1-1, to=2-1]
	\arrow["\simeq", from=1-1, to=1-2]
	\arrow[from=1-2, to=1-3]
	\arrow["\simeq", from=2-1, to=2-2]
	\arrow[from=1-2, to=2-2]
	\arrow[from=2-2, to=2-3]
	\arrow[from=1-3, to=2-3]
\end{tikzcd}\]
that $((\pi_{\mathcal{E}_G})_!)_{/ \mathbf{1}_{\mathcal{E}_G}}$ is given by the functor $\mathcal{E}_G \to \mathcal{S}_{\pi_!G}$ taking any object $X$ to $(\pi_\mathcal{E})_!X$ with its induced $(\pi_\mathcal{E})_!G$-action. By composing the two horizontal morphisms on the top we see that the quotient functor $G / \emptyinput: \mathcal{E}_G \to \mathcal{E}$ is the extra left adjoint of an \'{e}tale geometric morphism, so that Proposition \ref{essential shape} yields the following result: 

\begin{proposition} \label{locally contractible quotients}
For any object $X$ in $\mathcal{E}_G$ the comparison morphism $\pi_!X / \pi_!G \to \pi_!(X / G)$ is an isomorphism in $\mathcal{S}$. \qed
\end{proposition} 

\subsubsection{Nerves}	\label{snerves} 

We now discuss the main tool for calculating shapes in this article. Until the end of \S \ref{snerves}, $\mathcal{E}$ denotes a locally contractible $\infty$-topos. Moreover, let $A$ be a small $\infty$-category together with a functor $u: A \to \mathcal{E}$. 

\begin{proposition}	\label{left nerve} 
If the image of $u: A \to \mathcal{E}$ is spanned by contractible objects then the functors $\colim: [A^{\op}, \mathcal{S}] \to \mathcal{S}$ and $(\pi_\mathcal{E})_! \circ u_!$ are canonically equivalent. 
\end{proposition} 

\begin{proof} 
By assumption the composition of $A \xrightarrow{u} \mathcal{E} \xrightarrow{\pi_!} \mathcal{S}$ is equivalent to the constant functor $a \mapsto \mathbf{1}$, so the equivalence is obtained by extending by colimits. 
\end{proof} 

In the following two statements $C \subseteq \mathcal{E}$ denotes a small subcategory spanned by contractible objects and generating $\mathcal{E}$ under colimits. 

\begin{proposition}	\label{prenerves} 
The shape functor $(\pi_{\mathcal{E}})_!$ is canonically equivalent to $\colim_{c \in C}\mathcal{E}(c,\emptyinput)$. 
\end{proposition} 

\begin{proof} 
We observe that $(\pi_{[C^{\op}, \mathcal{S}]})_! c = \colim C(\emptyinput, c) = \mathbf{1}$ for every object $c$ in $C$, and apply first Proposition \ref{colim extend} and then Proposition \ref{aspherical embedding} to the geometric morphism $\mathcal{E} \to [C^{\op}, \mathcal{S}]$. 
\end{proof} 

\begin{theorem}	\label{nerves}
Assume that $u: A \to \mathcal{E}$ factors (uniquely) through $C \hookrightarrow \mathcal{E}$, and that the functor $u: A \to C$ is initial, then the natural transformation 
$$	\colim \circ u^* \to (\pi_{\mathcal{E}})_! $$
is an equivalence. 	

Moreover, both $u_!$ and $u^*$ preserve weak equivalences, and induce and adjoint equivalence as indicated in the following diagram:  
% https://q.uiver.app/#q=WzAsNCxbMCwwLCJbQV57XFxvcH0sIFxcbWF0aGNhbHtTfV0iXSxbMCwxLCJcXG1hdGhjYWx7U31fey9BX1xcc2ltZXF9Il0sWzEsMCwiXFxtYXRoY2Fse0V9Il0sWzEsMSwiXFxtYXRoY2Fse1N9X3svXFxwaV8hIFxcbWF0aGJmezF9fSJdLFswLDFdLFsyLDNdLFswLDIsInVfISIsMCx7Im9mZnNldCI6LTJ9XSxbMiwwLCJ1XioiLDAseyJvZmZzZXQiOi0yfV0sWzEsMywiXFxzaW1lcSIsMCx7Im9mZnNldCI6LTJ9XSxbMywxLCJcXHNpbWVxIiwwLHsib2Zmc2V0IjotMn1dLFs2LDcsIiIsMCx7ImxldmVsIjoxLCJzdHlsZSI6eyJuYW1lIjoiYWRqdW5jdGlvbiJ9fV0sWzgsOSwiIiwwLHsibGV2ZWwiOjEsInN0eWxlIjp7Im5hbWUiOiJhZGp1bmN0aW9uIn19XV0=
\[\begin{tikzcd}
	{[A^{\op}, \mathcal{S}]} & {\mathcal{E}} \\
	{\mathcal{S}_{/A_\simeq}} & {\mathcal{S}_{/\pi_! \mathbf{1}}}
	\arrow[from=1-1, to=2-1]
	\arrow[from=1-2, to=2-2]
	\arrow[""{name=0, anchor=center, inner sep=0}, "{u_!}", shift left=2, from=1-1, to=1-2]
	\arrow[""{name=1, anchor=center, inner sep=0}, "{u^*}", shift left=2, from=1-2, to=1-1]
	\arrow[""{name=2, anchor=center, inner sep=0}, "\simeq", shift left=2, from=2-1, to=2-2]
	\arrow[""{name=3, anchor=center, inner sep=0}, "\simeq", shift left=2, from=2-2, to=2-1]
	\arrow["\dashv"{anchor=center, rotate=-90}, draw=none, from=0, to=1]
	\arrow["\dashv"{anchor=center, rotate=-90}, draw=none, from=2, to=3]
\end{tikzcd}\] 
\end{theorem} 

\begin{proof} 
The two diagrams 
% https://q.uiver.app/#q=WzAsNixbMCwwLCJbQV57XFxvcH0sIFxcbWF0aGNhbHtTfV0iXSxbMiwwLCJcXG1hdGhjYWx7RX0iXSxbMSwxLCJcXG1hdGhjYWx7U30iXSxbMywwLCJbQV57XFxvcH0sIFxcbWF0aGNhbHtTfV0iXSxbNSwwLCJcXG1hdGhjYWx7RX0iXSxbNCwxLCJcXG1hdGhjYWx7U30iXSxbMCwxLCJ1XyEiXSxbMCwyLCJcXGNvbGltIiwyXSxbMSwyLCJcXHBpXyEiXSxbNCwzLCJ1XioiLDJdLFs0LDUsIlxccGlfISJdLFszLDUsIlxcY29saW0iLDJdXQ==
\[\begin{tikzcd}
	{[A^{\op}, \mathcal{S}]} && {\mathcal{E}} & {[A^{\op}, \mathcal{S}]} && {\mathcal{E}} \\
	& {\mathcal{S}} &&& {\mathcal{S}}
	\arrow["{u_!}", from=1-1, to=1-3]
	\arrow["\colim"', from=1-1, to=2-2]
	\arrow["{\pi_!}", from=1-3, to=2-2]
	\arrow["{u^*}"', from=1-6, to=1-4]
	\arrow["{\pi_!}", from=1-6, to=2-5]
	\arrow["\colim"', from=1-4, to=2-5]
\end{tikzcd}\]
commute, the first one by Proposition \ref{left nerve}, and the second one by the following calculation (obtained using Proposition \ref{prenerves}): $(\pi_\mathcal{E})_! X = \colim_{c \in C} \mathcal{E}(c,X) = \colim_{a \in A} \mathcal{E}(ua, X)$. Thus, both $u_!$ and $u^*$ preserve weak equivalences, inducing the indicated adjoint equivalence by Corollary \ref{topos loc} and \cite[Prop.~7.1.14]{dcC2019}. 
\end{proof} 
 
We will often refer to any functor $u: A \to \mathcal{E}$ to which we intend to apply Theorem \ref{nerves} as a \Emph{nerve diagram}, and the functor $[A^{\op}, \mathcal{S}] \leftarrow \mathcal{E}: u^*$ as a \Emph{nerve}. In the examples considered this article, the functor $\alpha: A \to C$ will usually induce a bijection on objects.

\begin{remark}	\label{nc}
Let $u,v: A \to \mathcal{E}$ be two nerve diagrams satisfying the conditions of Theorem \ref{nerves} (the small subcategories $C$ are not assumed to be the same for $u$ and $v$). Any natural transformation $u \to v$ induces a natural transformation $u^* \leftarrow v^*$, and by the universal property of localisations we obtain a diagram
\begin{equation}	\label{ncd}
% https://q.uiver.app/#q=WzAsNCxbMCwwLCJbQV57XFxvcH0sIFxcbWF0aGNhbHtTfV0iXSxbMSwwLCJcXG1hdGhjYWx7RX0iXSxbMSwxLCJcXG1hdGhjYWx7U31fey9cXHBpXyFcXG1hdGhiZnsxfX0iXSxbMCwxLCJcXG1hdGhjYWx7U31fey9BX1xcc2ltZXF9Il0sWzEsMCwidV4qIiwxLHsiY3VydmUiOjN9XSxbMSwwLCJ2XioiLDEseyJjdXJ2ZSI6LTN9XSxbMSwyXSxbMCwzXSxbMiwzLCIiLDEseyJjdXJ2ZSI6M31dLFsyLDMsIiIsMCx7ImN1cnZlIjotM31dLFs1LDQsIiIsMCx7InNob3J0ZW4iOnsic291cmNlIjoyMCwidGFyZ2V0IjoyMH19XSxbOSw4LCIiLDEseyJzaG9ydGVuIjp7InNvdXJjZSI6MjAsInRhcmdldCI6MjB9fV1d
\begin{tikzcd}[row sep=2.25em]
	{[A^{\op}, \mathcal{S}]} & {\mathcal{E}} \\
	{\mathcal{S}_{/A_\simeq}} & {\mathcal{S}_{/\pi_!\mathbf{1}}.}
	\arrow[""{name=0, anchor=center, inner sep=0}, "{u^*}"{description}, curve={height=18pt}, from=1-2, to=1-1]
	\arrow[""{name=1, anchor=center, inner sep=0}, "{v^*}"{description}, curve={height=-18pt}, from=1-2, to=1-1]
	\arrow[from=1-2, to=2-2]
	\arrow[from=1-1, to=2-1]
	\arrow[""{name=2, anchor=center, inner sep=0}, curve={height=18pt}, from=2-2, to=2-1]
	\arrow[""{name=3, anchor=center, inner sep=0}, curve={height=-18pt}, from=2-2, to=2-1]
	\arrow[shorten <=5pt, shorten >=5pt, Rightarrow, from=1, to=0]
	\arrow[shorten <=5pt, shorten >=5pt, Rightarrow, from=3, to=2]
\end{tikzcd}
\end{equation} 
As the two functors $S_{/A_\simeq} \leftarrow \mathcal{S}_{\pi_! \mathbf{1}}$ are equivalences they restrict to equivalences $A_\simeq \leftarrow \pi_!\mathbf{1}$. (This follows e.g.\ from \cite[Th.~2.39]{dAjF2020}, or the fact that $\mathcal{S}_{/\emptyinput}: \mathcal{S} \to \Top$ is fully faithful.) The functor from morphisms $A_\simeq \leftarrow \pi_!\mathbf{1}$ to colimit preserving functors $[A_\simeq^{\op}, \mathcal{S}] \leftarrow [(\pi_!\mathbf{1})^{\op}, \mathcal{S}]$ is fully faithful, and thus the lower natural transformation in (\ref{ncd}) must be a natural isomorphism.   \qede
\end{remark} 

We will repeatedly use Propositions \ref{homotopy initial} \& \ref{monoidal homotopy initial} below to verify the conditions of the above proposition. 

\begin{definition} \label{interval}
Let $A$ be an ordinary category, admitting a final object $\mathbf{1}$, then an object $I$ in $A$ with two morphisms $\mathbf{1} \rightrightarrows I$ is called an \Emph{interval} in $A$. If $A$ admits an initial object $\mathbf{0}$, and the square 
% https://q.uiver.app/#q=WzAsNCxbMCwwLCJcXG1hdGhiZnswfSJdLFsxLDAsIlxcbWF0aGJmezF9Il0sWzEsMSwiSSJdLFswLDEsIlxcbWF0aGJmezF9Il0sWzAsMV0sWzEsMl0sWzAsM10sWzMsMl1d
\[\begin{tikzcd}
	{\mathbf{0}} & {\mathbf{1}} \\
	{\mathbf{1}} & I
	\arrow[from=1-1, to=1-2]
	\arrow[from=1-2, to=2-2]
	\arrow[from=1-1, to=2-1]
	\arrow[from=2-1, to=2-2]
\end{tikzcd}\]
is a pullback, then $I$ is \Emph{separating interval}. \qede
\end{definition} 

\begin{example} 
Let $\mathcal{E}$ be an ordinary topos, then the subobject classifier $\Omega$ in $\mathcal{E}$ canonically has the structure of a separating interval. The first morphism $\mathbf{1} \to \Omega$ is given by the universal monomorphism, and the second morphism $\mathbf{1} \to \Omega$ classifies the subobject $\mathbf{0} \to \mathbf{1}$. \qede
\end{example} 

\begin{definition} 
Let $A$ be an ordinary category equipped with an interval $I$, then an \Emph{$I$-homotopy} between two maps $f,g: a \rightrightarrows a'$ is a commutative diagram
% https://q.uiver.app/#q=WzAsNCxbMSwxLCJJIFxcdGltZXMgYSJdLFsyLDEsImEnIl0sWzAsMCwiXFxtYXRoYmZ7MX0gXFx0aW1lcyBhIl0sWzAsMiwiXFxtYXRoYmZ7MX0gXFx0aW1lcyBhIl0sWzAsMV0sWzIsMF0sWzMsMF0sWzIsMSwiZiIsMCx7ImN1cnZlIjotMX1dLFszLDEsImciLDIseyJjdXJ2ZSI6MX1dXQ==
\[\begin{tikzcd}
	{\mathbf{1} \times a} \\
	& {I \times a} & {a'} \\
	{\mathbf{1} \times a}
	\arrow[from=2-2, to=2-3]
	\arrow[from=1-1, to=2-2]
	\arrow[from=3-1, to=2-2]
	\arrow["f", curve={height=-6pt}, from=1-1, to=2-3]
	\arrow["g"', curve={height=6pt}, from=3-1, to=2-3]
\end{tikzcd}\]
and $f$ and $g$ are called \Emph{$I$-homotopic} if there exists an $I$-homotopy between $f$ and $g$. A map $f: a \to a'$ is an \Emph{$I$-homotopy equivalence} if there exists a map $a \leftarrow a':g$, such that $gf$ and $fg$ are $I$-homotopic to $\id_a$ and $\id_{a'}$, respectively. An object $a$ in $A$ is \Emph{$I$-contractible}, if the unique morphism $a \to \mathbf{1}$ is an $I$-homotopy equivalence.  \qede
\end{definition} 

\begin{proposition}	\label{homotopy initial}
Let $(A, I)$ and $(B,J)$ be pairs consisting of small ordinary categories together with an interval, and let $u: A \to B$ be a functor carrying $I$ to $J$ (including the inclusions of the final object, which $u$ must then preserve). Assume that 
\begin{enumerate}[label = \normalfont{(\alph*)}]
\item	\label{pres products}	$\pi_!: [A^{\op}, \mathcal{S}] \to \mathcal{S}$ preserves finite products, and that
\item					every object in $B$ is $J$-contractible
\end{enumerate} 
then $u$ is initial. 
\end{proposition} 

\begin{proof}
The functor $u$ is initial iff for every object $b$ in $B$ the shape of $u^*b$ is contractible (see \cite[Cor.~4.4.31]{dcC2019}). Let $J \times b \to b$ be an $J$-contraction of $b$, then the unit morphisms produce a diagram
% https://q.uiver.app/#q=WzAsNyxbNCwxLCJ1XipiLCJdLFsyLDEsInVeKnVfIUkgXFx0aW1lcyB1XipiIFxcY29uZyAgdV4qKEogXFx0aW1lcyBiKSJdLFsyLDAsInVeKnVfIVxcbWF0aGJmezF9X0EgXFx0aW1lcyB1XipiIFxcY29uZyB1XiooXFxtYXRoYmZ7MX1fQSBcXHRpbWVzIGIpIl0sWzIsMiwidV4qdV8hXFxtYXRoYmZ7MX1fQSBcXHRpbWVzIHVeKmIgXFxjb25nIHVeKihcXG1hdGhiZnsxfV9BIFxcdGltZXMgYikiXSxbMCwxLCJJIFxcdGltZXMgdV4qYiJdLFswLDAsInVeKmIgXFxjb25nIFxcbWF0aGJmezF9X0EgXFx0aW1lcyB1XipiIl0sWzAsMiwidV4qYiBcXGNvbmcgXFxtYXRoYmZ7MX1fQSBcXHRpbWVzIHVeKmIiXSxbMSwwXSxbMiwxXSxbMiwwLCJcXGlkIiwwLHsiY3VydmUiOi0yfV0sWzMsMV0sWzMsMCwiMCIsMix7ImN1cnZlIjoyfV0sWzQsMV0sWzUsNF0sWzYsNF0sWzUsMl0sWzYsM11d
\[\begin{tikzcd}
	{u^*b \cong \mathbf{1}_A \times u^*b} && {u^*u_!\mathbf{1}_A \times u^*b \cong u^*(\mathbf{1}_A \times b)} \\
	{I \times u^*b} && {u^*u_!I \times u^*b \cong  u^*(J \times b)} && {u^*b,} \\
	{u^*b \cong \mathbf{1}_A \times u^*b} && {u^*u_!\mathbf{1}_A \times u^*b \cong u^*(\mathbf{1}_A \times b)}
	\arrow[from=2-3, to=2-5]
	\arrow[from=1-3, to=2-3]
	\arrow["\id", curve={height=-12pt}, from=1-3, to=2-5]
	\arrow[from=3-3, to=2-3]
	\arrow["0"', curve={height=12pt}, from=3-3, to=2-5]
	\arrow[from=2-1, to=2-3]
	\arrow[from=1-1, to=2-1]
	\arrow[from=3-1, to=2-1]
	\arrow[from=1-1, to=1-3]
	\arrow[from=3-1, to=3-3]
\end{tikzcd}\]
showing that $u^*b$ is $I$-contractible by \ref{pres products}. 
\end{proof} 

\begin{proposition}	\label{monoidal homotopy initial}
Let $(B,J)$ be a pair consisting of a small ordinary category together with an interval. Let $u: \Cube \to B$ be a functor carrying the interval $\Cube^{\; \! 1}$ to $J$ (including the inclusions of the final object, which $u$ must then preserve). If every object in $B$ is $J$-contractible then $u$ is initial. 
\end{proposition} 

\begin{proof}
The functor $u$ is initial iff for every object $b$ in $B$ the shape of $u^*b$ is contractible (see \cite[Cor.~4.4.31]{dcC2019}). We will require the following claim, which we prove below. 

\noindent\underline{Claim:}	There exists a natural morphism $X_1 \otimes X_2 \to X_1 \times X_2$. \par
Let $J \times b \to b$ be an $J$-contraction of $b$, then the unit morphisms produce a diagram
% https://q.uiver.app/#q=WzAsMTAsWzYsMSwidV4qYiwiXSxbNCwxLCJ1Xip1XyFcXEN1YmVeMSBcXHRpbWVzIHVeKmIgXFxjb25nICB1XiooSiBcXHRpbWVzIGIpIl0sWzQsMCwidV4qdV8hXFxtYXRoYmZ7MX1fe1xcc2NyaXB0c2l6ZSBcXEN1YmV9IFxcdGltZXMgdV4qYiBcXGNvbmcgdV4qKFxcbWF0aGJmezF9X3tcXHNjcmlwdHNpemUgXFxDdWJlfSBcXHRpbWVzIGIpIl0sWzQsMiwidV4qdV8hXFxtYXRoYmZ7MX1fe1xcc2NyaXB0c2l6ZSBcXEN1YmV9IFxcdGltZXMgdV4qYiBcXGNvbmcgdV4qKFxcbWF0aGJmezF9X3tcXHNjcmlwdHNpemUgXFxDdWJlfSBcXHRpbWVzIGIpIl0sWzIsMSwiXFxDdWJlXjEgXFx0aW1lcyB1XipiIl0sWzIsMCwiXFxtYXRoYmZ7MX1fe1xcc2NyaXB0c2l6ZSBcXEN1YmV9IFxcdGltZXMgdV4qYiJdLFsyLDIsIiBcXG1hdGhiZnsxfV97XFxzY3JpcHRzaXplIFxcQ3ViZX0gXFx0aW1lcyB1XipiIl0sWzAsMCwidV4qYiBcXGNvbmcgXFxtYXRoYmZ7MX1fe1xcc2NyaXB0c2l6ZSBcXEN1YmV9IFxcb3RpbWVzIHVeKmIiXSxbMCwyLCJ1XipiIFxcY29uZyBcXG1hdGhiZnsxfV97XFxzY3JpcHRzaXplIFxcQ3ViZX0gXFxvdGltZXMgdV4qYiJdLFswLDEsIlxcQ3ViZV4xIFxcb3RpbWVzIHVeKmIiXSxbMSwwXSxbMiwxXSxbMiwwLCJcXGlkIl0sWzMsMV0sWzMsMCwiMCIsMl0sWzQsMV0sWzUsNF0sWzYsNF0sWzUsMl0sWzYsM10sWzcsNV0sWzgsNl0sWzksNF0sWzcsOV0sWzgsOV1d
\[\begin{tikzcd}
	{u^*b \cong \mathbf{1}_{\scriptsize \Cube} \otimes u^*b} && {\mathbf{1}_{\scriptsize \Cube} \times u^*b} && {u^*u_!\mathbf{1}_{\scriptsize \Cube} \times u^*b \cong u^*(\mathbf{1}_{\scriptsize \Cube} \times b)} \\
	{\Cube^1 \otimes u^*b} && {\Cube^1 \times u^*b} && {u^*u_!\Cube^1 \times u^*b \cong  u^*(J \times b)} && {u^*b,} \\
	{u^*b \cong \mathbf{1}_{\scriptsize \Cube} \otimes u^*b} && { \mathbf{1}_{\scriptsize \Cube} \times u^*b} && {u^*u_!\mathbf{1}_{\scriptsize \Cube} \times u^*b \cong u^*(\mathbf{1}_{\scriptsize \Cube} \times b)}
	\arrow[from=2-5, to=2-7]
	\arrow[from=1-5, to=2-5]
	\arrow["\id", from=1-5, to=2-7]
	\arrow[from=3-5, to=2-5]
	\arrow["0"', from=3-5, to=2-7]
	\arrow[from=2-3, to=2-5]
	\arrow[from=1-3, to=2-3]
	\arrow[from=3-3, to=2-3]
	\arrow[from=1-3, to=1-5]
	\arrow[from=3-3, to=3-5]
	\arrow[from=1-1, to=1-3]
	\arrow[from=3-1, to=3-3]
	\arrow[from=2-1, to=2-3]
	\arrow[from=1-1, to=2-1]
	\arrow[from=3-1, to=2-1]
\end{tikzcd}\]
showing that $u^*b$ is $\Cube^{\; \! 1}$-contractible because $\pi_! (X_1 \otimes X_2) \simeq \pi_! X_1 \times \pi_! X_2$ for all cubical sets $X_1, X_2$ (see \cite[Cor.~8.4.32]{dcC2006}).    \\\\
\underline{Proof of claim:}	We note that for any two cubical sets $X_1, X_2$ there are canonical morphisms $X_1 \otimes X_2 \to X_i$ ($i = 1,2$). To see this, note that for any $k_1, {k_2} \in \mathbf{N}$ we have projection maps (in $\Set$) $\Cube^{\;\! k_1} \otimes \Cube^{\;\! k_2} \cong \{0,1\}^{k_1} \times \{0,1\}^{k_2} \to \{0,1\}^{k_i}$ for $i = 1,2$;  the canonical morphisms $X_1 \otimes X_2 \to X_i$ ($i = 1,2$) are then obtained by extending by colimits, yielding the desired morphism.
\end{proof} 

\begin{proposition}	\label{contract interval} 
With notation as in Proposition \ref{monoidal homotopy initial}, if $B$ admits products, $u: \Cube \to B$ is monoidal, and every object in $B$ is a finite product of $J$, then $u$ is initial if $J$ is $J$-contractible. 
\end{proposition} 

\begin{proof} 
It is enough to show that if the objects $b$ and $b'$ are $J$ contractible, then so is $b \times b'$. So let $J \times b \to b$ and $J \times b' \to b'$ be contractions, then the composition of $J \times b \times b' \to J \times J  \times b \times b' \to b \times b'$ is a contraction of $b \times b'$, where the first morphism is induced by the diagonal morphism $J \to J \times J$. 
\end{proof} 

\subsection{Fractured $\infty$-toposes and shapes}	\label{Fractured toposes and local contractibility}

We now prove the result that will allow us to exhibit $\Diff^r$ as a locally contractible topos. This result may be viewed as a vast generalisation of the techniques underlying \cite[Lm.~6.1.5]{dcC2003}. Throughout this subsection $	\cadjunction{j_!: \mathcal{E}^{\corp}}	{\mathcal{E}: j^*}	$ denotes a fractured $\infty$-topos. 

\begin{theorem}	\label{corporeal shape} 
For any corporeal object $X$ the geometric morphism $	\cadjunction{j_!: \mathcal{E}_{/X}^{\corp}}	{\mathcal{E}_{/X}: j^*}	$ is a local geometric morphism. 
\end{theorem} 

\begin{proof} 
This is a consequence of property \ref{ft2} of in the definition of fractured $\infty$-toposes, and Proposition \ref{ff shape}. 
\end{proof} 

The following result could be viewed as a corollary of the above, but we note that it follows more immediately from Proposition \ref{essential shape}.  

\begin{theorem}	\label{jps} 
The functor $j_!: \mathcal{E}^{\corp} \to \mathcal{E}$ preserves shapes. \qed
\end{theorem} 

Thus, the cohomology of a geometric object such as a scheme with coefficients in a locally constant sheaf is the same when computed in its big or small topos. For us, Theorem \ref{corporeal shape} provides a way of showing that a topos is locally contractible, as seen in the following corollary. 

\begin{corollary} 
Let $C \subseteq \mathcal{E}^{\corp}$ be a small subcategory, spanned by contractible objects, and generating $\mathcal{E}^{\corp}$ under colimits, then $j_!C \subseteq \mathcal{E}$ is a small subcategory, spanned by contractible objects, generating $\mathcal{E}$ under colimits. \qed
\end{corollary} 

In other words, if $\mathcal{E}^{\corp}$ is locally contractible, then so is $\mathcal{E}$. 

\begin{remark} 
By \cite[Rmk.~20.3.2.6]{jL2017} the subcategory of $\mathcal{E}_{/X}$ spanned by admissible morphisms is an $\infty$-topos for \emph{any} object $X$ in $\mathcal{E}$, and it can be shown that the inclusion of this subcategory into $\mathcal{E}_{/X}$ induces an equivalence on shapes. \qede
\end{remark} 

\section{Homotopy theory in locally contractible ($\infty$-)toposes}	\label{Homotopy theory in locally contractible}

Fix a \Emph{relative $\infty$-category} $(C,W)$, i.e.\ an $\infty$-category $C$ together with a subcategory $W$ containing all isomorphisms. It is then natural to study the relationship between $C$ and its localisation $W^{-1}C$; in particular, one may ask which limits in $W^{-1}C$ may be obtained via constructions in $C$. 

\begin{definition}	\label{holim}
Let $K$ be a simplicial set, then a functor $p:K^{\lhd} \to C$ is called a \Emph{homotopy limit} of $p|_{K}: K \to C$ if the composition of $K^{\lhd} \to C \to W^{-1}C$ is a limit of the composition of $K \xrightarrow{p|_K} C \to W^{-1}C$. A functor $K^{\rhd} \to C$ is a \Emph{homotopy colimit} if $(K^{\rhd})^{\op} \to C^{\op}$ is a homotopy limit. \qede
\end{definition} 

In particular, a (co)limit in $C$ is a homotopy (co)limit iff it is carried to a (co)limit by $C \to W^{-1}C$. 

While at the level of generality of Definition \ref{holim} the theories of homotopy limits and colimits are dual to each other, in this article homotopy limits and colimits have very different flavours. This is because the localisation functors under consideration of are all of the form $C \to \mathcal{S}$ with $C$ some subcategory of a locally contractible $\infty$-topos $\mathcal{E}$, and the localisation functor is simply given by the restriction of $\pi_!$ to $C$. Thus, when $C = \mathcal{E}$ all colimits are homotopy colimits. When $C \subsetneq \mathcal{E}$ we can exhibit many colimits in $C$ as homotopy colimits by showing that they are preserved by the inclusion $C \hookrightarrow \mathcal{E}$. This approach is explored in \S \ref{truncated} \& \S \ref{Concrete objects} where $C$ consists of $n$-truncated objects and concrete objects (to which we also give a brief introduction), respectively.

Commuting limits past $(\pi_!)|_C$ is considerably harder, and requires different techniques. To this end we develop the basic theory of homotopical calculi (e.g., model structures) on $\infty$-categories in \S \ref{Basic theory of homotopical calculi}, and then use the machinery of \S \ref{snerves} combined with test categories to construct homotopical calculi on locally contractible $\infty$-toposes in \S \ref{Homotopical calculi in locally contractible toposes}. 

\subsection{Colimits of $n$-truncated objects in $\infty$-toposes}	\label{truncated}

Let $\mathcal{E}$ be a fixed $\infty$-topos, and $n \geq -2$. In this subsection we will show that many colimits of $n$-truncated objects in $\mathcal{E}$ are again $n$-truncated.

\begin{proposition}	\label{pushout}
Consider a pushout square in $\mathcal{E}$
% https://q.uiver.app/?q=WzAsNCxbMCwwLCJYIl0sWzEsMCwiWCciXSxbMCwxLCJZIl0sWzEsMSwiWSciXSxbMCwxLCIiLDAseyJzdHlsZSI6eyJ0YWlsIjp7Im5hbWUiOiJob29rIiwic2lkZSI6InRvcCJ9fX1dLFswLDJdLFsyLDMsIiIsMix7InN0eWxlIjp7InRhaWwiOnsibmFtZSI6Imhvb2siLCJzaWRlIjoidG9wIn19fV0sWzEsM11d
\[\begin{tikzcd}
	{X} & {X'} \\
	{Y} & {Y'}
	\arrow[from=1-1, to=1-2, hook]
	\arrow[from=1-1, to=2-1]
	\arrow[from=2-1, to=2-2, hook]
	\arrow[from=1-2, to=2-2]
\end{tikzcd}\]
for which $X,X',Y$ are $n$-truncated, and in which the top horizontal map (and thus also the bottom horizontal map; see \cite[Prop.~2.2.6]{mAgBeFaJ2020}) is a monomorphism, then $Y'$ is $n$-truncated. \qed
\end{proposition}

\begin{proposition}	\label{filtered}
The inclusion $\mathcal{E}_{\leq n} \hookrightarrow \mathcal{E}$ commutes with filtered colimits. \qed
\end{proposition}

\begin{proposition}	\label{coproducts}
The inclusion $\mathcal{E}_{\leq n} \hookrightarrow \mathcal{E}$ commutes with coproducts. \qed
\end{proposition}

\begin{proposition}	\label{retract} 
The subcategory of $n$-truncated objects is closed under retracts.  \qed
\end{proposition}

\begin{corollary}	\label{truncated colimits}
Let $A$ be a small category, and $X: A \to \mathcal{E}$ a functor $(n \in \mathbf{N})$. If 
\begin{enumerate}[label = {\normalfont \arabic*.}]
\item	$X$ is a wedge in which one leg is a monomorphism, 
\item	$A$ is filtered, or 
\item$A$ is discrete, 
\end{enumerate} 
then restricted shape functor $\pi_!|_{\mathcal{E}_{\leq n}} \to \Pro(\mathcal{S})$ preserves the colimit of $X$ . \qed
\end{corollary} 

\paragraph{Discussion of the proofs of Propositions \ref{pushout} - \ref{retract}}

All four propositions may be proved by first checking the statement for simplicial sets equipped with the Kan-Quillen model structure, so that they are true in $\mathcal{S}$. In any presheaf $\infty$-topos the statements can be checked pointwise. The general statements then follow from the fact that left exact functors preserve monomorphisms and truncation. \par
We believe that it would be conceptually pleasing to have proofs of these statements which rely on descent (similar to e.g.\ \cite[Prop.~2.2.6]{mAgBeFaJ2020}) rather than the fact that every $\infty$-topos is a left exact localisation of a presheaf $\infty$-category. A proof of this form of a generalisation of Proposition \ref{retract} was suggested to us by Bastiaan Cnossen. 

\begin{proposition} 
Let $C$ be a finitely complete $\infty$-category, then $n$-truncated maps in $C$ are closed under retracts. 
\end{proposition} 

\begin{proof}
Let 
% https://q.uiver.app/?q=WzAsNixbMCwwLCJ4JyJdLFswLDEsInknIl0sWzEsMCwieCJdLFsxLDEsInkiXSxbMiwwLCJ4JyJdLFsyLDEsInknIl0sWzAsMV0sWzIsM10sWzQsNV0sWzAsMl0sWzIsNF0sWzEsM10sWzMsNV1d
\[\begin{tikzcd}
	{x'} & x & {x'} \\
	{y'} & y & {y'}
	\arrow[from=1-1, to=2-1]
	\arrow[from=1-2, to=2-2]
	\arrow[from=1-3, to=2-3]
	\arrow[from=1-1, to=1-2]
	\arrow[from=1-2, to=1-3]
	\arrow[from=2-1, to=2-2]
	\arrow[from=2-2, to=2-3]
\end{tikzcd}\]
be a retract diagram in which $x \to y$ is $n$-truncated, then we wish to show that $x' \to y'$ is likewise $n$-truncated. For $n = -2$ the statement is clear, so assume that $n > -2$. Then we obtain a new retract diagram
% https://q.uiver.app/?q=WzAsNixbMCwwLCJ4JyJdLFswLDEsIngnIFxcdGltZXNfe3knfXgnIl0sWzEsMCwieCJdLFsxLDEsInggXFx0aW1lc195IHgiXSxbMiwwLCJ4JyJdLFsyLDEsIngnIFxcdGltZXNfe3knfXgnIl0sWzAsMV0sWzIsM10sWzQsNV0sWzAsMl0sWzIsNF0sWzEsM10sWzMsNV1d
\[\begin{tikzcd}
	{x'} & x & {x'} \\
	{x' \times_{y'}x'} & {x \times_y x} & {x' \times_{y'}x'}
	\arrow[from=1-1, to=2-1]
	\arrow[from=1-2, to=2-2]
	\arrow[from=1-3, to=2-3]
	\arrow[from=1-1, to=1-2]
	\arrow[from=1-2, to=1-3]
	\arrow[from=2-1, to=2-2]
	\arrow[from=2-2, to=2-3]
\end{tikzcd}\]
and the general statement follows by induction.  
\end{proof}

\subsection{Concrete objects}	\label{Concrete objects} 

Throughout this subsection $\mathcal{E}$ denotes an \emph{ordinary} topos. Just as for $\infty$-toposes, $\mathcal{E}$ is called \emph{local} if the global sections functor admits a right adjoint, which by the same argument as for $\infty$-toposes is fully faithful. We first define the full subcategory $\mathcal{E}_{\concr}$ of concrete objects in $\mathcal{E}$ and discuss some of their basic properties, before exhibiting various colimits which are preserved by the inclusion $\mathcal{E}_{\concr} \hookrightarrow \mathcal{E}$ in \S \ref{Colimits of concrete objects in a local topos}. 

\begin{definition} 
An object $X$ in $\mathcal{E}$ is \Emph{concrete} if the canonical morphism $X \to \pi^! \pi_* X$ is a monomorphism. The subcategory of $\mathcal{E}$ spanned by concrete objects is denoted by $\mathcal{E}_{\concr}$. \qede
\end{definition} 

%\begin{remark}
%More general definitions than the above are possible, e.g., for $\infty \geq m \geq n \geq 0$ one could assume that $\mathcal{E}$ is only $m$-truncated and $X \to \pi^! \pi_* X$, $n$-truncated, but we are not aware of any such examples. \qede
%\end{remark}

A concrete object in $\mathcal{E}$ may be thought of a set together with extra structure, making it into an object in $\mathcal{E}$. The functor $\pi_*: \mathcal{E}_{\concr} \to \Set$ is moreover faithful (but not full, in general). To see this let $X,Y$ be two concrete objects together with morphisms $X \rightrightarrows Y$ whose image agree in $\Set(\pi_*X, \pi_*Y)$, then we obtain a diagram 
% https://q.uiver.app/#q=WzAsNCxbMCwwLCJYIl0sWzAsMSwiWSJdLFsxLDAsIlxccGleIVxccGlfKlgiXSxbMSwxLCJcXHBpXiFcXHBpXypZIl0sWzAsMSwiIiwwLHsib2Zmc2V0IjotMX1dLFswLDEsIiIsMix7Im9mZnNldCI6MX1dLFswLDIsIiIsMix7InN0eWxlIjp7InRhaWwiOnsibmFtZSI6Imhvb2siLCJzaWRlIjoidG9wIn19fV0sWzIsM10sWzEsMywiIiwyLHsic3R5bGUiOnsidGFpbCI6eyJuYW1lIjoiaG9vayIsInNpZGUiOiJ0b3AifX19XV0=
\[\begin{tikzcd}
	X & {\pi^!\pi_*X} \\
	Y & {\pi^!\pi_*Y},
	\arrow[shift left, from=1-1, to=2-1]
	\arrow[shift right, from=1-1, to=2-1]
	\arrow[hook, from=1-1, to=1-2]
	\arrow[from=1-2, to=2-2]
	\arrow[hook, from=2-1, to=2-2]
\end{tikzcd}\]
and we see that $X \rightrightarrows Y$ are equalised by the monomorphism $Y \hookrightarrow \pi^! \pi_*Y$. Thus a morphism $X \to Y$ in $\mathcal{E}_{\concr}$ may be viewed as a morphism $\pi_*X \to \pi_*Y$ on underlying sets, respecting the extra structure making the sets $\pi_*X, \pi_*Y$ into objects in $\mathcal{E}_{\concr}$. This perspective is used for instance in Example \ref{retract embedding}. 

\begin{example}	\label{concrete presheaves}
For any small category $A$ which admits a final object, the topos $\widehat{A}$ is local. To see this, observe that $\pi_*: \widehat{A} \to \Set$ is simply given by evaluation at the final object, and thus commutes with colimits; therefore, it admits a right adjoint by the adjoint functor theorem, which is given by sending any set $X$ to $a \mapsto \Set(A(\mathbf{1}_A, a),X)$. Concrete objects in $\widehat{A}$ are then referred to as \Emph{concrete presheaves on $A$}. A concrete presheaf on $A$ is given by a set $X$ together with a subset of $\Set(A(\mathbf{1}_A, a),X)$ for every object $a$ in $A$; these subsets are then required to be closed under precomposing by morphisms in $A$.  \par
This observation applies to the topos of simplicial sets $\widehat{\Delta}$, with the functor $\pi^!$ being exhibited by $\mathrm{cosk}_0: \Set \hookrightarrow \widehat{\Delta}$. The concrete objects are then those simplicial sets $X$ such that for any $(n + 1)$-tuple $(x_0, \ldots, x_n) \in X_0^{(n+1)}$ there exists at most one $n$-simplex with precisely these vertices. \qede
\end{example} 

\begin{proposition} 
The inclusion $\mathcal{E}_{\concr} \hookrightarrow \mathcal{E}$ admits a left adjoint. 
\end{proposition} 

\begin{proof}
Recall that in any topos the epimorphisms and the monomorphisms form an orthogonal factorisation system. Let $X$ be an object in $\mathcal{E}$, then $X \to \pi^! \pi_* X$ may be factored uniquely as $X \twoheadrightarrow X' \hookrightarrow \pi^!\pi_*X$. Consider any map $X \to Y$, where $Y$ is a diffeological space, then the lifting problem 
% https://q.uiver.app/?q=WzAsNSxbMCwwLCJYIl0sWzAsMSwiWCciXSxbMiwwLCJZIl0sWzIsMSwiXFxwaV4hXFxwaV8qWSJdLFsxLDEsIlxccGleIVxccGlfKlgiXSxbMCwxLCIiLDAseyJzdHlsZSI6eyJoZWFkIjp7Im5hbWUiOiJlcGkifX19XSxbMCwyXSxbMiwzLCIiLDIseyJzdHlsZSI6eyJ0YWlsIjp7Im5hbWUiOiJob29rIiwic2lkZSI6InRvcCJ9fX1dLFsxLDQsIiIsMCx7InN0eWxlIjp7InRhaWwiOnsibmFtZSI6Imhvb2siLCJzaWRlIjoidG9wIn19fV0sWzQsM11d
\[\begin{tikzcd}
	X && Y \\
	{X'} & {\pi^!\pi_*X} & {\pi^!\pi_*Y}
	\arrow[two heads, from=1-1, to=2-1]
	\arrow[from=1-1, to=1-3]
	\arrow[hook, from=1-3, to=2-3]
	\arrow[hook, from=2-1, to=2-2]
	\arrow[from=2-2, to=2-3]
\end{tikzcd}\]
admits a unique solution, exhibiting the universality of $X \to X'$. 
\end{proof} 

\begin{definition} 
The left adjoint of the inclusion $\mathcal{E}_{\concr} \hookrightarrow \mathcal{E}$ (which exists by the preceding proposition) is called the \Emph{concretisation}. \qede
\end{definition} 

\begin{proposition} 	\label{cp}
The category $\mathcal{E}_{\concr}$ is presentable. 
\end{proposition} 

\begin{proof}
The pair $(\pi^!, \pi_*)$ is a geometric embedding, so that $\Set$ is a $\kappa$-accessible subcategory of $\mathcal{E}$ for some regular cardinal $\kappa$, i.e.\ $\pi^!: \Set \hookrightarrow \mathcal{E}$ commutes with $\kappa$-filtered colimits. We claim that $\mathcal{E}_{\concr} \hookrightarrow \mathcal{E}$ likewise commutes with $\kappa$-filtered colimits. Let $A$ be a $\kappa$-filtered category, and consider a functor $X: A \to \mathcal{E}_{\concr}$; as filtered colimits, and a fortiori $\kappa$-filtered colimits preserve monomorphisms, the canonical map $\colim X \to \colim \pi^! \pi_* X \xrightarrow{\cong} \pi^! \pi_* \colim X$ is a monomorphism, so that $\colim X$ is concrete. 
\end{proof}

\subsubsection{Colimits of concrete objects in a local topos}	\label{Colimits of concrete objects in a local topos}

We now discuss which colimits in $\mathcal{E}_{\concr}$ are preserved by the inclusion $\mathcal{E}_{\concr} \hookrightarrow \mathcal{E}$. 

\begin{definition}	\label{embedding} 
A monomorphism $X \hookrightarrow Y$ in $\mathcal{E}_{\concr}$ is called an \Emph{embedding} if 
% https://q.uiver.app/?q=WzAsNCxbMCwwLCJYIl0sWzEsMCwiWSJdLFswLDEsIlxccGleIVxccGlfKlgiXSxbMSwxLCJcXHBpXiFcXHBpXypZIl0sWzAsMSwiIiwwLHsic3R5bGUiOnsidGFpbCI6eyJuYW1lIjoiaG9vayIsInNpZGUiOiJ0b3AifX19XSxbMCwyLCIiLDIseyJzdHlsZSI6eyJ0YWlsIjp7Im5hbWUiOiJob29rIiwic2lkZSI6InRvcCJ9fX1dLFsyLDMsIiIsMix7InN0eWxlIjp7InRhaWwiOnsibmFtZSI6Imhvb2siLCJzaWRlIjoidG9wIn19fV0sWzEsMywiIiwwLHsic3R5bGUiOnsidGFpbCI6eyJuYW1lIjoiaG9vayIsInNpZGUiOiJ0b3AifX19XV0=
\[\begin{tikzcd}
	X & Y \\
	{\pi^!\pi_*X} & {\pi^!\pi_*Y}
	\arrow[hook, from=1-1, to=1-2]
	\arrow[hook, from=1-1, to=2-1]
	\arrow[hook, from=2-1, to=2-2]
	\arrow[hook, from=1-2, to=2-2]
\end{tikzcd}\] 
is a pullback square.	\qede
\end{definition} 

\begin{example}	\label{retract embedding} 
Any retract $X \xhookrightarrow{i} Y \xtwoheadrightarrow{r} X$ is an embedding. To see this, for any object $Z$ and any morphisms $f,g$ making the outer square in the diagram 
% https://q.uiver.app/#q=WzAsNSxbMSwxLCJYIl0sWzIsMSwiWSJdLFsxLDIsIlxccGleIVxccGlfKlgiXSxbMiwyLCJcXHBpXiFcXHBpXypZIl0sWzAsMCwiWiJdLFswLDEsIiIsMCx7InN0eWxlIjp7InRhaWwiOnsibmFtZSI6Imhvb2siLCJzaWRlIjoidG9wIn19fV0sWzEsMCwiIiwwLHsiY3VydmUiOjMsInN0eWxlIjp7ImhlYWQiOnsibmFtZSI6ImVwaSJ9fX1dLFsyLDMsIiIsMCx7InN0eWxlIjp7InRhaWwiOnsibmFtZSI6Imhvb2siLCJzaWRlIjoidG9wIn19fV0sWzMsMiwiIiwwLHsiY3VydmUiOjMsInN0eWxlIjp7ImhlYWQiOnsibmFtZSI6ImVwaSJ9fX1dLFswLDIsIlxcaW90YV9YIiwyLHsic3R5bGUiOnsidGFpbCI6eyJuYW1lIjoiaG9vayIsInNpZGUiOiJ0b3AifX19XSxbMSwzLCJcXGlvdGFfWSIsMCx7InN0eWxlIjp7InRhaWwiOnsibmFtZSI6Imhvb2siLCJzaWRlIjoidG9wIn19fV0sWzQsMiwiZiIsMix7ImN1cnZlIjo0fV0sWzQsMSwiZyIsMCx7ImN1cnZlIjotNH1dLFs0LDAsImgiLDAseyJzdHlsZSI6eyJib2R5Ijp7Im5hbWUiOiJkYXNoZWQifX19XV0=
\[\begin{tikzcd}
	Z \\
	& X & Y \\
	& {\pi^!\pi_*X} & {\pi^!\pi_*Y}
	\arrow[hook, from=2-2, to=2-3]
	\arrow[curve={height=18pt}, two heads, from=2-3, to=2-2]
	\arrow[hook, from=3-2, to=3-3]
	\arrow[curve={height=18pt}, two heads, from=3-3, to=3-2]
	\arrow["{\iota_X}"', hook, from=2-2, to=3-2]
	\arrow["{\iota_Y}", hook, from=2-3, to=3-3]
	\arrow["f"', curve={height=24pt}, from=1-1, to=3-2]
	\arrow["g", curve={height=-24pt}, from=1-1, to=2-3]
	\arrow["h", dashed, from=1-1, to=2-2]
\end{tikzcd}\]
commute, there exists a unique map $h: \pi_*Z \to \pi_*X$ (indicated by the dashed arrow in the diagram), such that the triangles $(\iota_X, f, h)$ and $(i, g, h)$ commute on \emph{underlying sets}. We must show that $h$ is in the image of $\mathcal{E}_{\concr}(Z,Y) \hookrightarrow \Set(\pi_*Z, \pi_*X)$. Indeed, $h$ may be written as $\pi_* i \circ \pi_* r \circ h$, and $\pi_* r \circ h = \pi_*g$ by assumption, so that $h = \pi_* (r \circ g)$. \qede

%we wish to show that $r \circ g$ is the unique morphism such that  $\iota_X \circ g \circ r = f$ and $i \circ r \circ g = g$. As $i$ is a monomorphism, uniqueness is automatic. As $(\pi^!\pi_*i) $ is monomorphism, showing $\iota_X \circ g \circ f = f$ amounts to showing $(\pi^!\pi_*i)  \circ \iota_X \circ g \circ f = (\pi^!\pi_*i)  \circ f$, and we have 
%$$	\begin{array}{rcl}
%	(\pi^!\pi_*i)  \circ \iota_X \circ r \circ g	&	=	&	(\pi^!\pi_*i) \circ \iota_X \circ r \circ g	\\
%	{}					&	=	&	   
%	\end{array}
%$$
% As $\iota_Y$ is a monomorphism, showing $i \circ r \circ g = g$ amounts to showing $\iota_Y \circ i \circ r \circ g = \iota_Y \circ g$, and we have 
%$$	\begin{array}{rcl}
%	\iota_Y \circ i \circ r \circ g	&	=	&	(\pi^!\pi_*i) \circ \iota_X \circ r \circ g	\\
%	{}					&	=	&	   
%	\end{array}
%$$
\end{example} 

\begin{proposition}	\label{concrete pushout} 
Consider a span % https://q.uiver.app/?q=WzAsMyxbMCwxLCJYIl0sWzEsMCwiWiJdLFsxLDIsIlkiXSxbMCwxLCIiLDAseyJzdHlsZSI6eyJ0YWlsIjp7Im5hbWUiOiJob29rIiwic2lkZSI6InRvcCJ9fX1dLFswLDIsIiIsMix7InN0eWxlIjp7InRhaWwiOnsibmFtZSI6Imhvb2siLCJzaWRlIjoidG9wIn19fV1d
\[\begin{tikzcd}[row sep = small]
	& Z \\
	X \\
	& Y
	\arrow[hook, from=2-1, to=1-2]
	\arrow[hook, from=2-1, to=3-2]
\end{tikzcd}\]
in $\mathcal{E}_{\concr}$, where $X \hookrightarrow Y$ is monomorphism, and $X \hookrightarrow Z$ is an embedding, then the pushout of the above diagram in the $\infty$-category associated to $\mathcal{E}$ (see {\normalfont\cite[Prop.~6.4.5.7]{jL2009}}) is again an object of $\mathcal{E}_{\concr}$. 
\end{proposition} 

\begin{proof} 
By Proposition \ref{pushout} it is enough to show that the pushout in $\mathcal{E}$ is again in $\mathcal{E}_{\concr}$. 

We must show that the map $Y \sqcup_X Z \to \pi_!\pi_*(Y \sqcup_X Z)$ in 
\begin{equation}	\label{cube diagram} 
\begin{tikzpicture}[scale = 1.25, yscale = -1, cross line/.style={preaction={draw=white, -, line width=8pt}}, commutative diagrams/every diagram] 
	\node	(TU)		at	(2.49,0)			{$Z$};
	\node	(LU)		at	(0.73,0.93)		{$X$};
	\node	(RU)		at	(5.05,0.46)		{$Y \sqcup_X Z$};
	\node 	(BU)		at	(3.67,1.59)		{$Y$};
	\node	(TD)		at	(2.56,2.5)			{$\pi^!\pi_*Z$};
	\node	(LD)		at	(0.96,3.76)		{$\pi^!\pi_*X$};
	\node	(RD)		at	(4.96,3.15)		{$\pi^!\pi_*(Y \sqcup_X Z)$};
	\node 	(BD)		at	(3.67,4.63)		{$\pi^!\pi_*Y$};

	\path[commutative diagrams/.cd, every arrow, every label]
		(LU)		edge		[commutative diagrams/hook]									node					{}		(TU)
		(TU)		edge		[commutative diagrams/hook]									node					{}		(RU)
		(BU)		edge		[commutative diagrams/hook]									node					{}		(RU)
		(LU)		edge		[commutative diagrams/hook]									node					{}		(LD)
		(TU)		edge		[commutative diagrams/hook]									node		[near end]		{}		(TD)
		(LD)		edge		[commutative diagrams/hook]									node					{}		(TD)
		(TD)		edge		[commutative diagrams/hook]									node					{}		(RD)
		(LU)		edge		[cross line, commutative diagrams/hook, commutative diagrams/hook]	node		[near start]	{}		(BU)
		(LD)		edge		[commutative diagrams/hook]									node					{}		(BD)
		(BU)		edge		[cross line, commutative diagrams/hook]							node		[swap]		{}		(BD)
		(RU)		edge																node					{}		(RD)
		(BD)		edge		[commutative diagrams/hook]									node		[swap]		{}		(RD);
				
\end{tikzpicture}
\end{equation}  
is a monomorphism. Let $T$ be any object in $T$, and $f,g: T \rightrightarrows Y \sqcup_X Z$, any pair of maps, then we will show that if their compositions with $Y \sqcup_X Z \to \pi_!\pi_*(Y \sqcup_X Z)$ are equal, then so are $f$ and $g$. 

First, we consider the \emph{special case} in which each of the morphisms $f$ and $g$ factor through either $Y \hookrightarrow Y \sqcup_X Z$ or $Z \hookrightarrow Y \sqcup_X Z$. If {both} maps together factor through the same inclusion, then $f = g$ because $Y \hookrightarrow \pi^!\pi_*(Y \sqcup_X Z)$ and $Z \hookrightarrow \pi^!\pi_*(Y \sqcup_X Z)$ are monomorphisms. Thus, assume w.l.o.g.\ that $f$ factors through $Y$ and $g$ factors through $Z$. Observe that the bottom square in (\ref{cube diagram}) is a pullback square by Proposition \ref{pushout}, \cite[Prop.~2.2.6]{mAgBeFaJ2020}, and the fact that $\pi_*$ and $\pi^!$ preserve limits, so that we obtain a morphism $T \to \pi^!\pi_*X$ from the commutative square  
 % https://q.uiver.app/?q=WzAsNCxbMCwxLCJUIl0sWzEsMCwiXFxwaV4hXFxwaV8qWiJdLFsxLDIsIlxccGleIVxccGlfKlkiXSxbMiwxLCJcXHBpXiFcXHBpXyooWVxcc3FjdXBfWFopIl0sWzAsMSwiIiwwLHsic3R5bGUiOnsidGFpbCI6eyJuYW1lIjoiaG9vayIsInNpZGUiOiJ0b3AifX19XSxbMCwyLCIiLDIseyJzdHlsZSI6eyJ0YWlsIjp7Im5hbWUiOiJob29rIiwic2lkZSI6InRvcCJ9fX1dLFsxLDMsIiIsMCx7InN0eWxlIjp7InRhaWwiOnsibmFtZSI6Imhvb2siLCJzaWRlIjoidG9wIn19fV0sWzIsMywiIiwyLHsic3R5bGUiOnsidGFpbCI6eyJuYW1lIjoiaG9vayIsInNpZGUiOiJ0b3AifX19XV0=
\[\begin{tikzcd}
	& {\pi^!\pi_*Z} \\
	T && {\pi^!\pi_*(Y\sqcup_XZ)} \\
	& {\pi^!\pi_*Y}
	\arrow[hook, from=2-1, to=1-2]
	\arrow[hook, from=2-1, to=3-2]
	\arrow[hook, from=1-2, to=2-3]
	\arrow[hook, from=3-2, to=2-3]
\end{tikzcd}\]
and thus a morphism $T \to X$ from the induced commutative square 
% https://q.uiver.app/?q=WzAsNCxbMCwwLCJUIl0sWzEsMCwiWiJdLFsxLDEsIlxccGleIVxccGlfKloiXSxbMCwxLCJcXHBpXiFcXHBpXypYIl0sWzAsMV0sWzEsMiwiIiwwLHsic3R5bGUiOnsidGFpbCI6eyJuYW1lIjoiaG9vayIsInNpZGUiOiJ0b3AifX19XSxbMywyLCIiLDIseyJzdHlsZSI6eyJ0YWlsIjp7Im5hbWUiOiJob29rIiwic2lkZSI6InRvcCJ9fX1dLFswLDNdXQ==
\[\begin{tikzcd}
	T & Z \\
	{\pi^!\pi_*X} & {\pi^!\pi_*Z}
	\arrow[from=1-1, to=1-2]
	\arrow[hook, from=1-2, to=2-2]
	\arrow[hook, from=2-1, to=2-2]
	\arrow[from=1-1, to=2-1]
\end{tikzcd}\]
and the fact that $X \hookrightarrow Z$ is an embedding. The composition of $T \to X \to Z$ yields $g$ by construction. To see that the composition of $T \to X \to Y$ yields $f$ we further compose with the monomorphism $Y \hookrightarrow \pi^!\pi_*(Y \sqcup_X Z)$ which is equal to $f$ composed with the same monomorphism.

For the general statement consider the effective epimorphism $\bigcup_{i = 1}^4 U_i \to T$, where 
$$	\begin{array}{rcl}
	U_1	&	=	&	f^*Y \times_{Y \sqcup_X Z , \, g|_{f^*Y}}Y	\\
	U_2	&	=	&	f^*Y \times_{Y \sqcup_X Z , \, g|_{f^*Y}}Z	\\
	U_3	&	=	&	f^*Z \times_{Y \sqcup_X Z , \, g|_{f^*Z}}Y	\\
	U_4	&	=	&	f^*Z \times_{Y \sqcup_X Z , \, g|_{f^*Z}}Z	
	\end{array}
$$
then the compositions of $U_i \to T \xrightarrow{f} Y\sqcup_XZ$ and $U_i \to T \xrightarrow{g} Y\sqcup_XZ$  factor through $Y$ or $Z$ for all $i$. By the above discussion, $\bigcup_{i = 1}^4 U_i \to T$ equalises $T \rightrightarrows Y\sqcup_XZ$, and thus $f = g$, as $\bigcup_{i = 1}^4 U_i \to T$ is an effective epimorphism. %First define $Z|_f \defeq Z \times_{Y \sqcup_X Z ,f}T$ and $Z|_g \defeq Z \times_{Y \sqcup_X Z g}T$, then we set 
\end{proof} 

\begin{lemma}	\label{filtered mono} 
Let $\mathcal{F}$ be an $\infty$-topos, $I$ a filtered category, and $X: I \to \mathcal{F}$ a diagram such that $X_i \hookrightarrow X_j$ is a monomorphism for all morphisms $i \to j$ in $I$, then $X_i \to \colim X$ is likewise a monomorphism for all $i$ in $I$. 
\end{lemma} 

\begin{proof} 
Denote by $I_{i \leq}$ the full subcategory of $I$ spanned by those objects admitting a morphism from $i$, then $I_{\leq i}$ is again filtered, and the functor $I_{i \leq} \to I$ is final, so that the canonical morphism $\colim_{k \in I_{i \leq}}X_k \to \colim_{k \in I}X_k$ is an isomorphism. As $I_{i \leq}$ is filtered, and thus connected, the morphism $X_i \hookrightarrow \colim_{k \in I_{i \leq}}X_k$ may be written as $\colim_{k \in I_{i \leq}} X_i \to \colim_{k \in I_{i \leq}} X_k$, and is a monomorphism, because filtered colimits commute with finite limits in $\infty$-toposes. 
\end{proof} 

\begin{proposition}	\label{concrete filtered} 
Let $I$ be a filtered category, and $X: I \to \mathcal{E}_{\concr}$ a diagram such that $X_i \to X_j$ is a monomorphism for all morphisms $i \to j$ in $I$, then the colimit of $X$ in the $\infty$-category associated to $\mathcal{E}$ (see {\normalfont\cite[Prop.~6.4.5.7]{jL2009}}) is again in $\mathcal{E}_{\concr}$. 
\end{proposition} 

\begin{proof} 
By Proposition \ref{filtered} it is enough to show that the colimit of $X: I \to \mathcal{E}$ in $\mathcal{E}$ is again in $\mathcal{E}_{\concr}$. 

Denote by $X$ the colimit of $X: I \to \mathcal{E}$. Let $T$ be any object in $T$, and $f,g: T \rightrightarrows X$ be a pair of morphisms, then we will show that if their compositions with $X \to \pi^!\pi_*X$ are equal, then so are $f$ and $g$. By the same technique used in the last paragraph of the proof of Proposition \ref{concrete pushout} we may assume that $f$ and $g$ each factor through $X_i \to X$ and $X_j \to X$ respectively, and by the filteredness of $I$ we may assume w.l.o.g.\ that $i = j$. Consider the square 
% https://q.uiver.app/?q=WzAsNCxbMCwwLCJYX2kiXSxbMCwxLCJcXHBpXiFcXHBpXypYX2kiXSxbMSwxLCJcXHBpXiFcXHBpXypYIl0sWzEsMCwiWCJdLFswLDEsIiIsMCx7InN0eWxlIjp7InRhaWwiOnsibmFtZSI6Imhvb2siLCJzaWRlIjoidG9wIn19fV0sWzEsMiwiIiwwLHsic3R5bGUiOnsidGFpbCI6eyJuYW1lIjoiaG9vayIsInNpZGUiOiJ0b3AifX19XSxbMywyXSxbMCwzLCIiLDIseyJzdHlsZSI6eyJ0YWlsIjp7Im5hbWUiOiJob29rIiwic2lkZSI6InRvcCJ9fX1dXQ==
\[\begin{tikzcd}
	{X_i} & X \\
	{\pi^!\pi_*X_i} & {\pi^!\pi_*X}
	\arrow[hook, from=1-1, to=2-1]
	\arrow[hook, from=2-1, to=2-2]
	\arrow[from=1-2, to=2-2]
	\arrow[hook, from=1-1, to=1-2]
\end{tikzcd}\]
in which $X_i \hookrightarrow X$ is a monomorphism by Lemma \ref{filtered mono}, and therefore also $\pi^!\pi_*X_i \hookrightarrow \pi^!\pi_*X$, as $\pi^!\pi_*$ preserves limits. The compositions of the lifts of $f$ and $g$ to $T \to  X_i$ with  the monomorphism $X_i \hookrightarrow \pi^!\pi_* X$ are equal by assumption, and thus so are $f$ and $g$. 
\end{proof} 

\begin{proposition} 
Any coproduct of concrete objects in the $\infty$-topos associated to $\mathcal{E}$ (see {\normalfont\cite[Prop.~6.4.5.7]{jL2009}}) is again in $\mathcal{E}_{\concr}$.  
\end{proposition} 

\begin{proof} 
By Proposition \ref{coproducts} it is enough to show that any coproduct in $\mathcal{E}$ is again in $\mathcal{E}_{\concr}$. \\

\noindent\underline{Claim:}	For any object $E$ in $\mathcal{E}$ the map $\varnothing \to E$ is an embedding. \\

By induction it then follows form Proposition \ref{concrete pushout} that any finite coproduct of concrete objects is concrete. An arbitrary coproduct is the filtered colimit of all its finite subcoproducts so that the proposition follows from Proposition \ref{concrete filtered}. \\

\noindent\underline{Proof of claim:}	We must show that 
% https://q.uiver.app/#q=WzAsNCxbMCwwLCJcXHZhcm5vdGhpbmciXSxbMCwxLCJFIl0sWzEsMCwiXFxwaV4hXFxwaV8qXFx2YXJub3RoaW5nIl0sWzEsMSwiXFxwaV4hXFxwaV8qRSJdLFswLDFdLFswLDJdLFsyLDNdLFsxLDNdXQ==
\[\begin{tikzcd}
	\varnothing & {\pi^!\pi_*\varnothing} \\
	E & {\pi^!\pi_*E}
	\arrow[from=1-1, to=2-1]
	\arrow[from=1-1, to=1-2]
	\arrow[from=1-2, to=2-2]
	\arrow[from=2-1, to=2-2]
\end{tikzcd}\]
is a pullback. The claim will follow from showing that for any map $A \to \pi^!\pi_*\varnothing$ we must have $A = \varnothing$. As $\pi_*$ is a left adjoint we have $\pi* \varnothing = \varnothing$, so that $A \to \pi^!\pi_*\varnothing = \pi^!\varnothing$ corresponds to a map $\pi_*A \to \varnothing$ so that $\pi_*A = \varnothing$. But then we have $A \to \pi^*\pi_*A = \varnothing$, so that $A = \varnothing$. 
\end{proof} 

We then obtain the following corollary of the above propositions: 

\begin{corollary}	\label{chcl}
Let $A$ be a small category, and $X: A \to \mathcal{E}$ a functor $(n \in \mathbf{N})$. If 
\begin{enumerate}[label = {\normalfont \arabic*.}]
\item	$X$ is a wedge in which one leg is an embedding, and the other a monomorphism, 
\item	$A$ is filtered, or 
\item$A$ is discrete, 
\end{enumerate} 
then restricted shape functor $\pi_!|_{\mathcal{E}_{\concr}} \to \Pro(\mathcal{S})$ preserves the colimit of $X$ . \qed
\end{corollary} 

\subsection{Basic theory of homotopical calculi}	\label{Basic theory of homotopical calculi}

Here we construct homotopy (co)limits in a general relative category $(C,W)$. Let us begin with the simplest case of a homotopy (co)limit. By \cite[Prop.~7.1.10]{dcC2019} the localisation functor $\gamma: C \to W^{-1}C$ is both initial and final, so that if $x_0$ is an initial or final object of $C$, then $\gamma(x_0)$ is an initial or final object of $W^{-1}C$. Thus, if $C$ has a final object, then $W^{-1}C$ admits all finite limits iff it admits all pullbacks, and moreover admits all limits if it furthermore admits all products. Thus we will focus on the construction of homotopy pullbacks, and assume that $C$ admits all finite pullbacks. This leads us to consider the following definition.  

\begin{definition} 
A morphism is \Emph{sharp} if pullbacks along it are homotopy pullbacks (see Remark \ref{Rezk}). \qede
\end{definition} 

In order to recognise sharp morphisms, we abstract the properties of right proper model categories. 

\begin{definition}	\label{right proper} 
An object $x$ in $C$ is called \Emph{right proper} if the canonical functor $$W_{/x}^{-1}C_{/x} \to (W^{-1}C)_{/x}$$ is an equivalence. The relative category $(C,W)$ is called \Emph{right proper} if all objects in $C$ are right proper. \qede
\end{definition} 

\begin{notation} 
If an object $x$ in $C$ is right proper, then we will denote the $\infty$-category $(W^{-1}C)_{/x}$ by $W^{-1}C_{/x}$.  \qede
\end{notation} 

\begin{remark} 
A model category is right proper in the usual sense iff its underlying relative category is right proper. This may be seen by combining \cite[Prop.~2.7]{cR1998} with \cite[Cor.~7.6.13]{dcC2019}\footnote{Rezk's proof of \cite[Prop.~2.7]{cR1998} can be interpreted verbatim in model $\infty$-categories, so that the remark is in fact true for model $\infty$-categories.}.	\qede
\end{remark} 

\begin{remark}	\label{Rezk}
Let $f: x' \to x$ be a morphism in $C$, then recall that it is \emph{sharp} in the sense of Rezk (see \cite[\S 2]{cR1998}), if for every morphism $b \to x$ and every weak equivalence $a \xrightarrow{\sim} b$ there exists a diagram 
% https://q.uiver.app/?q=WzAsNixbMCwwLCJhJyJdLFsxLDAsImInIl0sWzIsMCwieCciXSxbMCwxLCJhIl0sWzEsMSwiYiJdLFsyLDEsIngiXSxbMCwxXSxbMSwyXSxbMCwzXSxbMyw0XSxbNCw1XSxbMSw0XSxbMiw1XV0=
\begin{equation}	\label{Rezk sharp} 
\begin{tikzcd}
	{a'} & {b'} & {x'} \\
	a & b & x
	\arrow[from=1-1, to=1-2]
	\arrow[from=1-2, to=1-3]
	\arrow[from=1-1, to=2-1]
	\arrow[from=2-1, to=2-2]
	\arrow[from=2-2, to=2-3]
	\arrow[from=1-2, to=2-2]
	\arrow[from=1-3, to=2-3]
\end{tikzcd}
\end{equation} 
in which all squares are pullbacks and such that $a' \to b'$ is a weak equivalence. 
If $(C,W)$ is right proper, then a morphism in $C$ is sharp in our sense iff it is sharp in the sense of Rezk.

 To see this, first assume that $x' \to x$ is sharp in our sense, then it is sharp in the sense of Rezk, because for every diagram of the form (\ref{Rezk sharp}) the rightmost and outer squares are homotopy pullbacks, so that the leftmost square is a homotopy pullback. Thus, if $a \to b$ is a weak equivalence, then $a' \to b'$ is a weak equivalence. 
 
 Conversely, if $x' \to x$ is sharp in the sense of Rezk, then the functor $C_{/x'} \leftarrow C_{/x}: f^*$ preserves weak equivalences, so that  \cite[Prop.~7.1.14]{dcC2019} yields, canonically, a commutative diagram 
% https://q.uiver.app/?q=WzAsNCxbMCwwLCJDX3sveCd9Il0sWzEsMCwiQ197L3h9Il0sWzAsMSwiV157LTF9Q197L3gnfSJdLFsxLDEsIldeey0xfUNfey94fSJdLFswLDEsImZfISIsMCx7Im9mZnNldCI6LTJ9XSxbMSwwLCJmXioiLDAseyJvZmZzZXQiOi0yfV0sWzIsMywiZl8hIiwwLHsib2Zmc2V0IjotMn1dLFszLDIsImZeKiIsMCx7Im9mZnNldCI6LTJ9XSxbMCwyLCIiLDEseyJvZmZzZXQiOjJ9XSxbMSwzLCIiLDEseyJvZmZzZXQiOjJ9XSxbNCw1LCIiLDAseyJsZXZlbCI6MSwic3R5bGUiOnsibmFtZSI6ImFkanVuY3Rpb24ifX1dLFs2LDcsIiIsMCx7ImxldmVsIjoxLCJzdHlsZSI6eyJuYW1lIjoiYWRqdW5jdGlvbiJ9fV1d
\[\begin{tikzcd}
	{C_{/x'}} & {C_{/x}} \\
	{W^{-1}C_{/x'}} & {W^{-1}C_{/x}}
	\arrow[""{name=0, anchor=center, inner sep=0}, "{f_!}", shift left=2, from=1-1, to=1-2]
	\arrow[""{name=1, anchor=center, inner sep=0}, "{f^*}", shift left=2, from=1-2, to=1-1]
	\arrow[""{name=2, anchor=center, inner sep=0}, "{f_!}", shift left=2, from=2-1, to=2-2]
	\arrow[""{name=3, anchor=center, inner sep=0}, "{f^*}", shift left=2, from=2-2, to=2-1]
	\arrow[shift right=2, from=1-1, to=2-1]
	\arrow[shift right=2, from=1-2, to=2-2]
	\arrow["\dashv"{anchor=center, rotate=-90}, draw=none, from=0, to=1]
	\arrow["\dashv"{anchor=center, rotate=-90}, draw=none, from=2, to=3]
\end{tikzcd}\]
The pullback of any morphism $y \to x$ along $f$ in $C$ thus yields the pullback of $y \to x$ along $f$ in $W^{-1}C$. \qede
\end{remark} 

Luckily, the main type of relative $\infty$-category of interest in this article is right proper: 

\begin{proposition} 
Let $\mathcal{E}$ be a locally contractible $\infty$-topos, then $\mathcal{E}$ together with its class $W$ of shape equivalences is a right proper relative $\infty$-category. 
\end{proposition} 

\begin{proof} 
By Corollary \ref{topos loc} the comparison functor $W_{/E}^{-1}(\mathcal{E})_{/\pi_! E} \to W^{-1}(\mathcal{E}_{/E})$ is given by the functor $\mathcal{S}_{/ \pi_! E \to \pi_! 1} \to \mathcal{S}_{/ \pi_! E}$, which is an equivalence by \cite[Prop.~4.1.1.8]{jL2009} and the fact that the inclusion $\Delta^{\{0\}} \to \Delta^1$ is initial. 
\end{proof} 

We now introduce our main tool for recognising sharp morphisms in a relative $\infty$-category. 

\begin{definition}	\label{fibration structure} 
A \Emph{fibration structure} on $(C,W)$ consists of a subcategory $\mathrm{Fib} \subseteq C$, such that $W$ and $\mathrm{Fib}$ satisfy the following conditions: 
	\begin{enumerate}[label = (\alph*)]
	\item	$\mathrm{Fib}$ contains all equivalences in $C$.
	\item $W$ satisfies the $2$-out-of-$3$ property.
	\end{enumerate} 
The morphisms in $W$, $\mathrm{Fib}$, and $\mathrm{Fib} \cap W$ are called \emph{weak equivalences}, \emph{fibrations}, and \emph{trivial fibrations} respectively. An object $x$ for which some (and therefore any) morphism to a final object of $C$ is a fibration is called \emph{fibrant}. Furthermore: 
	\begin{enumerate}[label = (\alph*), resume]
	\item	In any diagram
% https://q.uiver.app/?q=WzAsMyxbMSwwLCJ4JyJdLFsxLDEsIngiXSxbMCwxLCJ5Il0sWzAsMSwiZiJdLFsyLDFdXQ==
\[\begin{tikzcd}
	& {x'} \\
	y & x
	\arrow["f", from=1-2, to=2-2]
	\arrow[from=2-1, to=2-2]
\end{tikzcd}\]
such that $f$ is either a fibration, or trivial fibration, the pullback is again a fibration of trivial fibration, respectively. 
	\item	Any morphism $x \to y$ admits a factorisation $x \to x' \to y$ such that $x' \to x'$ is a weak equivalence, and $x' \to y$ is a fibration. 
	\end{enumerate} 

An $\infty$-category equipped with a fibration structure is called a \Emph{fibration $\infty$-category}. \par
Dually, a subcategory $\mathrm{Cof} \subseteq C$ is a \Emph{cofibration structure} on $C$, if $\mathrm{Cof}^{\op}$ is a fibration structure on $(C^{\op},W^{\op})$.  An $\infty$-category equipped with a cofibration structure is called a \Emph{cofibration $\infty$-category}.  \qede
\end{definition} 

\begin{remark} 
Our notion of \emph{fibration structure} is slightly stronger than the notion of \emph{$\infty$-category with weak equivalences and fibrations} considered in \cite[Def.~7.4.12]{dcC2019}. \qede
\end{remark} 

\begin{example} 
The classes of weak equivalences and fibrations of any $\infty$-model category (see \S \ref{model structure}) form a fibration structure, which moreover satisfy the condition of Proposition \ref{ap} if it admits all limits. All fibration structures considered in this article will be of this form (however, see Remark \ref{squishy remark}). \qede
\end{example} 

From now on we assume that $(C,W)$ is equipped with a fibration structure $\mathrm{Fib}$. 

\begin{proposition} 
Let 
\begin{equation}	\label{derived pullback}
% https://q.uiver.app/#q=WzAsNCxbMCwwLCJ5JyJdLFswLDEsInkiXSxbMSwwLCJ4JyJdLFsxLDEsIngiXSxbMCwxXSxbMCwyXSxbMiwzXSxbMSwzLCJmIl1d
\begin{tikzcd}
	{y'} & {x'} \\
	y & x
	\arrow[from=1-1, to=2-1]
	\arrow[from=1-1, to=1-2]
	\arrow[from=1-2, to=2-2]
	\arrow["f", from=2-1, to=2-2]
\end{tikzcd}
\end{equation} 
be pullback square in $C$, where $y$ and $x$ are proper, and $x' \to x$ is a fibration, then the square is a homotopy pullback. 
\end{proposition} 

The following proof is similar to the last part of Remark \ref{Rezk}. 

\begin{proof} 
Equip the relative $\infty$-category $(C_{/y}, W_{/y})$ with the cofibration structure in which all morphisms are cofibrations, and $(C_{/x}, W_{/x})$ with the fibration structured induced by $\mathrm{Fib}$,  then, by \cite[Th.~7.5.30]{dcC2019} the derived functors of $f_!$ and $f^*$ exist and are canonically adjoint to each other. As in Remark \ref{Rezk} the functor $\mathbf{L}f_!: W^{-1}C_{/y} \to W^{-1}C_{/x}$ is given by postcomposition. Thus, $\mathbf{L}f_!$ gives us the lower triangle in 
% https://q.uiver.app/#q=WzAsNCxbMCwwLCJcXG1hdGhiZntSfWZeKngnIl0sWzAsMSwieSJdLFsxLDEsIngiXSxbMSwwLCJ4JyJdLFswLDFdLFsxLDJdLFswLDJdLFszLDJdLFswLDNdXQ==
\[\begin{tikzcd}
	{\mathbf{R}f^*x'} & {x'} \\
	y & x
	\arrow[from=1-1, to=2-1]
	\arrow[from=2-1, to=2-2]
	\arrow[from=1-1, to=2-2]
	\arrow[from=1-2, to=2-2]
	\arrow[from=1-1, to=1-2]
\end{tikzcd}\]
and the counit gives us the upper triangle, exhibiting $\mathbf{R}f^*x'$ as a pullback. We now need to show that the above square is isomorphic to the image of (\ref{derived pullback}) under $C \to W^{-1}C$.  

As $x' \to x$ is fibrant in $C_{/x}$ we obtain a canonical isomorphism $y' \xrightarrow{\simeq} \mathbf{R}f^*x' $ in $W^{-1}C_{/y}$. Applying $\mathbf{L}f_!$ yields the diagram $\Delta^2 \to W^{-1}C_{/x}$ given by 
% https://q.uiver.app/#q=WzAsNCxbMCwwLCJ5JyJdLFswLDEsInkiXSxbMSwxLCJ4Il0sWzEsMCwiXFxtYXRoYmZ7Un1mXip4JyJdLFswLDFdLFsxLDJdLFswLDJdLFszLDJdLFswLDNdLFszLDEsIiIsMSx7ImNvbG91ciI6WzAsMCwxMDBdLCJzdHlsZSI6eyJoZWFkIjp7Im5hbWUiOiJub25lIn19fV0sWzMsMV1d
\[\begin{tikzcd}
	{y'} & {\mathbf{R}f^*x'} \\
	y & x.
	\arrow[from=1-1, to=2-1]
	\arrow[from=2-1, to=2-2]
	\arrow[from=1-1, to=2-2]
	\arrow[from=1-2, to=2-2]
	\arrow[from=1-1, to=1-2]
	\arrow[crossing over, from=1-2, to=2-1]
	\arrow[from=1-2, to=2-1]
\end{tikzcd}\]

Denote by $C_{/x}^{\fib}$ the full subcategory of $C_{/x}$ spanned by the fibrant objects, and by $\gamma: C_{/x} \to W^{-1}C_{/x}$ and $\gamma': C^{\fib}_{/x} \to W^{-1}C^{\fib}_{/x}$ the respective localisation functors. By \cite[Lm.~7.5.24]{dcC2019} and the fact that $f_!$ preserves weak equivalences, we see that under the equivalence of $\infty$-categories $[W^{-1}C_{/x}^{\fib}, W^{-1} C_{/x}] \xrightarrow{\simeq}[C_{/x}^{\fib}, W^{-1} C_{/x}]_W$ the transformation obtained by whiskering the counit $\mathbf{L}f_! \mathbf{R}f^* \to \id_{W^{-1}C_{/x}}$ by the functor $W^{-1}C_{/x}^{\fib} \xrightarrow{\simeq} W^{-1}C_{/x}$ corresponds to the transformation obtained by whiskering $f_!f^* \to \id_{C_{/x}}$ with the inclusion $C_{/x}^{\fib} \hookrightarrow C_{/x}$ under the natural isomorphism $\gamma \circ f_! \circ f^*|_{C_{/x}^{\fib}} \xrightarrow{\simeq} \mathbf{L} f_! \circ \mathbf{R} f^*|_{W^{-1}C_{/x}^{\fib}} \circ \gamma'$ yielding a diagram $\Delta^2 \to W^{-1}C_{/x}$ given by

% https://q.uiver.app/#q=WzAsNCxbMCwwLCJ5JyJdLFsxLDAsIngnIl0sWzAsMSwiXFxtYXRoYmZ7Un1mXip4JyJdLFsxLDEsIngiXSxbMCwxXSxbMCwyXSxbMiwzXSxbMSwzXSxbMCwzXSxbMiwxLCIiLDEseyJzdHlsZSI6eyJib2R5Ijp7Im5hbWUiOiJub25lIn0sImhlYWQiOnsibmFtZSI6Im5vbmUifX19XSxbMiwxXV0=
\[\begin{tikzcd}
	{y'} & {x'} \\
	{\mathbf{R}f^*x'} & x.
	\arrow[from=1-1, to=1-2]
	\arrow[from=1-1, to=2-1]
	\arrow[from=2-1, to=2-2]
	\arrow[from=1-2, to=2-2]
	\arrow[from=1-1, to=2-2]
	\arrow[crossing over, from=2-1, to=1-2]
	\arrow[from=2-1, to=1-2]
\end{tikzcd}\]
\end{proof} 

\begin{corollary}	\label{as}
Let $(C,W,\mathrm{Fib})$ be a fibration category such that $(C,W)$ is right proper, then every fibration is sharp. \qed
\end{corollary} 

\begin{corollary}	\label{sharp topos} 
Let $\mathcal{E}$ be a locally contractible $\infty$-topos, and let $\mathrm{Fib}$ be a fibration structure on $\mathcal{E}$ w.r.t.\ the shape equivalences, then any fibration is sharp w.r.t.\ the shape equivalences. 
\qed
\end{corollary} 

The following simple proposition offers an effective method for detecting sharp morphisms. 

\begin{proposition}	\label{sd}
Let $(C,W), (C',W')$ be relative $\infty$-categories with pullbacks, and let $f: C \to C'$ be a functor. Assume that 
\begin{enumerate}[label = {\normalfont (\alph*)}]
\item $f$ preserves pullbacks, 
\item $fW \subseteq W'$, and 
\item	the induced functor $W^{-1}C \to W'^{-1}C'$ is an equivalence of $\infty$-categories, 
\end{enumerate} 
then any morphism $x \to y$ in $C$ is sharp if $fa \to fb$ is. 
\end{proposition} 

\begin{proof} 
Any pullback along $x \to y$ is sent to a pullback along $fx \to fy$ which is sent to a pullback in $W'^{-1}C'$, so that any pullback along $x \to y$ is sent to a pullback in $W^{-1}C$. 
\end{proof} 

We now discuss not necessarily finite homotopy limits. 

\begin{proposition}[{\cite[Prop.~7.7.4]{dcC2019}}] \label{ap}
If an arbitrary product of fibrant objects in $C$ is again fibrant, and an arbitrary product of trivial fibrations is again a trivial fibration, then arbitrary products of fibrant objects are homotopy products. \qed
\end{proposition} 

\begin{remark} 
Model categories and $\infty$-categories are frequently viewed as providing competing foundations for homotopy theory  (see \cite{MO78400}). In reality, the axioms for model categories can be interpreted without difficulty in the setting of $\infty$-categories (see \S \ref{model structure}), not just ordinary categories, and one observes that model structures are simply tools for studying localisations. Any $\infty$-category may be obtained as the localisation of an ordinary relative category  (see \cite[Prop.~7.3.15]{dcC2019}, \cite{cBdK2012a}), and any presentable $\infty$-category may be obtained as the localisation of a combinatorial simplicial model category (see \cite[Prop.~A.3.7.6]{jL2009} \& \cite[Th.~1.3.4.20]{jL2012} \& \cite[Th.~7.5.18]{dcC2019}). Before the work of Joyal, Simpson, To\"{e}n, Rezk,  Lurie and many others it was simply not practical to present a given $\infty$-category in any other way than as an ordinary relative category (or a simplicially enriched category). Thus, nowadays, one has a \emph{choice} of whether one wishes to work in a given $\infty$-category $C$, or whether one wishes to view $C$ as the localisation of some other ($\infty$-)category $D$. The optimal choice of $D$ does not necessarily have to be an ordinary category, as seen in Mazel-Gee's generalisation of the Goerss-Hopkins obstruction theorem (see \cite{aMG2016t}), and in our applications to differentiable sheaves in this article. \qede
\end{remark}

\subsection{Constructing homotopical calculi in locally contractible ($\infty$-)toposes}	\label{Homotopical calculi in locally contractible toposes}

In \S \ref{Basic theory of homotopical calculi}  we saw how fibration structures are well suited to identifying homotopy limits in (subcategories of) locally contractible $\infty$-toposes; this subsection concerns their construction using test categories. 

Throughout this subsection $A$ denotes a small ordinary category. In Example \ref{plc} we saw that $[A^{\op}, \mathcal{S}]$ models the $\infty$-category $\mathcal{S}_{/A_\simeq}$ in the sense that taking colimits produces a localisation $[A^{\op}, \mathcal{S}] \to \mathcal{S}_{/A_\simeq}$. In the special case $A = \Delta$ something rather remarkable happens. The restriction of the functor $[\Delta^{\op}, \mathcal{S}] \to \mathcal{S}_{/\Delta_\simeq} \xrightarrow{\sim} \mathcal{S}$ to $\widehat{\Delta} \to \mathcal{S}$ is still a localisation, exhibiting the classical way in which homotopy types are modelled by simplicial sets. As the construction of the model category of simplicial sets is quite involved, one might expect this phenomenon to be particular to $\Delta$, but it turns out to be surprisingly common. Qualitatively, categories for which this phenomenon arise are precisely \emph{test categories}. 

The theory of test categories is outlined in \S \ref{Test categories}, with a focus on how ordinary categories of set-valued presheaves on test categories model slice $\infty$-categories of $\mathcal{S}$. Then in \S \ref{Transferring model structures} we discuss how to construct model structures on $\infty$-categories of homotopy-type-valued presheaves on test categories, which may then be transferred to locally contractible $\infty$-toposes via the nerves of \S \ref{snerves}. 

\subsubsection{Test categories}	\label{Test categories}

The basic ideas discussed in this subsection are essentially all due to Grothendieck, and were first outlined in \cite{aG1983}. A systematic account of Grothendieck's theory is given by Maltsiniotis in \cite{gM2005}. The theory of test categories, and in particular its model categorical aspects, are further developed in \cite{dcC2006}.  \par
The starting point for understanding the phenomenon discussed in the introduction of \S \ref{Homotopical calculi in locally contractible toposes} is the following fact: Recall that the classifying space of an $\infty$-category is nothing but the homotopy type obtained by inverting all its arrows, and furthermore, that the classifying space construction is left adjoint to the inclusion of $\mathcal{S}$ into $\Cat$, the $\infty$-category of $\infty$-categories. Then, paralleling the situation for $[\Delta^{\op}, \mathcal{S}]$, the restriction of the classifying  space functor $(\emptyinput)_\simeq$ to the $(2,1)$-category $\Cat_{(1,1)}$ of ordinary categories exhibits $\mathcal{S}$ as a localisation of $\Cat_{(1,1)}$:
% https://q.uiver.app/#q=WzAsMyxbMSwwLCJcXENhdCJdLFsyLDAsIlxcbWF0aGNhbHtTfSJdLFswLDAsIlxcQ2F0X3soMSwxKX0iXSxbMCwxLCIoXFxlbXB0eWlucHV0KV9cXHNpbWVxIiwwLHsib2Zmc2V0IjotMn1dLFsxLDAsIiIsMCx7Im9mZnNldCI6LTJ9XSxbMiwwLCIiLDAseyJzdHlsZSI6eyJ0YWlsIjp7Im5hbWUiOiJob29rIiwic2lkZSI6InRvcCJ9fX1dLFszLDQsIiIsMCx7ImxldmVsIjoxLCJzdHlsZSI6eyJuYW1lIjoiYWRqdW5jdGlvbiJ9fV1d
\[\begin{tikzcd}
	{\Cat_{(1,1)}} & \Cat & {\mathcal{S}}
	\arrow[""{name=0, anchor=center, inner sep=0}, "{(\emptyinput)_\simeq}", shift left=2, from=1-2, to=1-3]
	\arrow[""{name=1, anchor=center, inner sep=0}, shift left=2, from=1-3, to=1-2]
	\arrow[hook, from=1-1, to=1-2]
	\arrow["\dashv"{anchor=center, rotate=-90}, draw=none, from=0, to=1]
\end{tikzcd}\]
The $(2,1)$-category $\Cat_{(1,1)}$ itself is the localisation of the ordinary category of ordinary categories $\Cat'_{(1,1)}$ (along the equivalences of categories).

The fact that $\mathcal{S}$ is a localisation of $\Cat'_{(1,1)}$ has been known in essence since \cite[Cor.~3.3.1]{lI1972} (specifically, that the category of elements of a simplicial set encodes the same homotopy type as the simplicial set itself is shown in \cite[Th.~3.3.ii]{lI1972}. Illusie attributes the ideas presented in  \cite[\S 3.3]{lI1972} to Quillen; see also \cite{dQ1973}). Moreover, Thomasson shows that the relative category $\Cat'_{(1,1)}$ together with the weak equivalences induced by $(\emptyinput)_{\simeq}$ is right proper (see Definition \ref{right proper}), by exhibiting a right proper model structure on $\Cat'_{(1,1)}$ by right transferring the Kan-Quillen model structure (which is right proper) along the functor $\mathrm{Ex}^2 \circ N: \Cat'_{(1,1)} \to \widehat{\Delta}$ (see \cite{rT1980}). Thus, the category $(\Cat'_{(1,1)})_{/A}$ is a model for $\mathcal{S}_{/A_\simeq}$; a model which turns out to be particularly convenient for determining conditions on $A$ such that $\colim: \widehat{A} \to \mathcal{S}_{/A_\simeq}$ is a localisation. In a first instance, we will focus on the special case when $A_\simeq = 1$. Then, $\colim: [A^{\op}, \mathcal{S}] \to \mathcal{S}$ factors as $[A^{\op}, \mathcal{S}] \xrightarrow{A_{/\emptyinput}} \Cat \xrightarrow{(\emptyinput)_\simeq} \mathcal{S}$, which restricts to $\widehat{A} \xrightarrow{A_{/\emptyinput}} \Cat'_{(1,1)} \xrightarrow{(\emptyinput)_\simeq} \mathcal{S}$. Thus, $A_{/\emptyinput}$ models the left adjoint of the adjunction $\cadjunction{\colim: [A^{\op}, \mathcal{S}]}	{\mathcal{S}}$.

The functor $A_{/\emptyinput}$ also admits a right adjoint given by $N_A: C \mapsto (a \mapsto \Hom(A_{/a}, C))$.  The category $A$ is a \Emph{weak test category} if $N_A$ sends functors $C \to D$ such that $C_\simeq \to D_\simeq$ is an isomorphism to shape equivalence, and if the resulting adjunction $\cadjunction{W^{-1} \widehat{A}}	{\mathcal{S}}$ is an adjoint equivalence. We can now state the main definition of this subsection: 

\begin{definition} 
The category $A$ is a \Emph{local test category} if $A_{/a}$ is a weak test category for all $a$ in $A$. The category $A$ is a \Emph{test category} if it is a local test category, and if moreover $A_\simeq = 1$. \qede
\end{definition} 

\begin{theorem}[{\cite[Cor.~4.4.20]{dcC2006}}]	\label{ltch}
If $A$ is a local test category, then the composition of the functors  ${A_{/\emptyinput}: \widehat{A}} \to	{(\Cat_{(1,1)})_{/A}} \to \mathcal{S}_{A_\simeq}$ is a localisation of $\widehat{A}$ along the shape equivalences. \qed
\end{theorem} 

\begin{definition} 
Let $A$ be a small ordinary category, then a presheaf $X$ on $A$ is called \Emph{locally aspherical} if $(a \times X)_{\simeq} = 1$ for all $a \in A$. \qede
\end{definition} 

One of the key features of weak test categories is that they admit many characterisations, as seen in the following theorem. 

\begin{theorem}[{\cite[Th.~1.5.6]{gM2005} \& \cite[Thms.~1.4.3~\&~4.1.19~\&~4.2.15]{dcC2006}}]	\label{test model}
The following are equivalent: 
	\begin{enumerate}[label = {\normalfont (\Roman*)}]
	\item	\label{itc1}	$A$ is a local test category.
	\item	\label{itc2}	The subobject classifier of $\widehat{A}$ is locally aspherical.
	\item \label{itc3}The category $\widehat{A}$ admits a locally aspherical separating interval (see Definition \ref{interval}). 
	\item	\label{mtc2}	Any morphism in $\widehat{A}$ with the right lifting property against all monomorphisms is a shape equivalence. 
	\item	\label{mtc3}	The category $\widehat{A}$ admits a (cofibrantly generated) model structure in which the weak equivalences are the shape equivalences, and the cofibrations are the monomorphisms.  
	\end{enumerate}
\qed
\end{theorem} 

\begin{proposition}	\label{sifted}
The following are equivalent: 
	\begin{enumerate}[label = {\normalfont (\Roman*)}]
	\item	\label{sifted 1}	$A$ is sifted (see {\normalfont\cite[Def.~5.5.8.1]{jL2009}}).
	\item	\label{sifted 2}	$A_{/ \simeq} = 1$ and $(A_{/ a \times a'})_\simeq = 1$ for all $a, a' \in A$. 
	\item	\label{sifted 3}	$\colim: \widehat{A} \to \mathcal{S}$ preserves finite products. 
	\end{enumerate} 
\end{proposition} 

\begin{proof} 
The implication \ref{sifted 1} $\implies$ \ref{sifted 3} follows from \cite[Lm.~5.5.8.11]{jL2009}, and \ref{sifted 2} is a special case of \ref{sifted 3}, establishing \ref{sifted 3} $\implies$ \ref{sifted 2}, and \ref{sifted 2} $\implies$ \ref{sifted 1} follows form applying \cite[Th.~4.1.3.1]{jL2009} to \cite[Def.~5.5.8.1]{jL2009}. 
\end{proof} 

\begin{definition} 
The category $A$ is a \Emph{strict test category} if it is a local test category and satisfies the equivalent conditions of Proposition \ref{sifted}. \qede
\end{definition} 

Applying Theorem \ref{test model} to strict test categories yields the following recognition theorem. 

\begin{proposition}	\label{sistc}
Let $A$ be a small category admitting finite products and a representable separating interval on $\widehat{A}$, then $A$ is a strict test category. \qed
\end{proposition} 

In \cite[\S 1.8]{gM2005} Cisinski and Maltsiniotis develop more sophisticated tools for recognising strict test categories, and produces some surprising examples thereof, such as the monoid of increasing functions $\mathbf{N} \to \mathbf{N}$ (see \cite[Ex.~1.8.15]{gM2005}). 

\paragraph{Test toposes}

We give a very brief introduction to the theory of local test toposes developed in \cite{dcC2003}. Throughout our discussion on test toposes, $\mathcal{E}$ denotes an ordinary topos generated under colimits by a set of contractible objects, by which we mean objects which have contractible shape in the \emph{hypercompletion} of the $\infty$-topos associated to $\mathcal{E}$ (in the sense of \cite[Prop.~6.4.5.7]{jL2009}), which we denote by $\mathcal{E}_\infty$. 

We begin with the following generalisation of Theorem \ref{test model}, which we then use to give a definition of local test toposes.

\begin{theorem}[{\cite[Th.~4.2.8]{dcC2003}}]	\label{main cisinski} 
The following are equivalent: 
	\begin{enumerate}[label = \normalfont{(\Roman*)}]
	\item	For any object $X$ in $\mathcal{E}$ the projection map $X \times \Omega_\mathcal{E} \to X$ is a shape equivalence.  
	\item	\label{ctf}	Any morphism in $\mathcal{E}$ with the right lifting property against all monomorphisms is a shape equivalence; 
	\item	\label{tcc} There exists a subcategory of $\mathcal{E}$ spanned by objects of contractible shape, which is moreover a local test category and which generates $\mathcal{E}$ under colimits. 
	\item	There exists a (necessarily unique as well as cofibrantly generated) model structure on $\mathcal{E}$ in which the weak equivalences are the shape equivalences, and in which the cofibrations are the monomorphisms. \\ \qed
	\end{enumerate}
\end{theorem}

\begin{definition}	\label{test topos} 
An ordinary topos satisfying the equivalent conditions of Theorem \ref{main cisinski} is called a \Emph{local test topos}. A local test topos with trivial shape is a \Emph{test topos}. A test topos, whose shape functor commutes with finite products, is a \Emph{strict test topos}. On any topos, the model structure given by Theorem \ref{main cisinski} is referred to as the \Emph{canonical model structure}.   \qede
\end{definition} 

\begin{proposition}	\label{ttm}
Assume $\mathcal{E}$ is a local test topos, then $\mathcal{E}_\infty$ is locally contractible, and the composition $\mathcal{E} \hookrightarrow \mathcal{E}_\infty \xrightarrow{\pi_!} \mathcal{S}_{/\pi_! \mathbf{1}_{\mathcal{E}_\infty}}$ is a localisation. 
\end{proposition} 

\begin{proof} 
The proposition is equivalent to the statement that the inclusion $\mathcal{E} \hookrightarrow \mathcal{E}_\infty$ induces an equivalence of $\infty$-categories upon localising along shape equivalences. Let $C \subseteq \mathcal{E}$ be a subcategory as in \ref{tcc} of Theorem \ref{main cisinski}. Consider the diagram 
% https://q.uiver.app/#q=WzAsNCxbMCwwLCJbQ157XFxvcH0sIFxcbWF0aGNhbHtTfV0iXSxbMSwwLCJcXG1hdGhjYWx7RX1fXFxpbmZ0eSJdLFswLDEsIlxcd2lkZWhhdHtDfSJdLFsxLDEsIlxcbWF0aGNhbHtFfSJdLFswLDEsIiIsMCx7Im9mZnNldCI6LTJ9XSxbMSwwLCIiLDAseyJvZmZzZXQiOi0yfV0sWzIsMCwiIiwwLHsic3R5bGUiOnsidGFpbCI6eyJuYW1lIjoiaG9vayIsInNpZGUiOiJ0b3AifX19XSxbMywxLCIiLDAseyJzdHlsZSI6eyJ0YWlsIjp7Im5hbWUiOiJob29rIiwic2lkZSI6InRvcCJ9fX1dLFsyLDMsIiIsMSx7Im9mZnNldCI6LTJ9XSxbMywyLCIiLDEseyJvZmZzZXQiOi0yfV0sWzQsNSwiIiwwLHsibGV2ZWwiOjEsInN0eWxlIjp7Im5hbWUiOiJhZGp1bmN0aW9uIn19XSxbOCw5LCIiLDEseyJsZXZlbCI6MSwic3R5bGUiOnsibmFtZSI6ImFkanVuY3Rpb24ifX1dXQ==
\[\begin{tikzcd}
	{[C^{\op}, \mathcal{S}]} & {\mathcal{E}_\infty} \\
	{\widehat{C}} & {\mathcal{E}}
	\arrow[""{name=0, anchor=center, inner sep=0}, shift left=2, from=1-1, to=1-2]
	\arrow[""{name=1, anchor=center, inner sep=0}, shift left=2, from=1-2, to=1-1]
	\arrow[hook, from=2-1, to=1-1]
	\arrow[hook, from=2-2, to=1-2]
	\arrow[""{name=2, anchor=center, inner sep=0}, shift left=2, from=2-1, to=2-2]
	\arrow[""{name=3, anchor=center, inner sep=0}, shift left=2, from=2-2, to=2-1]
	\arrow["\dashv"{anchor=center, rotate=-90}, draw=none, from=0, to=1]
	\arrow["\dashv"{anchor=center, rotate=-90}, draw=none, from=2, to=3]
\end{tikzcd}\]	
then the top adjunction is a geometric embedding by \cite[Cor.~20.4.3.3~\&~Prop.~20.4.5.1]{jL2017}, and a local shape equivalence by Proposition \ref{colim extend}, so that the right adjoint is shape preserving by Proposition \ref{aspherical embedding}. Thus the unit and counit are natural weak equivalences, and the same is true of the bottom adjunction, as it is a restriction of the top one, so that both adjunctions descend to equivalences of $\infty$-categories after upon localising by \cite[Prop.~7.1.14]{dcC2019}.  The left vertical functor induces an equivalence upon localising by Theorem \ref{ltch}, so that the right vertical functor induces an equivalence upon localising by the 2-out-of-3 property. 
\end{proof} 

\begin{remark} 
Proposition \ref{ttm} fails if we do not assume that $\mathcal{E}_\infty$ is hypercomplete, because then $\mathcal{E}_\infty$ may no longer be generated by objects in $C$ under colimits (see \cite{mA2021}).  \qede 
\end{remark} 

\begin{lemma}	\label{hyperslice} 
Let $X$ be a $1$-truncated object of $\mathcal{E}_\infty$, then $(\mathcal{E}_\infty)_{/X}$ is equivalent to the hypercompletion of the $\infty$-topos associated to $\mathcal{E}_{/X}$. 
\end{lemma}

\begin{proof} 
Denote by $(\mathcal{E}_{/X})_\infty$ the hypercompletion of the $\infty$-topos associated to $\mathcal{E}_{X}$ then we obtain a commutative square 
% https://q.uiver.app/#q=WzAsNCxbMCwwLCJcXG1hdGhjYWx7RX1fey9YfSJdLFswLDEsIlxcbWF0aGNhbHtFfSJdLFsxLDEsIlxcbWF0aGNhbHtFfV9cXGluZnR5Il0sWzEsMCwiKFxcbWF0aGNhbHtFfV97L1h9KSBfXFxpbmZ0eSJdLFswLDFdLFsyLDEsIiIsMix7InN0eWxlIjp7InRhaWwiOnsibmFtZSI6Imhvb2siLCJzaWRlIjoiYm90dG9tIn19fV0sWzMsMl0sWzMsMCwiIiwwLHsic3R5bGUiOnsidGFpbCI6eyJuYW1lIjoiaG9vayIsInNpZGUiOiJib3R0b20ifX19XV0=
\[\begin{tikzcd}
	{\mathcal{E}_{/X}} & {(\mathcal{E}_{/X}) _\infty} \\
	{\mathcal{E}} & {\mathcal{E}_\infty.}
	\arrow[from=1-1, to=2-1]
	\arrow[hook', from=2-2, to=2-1]
	\arrow[from=1-2, to=2-2]
	\arrow[hook', from=1-2, to=1-1]
\end{tikzcd}\]
By the universal property of $(\mathcal{E}_\infty)_{/X}$ (see \cite[Rmk.~6.3.5.8]{jL2009}) we may exhibit $(\mathcal{E}_{/X}) _\infty$ as a subcategory of $(\mathcal{E}_\infty)_{/X}$. Conversely, $(\mathcal{E}_\infty)_{/X}$ is hypercomplete by \cite[Th.~6.5.3.12]{jL2009}, so by the universal property of $(\mathcal{E}_{/X}) _\infty$ the $\infty$-topos $(\mathcal{E}_\infty)_{/X}$ is a subcategory of $(\mathcal{E}_{/X}) _\infty$. 
\end{proof} 

\begin{proposition}[{\cite[Cor.~5.3.20~\&~Cor.~4.2.12]{dcC2003}}]	\label{test proper}
Any local test topos -- viewed as relative category with its shape equivalences as weak equivalences -- is proper. 
\end{proposition} 

\begin{proof} 
By the Lemma \ref{hyperslice} the composition of the functors $\mathcal{E}_{X} \hookrightarrow (\mathcal{E}_\infty)_{/X} \xrightarrow{\pi_!} \mathcal{E}_{/\pi_!X}$ is a localisation. 
\end{proof} 

We finish with an application of Theorem \ref{main cisinski} to equivariant homotopy theory. 

\begin{theorem}	\label{equivariant test} 
Assume that $\mathcal{E}$ is a strict test topos, and that $G$ is a group object in $\mathcal{E}$, then $\mathcal{E}_G$ is a test category. A morphism in $\mathcal{E}_G$ is a shape equivalence iff its underlying morphism in $\mathcal{E}$ is, and the induced functor $\mathcal{E}_G \to \mathcal{S}_G$ is a localisation along the shape equivalences in $\mathcal{E}_G$.  
\end{theorem} 

\begin{proof} 
From the equivalence of $\infty$-categories $(\mathcal{E}_\infty)_G = (\mathcal{E}_\infty)_{/BG}$ we see that $\mathcal{E}_G$  is equivalent to $\left ( (\mathcal{E}_\infty)_{/BG} \right)_{\leq 0}$. Let $C$ be a small subcategory of $\mathcal{E}$ spanned by objects of contractible shape generating $\mathcal{E}$ (and thus $\mathcal{E}_\infty$) under colimits, then $C_{/BG}$ is an ordinary category whose objects are of contractible shape and generate $(\mathcal{E}_\infty)_{/BG}$ under colimits. We will check that $\left ( (\mathcal{E}_\infty)_{/BG} \right)_{\leq 0}$ satisfies \ref{ctf} of Theorem \ref{main cisinski}, verifying the first part of the theorem.  Let $X \to Y$ be a morphism in $\left ( (\mathcal{E}_\infty)_{/BG} \right)_{\leq 0}$ lifting against all monomorphisms, then the underlying morphism of $X \to Y$ in $\mathcal{E}$ lifts against all monomorphisms (and is thus a shape equivalence), as any lifting problem against $X \to Y$ in $\mathcal{E}$ may be promoted to one in $\mathcal{E}_{/BG}$ by composing with the morphism $Y \to BG$. 

Next, the induced functor $\mathcal{E}_G \to \mathcal{S}_G$ is a localisation by the following diagram:  
% https://q.uiver.app/#q=WzAsNixbMCwwLCJcXG1hdGhjYWx7RX1fRyJdLFsxLDAsIihcXG1hdGhjYWx7RX1fXFxpbmZ0eSlfRyJdLFsyLDAsIlxcbWF0aGNhbHtTfV97XFxwaV8hR30iXSxbMCwxLCJcXGxlZnQgKCAoXFxtYXRoY2Fse0V9X1xcaW5mdHkpX3svQkd9IFxccmlnaHQpX3tcXGxlcSAwfSJdLFsxLDEsIihcXG1hdGhjYWx7RX1fXFxpbmZ0eSlfey9CR30gIl0sWzIsMSwiXFxtYXRoY2Fse1N9X3svQlxccGlfIUd9Il0sWzAsMSwiIiwwLHsic3R5bGUiOnsidGFpbCI6eyJuYW1lIjoiaG9vayIsInNpZGUiOiJ0b3AifX19XSxbMSwyXSxbMCwzLCJcXHNpbWVxIiwyXSxbMyw0LCIiLDAseyJzdHlsZSI6eyJ0YWlsIjp7Im5hbWUiOiJob29rIiwic2lkZSI6InRvcCJ9fX1dLFsxLDQsIlxcc2ltZXEiLDJdLFsyLDUsIlxcc2ltZXEiLDJdLFs0LDVdXQ==
\[\begin{tikzcd}
	{\mathcal{E}_G} & {(\mathcal{E}_\infty)_G} & {\mathcal{S}_{\pi_!G}} \\
	{\left ( (\mathcal{E}_\infty)_{/BG} \right)_{\leq 0}} & {(\mathcal{E}_\infty)_{/BG} } & {\mathcal{S}_{/B\pi_!G}}
	\arrow[hook, from=1-1, to=1-2]
	\arrow[from=1-2, to=1-3]
	\arrow["\simeq"', from=1-1, to=2-1]
	\arrow[hook, from=2-1, to=2-2]
	\arrow["\simeq"', from=1-2, to=2-2]
	\arrow["\simeq"', from=1-3, to=2-3]
	\arrow[from=2-2, to=2-3]
\end{tikzcd}\]
Finally, a morphism $X \to Y$ in $\mathcal{E}_G$ is a shape equivalence iff $(\pi_\mathcal{E})_!X \to (\pi_\mathcal{E})Y$ is an isomorphism in $\mathcal{S}_{\pi_!G}$, iff $(\pi_\mathcal{E})_!X \to (\pi_\mathcal{E})Y$ is an isomorphism in $\mathcal{S}$, iff the underlying morphism of $X \to Y$ in $\mathcal{E}$ is a shape equivalence. 
\end{proof} 

\subsubsection{Transferring model structures to locally contractible ($\infty$-)toposes}	\label{Transferring model structures}

Here we finally construct model structures on locally contractible $\infty$-toposes and test toposes for which the weak equivalences are the shape equivalences. We begin by recalling some basic theory of cofibrantly generated model $\infty$-categories, in particular,  two theorems on constructing and transferring cofibrantly generated model structures, respectively, which are classical in the ordinary categorical setting.  Then, for any local test category $A$ we extend the canonical model structure on $\widehat{A}$ to $[A^{\op}, \mathcal{S}]$. Finally, we transfer the model structure on $[A^{\op}, \mathcal{S}]$ to locally contractible $\infty$-toposes, and the model structure on $\widehat{A}$ to test toposes. 

\begin{definition} 
A complete and cocomplete model $\infty$-category $M$ is \Emph{cofibrantly generated} if there exist \emph{sets} $I,J$ of morphisms in $M$ such that 
\begin{enumerate} 
\item	$C = {}^{\boxslash}(I^{\boxslash})$, 
\item	$C \cap W = {}^{\boxslash}(J^{\boxslash})$, and
\item	$I$ and $J$ permit the small object argument (see \cite[\S 3.5]{aMG2014}). 
\end{enumerate} \ \qede
\end{definition} 

\begin{definition} 
Let $M$ be a cofibrantly generated model $\infty$-category, then a \Emph{relative $I$-complex ($J$-complex)} is any morphism which can be written as the transfinite composition (see \cite[Def.~1.4.2]{DAGX}) of pushouts of morphisms in $I$ ($J$). \qede
\end{definition} 

By \cite[Prop.~1.4.7]{DAGX} any \emph{set} of morphisms in a presentable $\infty$-category admits the small object argument. 

\begin{warning} 
Let $I$ be a set of morphisms in an $\infty$-category $C$ satisfying the small objects argument, then the attendant factorisation of any morphism in $C$ into a relative $I$-complex followed by a morphism in $I^{\boxslash}$ is \emph{not} functorial. See \cite[Warning~1.4.8]{DAGX} and \cite[Rmk.~3.7]{aMG2014}. \qede
\end{warning} 

\begin{proposition} \label{hcg} 
Let $M$ be a presentable $\infty$-category, let $W \subseteq M$ be a subcategory, which is closed under retracts, and satisfies the 2-out-of-3 property. Suppose that $I$ and $J$ are \emph{sets} of homotopy classes of maps such that 
\begin{enumerate}[label = {\normalfont (\alph*)}]
\item	\label{hcga}	$^{\boxslash}(J^{\boxslash}) \subseteq \, ^{\boxslash}(I^{\boxslash})  \cap W$ 
\item	\label{hcgb}	$I^{\boxslash} \subseteq J^{\boxslash} \cap W$
\item	and either 
	\begin{itemize}
	\item[{\normalfont (c$_1$)}]	\label{c1}	$^{\boxslash}(J^{\boxslash}) = \, ^{\boxslash}(I^{\boxslash})  \cap W$, or % {\normalfont(}i.e., $^{\boxslash}(J^{\boxslash}) = \, ^{\boxslash}(I^{\boxslash})  \cap W\,${\normalfont )}, or
	\item[{\normalfont (c$_1$)}]	$ I^{\boxslash} = J^{\boxslash} \cap W$, % {\normalfont (}i.e.\ $ I^{\boxslash} = J^{\boxslash} \cap W${\normalfont )},
	\end{itemize}
\end{enumerate} 
then the $I$ and $J$ define a cofibrantly generated model structure on $M$ whose weak equivalences are $W$.
\end{proposition} 

\begin{proof} 
In either case, by \cite[Prop.~1.4.7]{DAGX} the pairs ($^{\boxslash}(J^{\boxslash}),{}^{\boxslash}(I^{\boxslash})  \cap W)$ and $(I^{\boxslash}, J^{\boxslash} \cap W)$ satisfy the conditions of Proposition \ref{wmsa}.
\end{proof} 

\begin{proposition} \label{Hirschhorn} 
Let $M$ be a cofibrantly generated model $\infty$-category with generating cofibrations $I$ and generating trivial cofibrations $J$, let $N$ be a presentable $\infty$-category, and consider an adjunction  $\cadjunction{f: M}{N : u}$. If the functor $u$ takes relative $fJ$-cell complexes to weak equivalences, then 
\begin{itemize}
\item[{\normalfont (1)}]	\label{h1}	the $\infty$-category $N$ admits a cofibrantly generated model structure whose weak equivalences are those morphisms carried to weak equivalences by $u$, and with generating cofibrations and trivial cofibrations given by $fI$ and $fJ$ respectively, and
\item[{\normalfont (2)}]	\label{h2}	the adjunction $\cadjunction{f: M}{N : u}$ is a Quillen adjunction. 
\end{itemize}	
\end{proposition} 

\begin{proof} 
The condition in the proposition precisely ensures that $fI$ and $fJ$ satisfy \ref{hcga} of Proposition \ref{hcg}, and the two conditions \ref{hcgb} and (c$_1$) are satisfied by Proposition \ref{wfsa}. 
\end{proof}

We can now extend the canonical model structure. The following proposition generalises \cite[Th.~4.4]{aMG2014}.

\begin{proposition}	\label{EMG} 
Let $A$ be a local test category, then there exists a (necessarily unique) cofibrantly generated model structure on $[A^{\op}, \mathcal{S}]$ whose weak equivalences are the shape equivalences, and whose trivial fibrations are characterised by having the right lifting property against the monomorphisms in $\widehat{A}$. \par
Furthermore, if $I$ and $J$ are generating cofibrations and trivial cofibrations, respectively, of the canonical model structure on $\widehat{A}$, then these generate the model structure on $[A^{\op}, \mathcal{S}]$. 
\end{proposition} 

\begin{proof}
Let $I$ and $J$ be generating cofibrations and trivial cofibrations, respectively, of the canonical model structure on $\widehat{A}$. Any morphism $X \to Y$ which lifts against all monomorphisms in $\widehat{A}$ clearly lifts against $I$. Conversely, assume that $X \to Y$ lifts against $I$. Any monomorphism may be constructed as a retract of an $I$-cellular map which by Lemmas \ref{pushout} - \ref{retract} is again a morphism in $\widehat{A}$, so that $X \to Y$ lifts against all monomorphisms between objects in $\widehat{A}$ by \cite[Cor.~1.4.10]{DAGX}.

We will now verify that the set of shape equivalences $W$ together with $I,J$ satisfy \ref{hcga}, and \ref{hcgb}, (c$_2$) of Proposition \ref{hcg}. \\\\
\noindent\underline{Proof of \ref{hcga}:}	By Lemmas \ref{pushout} - \ref{retract} all colimits involved in constructing the morphisms in $^{\boxslash}(J^{\boxslash})$ are homotopy colimits. As all morphisms in $J$ are weak equivalences, the morphisms in $^{\boxslash}(J^{\boxslash})$ must be weak equivalences. \\\\
\noindent\underline{Proof of \ref{hcgb}:}	The inclusion $I^{\boxslash} \subseteq J^{\boxslash}$ is clear as $J \subseteq\, ^{\boxslash}(I^{\boxslash})$, so we need to show $I^{\boxslash} \subseteq W$. So, let $X \to Y$ be a morphism in $I^{\boxslash}$. \par
First, we show that it is enough to prove the statement in the case when $Y$ is representable. For all objects $a$ in $A$, and all maps $a \to Y$ the morphism $a \times_YX \to X$ is in $I^{\boxslash}$. If these morphisms are in $W$, then $X \to Y$ is in $W$ by faithful descent, as the morphism can be written as a colimit indexed by  $A_{/Y} \to A$.    	\par
So, assume that $Y$ is representable. As a morphism in $A_{/Y}$ is a monomorphism iff it is a monomorphism in $A$, we may furthermore assume that $A$ has a final object, and that $Y$ is such a final object. \par
As the shape of the presheaf represented by the final object in $A$ is contractible, it is enough to show that the shape of $X$ is contractible. Now, the shape of $X$ is given by $(A_{/X})_\simeq \, \simeq \Ex^\infty A_{/X}$, so that any map $S^k \to \pi_! X$ ($k \geq 0$) may be represented by a map $\Sd^n\partial\Delta^k \to A_{/X}$ for some $n \geq 0$. If $n \geq 1$, then $\Sd^n$ is a finite poset, and therefore a finite direct category. We will show that for any finite direct category $I$ and any functor $I \to A_{/X}$ we obtain a factorisation 
% https://q.uiver.app/?q=WzAsMyxbMCwwLCJJX3tcXHNpbWVxfSJdLFsxLDAsIihBX3svWH0pX1xcc2ltZXEiXSxbMCwxLCIqIl0sWzAsMV0sWzAsMl0sWzIsMSwiIiwyLHsic3R5bGUiOnsiYm9keSI6eyJuYW1lIjoiZGFzaGVkIn19fV1d
\begin{equation} \label{lift} 
\begin{tikzcd}
	{I_{\simeq}} & {(A_{/X})_\simeq} \\
	{*}
	\arrow[from=1-1, to=1-2]
	\arrow[from=1-1, to=2-1]
	\arrow[dashed, from=2-1, to=1-2]
\end{tikzcd}
\end{equation} 
Consider the diagram $f: I \to A$, and take a Reedy cofibrant replacement $\tilde{f} \xrightarrow{\sim} f$ in $\widehat{A}$ (see \cite[Prop.~7.4.19]{dcC2019}), then by an inductive application of \cite[Cor.~7.4.4]{dcC2019} and Lemmas \ref{pushout} \& \ref{filtered} we see that the colimit of $\tilde{f}$ is $0$-truncated. The map $I_\simeq \to {(A_{/X})_\simeq}$ corresponds to the map $\pi_! \colim \tilde{f} \to \pi_! X$. Consider a factorisation $\colim \tilde{f} \to c \to 1$ in $\widehat{A}$, where $\colim \tilde{f} \to c$ is a monomorphism, and $c \to 1$ is a trivial fibration, and thus a weak equivalence. By our assumption on $X$, we obtain a lift 
% https://q.uiver.app/?q=WzAsMyxbMCwwLCJcXGNvbGltIFxcdGlsZGV7Zn0iXSxbMSwwLCJYIl0sWzAsMSwiYyJdLFswLDFdLFswLDIsIiIsMix7InN0eWxlIjp7InRhaWwiOnsibmFtZSI6Imhvb2siLCJzaWRlIjoidG9wIn19fV0sWzIsMSwiIiwyLHsic3R5bGUiOnsiYm9keSI6eyJuYW1lIjoiZGFzaGVkIn19fV1d
\[\begin{tikzcd}
	{\colim \tilde{f}} & X \\
	c
	\arrow[from=1-1, to=1-2]
	\arrow[hook, from=1-1, to=2-1]
	\arrow[dashed, from=2-1, to=1-2]
\end{tikzcd}\]    \\
Taking the shape of this diagram yields the desired lift in (\ref{lift}). \par
\underline{Proof of (c$_2$):}	The proof of this fact for $A = \Delta$ is given in \cite[Prop.~7.9]{aMG2014}, and may be interpreted verbatim in our setting. 	
\end{proof}

We now finally construct model structures on locally contractible $\infty$-toposes and on test toposes. Both of these theorems should be compared to Theorem \ref{nerves}. 

\begin{proposition}	\label{infinity transfer} 
Let 
	\begin{enumerate}[label = \normalfont{(\roman*)}]
	\item	$\mathcal{E}$ be an $\infty$-topos, generated under small colimits by a small subcategory $C$ consisting of contractible objects (so that $\mathcal{E}$ is locally contractible), 
	\item	$A$, a small $\infty$-category, and 
	\item	 $u: A \to C$, a functor. 
	\end{enumerate} 
Assume that 
	\begin{enumerate}[label = \normalfont{(\alph*)}]
	\item	$u: A \to C$ is initial, and that 
	\item	$[A^{\op}, \mathcal{S}]$ admits a cofibrantly generated model structure with sets $I$ and $J$ of, respectively, generating cofibrations and generating trivial cofibrations, and in which the weak equivalences are the shape equivalences,  
	\end{enumerate}
there exists a cofibrantly generated model structure on $\mathcal{E}$ such that 
	\begin{enumerate}[label = \normalfont{(\arabic*)}] 
	\item	the weak equivalences are precisely the shape equivalences, 
	\item	the sets $u_!I$ and $u_!J$ are generating sets for the cofibrations and trivial cofibrations, respectively, and
	\item	the adjunction $\adjunction{u_!: [A^{\op}, \mathcal{S}]}	{\mathcal{E}: u^*}$  is a Quillen equivalence. 
	\end{enumerate} 
If moreover 
	\begin{itemize}
	\item[{\normalfont (c)}]	the inclusions $u{\ell} \hookrightarrow ud$ admit retracts for all morphisms ${\ell} \hookrightarrow d$ in $J$,
	\end{itemize}
then
	\begin{itemize}
	\item[{\normalfont (4)}]	all objects in the resulting model structure on $\mathcal{E}$ are fibrant. 
	\end{itemize}
\end{proposition} 

\begin{proof} 
We will use Proposition \ref{Hirschhorn} to transfer the model structure on $[A^{\op}, \mathcal{S}]$ to $\mathcal{E}$. By Theorem \ref{nerves} the weak equivalences in $\mathcal{E}$ created by $u^*$ are precisely the shape equivalences. The condition in the statement of Proposition \ref{Hirschhorn} is then trivially satisfied, because the shape functor $\pi_!: \mathcal{E} \to \mathcal{S}$ commutes with all colimits, so that we obtain a Quillen adjunction, which is a Quillen equivalence, again by Theorem \ref{nerves}. 
Conclusion (4) is obvious. 
\end{proof} 

\begin{theorem} \label{transfer}
Let 
	\begin{enumerate}[label = \normalfont{(\roman*)}]
	\item	$\mathcal{E}$ be an $\infty$-topos, generated under small colimits by a small subcategory $C$ of $\mathcal{E}_{\leq 0}$ consisting of contractible objects, 
	\item	$A$, a local test category, and 
	\item	 $u: A \to C$, a functor. 
	\end{enumerate} 
Assume that 
	\begin{enumerate}[label = \normalfont{(\alph*)}]
	\item	\label{lmsc2}	$u: A \to C$ is initial,
	\item	\label{lmsc3}	$u_!: [A^{\op}, \mathcal{S}] \to \mathcal{E}$ preserves $0$-truncated objects, and
	\item	\label{lmsc4}	$u_!: \widehat{A} \to \mathcal{E}_{\leq 0}$ preserves monomorphisms, 
	\end{enumerate}
then for any sets $I$ and $J$ of, respectively, generating cofibrations and generating trivial cofibrations for the canonical model structure on $\widehat{A}$, there exists a cofibrantly generated model structure on $\mathcal{E}_{\leq 0}$ such that 
	\begin{enumerate}[label = \normalfont{(\arabic*)}] 
	\item	\label{tms1}	the weak equivalences are precisely the shape equivalences, 
	\item	\label{tms2}	the sets $u_!I$ and $u_!J$ are generating sets for the cofibrations and trivial cofibrations, respectively, and
	\item	\label{tms3}	the adjunction $\adjunction{u_!: \widehat{A}}	{\mathcal{E}_{\leq 0}: u^*}$ is a Quillen equivalence. 
	\end{enumerate} 
If moreover 
	\begin{itemize}
	\item[{\normalfont (e)}]	the inclusions $u{\ell} \hookrightarrow ud$ admit retracts for all morphisms ${\ell} \hookrightarrow d$ in $J$,
	\end{itemize}
then
	\begin{itemize}
	\item[{\normalfont (4)}]	\label{cms4}	all objects in the resulting model structure on $\mathcal{E}_{\leq 0}$ are fibrant. 
	\end{itemize}
\end{theorem} 

The proof of Theorem \ref{transfer} is very similar to the proof of Proposition \ref{infinity transfer}. 

\begin{proof} 
The shape equivalences in $\mathcal{E}_{\leq 0}$ are created by $u^*$ by Theorem \ref{nerves}. The conditions of Proposition \ref{Hirschhorn} are satisfied by assumptions \ref{lmsc3} \& \ref{lmsc4} and Corollary \ref{truncated colimits}, so that $u_! \dashv u^*$ is a Quillen adjunction. By \ref{lmsc3} the unit and counit of the $\cadjunction{u_!: \widehat{A}}	{\mathcal{E}_{\leq 0}: u^*}$ coincide with the ones of $\cadjunction{u_!: [A^{\op}, \mathcal{S}]}	{\mathcal{E}: u^*}$, so that $u_! \dashv u^*$ is a Quillen equivalence by Proposition \ref{infinity transfer}. 

Conclusion (4) is obvious. 
\end{proof} 

We conclude this section with a discussion of some criteria for checking conditions \ref{lmsc2} - \ref{lmsc4} in Theorem \ref{transfer}. We have already seen that condition \ref{lmsc2} may be checked using Propositions \ref{homotopy initial} \& \ref{monoidal homotopy initial}. We add two simple criteria for verifying \ref{lmsc3} \& \ref{lmsc4} of Theorem \ref{transfer} in the case when $A = \Delta, \Cube$.

\begin{proposition}	\label{simp mono}  
Let $\mathcal{E}$ be an $\infty$-topos, and $u: \Delta \to \mathcal{E}_{\leq 0}$, a functor, and assume that the unique cocontinuous extension $u_!: [\Delta^{\op},\mathcal{S}] \to \mathcal{E}$ carries 
% https://q.uiver.app/#q=WzAsNCxbMCwwLCJcXHZhcm5vdGhpbmciXSxbMSwwLCJcXERlbHRhXntcXHsxXFx9fSJdLFswLDEsIlxcRGVsdGFee1xcezBcXH19Il0sWzEsMSwiXFxEZWx0YV4xIl0sWzAsMV0sWzAsMl0sWzIsM10sWzEsM11d
\[\begin{tikzcd}
	\varnothing & {\Delta^{\{1\}}} \\
	{\Delta^{\{0\}}} & {\Delta^1}
	\arrow[from=1-1, to=1-2]
	\arrow[from=1-1, to=2-1]
	\arrow[from=2-1, to=2-2]
	\arrow[from=1-2, to=2-2]
\end{tikzcd}\]
to a pullback, then 
	\begin{enumerate}[label = (\arabic*)]
	\item \label{ts1}	$u_!: [\Delta^{\op},\mathcal{S}] \to \mathcal{E}$ preserves $0$-truncated objects, and the restricted functor
	\item	\label{ts2} $u_!:  \widehat{\Delta} \to \mathcal{E}_{\leq 0}$ preserves monomorphisms. 
	\end{enumerate} 
\end{proposition} 

\begin{proof} 
By \cite[Lm.~2.1.9]{dcC2006} and the assumption in the statement of the proposition, the \v{C}ech nerve of the map $\coprod_{i = 0}^n u_! \Delta^{n-1} \xrightarrow{(d_0, \ldots, d_n)} u_! \Delta^n$ is given by the image under $u_!$ of the \v{C}ech nerve of $\coprod_{i = 0}^n \Delta^{n-1} \xrightarrow{(d_0, \ldots, d_n)}\Delta^n$, so that $u_! \partial \Delta^n \to u_! \Delta^n$ is monomorphism for all $n \geq 0$. Then \ref{ts1} follows from Propositions \ref{pushout} - \ref{coproducts} and the way in which $u_!X$ is constructed via cell attachments for any simplicial set $X$. Finally, \ref{ts2} follows from the fact that the monomorphism $X \to Y$ in $\widehat{\Delta}$ is obtained via a sequence of cell attachments, and the fact that monomorphisms are preserved under pushouts and filtered colimits. 
\end{proof} 

The proof of the following proposition is the same as the previous proof, except that it relies on \cite[Lm.~8.4.18]{dcC2006} instead of \cite[Lm.~2.1.9]{dcC2006}. 

\begin{proposition}	\label{cube mono}
Let $\mathcal{E}$ be an $\infty$-topos, and $u: \Cube \to \mathcal{E}_{\leq 0}$, a functor, and assume that the unique cocontinuous extension $u_!: [\Cube^{\;\! \op},\mathcal{S}] \to \mathcal{E}$ carries $(\delta_i^0,\delta_i^1): \Cube^{\;\! n-1} \sqcup \Cube^{\;\! n-1} \hookrightarrow \Cube^{\; \! n}$ to a monomorphism for all $n \geq i \geq 1$, then 
	\begin{enumerate}[label = (\arabic*)]
	\item \label{ts1}	$u_!: [\Cube^{\;\! \op},\mathcal{S}] \to \mathcal{E}$ preserves $0$-truncated objects, and the restricted functor
	\item	\label{ts2} $u_!:  \widehat{\Cube} \to \mathcal{E}_{\leq 0}$ preserves monomorphisms.  \\ \qed
	\end{enumerate} 
\end{proposition} 

The asymmetry between Propositions \ref{simp mono} \& \ref{cube mono} disappears in the following situation: 

\begin{corollary}	\label{cube trunc mono} 
Let $\mathcal{E}$ be an $\infty$-topos, and $u: \Cube \to \mathcal{E}_{\leq 0}$, a monoidal functor, and assume that the unique cocontinuous extension $u_!: [\Cube^{\;\! \op},\mathcal{S}] \to \mathcal{E}$ carries 
% https://q.uiver.app/#q=WzAsNCxbMCwwLCJcXHZhcm5vdGhpbmciXSxbMSwwLCJcXEN1YmVee1xcezFcXH19Il0sWzAsMSwiXFxDdWJlXntcXHswXFx9fSJdLFsxLDEsIlxcQ3ViZV4xIl0sWzAsMV0sWzAsMl0sWzIsM10sWzEsM11d
\[\begin{tikzcd}
	\varnothing & {\Cube^{\;\! \{1\}}} \\
	{\Cube^{\;\! \{0\}}} & {\Cube^{\; \! 1}}
	\arrow[from=1-1, to=1-2]
	\arrow[from=1-1, to=2-1]
	\arrow[from=2-1, to=2-2]
	\arrow[from=1-2, to=2-2]
\end{tikzcd}\]
to a pullback, then 
	\begin{enumerate}[label = (\arabic*)]
	\item \label{ts1}	$u_!: [\Cube^{\;\! \op},\mathcal{S}] \to \mathcal{E}$ preserves $0$-truncated objects, and the restricted functor
	\item	\label{ts2} $u_!:  \widehat{\Cube} \to \mathcal{E}_{\leq 0}$ preserves monomorphisms.  \\ \qed
	\end{enumerate} 
\end{corollary} 

\begin{proof} 
By assumption the morphism $(\delta_1^0,\delta_1^1): \Cube^{\;\! n-1} \sqcup \Cube^{\;\! 0} \hookrightarrow \Cube^{\; \! 1}$ is carried to a monomorphism, and the maps $(\delta_i^0,\delta_i^1): \Cube^{\;\! n-1} \sqcup \Cube^{\;\! n-1} \hookrightarrow \Cube^{\; \! n}$ may be rewritten as $\id_{\; \! \scriptCube^{ \; \! i-1}} \times (\delta_1^0,\delta_1^1) \times \id_{\; \! \scriptCube^{\; \! n-i}}$, so the corollary follows from Proposition \ref{cube mono}. 
\end{proof}

\part{Differentiable sheaves}	\label{Differentiable sheaves}

\section{Basic definitions and properties of differentiable sheaves}	\label{basic} 

We formally define the $\infty$-topos $\Diff^r$ of $r$-times differentiable sheaves, and apply the machinery of \S \ref{Fractured toposes} to exhibit $\Diff^r$ as a fractured $\infty$-topos, and derive some of its basic properties in \S \ref{dss}. Then, in \S \ref{Diffeological spaces} we collect some basic facts about diffeological spaces, and briefly discuss the classification of diffeological bundles. Finally, in \S \ref{compact manifolds} we give our first application of the fractured $\infty$-topos structure on $\Diff^r$ and show that closed manifolds are categorically compact in $\Diff^r$. Moreover, we show that manifolds with non-empty boundary or corners are \emph{not} categorically compact. 

\subsection{Differentiable sheaves}	\label{dss} 

Recall that 
	\begin{enumerate}[label = \arabic*.]
	\item	$\Mfd^r$ denotes the category of $r$-times differentiable (2$^{\text{nd}}$-countable, Hausdorff) manifolds and $r$-times differentiable maps, and
	\item	$\Cart^r$ denotes the full subcategory of $\Mfd^r$ spanned by the spaces of $\mathbf{R}^n$ ($0 \leq n < \infty$).
	\end{enumerate} 
On each of these small categories we denote by $\tau$ the Grothendieck topology in which a sieve on a manifold is a covering sieve iff it contains a covering consisting of jointly surjective open embeddings. %\par
	\begin{enumerate}[label = \arabic*.]
	\item	$\Mfd_{\et}^r$ denotes the category of $r$-differentiable manifolds and $r$-differentiable open embeddings. 
	\item	$\Cart_{\et}^r$ denotes the full subcategory of $\Mfd_{\et}^r$ spanned by the spaces for $\mathbf{R}^n$ ($0 \leq n < \infty$).
	\end{enumerate}	
On each of these (essentially) small subcategories we denote the restriction of $\tau$ by $\tau_{\et}$. 

\begin{definition}	\label{maindef}
An $\mathcal{S}$-valued sheaf on $\Cart^r$ is an \Emph{$r$-times differentiable sheaf}, and the category thereof is denoted by $\Diff^r$. Similarly, an $\mathcal{S}$-valued sheaf on $\Cart_{\et}^r$ is an \Emph{étale $r$-times differentiable stack}, and the category thereof is denoted by $\Diff_{\et}^r$.  	\qede
\end{definition} 

We provide a new comparison theorem for $\infty$-toposes of sheaves on \emph{ordinary} sites. See also \cite[Lm.~C.3]{mH2014}. 

\begin{proposition}	\label{comparison theorem} 
Let $C'$ be an \emph{ordinary} site, and let $u: C \hookrightarrow C'$ be a small subcategory endowed with the induced Grothendieck topology. Denote by $\mathcal{E}$ and $\mathcal{E}'$ the $\infty$-toposes of sheaves on $C$ and $C'$ , respectively. If $\mathcal{E}'$ is hypercomplete, and if every object in $C'$ may be covered by objects in $C$, then $u^*: [C^{\op}, \mathcal{S}] \leftarrow [(C')^{\op}, \mathcal{S}]$ restricts to an equivalence of $\infty$-categories $\mathcal{E} \xleftarrow{\simeq} \mathcal{E}'$. 
\end{proposition} 

\begin{proof} 
By the proof of \cite[Th.~III.4.1]{SGA4.1} the functor $u$ is both continuous and cocontinuous (\cite[Defs.~III.1.1~\&~III.2.1]{SGA4.1}), so that both $u_*$ and $u^*$ preserve local equivalences of set-valued sheaves, and we obtain an induced adjunction $\copadjunction{\mathcal{E}}	{\mathcal{E}'}$.  The counit of this adjunction is an equivalence since $u$ is fully faithful, and the left adjoint is conservative by \cite[Prop.~20.4.5.1]{jL2017}, so that the adjunction is an equivalence. 
\end{proof} 

%\begin{remark} 
%Let $\mathcal{E}$ be the ordinary topos of set-valued sheaves on a (not necessarily finitely complete) site $C$, then an argument similar to the one used in Proposition \ref{comparison theorem} shows that the hypercompletion the $\infty$-topos universally associated to $\mathcal{E}$ (in the sense of {\normalfont\cite[Prop.~6.4.5.7]{jL2009}}) is canonically equivalent to the hypercompletion of the $\infty$-topos of sheaves on $C$.	\qede
%\end{remark} 

Combining Proposition \ref{comparison theorem} with Corollary \ref{diff hypercomplete} below, we obtain the following result. 

\begin{proposition}	\label{ccm}
Denote by $u: \Cart^r \hookrightarrow \Mfd^r$ the canonical inclusion, then the functor $[(\Man^r)^{\op}, \mathcal{S}] \leftarrow [(\Cart^r)^{\op}, \mathcal{S}]: u^*$ restricts to an equivalence between $\Diff^r$ and the $\infty$-topos of sheaves on $\Man^r$. \qed
\end{proposition} 

\begin{remark} 
We have opted not to use good open covers in the proof of Proposition \ref{ccm}, as the question of their existence and whether they refine all open covers is subtle (as discussed at length in \cite{nlab:good_open_cover}). Moreover, arguments using good open covers may not carry over to other settings such as real analytic or complex, which we hope to explore in the future. \qede
\end{remark} 

\begin{lemma}
The triple $(\mathbf{Mfd}^r, \mathbf{Mfd}_{\et}^r, \tau)$ is a geometric site. 
\end{lemma}

\begin{proof}
Axioms \ref{as1} - \ref{as3} are clear. To prove \ref{as4}, consider the diagram
% https://q.uiver.app/?q=WzAsNixbMCwwLCJWJyJdLFsxLDAsIlUnIl0sWzIsMCwiViciXSxbMCwxLCJWIl0sWzEsMSwiVSJdLFsyLDEsIlYiXSxbMCwxLCIiLDAseyJzdHlsZSI6eyJ0YWlsIjp7Im5hbWUiOiJob29rIiwic2lkZSI6InRvcCJ9fX1dLFsxLDIsIiIsMCx7InN0eWxlIjp7ImhlYWQiOnsibmFtZSI6ImVwaSJ9fX1dLFszLDQsIiIsMCx7InN0eWxlIjp7InRhaWwiOnsibmFtZSI6Imhvb2siLCJzaWRlIjoidG9wIn19fV0sWzQsNSwiIiwwLHsic3R5bGUiOnsiaGVhZCI6eyJuYW1lIjoiZXBpIn19fV0sWzQsMSwiIiwxLHsic3R5bGUiOnsidGFpbCI6eyJuYW1lIjoiaG9vayIsInNpZGUiOiJ0b3AifX19XSxbMywwXSxbNSwyXV0=
\begin{equation}	\label{open retract}
\begin{tikzcd}
	{V'} & {U'} & {V'} \\
	V & U & V,
	\arrow[hook, from=1-1, to=1-2]
	\arrow[two heads, from=1-2, to=1-3]
	\arrow[hook, from=2-1, to=2-2]
	\arrow[two heads, from=2-2, to=2-3]
	\arrow[hook, from=2-2, to=1-2]
	\arrow[from=2-1, to=1-1]
	\arrow[from=2-3, to=1-3]
\end{tikzcd}
\end{equation} 
where $U \hookrightarrow U'$ is an open subset inclusion. Axiom \ref{as4} follows from \ref{as2}, after proving the following claim: \par
\underline{Claim:}	The leftmost square in (\ref{open retract}) is a pullback. \par
First, as monomorphisms have the left cancelling property, the map $V \to V'$ is a monomorphism.  Let $y' \in V' \cap U$, then $y'$ coincides with its image under $U \to V$, which shows that the leftmost  square induces a pullback on underlying sets. Next, consider a commutative  square
% https://q.uiver.app/?q=WzAsNCxbMCwwLCJWJyJdLFsxLDAsIlUnIl0sWzEsMSwiVSwiXSxbMCwxLCJXIl0sWzAsMSwiIiwwLHsic3R5bGUiOnsidGFpbCI6eyJuYW1lIjoiaG9vayIsInNpZGUiOiJ0b3AifX19XSxbMiwxLCIiLDIseyJzdHlsZSI6eyJ0YWlsIjp7Im5hbWUiOiJob29rIiwic2lkZSI6InRvcCJ9fX1dLFszLDBdLFszLDJdXQ==
\[\begin{tikzcd}
	{V'} & {U'} \\
	W & {U,}
	\arrow[hook, from=1-1, to=1-2]
	\arrow[hook, from=2-2, to=1-2]
	\arrow[from=2-1, to=1-1]
	\arrow[from=2-1, to=2-2]
\end{tikzcd}\]
then the canonical map of sets $W \to V$ is smooth, as it may  be written as the composition of $W \to U \to V$. 
\end{proof}

Applying Theorem \ref{fractured sheaves} we obtain the key result of this subsection: 

\begin{theorem} 
The $\infty$-category $\Diff^r$ is a fractured $\infty$-topos, whose $\infty$-topos of corporeal objects is given by $\Diff_{\et}^r$. \qed
\end{theorem} 

\begin{remark}	\label{mantop}
Observe that for any smooth manifold, the $\infty$-topos $(\Diff_{\et}^r)_{/M}$ is equivalent to the $\infty$-topos of sheaves on underlying topological space of $M$ (and is thus independent of $r$). \qede
\end{remark} 

By \cite[Th.~C.3]{dC2016o} and Proposition \ref{pot}, $\Diff_{\et}^\infty$ is equivalent to the $\infty$-topos of sheaves on the category of smooth manifolds and local diffeomorphisms, so that the fractured $\infty$-topos $\Diff^\infty$ coincides with the one considered in \cite[\S 6.1]{dC2020}. Carchedi moreover shows that $\Diff_{\et}^\infty$ coincides with 
	\begin{enumerate}
	\item	the $\infty$-category of $\infty$-toposes, locally ringed in $\mathbf{R}$-algebras, which can \'{e}tale locally be covered by manifolds. 
	\item	sheaves in $\Diff^r$ which may be presented by \'{e}tale groupoids.
	\end{enumerate} 
It is prima facie surprising that these two $\infty$-categories are $\infty$-toposes. These observations (in the generality of \cite[\S 5]{dC2020}) were key to the development of fractured $\infty$-toposes (see \cite[Rmk.~20.0.0.2]{jL2017}). 

\begin{warning} 
The functor $j_!: \Diff_{\et}^r \to \Diff^r$ does not preserve $0$-truncated objects. For example, $\mathbf{1}_{\Diff_{\et}^r}$ is mapped to the Haefliger stack (see \S \ref{The shape of the Haefliger stack}), which is $1$-truncated, but not $0$-truncated.  \qede
\end{warning} 

We finish this subsection by proving some basic properties about the fractured $\infty$-topos $\Diff^r$. 

\begin{proposition} 
The $\infty$-topos $\Diff^r$ is local. 
\end{proposition} 

\begin{proof} 
This follows immediately from Proposition \ref{fractured local}. 
\end{proof} 

\begin{lemma}	\label{cd}
Let $M$ be a connected, paracompact Hausdorff $r$-times differentiable manifold $M$, then the covering dimension of $M$ is $\leq \dim M$. 
\end{lemma} 

\begin{proof} 
By \cite[\S II.6.2]{EMS17} the covering dimension of $M$ is equal to the inductive dimension, which is $\leq \dim M$. 
\end{proof} 

\begin{proposition} 
For any manifold $M$ the $\infty$-topos $(\Diff_{\et}^r)_{/M}$ is hypercomplete. 
\end{proposition} 

\begin{proof} 
By Lemma \ref{cd}, \cite[Th.~7.2.3.6]{jL2009}, and Remark \ref{mantop} the $\infty$-topos $(\Diff_{\et}^r)_{/M}$ has homotopy dimension $\leq \dim M$, and is thus hypercomplete by \cite[Cor.~7.2.1.12]{jL2009}. 
\end{proof} 

By \cite[Lm.~A.3.9.]{jL2012} we obtain the following corollary:

\begin{corollary} 
For any $r$-times differentiable manifold $M$ the $\infty$-topos $(\Diff_{\et}^r)_{/M}$ has enough points. \qed
\end{corollary} 

\begin{proposition} 
The $\infty$-topos $\Diff_{\et}^r$ has enough points. 
\end{proposition} 

\begin{proof}
By Remark \ref{mantop} the adjunction $\cadjunction{x_*: \mathcal{S}}	{(\Diff_{\et}^r)_{/M}: x^*}$ at any point $x \in \mathbf{R}^d$ provides a point of $(\Diff_{\et}^r)_{/\mathbf{R}^d}$, and thus a point $\mathcal{S} \xrightarrow{x_*} (\Diff_{\et}^r)_{/\mathbf{R}^d} \xrightarrow{(\mathbf{R}^d)_*} \Diff_{\et}^r$. Thus, $\Diff_{\et}^r$ has enough points, as it is generated by the spaces $\mathbf{R}^d$ under colimits. 
\end{proof} 

By Corollary \ref{fep} we obtain the following result: 

\begin{corollary}[{\cite[Ex.~4.1.2]{dD1998}, \cite[Prop.~A.5.3]{aAaDpH2021}}]
The topos $\Diff^r$ has enough points. \qed
\end{corollary} 

By \cite[Rmk.~6.5.4.7]{jL2009} we obtain the following corollary: 

\begin{corollary}	\label{diff hypercomplete}
The $\infty$-topos $\Diff^r$ is hypercomplete. \qed 
\end{corollary}

\subsection{Diffeological spaces}	\label{Diffeological spaces}

Diffeological spaces are particularly nice ($0$-truncated) differentiable sheaves, which have a good notion of underlying set. We discuss their basic properties here and briefly discuss the classification of diffeological principal bundles in \S \ref{Diffeological spaces and descent}.
 
Observe that since $\Diff^r$ is local, so is $\Diff_{\leq 0}^r$. 

\begin{definition} 
A \Emph{diffeological space} $X$ is a concrete object in $\Diff_{\leq 0}^r$. A \Emph{plot} of $X$ is a map $\mathbf{R}^n \to X$. The collection of all plots of $X$ is called the \Emph{diffeology} of $X$.  \qede
\end{definition} 

\begin{convention} 
Let $X$ be a diffeological space, then plots of $X$ are usually identified with their images under $\pi_*$. \qede
\end{convention}

A diffeological space is thus a set $S$ together with a specified set of maps $\pi_*\mathbf{R}^d \to S$ for each $d \geq 0$ such that the resulting presheaf on $\Cart^r$ is a sheaf. 

\begin{remark}	\label{de} 
A monomorphism $X \hookrightarrow Y$ of diffeological spaces is an embedding (see Definition \ref{embedding}) iff for all $d \geq 0$ any map $\pi_* \mathbf{R}^d \to \pi_*X$ is a plot iff its composition with $\pi_*X \hookrightarrow \pi_*Y$ is. \qede
\end{remark}

\begin{definition} 
Let $Y$ be a diffeological space, and $X\subseteq \pi_* Y$, a subset, then the \Emph{subspace diffeology} on $X$ is the unique diffeology on $X$ in which a map $\pi_* \mathbf{R}^n \to \pi_*X$ is a plot iff it is a plot viewed as a map to $Y$. \qede
\end{definition} 

Thus, the subspace diffeology on $X$ is the unique diffeology making the inclusion $X \subseteq Y$ into an embedding. 

\begin{example}	\label{stupid simplex}
The standard simplex $\Delta^n$ with the subspace diffeology inherited from $\mathbf{R}^{n+1}$ is denoted by $\Delta^n_{\sub}$, and is referred to as the \Emph{closed $n$-simplex}.  \qede
\end{example} 

\begin{proposition}[{\cite[Lm.~2.64]{jW2012d}}]	\label{classical diff}
Write $\mathbf{R}_+^n \defeq \setbuilder{(x_1, \ldots, x_n) \in \mathbf{R}^n}	{x_1, \ldots, x_n \geq 0}$, and endow this set with the subspace diffeology inherited from $\mathbf{R}^n$. A map $f: \mathbf{R}_+^n \to \mathbf{R}$ is smooth iff it is the restriction of a smooth map $U \to \mathbf{R}$, where $U$ is an open neighbourhood of $\mathbf{R}_+^n$ in $\mathbf{R}^n$. 
\end{proposition} 

\begin{proof}
Write $s: \mathbf{R}^n \to \mathbf{R}^n, \, (x_1, \ldots, x_n) \mapsto (x_1^2, \ldots, x_n^2)$, then $f \circ s: \mathbf{R}^n \to \mathbf{R}$ is smooth, and moreover invariant under the action $(\mathbf{Z}^*)^n \times \mathbf{R}^n \to \mathbf{R}^n, \; \big( (\sigma_1, \ldots, \sigma_n), (x_1, \ldots, x_n) \big ) \mapsto (\sigma_1 x_1, \ldots, \sigma_n x_n)$ . By \cite{gS1975} there exists a smooth map $\widetilde{f}: \mathbf{R}^n \to \mathbf{R}$ such that $\widetilde{f} \circ s = f \circ s$. As $s$ restricts to a bijection on the underlying sets of $\mathbf{R}_+^n \to \mathbf{R}_+^n$, the maps $f$ and $\widetilde{f}$ agree on $\mathbf{R}_+^n$, so that $f$ is a restriction of $\widetilde{f}$. 
\end{proof}

\begin{corollary} 
Let $M$ be a smooth manifold with corners, and $N$ a smooth manifold \emph{without} corners, then a map $M \to N$ is smooth iff there exists a manifold $\widetilde{M}$ without corners, and an open embedding $M \subseteq \widetilde{M}$ and a smooth map $\widetilde{M} \to N$ which restricts to $M \to N$. In particular, a map $\Delta_{\sub}^n \to N$ is smooth iff there exists an open neighbourhood $U$ of $\Delta_{\sub}^n$ in $\mathbf{R}^{n + 1}$ and a smooth map $U \to N$ which restricts to $\Delta_{\sub}^n \to N$.  \qed
\end{corollary} 

\begin{example}	\label{simplex example}
Consider the unique cocontinuous functor $\widehat{\Delta} \to \Diff_{\leq 0}^r$ carrying $\Delta^n$ to $\Delta_{\sub}^n$ from Example \ref{stupid simplex} then this functor carries the simplicial sets $\partial \Delta^n$ and $\Lambda_k^n$ to diffeological spaces. These diffeological spaces \emph{are not} equipped with the subspace diffeology of $\Delta_{\sub}^n$. Write $\Lambda_1^2 \defeq u_! \Lambda_1^2$ and $\Lambda_{1, \sub}^2$ for the $1$-horn of $\Delta^2$ with the subdiffeology. For any path $[0,1] \to \Lambda_1^2$ passing through $\Delta^{\{1\}}$ for some time $t_0$ there must exist some open neighbourhood $U$ of $t_0$, which gets constantly mapped to $\Delta^{\{1\}}$.	\qede
\end{example} 

\subsubsection{Diffeological spaces and descent}	\label{Diffeological spaces and descent}

A recurring theme in this article is that many $\infty$-categories consisting of appropriate geometric objects may be profitably studied by embedding them into a suitable ambient $\infty$-topos. Applying this strategy to diffeological space enables us to recover the main theorem of \cite{eM2023} on the classification of diffeological principal bundles (in the sense of \cite[8.11]{pIZ2013}) as Corollary \ref{min} below. Thus result is not used in the rest of the article. 

\begin{theorem}	\label{db}
Let $G$ be a diffeological group, then for any pullback square 
% https://q.uiver.app/#q=WzAsNCxbMCwwLCJQIl0sWzAsMSwiQiJdLFsxLDEsIkJHIl0sWzEsMCwiXFxtYXRoYmZ7MX0iXSxbMCwxXSxbMSwyXSxbMywyXSxbMCwzXSxbMCwyLCIiLDEseyJzdHlsZSI6eyJuYW1lIjoiY29ybmVyIn19XV0=
\[\begin{tikzcd}
	P & {\mathbf{1}} \\
	B & BG,
	\arrow[from=1-1, to=2-1]
	\arrow[from=2-1, to=2-2]
	\arrow[from=1-2, to=2-2]
	\arrow[from=1-1, to=1-2]
	\arrow["\lrcorner"{anchor=center, pos=0.125}, draw=none, from=1-1, to=2-2]
\end{tikzcd}\]
the map $P \to B$ is a diffeological principal bundle. 
\end{theorem} 

\begin{proof}
First, we note that $P \to B$ is $0$-truncated, as it is the pullback of the $0$-truncated map $1 \to BG$. Next, for any plot $\mathbf{R}^d \to B$ the pullback $P|_{\mathbf{R}^d} \to \mathbf{R}^d$ admits local sections and is thus a diffeological principal $G$-bundle. By faithful descent, the space $P$ is the colimit of all spaces $P|_{\mathbf{R}^d}$. Denote by $P'$ the diffeological space universally associated to $P$, then by \cite[\S 6]{rGsL2012} the spaces  $P|_{\mathbf{R}^d}$ and  $P'|_{\mathbf{R}^d}$ are canonically isomorphic, so that $P'$ is likewise the colimit of all spaces $P|_{\mathbf{R}^d}$, and thus isomorphic to $P$. 
\end{proof} 

\begin{corollary}	\label{min}
The canonical functor from the groupoid of diffeological principal $G$-bundles on $B$ to $\Diff^r(B,BG)$ is an equivalence.
\end{corollary} 

\begin{proof} 
By \cite[Def.~5.1~\&~Rmk.~5.2]{dKjWsW2022} and by descent any diffeological principal $G$-bundle is classified by a map $B \to BG$. The functor from the groupoid of diffeological principal $G$-bundles on $B$ to $\Diff^r(B,BG)$ is fully faithful, and by Theorem \ref{db} it is essentially surjective. 
\end{proof} 

\subsection{Compact manifolds are compact}	\label{compact manifolds}

In this subsection we discuss the categorical compactness of manifolds in $\Diff^r$. We use the fractured $\infty$-topos structure to show that any closed manifold is categorically compact, by relating it to its underlying topological space, which is compact when viewed as an $\infty$-topos. Then we discuss the categorical compactness of other manifolds in \S \ref{On the compactness of non-closed manifolds}. 

\begin{theorem} \label{compact compact} 
Let $M$ be a closed manifold, then $\Diff^r(M,\emptyinput)$ commutes with filtered colimits.  
\end{theorem} 

\begin{proof} 
Let $A$ be a small filtered $\infty$-category, and let $X: A \to \Diff^r$ be a diagram, then 
$$	{\arraycolsep=1.5pt\def\arraystretch{1.6}
	\begin{array}{rcl}		
	\colim_\alpha \Diff^r(M,X_\alpha)	&	= 	&	\colim_\alpha \Diff^r(j_!M,X_\alpha)					\\
	{}							&	=	&	\colim_\alpha \Diff_{\et}^r(M, j^* X_\alpha)				\\
	{}							&	=	&	\colim_\alpha (\Diff_{\et})_{/M}^r(M, M \times j^* X_\alpha)	\\
	{}							&	\to	&	(\Diff_{\et})_{/M}^r(M, \colim_\alpha M \times j^* X_\alpha)	\\
	{}							&	=	&	(\Diff_{\et})_{/M}^r(M, M \times j^* \colim_\alpha X_\alpha)	\\
	{}							&	=	&	\Diff_{\et}^r(M, j^* \colim_\alpha X_\alpha)				\\
	{}							&	=	&	\Diff^r(j_! M, \colim_\alpha X_\alpha)					\\
	{}							&	=	&	\Diff^r( M, \colim_\alpha X_\alpha)					
	\end{array}
	}
$$
where the map in the fourth line is an isomorphism by \cite[Th.~7.3.1.16~\&~Rmk.~7.3.1.5]{jL2009}. 	
\end{proof}

\subsubsection{On the compactness of non-closed manifolds}	\label{On the compactness of non-closed manifolds}

\begin{proposition}	\label{ncnc} 
Any non-compact manifold is not categorically compact in $\Diff^r$. 
\end{proposition} 

\begin{proof} 
By assumption any such manifold $M$ admits a sequence $(x_i)$ such that $\{x_i\}$ is a closed subset of $M$, then the identity map $M \xrightarrow{\id} M = \colim_{n \in \mathbf{N}} M \setminus \{x_i\}_{i \geq n}$, does not factor through any of the manifolds $M \setminus \{x_i\}_{i \geq n}$ for $n \in \mathbf{N}$. 
\end{proof} 

\begin{theorem} 
Any connected manifold with non-empty corners is not categorically compact in $\Diff^r$ for $r \geq 2$. 
\end{theorem} 

\begin{proof} 
By Proposition \ref{ncnc} it is enough to prove the theorem for (topologically) compact manifolds. Moreover, as any finite coproduct of categorically compact manifolds is again compact, we may restrict to connected manifolds. We first prove the theorem for the unique compact $1$-dimensional manifold with corners, $[0,1]$, and then deduce the general case from this special case. 

Denote by $I$ the collection of all finite families of $r$-times differentiable maps of the form $\{\mathbf{R}^{d_i} \to [0,1]\}_{i = 0}^k$, where each map $\mathbf{R}^{d_i} \to [0,1]$ factors through either $[0,1)$ or $(0,1]$, then $I$ becomes a filtered poset under inclusion. For any member $C = \{\mathbf{R}^{d_i} \to [0,1]\}_{i = 0}^k$ of $I$ denote by $[0,1]_C$ the diffeological spaces consisting of the set $[0,1]$ together with the coarsest diffeology making all maps in $C$ differentiable, then by Proposition \ref{filtered mono} $\colim_{C \in E}[0,1]_C$ is diffeomorphic to $[0,1]$. We will show that the identity map $[0,1] \to [0,1]$ does not factor through $[0,1]_C$ for any $C \in I$, thus showing that $[0,1]$ is not compact.

Let us fix $C \in I$ as well as $f \in C$, which we assume w.l.o.g.\ factors through $[0,1)$. After reparametrising, we may view $f$ as a function $\mathbf{R}^d \to [0,\infty)$. We will show that for any $n > d$ the smooth map $\sigma_n: \mathbf{R}^n \to [0,\infty), \; (x_1, \ldots, x_{n}) \mapsto x_1^2 + \cdots + x_{n}^2$ does not factor through $f$ when restricted to any neighbourhood of $0 \in \mathbf{R}^n$. Thus, for sufficiently large $n$, $\sigma_n$ does not locally factor through any of the functions in $C$, so that $[0,1]_C$ has a strictly coarser diffeology than $[0,1]$. 

So, suppose to the contrary that there exists some neighbourhood $U$ of $0 \in \mathbf{R}^d$ such that $\sigma_n|_U$ factors through $f$ via a map $g: U \to \mathbf{R}^d$, and assume w.l.o.g.\ that $g(0) = 0$, then as $n > d$ the kernel of $d g|_0$ is non-trivial, and we may assume w.l.o.g.\ that $(1, 0, \ldots, 0)$ is in the kernel. Choose $\varepsilon > 0$ such that $(-\varepsilon, \varepsilon) \times \{0\} \times \cdots  \times \{0\} \subseteq U$, and write $h: (-\varepsilon, \varepsilon) \to \mathbf{R}^d$ for the map $x \mapsto g(x, 0, \ldots, 0)$, then, by assumption $f \circ h$ is given by $x \mapsto x^2$, so that $(f \circ h)'' = 2$.
On the other hand we have 
$$
{\arraycolsep= 3pt\def\arraystretch{3}
	\begin{array}{rc >{\displaystyle}l}
	(f \circ h)''(t)	&	=	&	\left( \sum_{i = 1}^n h_i'(t) \, \partial_i f\big (h_1(t), \ldots, h_n(t) \big ) \right)'	\\
	{}			&	=	&	\sum_{i = 1}^n \left ( h_i''(t) \, \partial_i  f\big (h_1(t), \ldots, h_n(t) \big ) + \sum_{j = 1}^n h_i'(t) \, h_j'(t) \, \partial_i \partial_j  f\big (h_1(t), \ldots, h_n(t) \big )	\right )
	\end{array}
}
$$
which evaluates to $0$ for $t = 0$ because for all $1 \leq i \leq n$ we have  $\partial_i  f\big (h_1(0), \ldots, h_n(0) \big ) = \partial_if(0)= 0$ (as $f$ has a local minimum at $0$) and $h_i'(0) = \mathrm{d}g|_0(1,0,\ldots,0) = 0$ by assumption, yielding a contradiction. 

Now, let $M$ be a manifold of dimension $ > 1$ with non-empty corners, then by assumption $M$ admits at least one chart $\mathbf{R}^{\dim M - 1} \times [0,\infty) \hookrightarrow M$. Consider the embedding 
$$\iota:  [0,1] \hookrightarrow \mathbf{R}^{\dim M - 1} \times [0,\infty), \quad \theta \mapsto (\cos \pi \theta, 0, \ldots, 0, \sin \pi \theta ).$$
Denote by $L$ the image of the embedding $[0,1] \xrightarrow{\iota} \mathbf{R}^{\dim M - 1} \hookrightarrow M$, and denote by $f: M \to \mathbf{R}$ an extension the diffeomorphism $L \to [0,1]$. With notation as above, assume that $f$ factors through $[0,1]_C \hookrightarrow [0,1]$ for some $C$ in $I$, then this implies that the identity map $\id_{[0,1]}: [0,1] \to [0,1]$ factors through $[0,1]_C \hookrightarrow [0,1]$, as $\id_{[0,1]}$ is equal to the composition of
$[0,1] \xrightarrow{\iota} \mathbf{R}^{\dim M - 1} \times [0,\infty) \hookrightarrow M \xrightarrow{f} [0,1],$
yielding a contradiction. 
\end{proof} 

It is possible to show that the category of $r$-times differentiable manifolds with corners, when equipped with the standard Grothendieck topology and open immersions, form a geometric site, yielding a fractured $\infty$-topos by Theorem \ref{fractured sheaves}, in which all topologically compact manifolds are categorically compact by the same argument as in Theorem \ref{compact compact}. Thus, in a sense, topologically compact manifolds with corners become categorically compact when corners are encoded as \emph{structure} rather than as a property. 

For $r = 0$ this geometric site yields a fractured $\infty$-topos, which is equivalent to $\Diff^0$, so in this case topologically compact manifolds with corners are also compact in $\Diff^0$. We don't know whether or not topologically compact $C^1$-manifolds with corners are categorically compact in $\Diff^1$. 

\section{Shapes, cofinality and differentiable sheaves}	\label{Shapes, cofinality and differentiable sheaves}

We first prove in \S \ref{dialct} that $\Diff^r$ is a locally contractible $\infty$-topos compatibly with its fractured $\infty$-topos structure, so that we may apply the technology of \S \ref{Shapes and cofinality} to $\Diff^r$: In \S \ref{underlying homotopy type} we prove Theorem \ref{homotopy types agree} from the introduction stating that various ways of extracting homotopy types from manifolds are equivalent. Finally, in \S \ref{locally contractible applications} we discuss some applications of the technology developed so far; we give a streamlined account of Carchedi's calculation of the shape of the Haefliger stack in \S \ref{The shape of the Haefliger stack}, and provide new, simpler proofs of classical descent theorems in algebraic topology such as Dugger and Isaksen's hypercovering theorem in \S \ref{Colimits of hypercoverings of topological spaces}. 

\subsection{$\Diff^r$ is a locally contractible $\infty$-topos}	\label{dialct}

\begin{lemma}	\label{diff loc contr}
The shape of $\mathbf{R}^d$ is contractible in $\Diff_{\et}^r$ for every $d \in \mathbf{N}$. 
\end{lemma} 

\begin{proof} 
The $\infty$-topos $(\Diff_{\et}^r)_{/\mathbf{R}^d}$ is equivalent to the $\infty$-topos of sheaves on the underlying topological space of $\mathbf{R}^d$. We will check that $(\Diff_{\et}^r)_{/\mathbf{R}^d}$ is contractible (and moreover locally contractible) by induction on $d$. 

The case $d = 0$ is clear.

Next, we check the case $d = 1$ using Corollary \ref{cc}.  Let $X$ be a set, then $H^0(\mathbf{R},X) = \TSpc(\mathbf{R},X) = X$, as $\mathbf{R}$ is connected. 
Let $G$ be any group, then $H^1(\mathbf{R}, G)$ is equivalent to the set of isomorphism classes of principle $G$-bundles on $\mathbf{R}$, which are constant (and thus all equivalent) by the standard argument that covering spaces on $\mathbf{R}$ are constant (see e.g.,  \cite[Lm.~5.1.2]{pS2007}). 
Finally, by Lemma \ref{cd} $\mathbf{R}$ has covering dimension $\leq 1$, and thus cohomological dimension $\leq 1$ by the discussion following \cite[Rmk.~7.2.2.19]{jL2009}. (Alternatively, one may prove that $\mathbf{R}$ has cohomological dimension $\leq 1$ using a similar argument to the one used to exhibit the triviality of covering spaces on $\mathbf{R}$, as shown in \cite[Lm.~5.1.1]{pS2007}.) 

Observe that $(\Diff_{\et}^r)_{/\mathbf{R}}$ is moreover locally contractible, as $\mathbf{R}$ has a basis given by open intervals, which are diffeomorphic to $\mathbf{R}$. Now, let $d > 1$, and assume  that $(\Diff_{\et}^r)_{/\mathbf{R}^{d-1}}$ has contractible shape and moreover is locally contractible, so that by Proposition \ref{topos preloc} the adjunction $\cadjunction{\pi_!: (\Diff_{\et}^r)_{/\mathbf{R}^{d-1}}}	{\mathcal{S}: \pi^*}$ is a reflection. Then, $(\Diff_{\et}^r)_{/\mathbf{R}^d} \leftarrow \mathcal{S}: \pi^*$ is a reflexive subcategory, as it is equivalent to the tensor product of $(\Diff_{\et}^r)_{/\mathbf{R}^{d-1}} \leftarrow \mathcal{S}: \pi^*$ and $(\Diff_{\et}^r)_{/\mathbf{R}} \leftarrow \mathcal{S}: \pi^*$ (see \cite[Ex.~A.2.8]{jL2012}), both of which are reflexive subcategories. 
\end{proof} 

\begin{corollary} 
The $\infty$-topos $\Diff_{\et}^r$ is locally contractible. \qed
\end{corollary} 

By Theorem \ref{corporeal shape} we then obtain the following corollary. 

\begin{corollary}	\label{rcid}
The shape of $\mathbf{R}^d$ is contractible in $\Diff^r$ for all $d \in \mathbf{N}$. \qed
\end{corollary} 

\begin{corollary} 
The $\infty$-topos $\Diff^r$ is locally contractible. \qed
\end{corollary} 

\begin{remark}	\label{other shapes} 
The functor $\pi_!: \Diff^r \to \mathcal{S}$ has been shown to exist many times before, e.g.\ in \cite[Prop.~8.3]{dD2001}, \cite[\S 4.4]{uS2013}, \cite[Prop.~3.1]{dC2016o}, \cite[Prop.~1.3]{dBEpBDBdP2019}, \cite{sB2023}, \cite[\S 4.3]{aAaDpH2021}, \cite{dP2022}. All of these sources rely on some variant of the nerve or Seifert-Van Kampen theorem (see \cite{kB1948}, \cite{jL1950}, \cite{aW1952}, \cite[\S 4]{gS1968}, \cite[Th.~1.1]{dDdI2004}, \cite[Th.~A.3.1]{jL2012}) to implement some version of the following argument: one shows that
\begin{enumerate}
\item $\colim: [(\Cart^r)^{\op}, \mathcal{S}] \to \mathcal{S}$ sends covers to colimits, and 
\item	constant presheaves on $\Cart^r$ are sheaves, 
\end{enumerate}
so that the adjunction $\cadjunction{\colim: [(\Cart^r)^{\op},\mathcal{S}]}	{\mathcal{S}: \mathrm{const}}$ restricts to $\cadjunction{\pi_!: \Diff^r}	{\mathcal{S}: \pi^*}$. We will discuss the specific argument used in  \cite{dC2016o} in more detail in \S \ref{The shape of the Haefliger stack}.

The proofs of some variants of the nerve and Seifert-Van Kampen theorem, in particular \cite[Th.~1.1]{dDdI2004} and \cite[Th.~A.3.1]{jL2012}, are quite involved. We will obtain these practically for free in \S \ref{Colimits of hypercoverings of topological spaces}. 
\qede 
\end{remark} 

\begin{corollary}	\label{shape finite products}
The shape functor $\pi_!: \Diff^r \to \mathcal{S}$ preserves finite products. 
\end{corollary} 

\begin{proof}
By Proposition \ref{aspherical embedding} the shape of any sheaf in $\Diff^r$ may be computed as the colimit of the corresponding presheaf on $\Cart^r$, but $\Cart^r$ has finite products, and is thus sifted. 
\end{proof} 

Several of the references listed in Remark \ref{other shapes} moreover show (some variant of) the following result: 

\begin{proposition} 
The shape functor $\pi_!: \Diff^r \to \mathcal{S}$ exhibits $\mathcal{S}$ as the localisation of $\Diff^r$ along the projection map $\mathbf{R}^1 \times X  \to  X$ for all differentiable sheaves $X$. 
\end{proposition}  

\begin{proof} 
Denote by $W$ the class of weak equivalences inverted by the localisation of $\Diff^r$ along the projection map $\mathbf{R}^1 \times X  \to  X$ for all differentiable sheaves $X$. By Corollary \ref{shape finite products} these projection maps are all shape equivalences, so $W$ is contained in the class of shape equivalence. On the other hand, all maps $\mathbf{R}^d \to \mathbf{1}$ can be decomposed into a sequence of projection maps $\mathbf{R}^d \to \mathbf{R}^{d-1} \to \cdots \to \mathbf{R} \to \mathbf{1}$, and are thus in $W$, and therefore by the 2-out-of-3 property, all maps in $\Cart^r$ are in $W$. Thus, the associated sheaf functor $[((\Cart^r)^{\op},\mathcal{S}] \twoheadrightarrow \Diff^r$ sends shape equivalences to morphisms in $W$, so that every shape equivalence in $\Diff^r$ is in $W$. 
\end{proof} 

\subsection{Comparing methods of calculating underlying homotopy types of differentiable sheaves}	\label{underlying homotopy type} 

We first construct various nerve diagrams in $\Diff^r$ in \S \ref{comparing nerves}, and show that the induced nerves all calculate shapes. Then, in \S \ref{change} we show that sending any $r$-times differentiable manifold to its underlying $s$-times differentiable manifold for $r \geq s \geq 0$ does not change its shape.  

\subsubsection{Nerves}	\label{comparing nerves} 

Here we consider five different nerve diagrams: 

	\begin{itemize}
	\item[]	$\mathbf{A}^\bullet: \Delta \to \Diff^r_{\leq 0}$ 
	\item[]	$\Delta_{\sub}^\bullet: \Delta \to \Diff^r_{\leq 0}$ 
	\item[]	$\Delta^\bullet: \Delta \to \Diff^r_{\leq 0}$ 
	\item[]	$\BBube^\bullet: \Cube \to \Diff^r_{\leq 0}$ 
	\item[]	$\Cube^\bullet: \Cube \to \Diff^r_{\leq 0}$ 
	\end{itemize}

In each case we will use Theorem \ref{nerves} to show that the five resulting nerves all calculate shapes. 

\paragraph{Extended simplices}

\begin{definition} 
Consider the cosimplicial object
$$	\begin{array}{rrcl}
	\mathbf{A}^\bullet:	&	\Delta^{\phantom{n}}	&	\to		&	\Cart^r														\\
	{}				&	[n]				&	\mapsto	&	\mathbf{A}^n \defeq \setbuilder{(x_0, \ldots, x_n) \in \mathbf{R}^{n+1}}	{x_0 + \cdots + x_n = 1},
	\end{array}
$$
then the spaces $\mathbf{A}^n$ for $n \geq 0$ are referred to as \Emph{extended simplices}. Moreover we write 
$$	\begin{array}{rcll}
	\partial \mathbf{A}^n		&	\defeq	&	\mathbf{A}_!^\bullet \partial \Delta^n,	&	n \geq 0				\\
	\mathbf{\Lambda}_k^n	&	\defeq	&	\mathbf{A}_!^\bullet \Lambda_k^n,	&	n \geq 1, \; n \geq k \geq 0.
	\end{array}
$$
\qede
\end{definition} 

\begin{proposition}	\label{extended nerve} 
The canonical natural transformation $\colim \circ \, (\mathbf{A}^\bullet)^* \to (\pi_{\Diff^r})_!$ is an equivalence. 
\end{proposition} 

\begin{proof} 
The image of $\mathbf{A}^\bullet$ is given by $\Cart^r$, and satisfies the conditions of Theorem \ref{nerves} by Corollary \ref{rcid}. The induced functor $\Delta \to \Cart^r$ is easily seen to satisfy the conditions of Proposition \ref{homotopy initial}, thus verifying the conditions of Theorem \ref{nerves}. 
\end{proof}

\paragraph{Closed simplices}

Consider the cosimplicial object 
$$	\begin{array}{rrcl}
	\Delta_{\sub}^\bullet:	&	\Delta^{\phantom{n}}	&	\to		&	\Diff^r_{\leq 0}														\\
	{}				&	[n]				&	\mapsto	&	\Delta_{\sub}^n.
	\end{array}
$$

\begin{proposition}	\label{csn}
The canonical natural transformation $\colim \circ \, (\Delta_{\sub}^\bullet)^* \to (\pi_{\Diff^r})_!$ is an equivalence. 
\end{proposition} 

\begin{proof} 
To see that the image $C$ of $\mathbf{A}^\bullet$ satisfies the conditions of Theorem \ref{nerves} we observe that the collection of translations of the standard inclusion $\Delta_{\sub}^d \hookrightarrow \mathbf{A}^d$ form a cover of $\mathbf{A}^d$ ($d \geq 0$), so that we may apply \cite[Prop.~20.4.5.1]{jL2017}. The induced functor $\Delta \to C$ is then easily seen to satisfy the conditions of Proposition \ref{homotopy initial}, thus verifying the conditions of Theorem \ref{nerves}.
\end{proof}

\paragraph{Kihara's simplices}

It has been a longstanding goal to establish a model structure on diffeological spaces (see e.g.\ \cite{jdCeW2014} and \cite{tHkS2018}). To this end Kihara endows the standard simplices with a new diffeology in \cite[\S~1.2]{hK2019}. With this diffeology the horn inclusions admit deformation retracts (see Proposition \ref{deformation retract}), allowing Kihara to mimic the construction of the model structure on topological spaces in \cite[\S II.3]{dQ1967}, and show that the resulting model category is Quillen equivalent to simplicial sets with the Kan-Quillen model structure. We need Kihara's simplices in order to construct objects satisfying the differentiable Oka principle in \S \ref{closure}. 

For the convenience of the reader, we repeat the construction of Kihara's simplices: For each $n \geq 1$ and each $0 \leq k \leq n$ we define the set
$$	A_k^n \defeq \setbuilder{(x_0, \ldots, x_n) \in \Delta^n}{x_k < 1}.	$$
We now proceed inductively: On $\Delta^0$ and $\Delta^1$ the diffeology is the subspace diffeology coming from $\mathbf{R}^1$ and $\mathbf{R}^2$, respectively.  Let $n > 1$, and assume that the diffeologies on the simplices $\Delta^m$ for $m < n$ have been defined, then we define a diffeology on $A_k^n$ by exhibiting this set as the underlying set of the quotient 
% https://q.uiver.app/?q=WzAsNCxbMSwwLCJcXERlbHRhXntuLTF9IFxcdGltZXMgWzAsMSkiXSxbMCwwLCJcXERlbHRhXntuLTF9IFxcdGltZXMgXFx7MFxcfSJdLFswLDEsIjEiXSxbMSwxLCJBX2tebiJdLFsxLDAsIiIsMCx7InN0eWxlIjp7InRhaWwiOnsibmFtZSI6Imhvb2siLCJzaWRlIjoidG9wIn19fV0sWzEsMiwiIiwyLHsic3R5bGUiOnsiaGVhZCI6eyJuYW1lIjoiZXBpIn19fV0sWzAsMywiIiwwLHsic3R5bGUiOnsiaGVhZCI6eyJuYW1lIjoiZXBpIn19fV0sWzIsMywiIiwyLHsic3R5bGUiOnsidGFpbCI6eyJuYW1lIjoiaG9vayIsInNpZGUiOiJ0b3AifX19XV0=
\[\begin{tikzcd}
	{\Delta^{n-1} \times \{0\}} & {\Delta^{n-1} \times [0,1)} \\
	1 & {A_k^n},
	\arrow[hook, from=1-1, to=1-2]
	\arrow[two heads, from=1-1, to=2-1]
	\arrow[two heads, from=1-2, to=2-2]
	\arrow[hook, from=2-1, to=2-2]
\end{tikzcd}\]
where $\Delta^{n-1} \times [0,1) \to A_n^n$ is given by $(x_0, \ldots, x_{n-1}; t) \mapsto ((1-t) \cdot x_0, \ldots, (1 - t) \cdot x_n, t)$. Finally, the diffeology on $\Delta^n$ is determined by the map $\coprod_{k = 0}^n A_k^n \twoheadrightarrow \Delta^n$.

\begin{proposition}[{\cite[\S~8]{hK2019}}]	\label{deformation retract} 
The horn inclusions $\Lambda_k^n \hookrightarrow \Delta^n$ for $n = 2$ and $n \geq k \geq 0$ admit a deformation retract.	\qed
\end{proposition} 

\begin{definition}	\label{Kiharas simplices} 
We write
$$	\begin{array}{rrcl}
	\Delta^\bullet:	&	\Delta^{\phantom{n}}	&	\to		&	\Diff^r	\\
	{}			&	[n]				&	\mapsto	&	\Delta^n
	\end{array}
$$
for the cosimplicial object sending each simplex $\Delta^n$ to the standard $n$-simplex endowed with the diffeology described above. The spaces $\Delta^n$ for $n \geq 0$ are referred to as \Emph{Kihara's simplices}. Moreover, we write 
$$	\begin{array}{rcll}
	\partial \Delta^n		&	\defeq	&	\Delta_!^\bullet \partial \Delta^n,	&	n \geq 0				\\
	\Lambda_k^n		&	\defeq	&	\Delta_!^\bullet \Lambda_k^n,		&	n \geq 1, \; n \geq k \geq 0
	\end{array}
$$
  \qede
\end{definition} 

The proof of the following proposition is completely analogous to the proof of Proposition \ref{csn}. 

\begin{proposition}	\label{Kihara nerve} 
The canonical natural transformation $\colim \circ \, (\Delta^\bullet)^* \to (\pi_{\Diff^r})_!$ is an equivalence. \qed
\end{proposition}

\paragraph{Extended cubes}

\begin{definition} 
We write 
$$	\begin{array}{rrcl}
	\BBube^\bullet:	&	\Cube^{\;\! \phantom{n}}	&	\to		&	\Diff^r			\\
	{}			&	\Cube^{\; \! n}			&	\mapsto	&	\mathbf{R}^n
	\end{array}
$$
for the unique symmetric monoidal functor determined by sending the morphisms $\delta^\xi: \Cube^{\; \!0} \to \Cube^{\; \!1}$ to $0 \mapsto \xi$  for $\xi = 0,1$  (see Proposition \ref{universal cube}). The spaces $\BBube^n$ for $n \geq 0$ are referred to as the \Emph{extended $n$-cubes}. \qede
\end{definition} 

\begin{proposition}	\label{BBube nerve}
The canonical natural transformation $\colim \circ \, (\BBube^\bullet)^* \to (\pi_{\Diff^r})_!$ is an equivalence. 
\end{proposition} 

\begin{proof} 
The image of $\BBube^\bullet$ is given by $\Cart^r$, and satisfies the conditions of Theorem \ref{nerves} by Corollary \ref{rcid}. The induced functor $\Cube \to \Cart^r$ is easily seen to satisfy the conditions Proposition \ref{monoidal homotopy initial}, thus verifying the conditions of Theorem \ref{nerves}. 
\end{proof} 

\paragraph{Closed cubes}

\begin{definition} 
We write 
$$	\begin{array}{rrcl}
	\Cube^{\;\! \bullet}:	&	\Cube^{\;\! \phantom{n}}	&	\to		&	\Diff^r		\\
	{}				&	\Cube^{\;\! n}				&	\mapsto	&	[0,1]^n
	\end{array}
$$
for the unique symmetric monoidal functor determined by sending the morphisms $\delta^\xi: \Cube^{\; \!0} \to \Cube^{\; \!1}$ to $0 \mapsto \xi$  for $\xi = 0,1$  (see Proposition \ref{universal cube}). The spaces $\Cube^{\; \! n}$ for $n \geq 0$ are referred to as the \Emph{closed $n$-cubes}. \qede
\end{definition} 

The following proposition may be proved using an obvious adaption of the proofs of Propositions \ref{BBube nerve} \& \ref{csn}. 

\begin{proposition} 
The canonical natural transformation $\colim \circ \, (\BBube^{\;\! \bullet})^* \to (\pi_{\Diff^r})_!$ is an equivalence. \qed
\end{proposition}

\subsubsection{Change of regularity}	\label{change}

\begin{theorem}	\label{change theorem} 
Let $r \geq s \geq 0$, and denote by $u: \Cart^r \to \Cart^s$ the forgetful functor, then the adjunction $\copadjunction{u_*:[(\Cart^r)^{\op}, \mathcal{S}]}	{[(\Cart^s)^{\op}, \mathcal{S}]: u^*}$ restricts to an essential geometric morphism $\copadjunction{u_*:\Diff^r}	{\Diff^s: u^*}$, such that $u_!: \Diff^r \to \Diff^s$ sends any $r$-times differentiable manifold to its underlying $s$-times differentiable manifold. 
\end{theorem} 

Setting $s = 0$ we obtain the following corollary: 

\begin{corollary} 
The underlying topological space of any $r$-times differentiable manifold calculates its shape. \qed
\end{corollary} 

The proof of Theorem \ref{change theorem} is similar to the proof of Proposition \ref{comparison theorem}. 

\begin{proof}[Proof of Theorem \ref{change theorem}]
First, by Proposition \ref{ccm} we may identify $\Diff^r$ and $\Diff^s$ with the $\infty$-toposes of sheaves on $\Mfd^r$ and $\Mfd^s$, respectively. We again denote the forgetful functor $\Mfd^r \to \Mfd^s$ by $u$, which is clearly cocontinuous (\cite[Def.~III.2.1]{SGA4.1}), so that $u^*$ preserves sieves, which shows that $u_*$ restricts to a functor $\Diff^r \to \Diff^s$. As $u$ preserves pullbacks along open embeddings, $u$ satisfies condition iii) of \cite[Prop.~III.1.11]{SGA4.1}, so that $u_!: \widehat{\Mfd^r} \to \widehat{\Mfd^s}$ preserves local equivalences, so that $u^*$ restricts to a functor $\Diff^r \leftarrow \Diff^s$. The functor $u_!: \Diff^r \to \Diff^s$ is obtained by composing the restriction of $u_!: [(\Mfd^r)^{\op}, \mathcal{S}] \to [(\Mfd^s)^{\op}, \mathcal{S}]$ to $\Diff^r$ with the sheafification functor $[(\Mfd^s)^{\op}, \mathcal{S}] \to \Diff^s$. 
\end{proof} 

\subsection{Applications}	\label{locally contractible applications} 

We now present two applications of the technology developed so far. In \S \ref{The shape of the Haefliger stack} we show that once we decompose $\Diff_{\et}^r$ into the coproduct (in $\Top$) of the $\infty$-toposes $\Diff_{\et,d}^r$ of $d$-dimensional \'{e}tale differentiable stacks, and moreover have Carchedi's result that the $d$-th Haefliger stack $\mathbf{H}^d$ is final in $\Diff_{\et,d}^r$ (see Theorem \ref{Haefliger final}), then the calculation of the shape of $\mathbf{H}^d$ (as an object $\Diff^r$) follows formally from the way in which $\Diff^r$ is a locally contractible $\infty$-topos, compatibly with its structure as a fractured $\infty$-topos. Then, in \S \ref{Colimits of hypercoverings of topological spaces} we observe that the shape of the sheaf on $\Diff^0$ represented by a topological space calculates its singular homotopy type, and thus, we are able to harness the descent properties of $\Diff^0$ to prove descent theorems in algebraic topology. We recover essentially for free Lurie's Seifert-Van Kampen theorem (see \cite[Th.~A.3.1]{jL2012}, Theorem \ref{lsvkt}), Dugger and Isaksen's hypercovering theorem (see \cite[Th.~1.1]{dDdI2004}, Theorem \ref{dihc}), and the folk theorem that the base space of any principal bundle is a homotopy quotient (see Theorem \ref{homotopy quotient}).

\subsubsection{The shape of the Haefliger stack}	\label{The shape of the Haefliger stack}

The underlying topological groupoid of $\Gamma^d$ (defined below), now known as the Haefliger groupoid, was introduced by Haefliger in \cite{aH1958}, with a view towards applications to the study of foliations. Its classifying space (in the sense of \cite{gS1968}) was first determined in \cite[Prop.~1.3]{gS1978}, and later Carchedi provided a new calculation of this classifying space in \cite[Th.~3.7]{dC2016o}. The proof we present here is essentially the same as Carchedi's, the only difference being that we may exhibit every step of the proof as a formal manipulation in the calculus afforded by a more systematic account of the theory of locally contractible $\infty$-toposes and their interactions with fractured $\infty$-toposes. (Incidentally, similar interactions between shapes and fractured $\infty$-toposes -- although not in this language -- are explored by Carchedi in a subsequent article, \cite{dC2021}, where GAGA like theorems are established for profinite shapes.)

Before turning to the Haefliger stack, we quickly explain how to decompose $\Diff_{\et, d}^r$ into a product of $\infty$-toposes. Denote by $\Cart_{\et, d}^r$ the category of $d$-dimensional $r$-times differentiable Cartesian spaces, and by $\Diff_{\et,d}^r$, the $\infty$-topos of $\mathcal{S}$-valued sheaves on $\Cart_{\et, d}^r$ -- the \Emph{$d$-dimensional \'{e}tale $r$-times differentiable stacks}. Observe that $\Cart_{\et, d}^r$ is equivalent to the monoid (viewed as a category) of $r$-times differentiable embeddings $\Emb^r(\mathbf{R}^d, \mathbf{R}^d)$. We will now examine how we may decompose $\Diff_{\et}^r$ into the $\infty$-toposes $\Diff_{\et,d}^r$ ($d \geq 0$). 

\begin{proposition}	\label{pot}
Let $\{\mathcal{E}_i\}_{i \in I}$ be a family of $\infty$-toposes indexed by a (small) set $I$.  
	\begin{enumerate}[label = {\normalfont (\arabic*)}]
	\item	The coproduct of $\{\mathcal{E}_i\}_{i \in I}$ in $\Top$ is given by the product of $\{\mathcal{E}_i\}_{i \in I}$ in the $\infty$-categories of $\infty$-categories. 
	\item	The structure geometric morphism $\iota_i: \mathcal{E}_i = \mathcal{E}_i \times \mathbf{1}_{\Cat}  \to \mathcal{E}_i \times \prod_{i \neq j}  \mathcal{E}_j = \prod_{i \in I} \mathcal{E}_i$ is essential for every $i \in I$. 
	\item	For any $i \in I$ and any object $X$ in $\mathcal{E}_i$ the geometric morphism $\iota_i: (\mathcal{E}_i)_{/X} \to (\prod_{i \in I} \mathcal{E}_i)_{/(\iota_i)_! X}$ is an equivalence. 
	\item	\label{pot4}For any sequence of objects $(X_i)_{i \in I} \in \prod_{i \in I} \mathcal{E}_i$ we have $(X_i)_{i \in I} = \coprod_{i \in I} (\iota_i)_! X_i$. 
	\item	\label{pot5}	Let $\{C_i\}_{i \in I}$ be a family of small $\infty$-categories, then 
		\begin{enumerate}%[label=\arabic*]
		\item[{\normalfont (5.1)}]	the equivalence $\prod_{i \in I} [C_i^{\op}, \mathcal{S}] = \underline{\Hom}\left(\coprod_{i \in I} C_i, \mathcal{S}\right) $ establishes a bijection
		$$\prod_{i \in I} \Big\{\text{Grothendieck topologies on } C_i\Big\} = \Big\{\text{Grothendieck topologies on }\coprod_{i \in I} C_i\Big\};$$ 
		\item[\normalfont{(5.2)}]	let $(\tau_i)_{i \in I}$ and $\tau$ be a pair of corresponding elements under the above bijection, then the functors $C_i \to \coprod_{i \in I} C_i$ are both continuous and cocontinuous ({\normalfont \cite[Defs.~III.1.1 \& III.2.1]{SGA4.1}}) and the induced essential geometric morphisms $\Sh_{C_i, \tau_i} \to \Sh_{\left (\coprod_{i \in I} C_i\right ) , \tau}$ exhibit $\Sh_{\left (\coprod_{i \in I} C_i\right ) , \tau}$ as the coproduct of $\{\Sh_{C_i, \tau_i}\}_{i \in I}$ (in $\Top$). 
		\end{enumerate} 
		\qed
	\end{enumerate} 
\end{proposition}

We defer the proof of the above proposition to the end of this subsection. We obtain the following corollary: 

\begin{proposition} 
The inclusions $\Cart^r_{\et, d} \hookrightarrow \Cart_{\et}^r$ induce essential geometric morphisms $\Diff_{\et, d}^r \to \Diff_{\et, r}^d$ exhibiting $ \Diff_{\et, r}^d$ as the coproduct of $\{\Diff_{\et, d}^r\}_{d \geq 0}$. \qed
\end{proposition} 

Now, consider the set-valued presheaf on the topological space $\mathbf{R}^d$ ($d \geq 0$) given by sending $U$ to the set of $r$-time differentiable embeddings of $U$ into $\mathbf{R}^d$. The \'{e}tal\'{e} space of this presheaf, denoted by $\Gamma_0^d$, may naturally be viewed as an object of $\Diff_{\et}^r$. Its underlying set consists of pairs $(x,\varphi)$ consisting of a point $x \in \mathbf{R}^d$ together with the germ of an embedding of a neighbourhood of $x$ into $\mathbf{R}^d$. Apart from the structure map $\Gamma_0^d \to \mathbf{R}^d$ there exists a second \'{e}tale map $\Gamma_0^d \to \mathbf{R}^d$ given by sending any element $(x,\varphi)$ of $\Gamma_0^d$ to $\varphi(x)$. Composition of germs endows the simplicial diagram $[n] \mapsto \Gamma_n^d \defeq \Gamma_0^d \times_{\mathbf{R}^d} \stackrel{(n + 1) \times}{\cdots}  \times_{\mathbf{R}^d} \Gamma_0^d$ with the structure of a groupoid object in $\Diff_{\et}^r$, called the \Emph{Haefliger groupoid}, and is denoted by $\Gamma^d$. The \Emph{$d$-th Haefliger stack}, denoted by $\mathbf{H}^d$, is then the \'{e}tale differentiable stack presented by $\Gamma^d$. As usual, we will identity $\Gamma^d$ and $\mathbf{H}^d$ with their images under $j_!$ in $\Diff^r$. The key to calculating the shape of the Haefliger stack is the following observation by Carchedi: 

\begin{theorem}[{\cite[Th.~3.3]{dC2019}}]	\label{Haefliger final}
The $d$-th Haefliger stack $\mathbf{H}^d$ is the final object in $\Diff_{\et, d}^r$.
\end{theorem}  

\begin{proof} 
It is enough to show that $\Diff_{\et, d}^r(\mathbf{R}^d, \mathbf{H}^d)$ is contractible. It is nonempty as it contains at least one element obtained by composing the identity map $\mathbf{R}^d \to \Gamma_0^d$ with the cover $\Gamma_0^d \to \mathbf{H}^d$. Let $f: \mathbf{R}^d \to \mathbf{H}^d$ be a map, then every point $\mathbf{R}^d$ admits a neighbourhood $U$ and a lift 
% https://q.uiver.app/#q=WzAsNCxbMCwwLCJVIl0sWzEsMCwiXFxHYW1tYV8wXmQiXSxbMSwxLCJcXG1hdGhiZntIfV9kIl0sWzAsMSwiXFxtYXRoYmZ7Un1eZCJdLFswLDEsIiIsMCx7InN0eWxlIjp7ImJvZHkiOnsibmFtZSI6ImRhc2hlZCJ9fX1dLFsxLDIsIiIsMCx7InN0eWxlIjp7ImhlYWQiOnsibmFtZSI6ImVwaSJ9fX1dLFszLDJdLFswLDMsIiIsMix7InN0eWxlIjp7InRhaWwiOnsibmFtZSI6Imhvb2siLCJzaWRlIjoidG9wIn19fV1d
\[\begin{tikzcd}
	U & {\Gamma_0^d} \\
	{\mathbf{R}^d} & {\mathbf{H}^d}.
	\arrow[dashed, from=1-1, to=1-2]
	\arrow[two heads, from=1-2, to=2-2]
	\arrow[from=2-1, to=2-2]
	\arrow[hook, from=1-1, to=2-1]
\end{tikzcd}\]
Choosing $U$ sufficiently small, we may assume that $U \dashrightarrow \Gamma_0^d$ is an embedding, and there exists a diffeomorphism between $U$ and its image in $\Gamma_0^d$, corresponding to a lift 
% https://q.uiver.app/#q=WzAsMyxbMSwwLCJcXEdhbW1hXmRfMSJdLFsxLDEsIlxcR2FtbWFfMF5kIl0sWzAsMSwiVSJdLFswLDFdLFsyLDEsIiIsMix7InN0eWxlIjp7InRhaWwiOnsibmFtZSI6Imhvb2siLCJzaWRlIjoidG9wIn19fV0sWzIsMF1d
\[\begin{tikzcd}
	& {\Gamma^d_1} \\
	U & {\Gamma_0^d}
	\arrow[from=1-2, to=2-2]
	\arrow[hook, from=2-1, to=2-2]
	\arrow[from=2-1, to=1-2]
\end{tikzcd}\]
so that $U \hookrightarrow \Gamma_0^d$ is equivalent to the standard inclusion. Performing this procedure for every point in $\mathbf{R}^d$, we see that $f$ may be represented by the identity map $\mathbf{R}^d \to \Gamma_0^d$. Finally, note that the only automorphism of he identity map $\mathbf{R}^d \to \Gamma_0^d$  in the groupoid  $\Diff_{\et, d}^r(\mathbf{R}^d, \mathbf{H}^d)$ is the identity.  
\end{proof} 

Applying Proposition \ref{pot}.\ref{pot4} we obtain the following corollary: 

\begin{corollary} 
The final object of $\Diff_{\et}^r$ is then given by $\coprod_d \mathbf{H}^d$. \qed
\end{corollary} 

We now calculate the shape of the $d$-th Haefliger stack ($d \geq 0$).

\begin{theorem}[{\cite[Prop.~1.3]{gS1978} \& \cite[Th.~3.7]{dC2016o}}]	\label{Haefliger shape} 
For all $d \geq 0$: 
$$(\pi_{\Diff^r})_! \mathbf{H}^d = B \Emb(\mathbf{R}^d, \mathbf{R}^d).$$
\end{theorem} 

\begin{proof} 
We have 
$$	\begin{array}{rcll}
	(\pi_{\Diff^r})_! \mathbf{H}^d	&	=	&	(\pi_{\Diff_{\et}^r})_! \mathbf{H}^d						&	\text{Th.\ \ref{jps}} 						\\
	{}						&	=	&	(\pi_{\Diff_{\et, d}^r})_! \mathbf{H}^d						&	\text{Props.\ \ref{pot} \& \ref{essential shape}} 	\\
	{}						&	=	&	(\pi_{\Diff_{\et, d}^r})_! \left(\mathbf{1}_{\Diff^r_{\et, d}}\right)	&	\text{Th.\ \ref{Haefliger final}} 				\\
	{}						&	=	&	\colim \mathbf{1}_{[(\Cart^r_{\et, d})^{\op}, \mathcal{S}]}		&	\text{Prop.\ \ref{prenerves}}				\\
	{}						&	=	&	(\Cart_{\et, d}^r)_\simeq								&	\text{Ex.\ \ref{plc}}						\\
	{}						&	=	&	B \Emb^r(\mathbf{R}^d, \mathbf{R}^d),					&
	\end{array}	$$
\end{proof} 

\begin{remark}	\label{Segal comp}
In order to obtain Segal's original result (\cite[Prop.~1.3]{gS1978}) on the classifying space of the underlying topological groupoid of $\Gamma^d$, it is enough to observe that
	\begin{enumerate}
	\item	$\mathbf{H}^d$ is given as the colimit of (the simplicial diagram) $\Gamma^d$, 
	\item	$u_!: \Diff^r \to \Diff^0$ preserves colimits,
	\item	applying $u_!$ to $\Gamma^d$ produces the underlying topological groupoid of $\Gamma^d$ (Theorem \ref{change theorem}), and
	\item	fat topological realisations are homotopy colimits (and that $\Delta_{\mathrm{inj}} \to \Delta$ is initial). 
	\end{enumerate} 
	\ \qede
\end{remark} 

We conclude this subsection by giving a sketch of Carchedi's proof of Theorem \ref{Haefliger shape}  in \cite{dC2016o}, before supplying a proof of Proposition \ref{pot}. First, Carchedi constructs the shape functors for $\Diff_{\et, d}^r$ and $\Diff^r$ (without identifying them as such) as follows:  Denote by $L: \TSpc \to \mathcal{S}$ the localisation functor, then the sequence of functors
	$$\Cart_{\et, d}^r \to \Mfd_{\et,d}^r	\to	\Mfd^r	\to	\TSpc	\to \mathcal{S}	$$
gives rise to the sequence of cocontinuous functors
	\begin{equation}	\label{Carchedi sequence} 
	[(\Cart_{\et, d}^r)^{\op}, \mathcal{S}]	\to \Diff_{\et, d}^r	\to	\Sh_{\Mfd_{\et,d}^r} ( = \Diff_{\et,d}^r)	\to	\Diff^r	\to \mathcal{S},
	\end{equation} 
as the composition $\Mfd^r	\to	\TSpc	\to \mathcal{S}$ preserves colimits of hypercovers by \cite[Th.~1.1]{dDdI2004} (and the fact that fat topological realisations are homotopy colimits), and because the functor $\Mfd_{\et,d}^r	\to	\Mfd^r$ is cocontinuous (see \cite[Def.~III.2.1]{SGA4.1}). Then, one observes that the composition of all the functors in (\ref{Carchedi sequence}) sends $\mathbf{R}^d$ to $\mathbf{1}_\mathcal{S}$ for all $d \geq 0$, so that by cocontinuity the composition is simply given by the colimit functor. Thus the shape of the $d$-th Haefliger stack ($d \geq 0$) is again given by $B \Emb^r(\mathbf{R}^d, \mathbf{R}^d)$. To obtain the comparison with Segal's result (as in Remark \ref{Segal comp}) it is enough to observe that the shape of the colimit of any simplicial diagram of (not necessarily 2$^\text{nd}$-countable, Hausdorff) manifolds is equivalent to the homotopy type of the fat topological realisation of the underlying simplicial diagram of topological spaces, again by \cite[Th.~1.1]{dDdI2004} and the fact that fat topological realisations are homotopy colimits. 

\begin{proof}[Proof of Proposition \ref{pot}] \ 
	\begin{enumerate}[label = (\arabic*)]
	\item	This is Proposition \cite[6.3.2.1]{jL2009}. 
	\item	The initial topos is given by $\mathbf{1}_{\Cat}$, and for any $\infty$-topos $\mathcal{E}$ the unique geometric morphism $\varnothing: \mathbf{1}_{\Cat} \to \mathcal{E}$ is essential, where the left adjoint to the pullback functor is given by sending the unique object of $\mathbf{1}_{\Cat}$ to the initial object of $\mathcal{E}$. 
	\item	 The functor $\Cat_{\mathbf{1}/} \to \Cat$ taking any pointed $\infty$-category $\mathbf{1} \xrightarrow{c} C$ to $C_{/c}$ is right adjoint to the cone functor, and thus preserves limits. The $\infty$-category $(\{\mathcal{E}_i\}_{i \in I})_{/(\iota_i)_! X}$ is obtained by taking the product of $(\prod_{i \neq j}\mathcal{E}_i)_{/\varnothing_! \mathbf{1}} = (\prod_{i \neq j}\mathcal{E}_i)_{/\varnothing} = \mathbf{1}$ and $(\mathcal{E}_i)_{/X}$. 
	\item	For any object $Y \in \prod_{i \in I} \mathcal{E}$ we have 
		$$
		\begin{array}{rcll} 
		\left( \prod_{i \in I} \mathcal{E}_i \right)\left( \coprod_{i \in I} (\iota_i)_! X_i, Y\right)	&	=	&	\prod_{i \in I} \left( \prod_{i \in I} \mathcal{E}_i \right)\left((\iota_i)_! X_i, Y\right)	&	{}	\\	
		{}																&	=	&	\prod_{i \in I} \left( \prod_{i \in I} \mathcal{E}_i \right)\left(X_i, \iota_i^* Y\right)	&	{}	\\
		{}																&	=	&	\left( \prod_{i \in I} \mathcal{E}_i \right)\left( (X_i)_{i \in I}, Y\right)			&	{}	
		\end{array}
		$$
where the last isomorphism follows from \cite[Lm.~6.3.3.6]{jL2009}. 
	\item	 \ 
		\begin{itemize}
		\item[(5.1)]	This is an immediate consequence of statement (3) of the theorem. 
		\item[(5.2)]	That the functors $C_i \to \coprod_{i \in I} C_i$ are both continuous and cocontinuous again follows from (3).  The equivalence $ \prod_{i \in I} [C_i^{\op}, \mathcal{S}] \xleftarrow{\simeq} \underline{\Hom}\left(\coprod_{i \in I} C_i, \mathcal{S}\right) $ restricts to a fully faithful functor $\Sh_{C_i, \tau_i} \hookleftarrow \Sh_{\left (\coprod_{i \in I} C_i\right ) , \tau}$. It remains to show that any presheaf on $\coprod_{i \in I} C_i$ which is sent to an object in $\Sh_{C_i, \tau_i}$ lies in $\Sh_{\left (\coprod_{i \in I} C_i\right ) , \tau}$. 
		\end{itemize}
	\end{enumerate} 
\end{proof} 

\subsubsection{Algebraic topology and descent}	\label{Colimits of hypercoverings of topological spaces}

Let $X$ be a topological space covered by two open sets $U$ and $V$ such that $X,U,V,U\cup V$ are connected, then by the Seifert-Van Kampen theorem the square 
% https://q.uiver.app/#q=WzAsNCxbMCwwLCJcXHBpXzFYIl0sWzAsMSwiXFxwaV8xVSJdLFsxLDAsIlxccGlfMVYiXSxbMSwxLCJcXHBpXzFVIFxcY2FwIFYiXSxbMSwwXSxbMiwwXSxbMywyXSxbMywxXV0=
\[\begin{tikzcd}
	{\pi_1X} & {\pi_1V} \\
	{\pi_1U} & {\pi_1U \cap V}
	\arrow[from=2-1, to=1-1]
	\arrow[from=1-2, to=1-1]
	\arrow[from=2-2, to=1-2]
	\arrow[from=2-2, to=2-1]
\end{tikzcd}\]
is a pushout (for any basepoint in $U \cap V$). In fact, more is true: Let $L: \TSpc \to \mathcal{S}$ be the localisation functor along the Serre-Quillen weak equivalences, then the pushout square in $\TSpc$
% https://q.uiver.app/#q=WzAsNCxbMCwwLCJYIl0sWzAsMSwiVSJdLFsxLDAsIlYiXSxbMSwxLCJVIFxcY2FwIFYiXSxbMSwwXSxbMiwwXSxbMywyXSxbMywxXV0=
\begin{equation}	\label{glueing square} 
\begin{tikzcd}
	X & V \\
	U & {U \cap V}
	\arrow[from=2-1, to=1-1]
	\arrow[from=1-2, to=1-1]
	\arrow[from=2-2, to=1-2]
	\arrow[from=2-2, to=2-1]
\end{tikzcd}
\end{equation} 
is carried by $L$ to a pushout square in $\mathcal{S}$, i.e., (\ref{glueing square}) is a homotopy pushout (see Definition \ref{holim}). Squares such as (\ref{glueing square}) encode glueing data for topological spaces, so that the Seifert-Van Kampen theorem reflects how descent for topological spaces interacts with their singular homotopy types. 

We give a quick proof of the statement that (\ref{glueing square}) is a homotopy pushout, which will function as a paradigm for our new proof of Theorem \ref{lsvkt} (Lurie’s Seifert-Van Kampen theorem), as well as the proofs of Theorem \ref{dihc} (Dugger and Isaksen’s hypercovering theorem) and Theorem \ref{homotopy quotient} (which states that the base space of any principal bundles is a homotopy quotient). Denote by $v: \Cart^0 \hookrightarrow \TSpc$ the inclusion of the category of Cartesian spaces into the category of topological spaces, then $v_!: [(\Cart^r)^{\op}, \mathcal{S}] \to \TSpc$ sends sieves generated by covers consisting of jointly surjective open embeddings to isomorphisms, so that we obtain an adjunction 
$$	\adjunction{v_!: \Diff^0}	{\TSpc: v^*}.	$$ 

\begin{proposition}	\label{sing shape} 
There exists a canonical natural equivalence: 
% https://q.uiver.app/#q=WzAsMyxbMiwwLCJcXFRTcGMiXSxbMCwwLCJcXERpZmZfe1xcbGVxIDB9XjAiXSxbMSwxLCJcXG1hdGhjYWx7U30iXSxbMCwxLCJ2XioiLDJdLFsxLDIsIlxccGlfISIsMl0sWzAsMl0sWzEsNSwiXFxzaW0iLDAseyJzaG9ydGVuIjp7InNvdXJjZSI6MTAsInRhcmdldCI6MTB9fV1d
\[\begin{tikzcd}
	{\Diff_{\leq 0}^0} && \TSpc \\
	& {\mathcal{S}}
	\arrow["{v^*}"', from=1-3, to=1-1]
	\arrow["{\pi_!}"', from=1-1, to=2-2]
	\arrow[""{name=0, anchor=center, inner sep=0}, from=1-3, to=2-2]
	\arrow["\sim", shorten <=4pt, shorten >=4pt, Rightarrow, from=1-1, to=0]
\end{tikzcd}\]
\end{proposition} 

In other words, for any topological space $X$, the shape of $v^*X$ is canonically equivalent to its singular homotopy type. 

\begin{proof} 
Observe that the standard simplex functor 
$$	\Delta \to \TSpc	$$
factors as $\Delta \xrightarrow{\Delta_!} \Diff^0 \xrightarrow{v_!} \TSpc$. Then, by Proposition \ref{csn} we obtain the diagram 
% https://q.uiver.app/#q=WzAsNCxbMCwwLCJcXHdpZGVoYXR7XFxEZWx0YX0iXSxbMSwwLCJcXERpZmZfe1xcbGVxIDB9XjAiXSxbMiwwLCJcXFRTcGMiXSxbMSwxLCJcXG1hdGhjYWx7U30iXSxbMiwxLCJ2XioiLDJdLFswLDMsIlxccGlfISIsMl0sWzEsMywiXFxwaV8hIl0sWzEsMCwiXFxEZWx0YV4qIiwyXSxbMCw2LCJcXHNpbSIsMCx7InNob3J0ZW4iOnsic291cmNlIjoyMCwidGFyZ2V0IjoyMH19XV0=
\[\begin{tikzcd}
	{\widehat{\Delta}} & {\Diff_{\leq 0}^0} & \TSpc \\
	& {\mathcal{S}}
	\arrow["{v^*}"', from=1-3, to=1-2]
	\arrow["{\pi_!}"', from=1-1, to=2-2]
	\arrow[""{name=0, anchor=center, inner sep=0}, "{\pi_!}", from=1-2, to=2-2]
	\arrow["{\Delta^*}"', from=1-2, to=1-1]
	\arrow["\sim", shorten <=6pt, shorten >=6pt, Rightarrow, from=1-1, to=0]
\end{tikzcd}\]
The desired natural equivalence is then obtained by whiskering. 
\end{proof} 

\begin{warning} 
Proposition \ref{sing shape} does not imply that the singular homotopy type of a topological space coincides with its shape. For example, the shape of the Hawaiian earring is not even representable. \qede
\end{warning} 

\begin{remark} 
Proposition \ref{sing shape} and the attendant Theorems \ref{lsvkt}, \ref{dihc}, \ref{homotopy quotient} remain true when we replace $\Diff^0$ with $\Diff^r$ for $r > 0$, but we find this circumstance bewildering, so we have opted to fix $r = 0$ until the end of this chapter. \qede
\end{remark} 

Now, observe that the commutative square  
%\begin{equation}	\label{tps}
% https://q.uiver.app/#q=WzAsNCxbMCwwLCJ2XipYIl0sWzAsMSwidl4qVSJdLFsxLDAsInZeKlYiXSxbMSwxLCJ2XipVIFxcY2FwIFYiXSxbMSwwXSxbMiwwXSxbMywyXSxbMywxXV0=
\[ \begin{tikzcd}
	{v^*X} & {v^*V} \\
	{v^*U} & {v^*U \cap V}
	\arrow[from=2-1, to=1-1]
	\arrow[from=1-2, to=1-1]
	\arrow[from=2-2, to=1-2]
	\arrow[from=2-2, to=2-1]
\end{tikzcd}\]
%\end{equation} 
is a pushout square in $\Diff^r$ (which can be seen, e.g., by pulling back along all continuous maps $\mathbf{R}^d \to v^*X$ ($d \geq 0$)), so that (\ref{glueing square}) is a homotopy pushout square by Proposition \ref{sing shape} and the fact that $\pi_!$ preserves colimits. 

\paragraph{Lurie's Seifert-Van Kampen theorem} 

We now prove Lurie's far reaching generalisation of the Seifert - Van Kampen Theorem: 

\begin{theorem}[{\cite[A.3.1]{jL2012}}]	\label{lsvkt}
Let $X$ be a topological space, and denote by $\Open_X$ the category of open subsets of $X$ (ordered by inclusion). Furthermore, let $A$ be a small category, and $\chi: A \to \Open_X$, a functor. Moreover, for each element $x \in X$ denote by $A_x$ the full subcategory of $A$ spanned by those objects $a \in A$ such that $x \in \chi (a)$. If $(A_x)_\simeq = \mathbf{1}_\mathcal{S}$ for each $x \in X$, then the cocone $A^\rhd \to \TSpc$ obtained by composing the uniqe cocone $A^\rhd \to \Open_X$ on $\chi$ with apex $X$ with the functor $\Open_X \to \TSpc$ is a homotopy colimit. \qed
\end{theorem} 

The version of the Seifert - Van Kampen Theorem presented above is then obtained by setting $A = U \leftarrow U \cap V \rightarrow V$, and letting $\chi$ be the inclusion $A \hookrightarrow \Open_X$.  

\begin{proof}[Proof of Theorem \ref{lsvkt}] 
The composition of $\Diff^0 \xleftarrow{u^*} \TSpc \leftarrow \Open_X$ sends any covering $\{U \subseteq V\}$ to a covering in $\Diff^0$ (as can be seen by pulling back the inclusions $u^*U \hookrightarrow u^*V$ along all maps $\mathbf{R}^d \to u^*V$), and moreover preserves finite limits, so that it preserves coving sieves, yielding a geometric morphism $\copadjunction{(u_X)_*: \Diff^0}	{\Sh_X: u_X^*}$ by Proposition \ref{ccm} and \cite[Cor.~20.4.3.2]{jL2017}. We must show that $A^{\rhd} \to \Open_X	\to	\Sh_X	\xrightarrow{u_X^*}	\Diff^0	\xrightarrow{\pi_!}	\mathcal{S}$ is a colimit. As $\Diff^0$ is hypercomplete and $u_X^*$ preserves $\infty$-connective morphisms, it is enough to show that $\colim \chi \to X$ is $\infty$-connected, which can be checked by showing that it is sent to an isomorphism by the stalk $x^*: \Sh_X \to \mathcal{S}$ for every elements $x \in X$ by \cite[Lm.~A.3.9.]{jL2012}. The left Kan extension of the constant functor $\mathbf{1}_\mathcal{S}: A_x \to \mathcal{S}$ given by $A \xrightarrow{\chi} \Sh_X \xrightarrow{x^*} \mathcal{S}$, so that we obtain 
	$$	\mathbf{1}_\mathcal{S}	=	\colim (\mathbf{1}_\mathcal{S}: A_x \to \mathcal{S})	=	\colim x^* \chi	=	x^* \colim \chi	\to	x^*X	= \mathbf{1}_\mathcal{S}		$$	
where the first isomorphism holds by assumption. 
\end{proof} 

In \cite{jL2012} Lurie first gives a technical proof of Corollary \ref{sieve}, from which he derives Theorem \ref{lsvkt} using arguments similar to those used in the proof of Theorem \ref{lsvkt}.   

\begin{corollary}	\label{sieve}	
Let $X$ be a topological space, and $R \hookrightarrow X$, a covering sieve (in $\widehat{\Open_X}$), then the cocone $R^\rhd \to \TSpc$ obtained by composing the colimiting cocone $R^\rhd \to \Open_X$ with $\Open_X \to \TSpc$ is a homotopy colimit. 
\end{corollary} 

\begin{proof} 
Set $A = R$, and $\chi$ equal to the inclusion $R \hookrightarrow \Open_X$. For every point $x \in X$ the category $A_x$ is filtered, and thus its classifying space is contractible. 
\end{proof} 

\paragraph{$\mathbf{R}$-epimorphisms}

Both Theorem \ref{dihc} and Theorem \ref{homotopy quotient} are most naturally expressed in a more general form than the original statements for which we require the notion of \emph{$\mathbf{R}$-epimorphism}. 

\begin{proposition}	\label{ree}
Let $X \to Y$ be a continuous map, then the following are equivalent: 
\begin{enumerate}[label = {\normalfont (\arabic*)}]
\item	The map $u^* X \to u^*Y$ in $\Diff^0$ is an effective epimorphism. 
\item	For every $d \geq 0$, every continuous map $\mathbf{R}^d \to Y$, and every point $x \in \mathbf{R}^d$ there exists a neighbourhood  $U$ of $x$ and a lift 
% https://q.uiver.app/#q=WzAsNCxbMSwxLCJcXG1hdGhiZntSfV5uIl0sWzIsMSwiWSJdLFsyLDAsIlgiXSxbMCwxLCJVIl0sWzAsMV0sWzIsMV0sWzMsMCwiIiwwLHsic3R5bGUiOnsidGFpbCI6eyJuYW1lIjoiaG9vayIsInNpZGUiOiJ0b3AifX19XSxbMywyLCIiLDIseyJzdHlsZSI6eyJib2R5Ijp7Im5hbWUiOiJkYXNoZWQifX19XV0=
\[\begin{tikzcd}
	&& X \\
	U & {\mathbf{R}^n} & Y
	\arrow[from=2-2, to=2-3]
	\arrow[from=1-3, to=2-3]
	\arrow[hook, from=2-1, to=2-2]
	\arrow[dashed, from=2-1, to=1-3]
\end{tikzcd}\]
\end{enumerate} 
\qed
\end{proposition}  

\begin{definition} 
A continuous map $X \to Y$ is an \Emph{$\mathbf{R}$-epimorphism} if it satisfies the equivalent conditions of Proposition \ref{ree}. \qede
\end{definition} 

\paragraph{Dugger and Isaksen's hypercovering theorem} 

\begin{definition} 
Let $X$ be a topological space, then a simplicial diagram $U: \Delta^{\op} \to \TSpc_{/X}$ is an \Emph{$\mathbf{R}$-hypercover} if $U^{\Delta^n} \to U^{\partial \Delta^n}$ is an $\mathbf{R}$-epimorphism for all $n \geq 0$. \qede
\end{definition} 

\begin{example} 
Any ordinary hypercover of a topological space is an $\mathbf{R}$-hypercover. \qede
\end{example} 

\begin{theorem}[{\cite[Th.~1.1]{dDdI2004}}]	\label{dihc}
Let $X$ be a topological space, and $U: \Delta^{\op} \to \TSpc_{/X}$, an $\mathbf{R}$-hypercover, then the corresponding cocone $\overline{U}: (\Delta^{\op})^{\rhd} \to \TSpc$ is a homotopy colimit.
\end{theorem} 

\begin{proof} 
The functor $\Diff^0 \leftarrow \TSpc: v^*$ preserves limits, and sends $\mathbf{R}$-epimorphism to effective epimorphisms by definition. Therefore, the composition of $(v^*_{/X}U)^{\partial \Delta^n} \xrightarrow{\simeq} v^*(U^{\partial \Delta^n}) \to v^* U^{\Delta^n}$ is an effective epimorphism for every $n \geq 1$, so that $v_{/X}^*U$ is a hypercover. Thus, $v^*\overline{U}$ is a colimit by descent, and we may apply Proposition \ref{sing shape}. 
\end{proof} 

\paragraph{Principal bundles}	\label{Principal bundles}

Until the end of this section $G$ denotes a topological group. Assume that $G$ acts on a topological space $X$. If the action is principal, then it is often taken for granted that $X/G$ is homotopically well-behaved. For an example of what what is meant by this: if in addition to being principal, $X$ is moreover contractible, then $X/G$ is a model for $BG$. To obtain a precise notion of this homotopical well-behavedness, we note that the localisation functor $L: \TSpc \to \mathcal{S}$ commutes with finite products, so that we obtain an action of $LG$ on $LX$. We then say that $X / G$ is a \Emph{homotopy quotient} of the action of $G$ on $X$ if the comparison map $LX / LG \to L (X / G)$ is an isomorphism.  We will prove in Theorem \ref{homotopy quotient} that if the action of $G$ on $X$ is principal, then $X / G$ is indeed a homotopy quotient. 

\begin{remark} 
It is often claimed, incorrectly, that $X / G$ is a homotopy quotient for any \emph{free} action. To see that this is not the case, let $G$ act on copy of itself equipped with the trivial topology, then the quotient is a point. If the quotient were a homotopy quotient, it would have to model the classifying space of $G$, which is only true if $G$ itself is weakly contractible. It is true however, that any free quotient in any strict test topos $\mathcal{E}$ is a homotopy quotient, as quotients by free actions commute with the inclusion of $\mathcal{E}$ into its associated hypercomplete $\infty$-topos. 
\qede 
\end{remark} 

Our notion of \emph{homotopy quotient} agrees with more traditional notions of homotopically well-behaved quotients. For example, the category of topological spaces with a continuous $G$-action, $\TSpc_G$, admits a model structure in which the weak equivalences are those equivariant maps whose underlying maps of topological spaces are weak equivalences (see \cite[Th.~VI.5.2]{jpM1996}), and one may ask when $X / G$ has the same weak homotopy type as $X / \! / G$, where $\emptyinput \, /\! / G$ denotes the derived functor of the quotient functor $\TSpc_G \to \TSpc$. These two notions agree by the following proposition. 

\begin{proposition} 
The functor $\TSpc_G \to \mathcal{S}_{LG}$ is a localisation along the weak equivalences in $\TSpc_G$. 
\end{proposition} 

\begin{proof} 
The functor $\mathcal{S}_{LG} \leftarrow \TSpc_G$ factors as $\mathcal{S}_{LG} \leftarrow (\widehat{\Delta})_{sG} \leftarrow \TSpc_G$ (where $sG$ is the total singular complex of $G$), so the proposition follows from Theorem \ref{equivariant test} and the fact that $(\widehat{\Delta})_{sG} \leftarrow \TSpc_G$ preserves weak equivalences and induces and equivalence of $\infty$-categories upon localisation (see \cite[1.7]{wDdK1984}). 
\end{proof} 

Using classical methods we are only aware of a proof of $X / \! / G \sim X / G$  for a principal action under the additional (mild) assumptions that $G$ is well pointed, and $X$ is a compactly generated weakly Hausdorff space: By \cite[9.2.10]{eR2014} $X / \! / G$ may be computed as the topological realisation of 
% https://q.uiver.app/#q=WzAsMyxbMiwwLCJYIl0sWzEsMCwiWCBcXHRpbWVzIEciXSxbMCwwLCJcXGNkb3RzIFxcZW5za2lwIFggXFx0aW1lcyBHIFxcdGltZXMgRyJdLFsxLDAsIiIsMCx7Im9mZnNldCI6LTF9XSxbMSwwLCIiLDIseyJvZmZzZXQiOjF9XSxbMiwxLCIiLDIseyJvZmZzZXQiOi0yfV0sWzIsMSwiIiwyLHsib2Zmc2V0IjoyfV0sWzIsMV1d
\begin{tikzcd}
	{\cdots \enskip X \times G \times G} & {X \times G} & X,
	\arrow[shift left, from=1-2, to=1-3]
	\arrow[shift right, from=1-2, to=1-3]
	\arrow[shift left=2, from=1-1, to=1-2]
	\arrow[shift right=2, from=1-1, to=1-2]
	\arrow[from=1-1, to=1-2]
\end{tikzcd}
and this topological realisation is weakly equivalent to $X / G$ by \cite[Props.~7.1 \& 8.5]{jpM1975} (which relies on technical pointset topological arguments). 

We will now systematically investigate the relationship between principal actions and homotopy quotients. 

\begin{definition} 
An \Emph{$\mathbf{R}$-principal $G$-bundle} is an $\mathbf{R}$-epimorphism $P \to B$ together with a fibre preserving action of $G$ on $P$, such that the shearing map $P \times G \to P \times_B P$ is a homeomorphism. \qede
\end{definition} 

\begin{example} 
Any principal $G$-bundles is an $\mathbf{R}$-principal $G$-bundle. \qede
\end{example} 

\begin{lemma}	\label{associated nice}
Let $P \to B$ be an $\mathbf{R}$-principal bundle, then the diagram %and $X$, a $G$-space then 
% https://q.uiver.app/#q=WzAsNCxbMCwwLCJcXERpZmZfe3ZeKkd9XjAiXSxbMCwxLCJcXERpZmZeMCJdLFsxLDAsIlxcVFNwY19HIl0sWzEsMSwiXFxUU3BjIl0sWzAsMSwidl4qUFxcdGltZXNfe3ZeKkd9IFxcZW1wdHlpbnB1dCIsMl0sWzIsMF0sWzIsMywiUFxcdGltZXNfRyBcXGVtcHR5aW5wdXQiXSxbMywxXV0=
\[\begin{tikzcd}
	{\Diff_{v^*G}^0} & {\TSpc_G} \\
	{\Diff^0} & \TSpc
	\arrow["{v^*P\times_{v^*G} \emptyinput}"', from=1-1, to=2-1]
	\arrow[from=1-2, to=1-1]
	\arrow["{P\times_G \emptyinput}", from=1-2, to=2-2]
	\arrow[from=2-2, to=2-1]
\end{tikzcd}\]
commutes. 
\end{lemma} 

\begin{proof} 
We will show that the natural transformation $v^*P \times_{v^*G} v^*(\emptyinput) \to v^*(P \times_G \emptyinput)$ -- obtained by whiskering $P \times \emptyinput: \TSpc \to \TSpc$ with $v^*(\emptyinput) / v^*G \to v^*(\emptyinput / G)$ -- is a natural isomorphism. 

Pulling back $X \times_G P \to B$ along $P \to B$ yields the Cartesian natural transformation 
%\begin{equation}	\label{cbc}
% https://q.uiver.app/#q=WzAsOCxbMywwLCJYIFxcdGltZXNfR1AiXSxbMywxLCJCIl0sWzIsMSwiUCJdLFsyLDAsIlggXFx0aW1lcyBQIl0sWzEsMSwiUCBcXHRpbWVzX0IgUCJdLFsxLDAsIlggXFx0aW1lcyBHIFxcdGltZXMgUCJdLFswLDAsIlxcY2RvdHMgXFxlbnNraXAgWCBcXHRpbWVzIEcgXFx0aW1lcyBHIFxcdGltZXMgUCJdLFswLDEsIlxcY2RvdHMgXFxlbnNraXAgUCBcXHRpbWVzX0IgUCBcXHRpbWVzX0IgUCJdLFswLDFdLFsyLDFdLFszLDJdLFszLDBdLFszLDEsIiIsMSx7InN0eWxlIjp7Im5hbWUiOiJjb3JuZXIifX1dLFs0LDIsIiIsMSx7Im9mZnNldCI6LTF9XSxbNSwzLCIiLDEseyJvZmZzZXQiOjF9XSxbNSwzLCIiLDEseyJvZmZzZXQiOi0xfV0sWzUsNF0sWzUsMiwiIiwxLHsic3R5bGUiOnsibmFtZSI6ImNvcm5lciJ9fV0sWzYsNV0sWzYsNSwiIiwxLHsib2Zmc2V0IjotMn1dLFs2LDUsIiIsMSx7Im9mZnNldCI6Mn1dLFs3LDRdLFs3LDQsIiIsMSx7Im9mZnNldCI6LTJ9XSxbNyw0LCIiLDEseyJvZmZzZXQiOjJ9XSxbNiw3XSxbNiw0LCIiLDEseyJzdHlsZSI6eyJuYW1lIjoiY29ybmVyIn19XSxbNCwyLCIiLDEseyJvZmZzZXQiOjF9XV0=
\[\begin{tikzcd}
	{\cdots \enskip X \times G \times G \times P} & {X \times G \times P} & {X \times P} & {X \times_GP} \\
	{\cdots \enskip P \times_B P \times_B P} & {P \times_B P} & P & B
	\arrow[from=1-4, to=2-4]
	\arrow[from=2-3, to=2-4]
	\arrow[from=1-3, to=2-3]
	\arrow[from=1-3, to=1-4]
	\arrow["\lrcorner"{anchor=center, pos=0.125}, draw=none, from=1-3, to=2-4]
	\arrow[shift left, from=2-2, to=2-3]
	\arrow[shift right, from=1-2, to=1-3]
	\arrow[shift left, from=1-2, to=1-3]
	\arrow[from=1-2, to=2-2]
	\arrow["\lrcorner"{anchor=center, pos=0.125}, draw=none, from=1-2, to=2-3]
	\arrow[from=1-1, to=1-2]
	\arrow[shift left=2, from=1-1, to=1-2]
	\arrow[shift right=2, from=1-1, to=1-2]
	\arrow[from=2-1, to=2-2]
	\arrow[shift left=2, from=2-1, to=2-2]
	\arrow[shift right=2, from=2-1, to=2-2]
	\arrow[from=1-1, to=2-1]
	\arrow["\lrcorner"{anchor=center, pos=0.125}, draw=none, from=1-1, to=2-2]
	\arrow[shift right, from=2-2, to=2-3]
\end{tikzcd}\]
%\end{equation} 
As $v^*$ preserves limits, we see that
%\begin{equation}	\label{cbc2}
% https://q.uiver.app/#q=WzAsMyxbMiwwLCJ2XiogWCBcXHRpbWVzIHZeKlAiXSxbMSwwLCJ2XiogWCBcXHRpbWVzIHZeKiAgRyBcXHRpbWVzIHZeKlAiXSxbMCwwLCJcXGNkb3RzIFxcZW5za2lwIHZeKiBYIFxcdGltZXMgdl4qICBHIFxcdGltZXMgdl4qICBHIFxcdGltZXMgdl4qUCJdLFsxLDAsIiIsMCx7Im9mZnNldCI6LTF9XSxbMSwwLCIiLDIseyJvZmZzZXQiOjF9XSxbMiwxLCIiLDIseyJvZmZzZXQiOi0yfV0sWzIsMSwiIiwyLHsib2Zmc2V0IjoyfV0sWzIsMV1d
\[\begin{tikzcd}
	{\cdots \enskip v^* X \times v^*  G \times v^*  G \times v^*P} & {v^* X \times v^*  G \times v^*P} & {v^* X \times v^*P}
	\arrow[shift left, from=1-2, to=1-3]
	\arrow[shift right, from=1-2, to=1-3]
	\arrow[shift left=2, from=1-1, to=1-2]
	\arrow[shift right=2, from=1-1, to=1-2]
	\arrow[from=1-1, to=1-2]
\end{tikzcd}\]
%\end{equation}
is the \v{C}ech complex both of $v^*X \times v^*P \to v^*(X \times_GP)$ and of $v^*X \times v^*P \to v^*X \times_{v^*G} v^*P$, so that the comparison map $v^*X \times_{v^*G} v^*P \to v^*(X \times_GP)$ is an isomorphism by descent. 
\end{proof} 

\begin{theorem}	\label{associated bundle} 
Let $P \to B$ be a an $\mathbf{R}$-principle $G$-bundle, and $X$, a $G$-space, then the comparison map $LX \times_{LG} LP \to L(X \times_G P)$ is an isomorphism in $\mathcal{S}$. 
\end{theorem}

\begin{proof} 
As $v^*P \times_{v^*G} v^*X = (v^*P \times v^*X) / v^*G$, the theorem follows from Lemma \ref{associated nice} and Propositions \ref{locally contractible quotients} \& \ref{sing shape}.
\end{proof} 

Setting $X = \mathbf{1}$ yields the following corollary: 

\begin{corollary}	\label{homotopy quotient} 
If $P \to B$ is an $\mathbf{R}$-principal $G$-bundle, then $B$ is the homotopy quotient of the $G$-space $P$. \qed
\end{corollary}

\begin{corollary}	\label{classifying space}
If $E \to B$ is a $\mathbf{R}$-principal $G$-bundle with $E$ weakly contractible, the comparison map $B(LG) \to LB$ is an isomorphism in $\mathcal{S}$. \qed
\end{corollary} 

We now turn to some classical theorems in (Borel) equivariant homotopy theory. Until the end of the section, let $E \to B$ denote a $\mathbf{R}$-principal $G$-bundle with $E$ weakly contractible. 

\begin{proposition}	\label{borel}
The functor $\emptyinput \times_G E: \TSpc_G \to \TSpc$ preserves weak equivalences, and the induced functor $\mathcal{S}_{LG} \to \mathcal{S}$ is canonical isomorphic to $\emptyinput / LG$. 
%For any $G$-space $X$ the homotopy types $LX / LG$ and $L(X \times_G E)$ are canonically isomorphic in $\mathcal{S}$. 
\end{proposition} 

\begin{proof} 
For any topological space $X$ the map $v^*X \times v^*E \to v^*X$ is a shape equivalence, so that the outer square in  
% https://q.uiver.app/#q=WzAsNixbMCwwLCJcXG1hdGhjYWx7U31fe0xHfSJdLFswLDEsIlxcbWF0aGNhbHtTfSJdLFsxLDAsIlxcRGlmZl9HXjAiXSxbMSwxLCJcXERpZmZeMCJdLFsyLDAsIlxcVFNwY19HIl0sWzIsMSwiXFxUU3BjIl0sWzAsMSwiXFxlbXB0eWlucHV0IC8gTEciLDJdLFsyLDBdLFsyLDMsIlxcZW1wdHlpbnB1dCAvIHZeKkciLDJdLFszLDFdLFs0LDJdLFs0LDUsIkVcXHRpbWVzX0cgXFxlbXB0eWlucHV0Il0sWzUsM10sWzUsMiwiXFxzaW0iLDJdXQ==
\[\begin{tikzcd}
	{\mathcal{S}_{LG}} & {\Diff_G^0} & {\TSpc_G} \\
	{\mathcal{S}} & {\Diff^0} & \TSpc
	\arrow["{\emptyinput / LG}"', from=1-1, to=2-1]
	\arrow[from=1-2, to=1-1]
	\arrow["{\emptyinput / v^*G}"', from=1-2, to=2-2]
	\arrow[from=2-2, to=2-1]
	\arrow[from=1-3, to=1-2]
	\arrow["{E\times_G \emptyinput}", from=1-3, to=2-3]
	\arrow[from=2-3, to=2-2]
	\arrow["\sim"', from=2-3, to=1-2, Rightarrow]
\end{tikzcd}\]
commutes by Lemma \ref{associated nice} and Proposition \ref{locally contractible quotients}.
\end{proof} 

In other words, the Borel construction calculates homotopy quotients. We end this part by modelling the equivalence $\mathcal{S}_{LG} = \mathcal{S}_{/BLG}$ in the setting of topological spaces. 

\begin{lemma}	\label{principal fibration} 
Any $\mathbf{R}$-principal $G$-bundle is sharp. 
\end{lemma} 

\begin{proof} 
As in \S \ref{comparing nerves} we can define an extended simplex functor $\mathbf{A}^\bullet: \Delta \to \TSpc$, which creates the standard weak equivalences in $\TSpc$ by combining Proposition \ref{extended nerve} with the argument used in the proof of Proposition \ref{sing shape}. Let $P \to B$ be an $\mathbf{R}$-principal $G$-bundle, then by Proposition \ref{sd} it is enough to show that $(\mathbf{A}^\bullet)^*P \to (\mathbf{A}^\bullet)^*B$ is a fibration in $[\Delta^{\op}, \mathcal{S}]$ (see Proposition \ref{EMG}). The square 
% https://q.uiver.app/#q=WzAsNCxbMCwwLCIoXFxtYXRoYmZ7QX1eXFxidWxsZXQpXyFcXExhbWJkYV9rXm4iXSxbMCwxLCJcXG1hdGhiZntBfV5uIl0sWzEsMCwiUCJdLFsxLDEsIkIiXSxbMCwxLCIiLDAseyJzdHlsZSI6eyJ0YWlsIjp7Im5hbWUiOiJob29rIiwic2lkZSI6InRvcCJ9fX1dLFsyLDNdLFswLDJdLFsxLDNdXQ==
\[\begin{tikzcd}
	{(\mathbf{A}^\bullet)_!\Lambda_k^n} & P \\
	{\mathbf{A}^n} & B
	\arrow[hook, from=1-1, to=2-1]
	\arrow[from=1-2, to=2-2]
	\arrow[from=1-1, to=1-2]
	\arrow[from=2-1, to=2-2]
\end{tikzcd}\]
admits a lift iff the square 
\label{tsl}
\begin{equation} 
% https://q.uiver.app/#q=WzAsNCxbMCwwLCIoXFxtYXRoYmZ7QX1eXFxidWxsZXQpXyFcXExhbWJkYV9rXm4iXSxbMCwxLCJcXG1hdGhiZntBfV5uIl0sWzEsMCwiUHxfe1xcbWF0aGJme0F9Xm59Il0sWzEsMSwiXFxtYXRoYmZ7QX1ebiJdLFswLDEsIiIsMCx7InN0eWxlIjp7InRhaWwiOnsibmFtZSI6Imhvb2siLCJzaWRlIjoidG9wIn19fV0sWzIsM10sWzAsMl0sWzEsMywiPSJdXQ==
\begin{tikzcd}
	{(\mathbf{A}^\bullet)_!\Lambda_k^n} & {P|_{\mathbf{A}^n}} \\
	{\mathbf{A}^n} & {\mathbf{A}^n}
	\arrow[hook, from=1-1, to=2-1]
	\arrow[from=1-2, to=2-2]
	\arrow[from=1-1, to=1-2]
	\arrow["{=}", from=2-1, to=2-2]
\end{tikzcd}
\end{equation} 
does. As both $\mathbf{A}^n$ and $(\mathbf{A}^\bullet)_!\Lambda_k^n$ are $\Delta$-generated, we may assume w.l.o.g., that $P|_{\mathbf{A}^n}$ is $\Delta$-generated, but as $P|_{\mathbf{A}^n} \to \mathbf{A}^n$ admits local sections, it is a principal $G$-bundle, and thus trivial, because $\mathbf{A}^n$ is contractible. Thus, $P|_{\mathbf{A}^n} \to \mathbf{A}^n$ is homeomorphic (over $\mathbf{A}^n$) to the projection map $\mathbf{A}^n \times G \to \mathbf{A}^n$, and (\ref{tsl}) admits a lift, as $(\mathbf{A}^\bullet)_!\Lambda_k^n \hookrightarrow \mathbf{A}^n$ admits a retract. The lemma then follows by Proposition \ref{wfsa}. 
\end{proof} 

\begin{remark} 
It is also possible to show that $\mathbf{R}$-principal $G$-bundles are fibrations in the Serre-Quillen model structure, but the proof is more involved.  \qede 
\end{remark} 

The following result is classical in the simplicial setting (see \cite{eDwDdK1980}); we were not able to find a reference in the topological setting. 

\begin{proposition}	\label{gte}
Both constituent functors in the adjunction
\begin{equation}	\label{tga}
	\adjunction{\emptyinput \times_B E: \TSpc_{/B}}	{\TSpc_G: \emptyinput \times_G E}	
\end{equation} 
preserve weak equivalences, and the induced adjunction 
$$	\adjunction{S_{/LB} = W^{-1}\TSpc_{/B}}	{W^{-1}\TSpc_G}	$$
is an adjoint equivalence. 
\end{proposition} 

\begin{proof} 
The functor $\emptyinput \times_G E$ preserves weak equivalences by Proposition \ref{borel}, and $\emptyinput \times_B EG$ preserves weak equivalences by Lemma \ref{principal fibration}.  The functor $\emptyinput \times_G E$ induces an equivalence on localisations by Proposition \ref{borel}, which by \cite[Prop.~7.14]{dcC2019} concludes the proof. 
\end{proof} 

\begin{remark} 
It is not difficult to prove Proposition \ref{gte} by showing by hand that the unit and counit of (\ref{tga}) are weak natural transformations. Moreover, it is possible to show that (\ref{tga}) is a Quillen adjunction, and thus, by Proposition \ref{gte}, a Quillen equivalence. \qede
\end{remark} 

\section{Homotopical calculi on differentiable sheaves}	\label{Homotopical calculi on differentiable sheaves}	

In \S \ref{model structures} we implement the technology of \S \ref{Transferring model structures} to construct numerous model structures on $\Diff^r$ and related ($\infty$-)categories, and discuss some of their properties. Then, in \S \ref{The differentiable Oka principle} we single out one of these model structures, the \emph{Kihara model structure}, and use it to prove Theorem \ref{manfib} which states that a large class of (possibly infinite dimensional) manifolds satisfies the \emph{smooth Oka principle}. We will spend the rest of this introduction explaining what the differentiable Oka principle is, why it is interesting, and our strategy for proving Theorem \ref{manfib}. 

In this introduction all topological spaces are assumed to belong to some convenient category, such as compactly or $\Delta$-generated spaces, and $\TSpc$ denotes the category of such spaces. Let $A,X$ be topological spaces with $A$, a CW complex, then $\TSpc(A,X)$ together with the compact open topology (denoted by $\underline{\TSpc}(A,X)$) is a model for $\mathcal{S}(LA,LX)$, where $L: \TSpc \to \mathcal{S}$ is the localisation functor. This follows from the fact that the model structure on $\TSpc$ is Cartesian, by which $A \times \emptyinput \dashv \underline{\TSpc}(A, \emptyinput)$ is a Quillen adjunction.  As all objects in $\TSpc$ are fibrant, both $A \times \emptyinput$ and $\underline{\TSpc}(A, \emptyinput)$ preserve weak equivalences, and $A \times \emptyinput \dashv \underline{\TSpc}(A, \emptyinput)$ descends to an adjunction on homotopy categories by \cite[Prop.~7.1.14]{dcC2019}:
% https://q.uiver.app/?q=WzAsNCxbMCwwLCJcXFRTcGMiXSxbMSwwLCJcXFRTcGMiXSxbMCwxLCJcXG1hdGhjYWx7U30iXSxbMSwxLCJcXG1hdGhjYWx7U30iXSxbMCwxLCJBIFxcdGltZXMgXFxlbXB0eWlucHV0IiwwLHsib2Zmc2V0IjotMn1dLFsxLDAsIlxcdW5kZXJsaW5le1xcVFNwY30oQSxcXGVtcHR5aW5wdXQpIiwwLHsib2Zmc2V0IjotMn1dLFsyLDMsIkxBIFxcdGltZXMgXFxlbXB0eWlucHV0IiwwLHsib2Zmc2V0IjotMn1dLFszLDIsIlxcbWF0aGNhbHtTfShMQSwgXFxlbXB0eWlucHV0KSIsMCx7Im9mZnNldCI6LTJ9XSxbMCwyXSxbMSwzXSxbNCw1LCIiLDAseyJsZXZlbCI6MSwic3R5bGUiOnsibmFtZSI6ImFkanVuY3Rpb24ifX1dLFs2LDcsIiIsMCx7ImxldmVsIjoxLCJzdHlsZSI6eyJuYW1lIjoiYWRqdW5jdGlvbiJ9fV1d
\[\begin{tikzcd}[column sep = 2.5 cm, row sep = large]
	\TSpc & \TSpc \\
	{\mathcal{S}} & {\mathcal{S}.}
	\arrow[""{name=0, anchor=center, inner sep=0}, "{A \times \emptyinput}", shift left=2, from=1-1, to=1-2]
	\arrow[""{name=1, anchor=center, inner sep=0}, "{\underline{\TSpc}(A,\emptyinput)}", shift left=2, from=1-2, to=1-1]
	\arrow[""{name=2, anchor=center, inner sep=0}, "{LA \times \emptyinput}", shift left=2, from=2-1, to=2-2]
	\arrow[""{name=3, anchor=center, inner sep=0}, "{\mathcal{S}(LA, \emptyinput)}", shift left=2, from=2-2, to=2-1]
	\arrow[from=1-1, to=2-1]
	\arrow[from=1-2, to=2-2]
	\arrow["\dashv"{anchor=center, rotate=-90}, draw=none, from=0, to=1]
	\arrow["\dashv"{anchor=center, rotate=-90}, draw=none, from=2, to=3]
\end{tikzcd}\]
The derived left adjoint is given by $LA \times \emptyinput$ by Corollary \ref{as}, and thus the derived right adjoint must be canonically equivalent to $\mathcal{S}(LA, \emptyinput)$. 

Now, moving on to the differentiable setting, let $M$ be a closed smooth manifold, and $N$ an arbitrary smooth manifold, then the set of smooth maps $\Diff^\infty(M,N)$ admits a canonical structure of a Fr\'{e}chet manifold (see \cite[Th.~1.11]{mGvG1973}). Via smoothing theory it is then possible to show that the homotopy type of this Fr\'{e}chet manifold is equivalent to the homotopy type of $\underline{\TSpc}(M,N)$ (where $M$, $N$ now denote the underlying topological spaces of the smooth manifolds $M$, $N$), which is equivalent to $\mathcal{S}(LM,LN)$, which is equivalent to $\mathcal{S}\big((\pi_{\Diff^\infty})_!M,(\pi_{\Diff^\infty})_!N\big)$ by Theorem \ref{change theorem}. By \cite[Lm~A.1.7]{kW2012} the Fr\'{e}chet manifold of smooth maps from $M$ to $N$ is canonically equivalent to $\underline{\Diff}^\infty(M,N)$, so it is natural to ask for which differentiable sheaves the internal mapping sheaf $\pi_! \underline{\Diff}^\infty(A,X)$ is a model for $\mathcal{S}(\pi_!  A, \pi_! X)$. More precisely (and from now on for $r$ no longer necessarily equal to $\infty$), by Corollary \ref{shape finite products} the shape functor $\pi_!: \Diff^r \to \mathcal{S}$ commutes with finite products so that we obtain a canonical map $\pi_! \underline{\Diff}^r(A,X) \to \mathcal{S}(\pi_!  A, \pi_! X)$ by applying $\pi_!$ to the evaluation map $\underline{\Diff}^r(A,X) \times A \to X$, and then taking the transpose of $\pi_!\underline{\Diff}^r(A,X) \times \pi_!A \to \pi_! X$. 

\begin{definition}	\label{Oka cofibrant} 
A differentiable sheaf $A$ satisfies the \Emph{differentiable Oka principle} or is \Emph{Oka cofibrant} if for every $r$-times differentiable sheaf $X$ the map $\pi_! \underline{\Diff}^r(A,X) \to \mathcal{S}(\pi_! A, \pi_! X)$ is an equivalence. \qede
\end{definition} 

\begin{remark} 
This terminology is inspired by work of Sati and Schreiber (e.g., \cite{hSuS2021}), where an object in $\Diff^\infty$ satisfying the differentiable Oka principle is said to satisfy the \emph{smooth} Oka principle. We have chosen the term \emph{differentiable} over \emph{smooth} to emphasise that in our setting $r$ is not necessarily equal to $\infty$. \qede
\end{remark} 

In Theorem \ref{Proof of the differentiable Oka principle} we prove that a large class of (possibly infinite dimensional) manifolds satisfy the differentiable Oka principle. We will now discuss our proof strategy:  Having constructed several model structures on $\Diff^r$ in \S \ref{model structures}, we might hope to prove Theorem \ref{Proof of the differentiable Oka principle} by showing that one of these satisfies the following three properties, so that we may argue similarly as in $\TSpc$: 
	\begin{enumerate} 
	\item	\label{ccf1} The model structure is Cartesian closed. 
	\item \label{ccf2} All objects are fibrant. 
	\item	All manifolds are cofibrant. 
	\end{enumerate} 
Unfortunately, unless $r = 0$ we are not able to get 1.\ and 2.\ simultaneously for any ``reasonable'' model structure by Proposition \ref{no such luck}. 

We thus bring the theory of \S \ref{Basic theory of homotopical calculi} to bear on our problem, which will allow us to think about homotopical calculi in a more flexible manner than allowed by model structures. Let us assume that we have already shown that a given differentiable sheaf $A$ is Oka cofibrant, and that $S \to D$ is a map between Oka cofibrant objects -- which we think of as constituting a ``cell inclusion'' -- then, if we attach our ``cell'' $D$ along a map $f: S \to A$, a natural way of showing that $A \cup_f D$ is also cofibrant is to show that the pullback 
% https://q.uiver.app/?q=WzAsNCxbMCwwLCJcXHVuZGVybGluZXtcXERpZmZ9KEFcXGN1cF9mRCwgWCkiXSxbMSwwLCJcXHVuZGVybGluZXtcXERpZmZ9KEQsIFgpIl0sWzAsMSwiXFx1bmRlcmxpbmV7XFxEaWZmfShBLCBYKSJdLFsxLDEsIlxcdW5kZXJsaW5le1xcRGlmZn0oUywgWCkiXSxbMCwyXSxbMSwzXSxbMCwxXSxbMiwzXV0=
\[\begin{tikzcd}
	{\underline{\Diff}(A\cup_fD, X)} & {\underline{\Diff}(D, X)} \\
	{\underline{\Diff}(A, X)} & {\underline{\Diff}(S, X)}
	\arrow[from=1-1, to=2-1]
	\arrow[from=1-2, to=2-2]
	\arrow[from=1-1, to=1-2]
	\arrow[from=2-1, to=2-2]
\end{tikzcd}\]
is a homotopy pullback. Thus, we would like to find morphisms $S \to D$ between objects satisfying the differentiable Oka principle such that the morphism $X^D \to X^S$ is sharp for every differentiable sheaf $X$.

\begin{definition} 
A morphism $S \to D$ in $\Diff^r$ is called an \Emph{Oka cofibration} if $X^D \to X^S$ is sharp for every differentiable sheaf $X$. \qede
\end{definition} 

Kihara's simplices $\Delta^\bullet: \Delta \to \Diff^r$ (see Definition \ref{Kiharas simplices}) induce one of the model structures discussed in \S \ref{model structures}, and our strategy is then to show that, while the Kihara model structure is not Cartesian closed, the Kihara horn inclusions are Oka cofibrations. To accomplish this we introduce a new class of sharp morphisms, the \emph{squishy fibrations} in \S \ref{squishy squish}, and show that the morphism $X^{\Delta^n} \to X^{\Lambda_k^n}$ is a squishy fibration for all horn inclusions and all $r$-times differentiable sheaves $X$. Then in \S \ref{Proof of the differentiable Oka principle} we show that a large class of (possibly infinite dimensional) \emph{smooth} manifolds are Oka cofibrant by relating them to simplicial complexes built using Kihara's simplices, thus proving Theorem \ref{manfib}. 

\subsection{Model structures on $\Diff^r$ and related $\infty$-categories}	\label{model structures} 

In this subsection we show that $\Diff^r_{\leq 0}$ is a test category, and construct multiple model structures on $\Diff^r$ and $\Diff^r_{\leq 0}$ using the technology of \S \ref{Transferring model structures}, after which we discuss some of their properties. In \S \ref{The Kihara model structure on diffeological spaces} we show that the model structure on $\Diff^r$ transferred using Kihara's simplices restricts to a model structure on $\Diff^r_{\concr}$, which is again Quillen equivalent to $\widehat{\Delta}$, thus recovering a theorem of Kihara. In \S \ref{The Quillen model structure on topological spaces} we recover Quillen's theorem that the Quillen adjunction $\cadjunction{\widehat{\Delta}}	{\TSpc}$ is a Quillen equivalence by making precise how the model structure on $\Diff^0_{\concr}$ further restricts to $\TSpc$ after ``applying a mild homotopy''. Furthermore, we sketch how this technique may be used to recover how homotopy colimits may be calculated using the bar construction without prior cofibrant replacement. 

\begin{proposition}	\label{cart strict test} 
The category $\Cart^r$ is a strict test category. 
\end{proposition} 

\begin{proof}
By Corollary \ref{sistc} it is enough to observe that $\mathbf{R}$ together with the inclusions of $\{0\}$ and $\{1\}$ is a separating interval. 
\end{proof}

\begin{theorem}[{\cite[Th.~6.1.8]{dcC2003}}]	\label{cst}
The topos $\Diff_{\leq 0}^r$ is a strict test topos. 
\end{theorem} 

\begin{proof} 
Combine the preceding proposition with Theorem \ref{main cisinski} and Corollary \ref{rcid}. 
\end{proof} 

By Propositions \ref{test proper} and \ref{truncated colimits} we obtain, respectively, the following two corollaries: 

\begin{corollary} 
The relative category $\Diff_{\leq 0}^r$ is proper. \qed
\end{corollary} 

\begin{corollary} 
The following are homotopy colimits in $\Diff_{\leq 0}^r$: 
\begin{enumerate}[label = {\normalfont \arabic*.}]
\item Pushouts along monomorphisms. 
\item Filtered colimits. 
\item	Coproducts. 
\end{enumerate} 
\qed
\end{corollary} 

%082308

We now discuss various model structures on $\Diff^r$ and $\Diff^r_{\leq 0}$. For all of the nerve diagrams discussed in \S \ref{comparing nerves} we have already verified that they satisfy the assumptions of Theorem \ref{nerves}, so that they satisfy condition \ref{lmsc2} of Proposition \ref{infinity transfer} and Theorem \ref{transfer}. The nerve diagrams are moreover readily seen to satisfy conditions \ref{lmsc3} and \ref{lmsc4} of Theorem \ref{transfer} using Proposition \ref{simp mono} and Corollary \ref{cube trunc mono}. We thus obtain the following proposition 

\begin{proposition} 
The pullback functors along the diagrams
	\begin{itemize}
	\item[]	$\mathbf{A}^\bullet: \Delta \to \Diff^r_{\leq 0}$ 
	\item[]	$\Delta_{\sub}^\bullet: \Delta \to \Diff^r_{\leq 0}$ 
	\item[]	$\Delta^\bullet: \Delta \to \Diff^r_{\leq 0}$ 
	\item[]	$\BBube^\bullet: \Cube \to \Diff^r_{\leq 0}$ 
	\item[]	$\Cube^\bullet: \Cube \to \Diff^r_{\leq 0}$ 
	\end{itemize}
of \S \ref{comparing nerves} all produce right transferred model structures on  $\Diff^r$ and $\Diff^r_{\leq 0}$ in which the weak equivalences are the shape equivalences.  \qed
\end{proposition} 

\begin{remark} 
A new proof that the model structure on $\Diff^r_{\leq 0}$ transferred along the adjunction $\cadjunction{(\Delta_{\sub}^\bullet)_!: \widehat{\Delta}}	{\Diff^r_{\leq 0}: (\Delta_{\sub}^\bullet)^*}$ is Quillen equivalent to the one on $\widehat{\Delta}$ is given in \cite[Th.~7.4]{dP2022}. His argument uses the nerve theorem in a similar way as discussed in \S \ref{other shapes}.  \qede
\end{remark} 

For us, the most important of these model structures is the following: 

\begin{definition} 
The model structures on $\Diff^r$ and $\Diff^r_{\leq 0}$ transferred along the pullback functor of the diagram $\Delta^\bullet: \Delta \to \Diff^r_{\leq 0}$ are both called the \Emph{Kihara model structure}, and (trivial) (co)fibrations in this model structure are called \Emph{Kihara (trivial) (co)fibrations}. \qede
 \end{definition} 

From Proposition \ref{deformation retract} we see: 

\begin{proposition} 
All objects in the Kihara model structure are fibrant. \qed
\end{proposition} 

From Proposition \ref{ap} we obtain the following corollary: 

\begin{corollary}[{\cite[Lm 5.10]{dBEpBDBdP2019}}]	\label{all products} 
The shape functor $\pi_!: \Diff^r \to \mathcal{S}$  commutes with arbitrary products. \qed
\end{corollary} 

Assume $r > 0$. We now exhibit the principle which shows that none of the model structures induced from the nerves functors in \S \ref{comparing nerves} can simultaneously satisfy conditions \ref{ccf1} \& \ref{ccf2} discussed in the introduction of this section. 

\begin{proposition}	\label{no such luck} 
In any model structure on $\Diff^r$ in which the weak equivalences are the shape equivalences, and in which $\{0\} \hookrightarrow \mathbf{R}$ or $\{0\} \hookrightarrow [0,1]$ is a (necessarily trivial) cofibration the following statements cannot \emph{both} be true.  
	\begin{enumerate}[label = {\normalfont \arabic*.}]
	\item	\label{dcf1} The model structure is Cartesian. 
	\item \label{dcf2} All objects are fibrant. 
	\end{enumerate} 	
\end{proposition} 

\begin{proof} 
We will prove the proposition under the assumption that $\{0\} \hookrightarrow [0,1]$ is a cofibration; the case when $\{0\} \hookrightarrow \mathbf{R}$ is a cofibration is similar. Assume that both \ref{dcf1}\ \& \ref{dcf2}\ hold. By \ref{dcf1}\ the pushout product $\iota$ of $\Delta^{\{0\}} \hookrightarrow \Delta^1$ is a trivial cofibration, and by \ref{dcf2}\ all trivial cofibrations admit a retract, which is not true of $\iota$. 
\end{proof} 

\begin{corollary} 
The Kihara model structures on $\Diff^r$ and $\Diff_{\leq 0}^r$ are not Cartesian closed. \qed
\end{corollary} 

All the other model structures induced by the nerves in \S \ref{nerves} are Cartesian closed: The $\mathbf{A}^\bullet$- and $\Delta_{\sub}^\bullet$-model structures by \cite[\S 8]{dP2022}, and the $\BBube^\bullet$- and $\Cube^\bullet$-model structures by Propositions \ref{universal cube} \& \ref{pp cubes}. 

\begin{corollary} 
Not all objects are fibrant in the model structures transferred from the nerves $\mathbf{A}^\bullet$, $\Delta_{\sub}^\bullet$, $\BBube^\bullet$, $\Cube^\bullet$. \qed
\end{corollary} 

\subsubsection{The Kihara model structure on diffeological spaces}	\label{The Kihara model structure on diffeological spaces}

Here we recover Kihara's model structure on diffeological spaces (for $r = \infty$), and show that the weak equivalences are the shape equivalences, and moreover that the restricted shape functor $\pi_!|_{\Diff_{\concr}^r}: \Diff^r_{\concr} \to \mathcal{S}$ again exhibits $\mathcal{S}$ as the localisation of $\Diff^r_{\concr}$ along the weak equivalences. Moreover we discuss several classes of colimits which are homotopy colimits in $\Diff^r_{\concr}$. 

\begin{proposition}	\label{restrict Kihara} 
The functor $(\Delta^\bullet)_!: \widehat{\Delta} \to \Diff^r_{\leq 0 }$ factors through $\Diff^r_{\concr} \hookrightarrow \Diff^r_{\leq 0 }$. 
\end{proposition} 

\begin{proof} 
The inclusions $\partial \Delta^n \hookrightarrow \Delta^n$ are embeddings, so that all colimits used to construct the realisation of any simplicial set in $\Diff^r_{\leq 0}$ are preserved by the inclusion $\Diff^r_{\concr} \hookrightarrow \Diff^r_{\leq 0 }$ by Corollary \ref{chcl}. 
\end{proof} 

\begin{remark} 
Observe that Proposition \ref{restrict Kihara} fails for the closed simplices $\Delta_{\mathrm{sub}}^n$, precisely because the maps $\partial \Delta_{\mathrm{sub}}^n \hookrightarrow \Delta_{\mathrm{sub}}^n$ are not embeddings. See \cite[\S 6]{dP2022} for a proof of this fact. \qede
\end{remark} 

\begin{theorem}[{\cite[Th.~1.3]{hK2019} \cite[Th.~1.1]{hK2017}}]
There exists a cofibrantly generated model structure on $\Diff_{\concr}^r$, such that 
\begin{enumerate}[label = {\normalfont (\arabic*)}]
\item	\label{tk1}	the weak equivalences are the shape equivalences, 
\item	the generating cofibrations and trivial cofibrations are given by $\{\partial \Delta^n \hookrightarrow \Delta^n\}_{n \geq 0}$ and $\{\Lambda_k^n \hookrightarrow \Delta^n\}_{n \geq 1, \; n \geq k \geq 0}$, 
\item	\label{tk3}	the adjunction $\adjunction{\widehat{\Delta}}{\Diff_{\concr}^r}$ is a Quillen equivalence, and
\item	\label{tk4}	all objects in $\Diff^r_{\concr}$ are fibrant. 
\end{enumerate}  
\end{theorem} 

\begin{proof}
We shall transfer the model structure from $\widehat{\Delta}$ using Proposition \ref{Hirschhorn}, and make heavy use of \ref{tk1}, which follows from Proposition \ref{Kihara nerve}. Thus, let $X$ be a diffeological space, and consider a map $f:\Lambda_k^n \to X$ ($n \geq 1$, $n \geq k \geq 0$), then $X \to X \cup_f \Delta^n$ is a $\Delta^1$-deformation retract (and thus a weak equivalence), since $\Lambda_k^n \hookrightarrow \Delta^n$ is one. The transfinite composition of $\{\Lambda_k^n \hookrightarrow \Delta^n\}$-cell-attachments is a weak equivalence by Proposition \ref{chcl}. Lastly, shape equivalences are closed under retract, because isomorphisms are closed under retracts in $\mathcal{S}$. 

By Proposition \ref{restrict Kihara} both adjoints in $\adjunction{\widehat{\Delta}}{\Diff_{\concr}^r}$ preserve weak equivalences, and the unit and counit are weak natural equivalences by Theorem \ref{transfer}, establishing \ref{tk3}. 

Finally, \ref{tk4} follows from the fact that all inclusions $\Lambda_k^n \to \Delta^n$ ($n \geq 1$, $n \geq k \geq 0$) are deformation retracts.
\end{proof} 

By Corollary \ref{chcl} we obtain the following result. 

\begin{proposition} 
The following classes of colimits are homotopy colimits in $\Diff_{\concr}^r$: 
	\begin{enumerate}[label = {\normalfont \arabic*.}]
	\item	Pushouts of embeddings along monomorphisms. 
	\item	Filtered colimits where all transition morphisms are monomorphisms. 
	\item	Coproducts
	\end{enumerate} 
\end{proposition} 

\subsubsection{The Quillen model structure on topological spaces}	\label{The Quillen model structure on topological spaces}

Milnor's result from \cite{jM1957} that the homotopy categories of CW complexes and Kan complexes are equivalent may be seen as the starting point of abstract homotopy theory, as it lays the groundwork for viewing homotopy types as objects in their own right, which may be presented in many different ways. Quillen refined Milnor's result in \cite{dQ1967} by showing that the adjunction
\begin{equation}	\label{sstsqe}
\cadjunction{{|} \emptyinput {}|: \widehat{\Delta}}	{\TSpc : s}
\end{equation} 
is a Quillen equivalence, providing a systematic framework in which to transfer homotopical arguments from one model category to the other. By \cite[Lms.~2.3.1~\&~2.3.2]{dQ1967} the model structure on $\TSpc$ is transferred from $\widehat{\Delta}$, so we are thus in a situation similar to the one encountered for the various cosimplicial diagrams $\Delta \to \Diff^0$ and $\Delta \to \Diff_{\concr}^0$ seen above, and we will show that our techniques may be used to give a conceptual proof of why (\ref{sstsqe}) is a Quillen equivalence. Before doing so, we give a brief historical overview of work on the equivalence of the homotopy theory of topological spaces and simplicial sets.     %already mention homotopy colimits TO DO 

It appears that Givier was the first to show that the counit of (\ref{sstsqe}) is a natural weak equivalence in \cite{jG1950} using combinatorial arguments. In \cite{jM1957} Milnor then shows that the unit of (\ref{sstsqe}) evaluated at any Kan complex is a weak equivalence, by providing an explicit equivalence on the level of connected components and fundamental groups, and then noting that a relative Hurewicz theorem gives isomorphisms of higher homotopy groups. Using a similar Hurewicz argument, Milnor shows that the total singular complex functor $s$ induces isomorphisms on homotopy groups, so that Milnor is able to recover Givier's result from his using the triangle identities of (\ref{sstsqe}). In \cite[\S VII.1]{pGmZ1967} Gabriel and Zisman provide a new proof of Milnor's theorem by showing that topological realisation $|\emptyinput |$ carries minimal fibrations to fibre bundles, and thus $|\emptyinput |$ also preserves fibration sequences (on homotopy categories), allowing them to obtain isomorphisms between the homotopy groups of any Kan complex $X$ and the topological space $|X|$ using Postnikov towers in the simplicial and topological contexts. Both Hovey and Goress-Jardine (see \cite{mH1999} and \cite{pGjJ1999}) begin by constructing the model structure on $\TSpc$, and then \emph{define} weak equivalences of simplicial sets as morphisms which are sent to weak equivalences of topological spaces. They then use Gabriel and Zisman's minimal fibrations and Milnor's theorem to construct the model structure on simplicial sets from the model structure on topological spaces. As weak equivalences are defined via topological realisation, they must show Givier's result that the counit is a weak equivalence in order to apply  \cite[Cor.~1.3.16]{mH1999}, which they again obtain from Milnor's theorem. %082510

We now turn to our proof that (\ref{sstsqe}) is a Quillen equivalence: As in our setup weak equivalences are created by the total singular complex functor, we need to show that the unit is a natural weak equivalence so that we may apply  \cite[Cor.~1.3.16]{mH1999}, i.e., we must recover Milnor's theorem, (which we do for all simplicial sets, not just Kan complexes). We denote by \linebreak
$|\emptyinput|_{\Diff^0_{\concr}}: \widehat{\Delta} \to \Diff_{\concr}^0$ and $|\emptyinput|_{\TSpc}: \widehat{\Delta} \to \TSpc$ the Yoneda extensions along the diagram $\Delta_{\sub}^\bullet:	\Delta \to \Diff^r_{\leq 0}$ (see \S \ref{comparing nerves}), so that $|\emptyinput|_{\TSpc}$ is just the usual topological realistion.  Observe that the subcategory $\Delta\TSpc \hookrightarrow \TSpc$ spanned by the $\Delta$-generated topological spaces (see \cite{jCgSeW2014}) is exhibited as a subcategory of $\Diff^0_{\concr}$ by $v^*$ (see \S \ref{Colimits of hypercoverings of topological spaces}). As hinted at above, one might then be tempted to implement the same strategy used for constructing a Quillen equivalence between $\widehat{\Delta}$ and $\Diff_{\concr}^0$ in \S \ref{The Kihara model structure on diffeological spaces}, but, unfortunately, the realisation functor $|\emptyinput|_{\Diff^0_{\concr}}: \widehat{\Delta} \to \Diff_{\concr}^0$ does not factor through $\Delta\TSpc$ (note, however, that the \emph{topological} realisation \emph{does} factor through $\Delta\TSpc \hookrightarrow \TSpc$). To see this, note for example that the topological space $|\Lambda_1^2|_{\TSpc}$ is homeomorphic to $[0,1]$, but that by Example \ref{simplex example} the object $|\Lambda_1^2|_{\Diff^0_{\concr}}$ is not even a topological space. Luckily, the realisation in $\Diff^0_{\concr}$ is close enough to the topological realisation for our above strategy to work after a slight modification. All we need to do is to show that for any simplicial set $X$ the morphism $|X|_{\Diff^0_{\concr}} \to |X|_{\TSpc}$ is an $\mathbf{R}$-homotopy equivalence. Then, from the commutative diagram 
% https://q.uiver.app/#q=WzAsMyxbMCwxLCJYIl0sWzEsMCwic3t8fVh7fH1fe1xcRGlmZl4wX3tcXGNvbmNyfX0iXSxbMSwyLCJze3x9WHt8fV97XFxUU3BjfSJdLFswLDFdLFsxLDJdLFswLDJdXQ==
\[\begin{tikzcd}
	& {s{|}X{|}_{\Diff^0_{\concr}}} \\
	X \\
	& {s{|}X{|}_{\TSpc}}
	\arrow[from=2-1, to=1-2]
	\arrow[from=1-2, to=3-2]
	\arrow[from=2-1, to=3-2]
\end{tikzcd}\]
we recover Milnor's theorem from the 2-out-of-3 property. 

We have just seen that while the map on underlying sets of $|X|_{\Diff^0_{\concr}} \to |X|_{\TSpc}$ is a bijection, it is not true that the map in the other direction is continuous. In order to remedy this, we construct below a homotopy $H^n: [0,1] \times \Delta^n \to \Delta^n$ (in $\Delta\TSpc$) for every $n \geq 1$, which deforms a neighbourhood of $\partial \Delta^n$ down to $\partial \Delta^n$ in such a way that the restriction of $H^n$ to any face $\Delta^{n-1}$ yields the homotopy $H^{n-1}$. For any simplicial set $X$ these homotopies assemble to the two homotopies $H^X: [0,1] \times |X|_{\TSpc} \to |X|_{\TSpc}$ and $H^X: [0,1] \times |X|_{\Diff^0_{\concr}}  \to |X|_{\Diff^0_{\concr}} $. We now prove that $|X|_{\TSpc}  \xrightarrow{H^X_1} |X|_{\Diff^0_{\concr}}$ is continuous, so that the maps $|X|_{\Diff^0_{\concr}} \to |X|_{\TSpc}$ and $H^X_1: |X|_{\TSpc} \to |X|_{\Diff^0_{\concr}}$ are then homotopy inverse to each other: 

Let $p: \mathbf{R}^d \to |X|_{\TSpc}$ be a continuous map, and let $x$ be a point in $\mathbf{R}^d$, then there exists a unique non-degenerate simplex $c_x: \Delta^n \hookrightarrow X$ with minimal $n \geq 0$ such that $p(x) \in|\Delta^n|_{\Diff^0_{\concr}} \subseteq |X|_{\Diff^0_{\concr}}$. For every factorisation $\Delta^n \xhookrightarrow{\delta} \Delta^m \hookrightarrow X$ of $c_x$ through a non-degenerate simplex with $m > n$ consider the subset $U_\delta \subseteq |\Delta^m|_{\TSpc}  \subset |X|_{\TSpc} $ given by the interior of the set of those points which are mapped to $|\partial \Delta^m|_{\TSpc}  \subset |X|_{\TSpc}$ by $H^m_1$.  The image $U$ of the canonical map $\coprod_\delta U_\delta \to |X|_{\TSpc}$ is open, and the composition of $p|_{p^{-1}U}: p^{-1}U \to |X|_{\TSpc} \to |X|_{\Diff^0_{\concr}}$ factors through $|\Delta^n|_{\Diff^0_{\concr}} \hookrightarrow |X|_{\Diff^0_{\concr}}$.  \\

\noindent\underline{Construction of $H_n$ for $n \geq 0$:}	For each $n \geq 0$ consider the smooth map $\Delta^n \to \mathbf{R}, \; x \mapsto 1 - \|x \|$, i.e., the map that radially measures the distance from any point in $\Delta^n$ to the unit sphere in $\mathbf{R}^{n+1}$, and take its gradient, which we view as an electric field. For $t \in [0,1]$ the map $H^n_t: \Delta^n \to \Delta^n$ is then given by viewing an element of $\Delta^n$ as a charged particle, which is pushed by the electric field for time $t$. If the particle hits a face of dimension $k$, the particle is pushed by the component of the force field parallel to the $k$-dimensional face, until either $t= 1$ of it hits a face of even lower dimension. \\

Lurie has also recently produced a proof of Milnor's theorem in \cite[\href{https://kerodon.net/tag/0142}{Tag 0142}]{kerodon}. We reframe one of Lurie's key arguments below to obtain yet another proof of Milnor's theorem, this time in the spirit of \S \ref{Colimits of hypercoverings of topological spaces}: \\

\noindent\underline{Claim:} Let $K$ be a finite simplicial set for which $|K|_{\Diff^0_{\concr}} \to |K|_{\TSpc}$ is a weak equivalence, then for any map $f: \partial \Delta^n \to K$ the map $|K \cup_f \Delta^n|_{\Diff^0_{\concr}} \to |K \cup_f \Delta^n|_{\TSpc}$ is likewise a weak equivalence. \\

From the claim it follows inductively that $|K|_{\Diff^0_{\concr}} \to |K|_{\TSpc}$ is a weak equivalence for all finite simplicial sets $K$. An arbitrary simplicial set $X$ may then be written as the filtered colimit of its finite simplicial subsets $\{K \subseteq X\}$. By \cite[Lm.~A.3]{dDdI2004} any map $\mathbf{R}^d \to |X|_{\TSpc}$ factors locally through $|K|_{\TSpc}$ for some finite simplicial subset $K \subseteq X$, so that the colimit of the functor $\{K \subseteq X\} \to \Diff^0_{\concr}, \; K \mapsto |X|_{\TSpc}$, is a topological space. Then for an arbitrary simplicial set $X$ the comparison map $|X|_{\Diff^0_{\concr}} \to |X|_{\TSpc}$ may be written as $\colim_{\{K \subseteq X\}}|K|_{\Diff^0_{\concr}}	\to	\colim_{\{K \subseteq X\}}|K|_{\TSpc}	$ and is thus a weak equivalence by the claim. \\
%$$	\begin{array}{rcl}
%	|X|_{\Diff^0_{\concr}}	&	=	&	|\colim_{\{K \subseteq X\}} K|_{\Diff^0_{\concr}}	\\
%	{}				&	=	&	\colim_{\{K \subseteq X\}}|K|_{\Diff^0_{\concr}}	\\
%	{}				&	\to	&	\colim_{\{K \subseteq X\}}|K|_{\TSpc}			\\
%	{}				&	=	&	|\colim_{\{K \subseteq X\}}K|_{\TSpc} 		\\
%	{}				&	=	&	|X|_{\TSpc}
%	\end{array}
%$$

\noindent\underline{Proof of claim:}	Denote by $0$ the centre of $\Delta^n$. As $| \emptyinput |_{\TSpc}$ preserves colimits, $K \cup_f \Delta^n$ is sent to $|K|_{\TSpc} \cup_f |\Delta^n|_{\TSpc}$, which can equivalently be written as the pushout 
% https://q.uiver.app/#q=WzAsNCxbMCwwLCJ8XFxEZWx0YV5ufF97XFxUU3BjfSBcXHNldG1pbnVzIFxcezBcXH0iXSxbMCwxLCJ8XFxEZWx0YV5ufF97XFxUU3BjfSJdLFsxLDAsInxLfF97XFxUU3BjfSBcXGN1cF9mICh8XFxEZWx0YV5ufF97XFxUU3BjfSBcXHNldG1pbnVzIFxcezBcXH0pIl0sWzEsMSwifEt8X3tcXFRTcGN9IFxcY3VwX2YgfFxcRGVsdGFebnxfe1xcVFNwY30iXSxbMCwxLCIiLDAseyJzdHlsZSI6eyJ0YWlsIjp7Im5hbWUiOiJob29rIiwic2lkZSI6InRvcCJ9fX1dLFswLDJdLFsxLDNdLFsyLDMsIiIsMix7InN0eWxlIjp7InRhaWwiOnsibmFtZSI6Imhvb2siLCJzaWRlIjoidG9wIn19fV1d
\[\begin{tikzcd}
	{|\Delta^n|_{\TSpc} \setminus \{0\}} & {|K|_{\TSpc} \cup_f (|\Delta^n|_{\TSpc} \setminus \{0\})} \\
	{|\Delta^n|_{\TSpc}} & {|K|_{\TSpc} \cup_f |\Delta^n|_{\TSpc}}
	\arrow[hook, from=1-1, to=2-1]
	\arrow[from=1-1, to=1-2]
	\arrow[from=2-1, to=2-2]
	\arrow[hook, from=1-2, to=2-2]
\end{tikzcd}\]
which is a homotopy pushout by Lurie's Seifert-Van Kampen theorem, Theorem \ref{lsvkt}.

\paragraph{Str{\o}m homotopy colimits are Serre homotopy colimits}	\label{strom}

Let $A$ be a small ordinary category, and $X: A \to \TSpc$, a diagram, then it was long a folklore theorem that the topological realisation of the simplicial topological space 
\begin{equation}	\label{chc}
% https://q.uiver.app/#q=WzAsNCxbMywwLCJcXGNvcHJvZF97YV8wIFxcaW4gQV57XFxEZWx0YV4wfX1YX3thXzB9Il0sWzIsMCwiXFxjb3Byb2Rfe2FfMCBcXHRvIGFfMSBcXGluIEFee1xcRGVsdGFeMX19WF97YV8wfSJdLFsxLDAsIlxcY29wcm9kX3thXzAgXFx0byBhXzEgXFx0byBhXzIgXFxpbiBBXntcXERlbHRhXjJ9fVhfe2FfMH0iXSxbMCwwLCJcXGNkb3RzIl0sWzEsMCwiIiwwLHsib2Zmc2V0IjotMX1dLFsxLDAsIiIsMix7Im9mZnNldCI6MX1dLFsyLDEsIiIsMix7Im9mZnNldCI6LTJ9XSxbMiwxLCIiLDIseyJvZmZzZXQiOjJ9XSxbMiwxXSxbMywyLCIiLDAseyJvZmZzZXQiOjN9XSxbMywyLCIiLDAseyJvZmZzZXQiOi0zfV0sWzMsMiwiIiwwLHsib2Zmc2V0IjotMX1dLFszLDIsIiIsMix7Im9mZnNldCI6MX1dXQ==
\begin{tikzcd}
	\cdots & {\coprod_{a_0 \to a_1 \to a_2 \in A^{\Delta^2}}X_{a_0}} & {\coprod_{a_0 \to a_1 \in A^{\Delta^1}}X_{a_0}} & {\coprod_{a_0 \in A^{\Delta^0}}X_{a_0}}
	\arrow[shift left, from=1-3, to=1-4]
	\arrow[shift right, from=1-3, to=1-4]
	\arrow[shift left=2, from=1-2, to=1-3]
	\arrow[shift right=2, from=1-2, to=1-3]
	\arrow[from=1-2, to=1-3]
	\arrow[shift right=3, from=1-1, to=1-2]
	\arrow[shift left=3, from=1-1, to=1-2]
	\arrow[shift left, from=1-1, to=1-2]
	\arrow[shift right, from=1-1, to=1-2]
\end{tikzcd}
\end{equation} 
yields a model for the homotopy colimit of $X$. In general, in a simplicial model category (such as $\TSpc$) one must replace the diagram $X$ with one which is objectwise cofibrant before applying the above construction in order to obtain a model for the homotopy colimit (see \cite[Rmk.~6.3.4]{eR2014}), however, in $\TSpc$ this is not necessary, as proved in \cite[Th.~A.7]{dDdI2004}. Our methods allow for a conceptual explanation for why the cofibrant replacement is not necessary. Unfortunately, producing a rigorous proof of \cite[Th.~A.7]{dDdI2004} using our methods ends up being considerably more involved than the original proof in \cite{dDdI2004}, so we content ourselves with a discussion of the special case when $A = \bullet \leftarrow \bullet \to \bullet$, in which case the geometric realisation of (\ref{chc}) yields the double mapping cylinder.

Consider a span 
% https://q.uiver.app/?q=WzAsMyxbMSwxLCJBIl0sWzAsMCwiQiJdLFsyLDAsIkMiXSxbMCwxXSxbMCwyXV0=
\begin{equation}	\label{mcs}
\begin{tikzcd}
	B && C \\
	& A
	\arrow[from=2-2, to=1-1]
	\arrow[from=2-2, to=1-3]
\end{tikzcd}
\end{equation} 
in $\TSpc$, then the double mapping cylinder $M$ may be more directly obtained as a colimit of 
% https://q.uiver.app/#q=WzAsNSxbMywxLCJBIl0sWzIsMCwiXFxEZWx0YV4xIFxcdGltZXMgQSJdLFs0LDAsIkMiXSxbMSwxLCJBIl0sWzAsMCwiQiJdLFswLDFdLFswLDJdLFszLDFdLFszLDRdXQ==
\[\begin{tikzcd}
	B && {\Delta^1 \times A} && C \\
	& A && A
	\arrow[from=2-4, to=1-3]
	\arrow[from=2-4, to=1-5]
	\arrow[from=2-2, to=1-3]
	\arrow[from=2-2, to=1-1]
\end{tikzcd} \]
or equivalently as the pushout of the diagram
% https://q.uiver.app/#q=WzAsMyxbMiwwLCJDIFxcc3FjdXBfQSAoQSBcXHRpbWVzIFxcRGVsdGFeMSkiXSxbMSwxLCJBIl0sWzAsMCwiQiBcXHNxY3VwX0EgKEEgXFx0aW1lcyBcXERlbHRhXjEpIl0sWzEsMCwiIiwwLHsic3R5bGUiOnsidGFpbCI6eyJuYW1lIjoiaG9vayIsInNpZGUiOiJ0b3AifX19XSxbMSwyLCIiLDAseyJzdHlsZSI6eyJ0YWlsIjp7Im5hbWUiOiJob29rIiwic2lkZSI6ImJvdHRvbSJ9fX1dXQ==
\[\begin{tikzcd}
	{B \sqcup_A (A \times \Delta^1)} && {C \sqcup_A (A \times \Delta^1)} \\
	& A
	\arrow[hook, from=2-2, to=1-3]
	\arrow[hook', from=2-2, to=1-1]
\end{tikzcd}\]
(which admits a natural homotopy equivalence to the diagram (\ref{mcs})). If we take the pushout $M'$ of 
% https://q.uiver.app/#q=WzAsMyxbMiwwLCJ2XipcXGJpZyhDIFxcc3FjdXBfQSAoQSBcXHRpbWVzIFxcRGVsdGFeMSlcXGJpZykiXSxbMSwxLCJ2XipBIl0sWzAsMCwidl4qXFxiaWcoQiBcXHNxY3VwX0EgKEEgXFx0aW1lcyBcXERlbHRhXjEpXFxiaWcpIl0sWzEsMCwiIiwwLHsic3R5bGUiOnsidGFpbCI6eyJuYW1lIjoiaG9vayIsInNpZGUiOiJ0b3AifX19XSxbMSwyLCIiLDAseyJzdHlsZSI6eyJ0YWlsIjp7Im5hbWUiOiJob29rIiwic2lkZSI6ImJvdHRvbSJ9fX1dXQ==
\[\begin{tikzcd}
	{v^*\big(B \sqcup_A (A \times \Delta^1)\big)} && {v^*\big(C \sqcup_A (A \times \Delta^1)\big)} \\
	& {v^*A}
	\arrow[hook, from=2-2, to=1-3]
	\arrow[hook', from=2-2, to=1-1]
\end{tikzcd}\]
in $\Diff^0$ we obtain an object in $\Diff_{\concr}^0$ by Proposition \ref{concrete pushout}, as both legs are embeddings. Now, while the comparison map $M' \to M$ induces a bijection on underlying sets it is not an isomorphism, just as was the case when comparing the realisations of any simplicial set in $\Diff^0_{\concr}$ and $\TSpc$. However, in a completely analogous way as we did for comparing realisations of simplicial sets, one may use $H^1$ to exhibit $M' \to M$ as a homotopy equivalence. For general diagrams, one can again show that the realisation of (\ref{chc}) calculated in $\Diff^0$ is a concrete sheaf, and then use the family of homotopies $H^n$, to show that the comparison map to the topological realisation of (\ref{chc}) is homotopy equivalence. 

Returning briefly to the question of exhibiting double mapping cylinders as homotopy pushouts, one may alternatively mimic Lurie's proof that $\cadjunction{\widehat{\Delta}}	{\TSpc}$ is a Quillen equivalence: The topological space $M$ is the pushout of 
% https://q.uiver.app/#q=WzAsMyxbMCwwLCJCIFxcc3FjdXBfQSBBIFxcdGltZXMgWzAsMSkiXSxbMSwxLCJBIFxcdGltZXMgKDAsMikiXSxbMiwwLCIoMCwyXSBcXHRpbWVzIEEgXFxzcWN1cF9BIEMiXSxbMSwwLCIiLDAseyJzdHlsZSI6eyJ0YWlsIjp7Im5hbWUiOiJob29rIiwic2lkZSI6ImJvdHRvbSJ9fX1dLFsxLDIsIiIsMix7InN0eWxlIjp7InRhaWwiOnsibmFtZSI6Imhvb2siLCJzaWRlIjoidG9wIn19fV1d
\[\begin{tikzcd}
	{B \sqcup_A A \times [0,1)} && {(0,2] \times A \sqcup_A C} \\
	& {A \times (0,2)}
	\arrow[hook', from=2-2, to=1-1]
	\arrow[hook, from=2-2, to=1-3]
\end{tikzcd}\]
(which also admits a natural homotopy equivalence to (\ref{mcs})), which is a homotopy pushout by Lurie's Seifert-Van Kampen theorem (Theorem \ref{lsvkt}). Unfortunately, we do not know how to generalise this approach to arbitrary diagrams in $\TSpc$. 

\subsection{The differentiable Oka principle}	\label{The differentiable Oka principle}
%082310

We now implement our plan for proving Theorem \ref{manfib} described in the introduction of this section. In \S \ref{squishy squish} we construct the squishy fibrations, and use them to exhibit the Kihara boundary inclusions $\partial \Delta^n \hookrightarrow \Delta^n$ as Oka cofibrations. Then, in \S \ref{closure} we discuss various closure properties -- such as being closed under $\Delta^1$-homotopy equivalence --  for differentiable sheaves satisfying the differentiable Oka principle. Using the closure properties discussed in \S \ref{closure} we show in \S \ref{Proof of the differentiable Oka principle} that simplicial complexes built using Kihara's simplices are Oka cofibrant, and then use an argument originally due to Segal and tom Dieck showing that a large class of (possibly infinite dimensional) differentiable manifolds are $\Delta^1$-homotopy equivalent to such simplicial complexes. Finally, in \S \ref{Counterexamples} we discuss some examples of objects not satisfying the differentiable Oka principle such as the long line. 
\subsubsection{Squishy fibrations}	\label{squishy squish} 

The squishy fibrations are defined using a cubical diagram of $\Fube: \Cube \to \Pro(\Diff^r)$ of \emph{squishy cubes}. To construct these, we first define a precursor, the  $\varepsilon$-squishy cubes $\Cube_{\; \! \varepsilon}: \Cube \to \Diff^r_{\leq 0}$ for all $0 < \varepsilon < \frac{1}{2}$, which induce the $\varepsilon$-squishy model structures on $\Diff^r$ and $\Diff^r_{\leq 0}$, then the squishy cubes are obtained as the pro-limit of the $\varepsilon$-squishy cubes. We then show that the squishy fibrations are sharp in Proposition \ref{sfs}, and prove the Kihara's horn inclusions are Oka cofibrations in Theorem \ref{squishy fibration}.

We will make frequent use of the following ancillary function throughout \S \ref{squishy squish}. 

\begin{notation} 
Let $0 <  \alpha < \beta < \frac{1}{2}$, then $\lambda_{\alpha}^\beta: [0,1] \to [0,1]$ denotes any map such that 
	\begin{enumerate}[label = (\alph*)]	
	\item	$\lambda_{\alpha}^\beta|_{\left[0, \alpha \right]} \equiv 0$, $\lambda_{\alpha}^\beta|_{\left[1- \alpha, 1\right]} \equiv 1$,
	\item	$\lambda_{\alpha}^\beta(t) = t$ for all $t \in \left[\frac{1}{2} (\beta + \alpha), 1- \frac{1}{2} (\beta + \alpha) \right]$, and
	\item	$\dot{\lambda}_{\alpha}^\beta(t) > 0$ for all $t \in \left(\alpha, 1- \alpha \right)$. 
	\end{enumerate} 
\qede
\end{notation}

\paragraph{$\varepsilon$-squishy intervals and cubes} 

Throughout this subsection we fix $0 < \varepsilon < \frac{1}{2}$. 
  
 \begin{definition} 
 The pushout of the span 
% https://q.uiver.app/#q=WzAsMyxbMCwwLCJbMCxcXHZhcmVwc2lsb25dIFxcY3VwIFsxLVxcdmFyZXBzaWxvbiwgMV0iXSxbMCwxLCJcXEN1YmVee1xcOiBcXCExfSJdLFsxLDAsIlxcezBcXH0gXFxjdXAgXFx7MVxcfSJdLFswLDEsIiIsMCx7InN0eWxlIjp7InRhaWwiOnsibmFtZSI6Imhvb2siLCJzaWRlIjoidG9wIn19fV0sWzAsMiwiIiwyLHsic3R5bGUiOnsiaGVhZCI6eyJuYW1lIjoiZXBpIn19fV1d
\[\begin{tikzcd}
	{[0,\varepsilon] \cup [1-\varepsilon, 1]} & {\{0\} \cup \{1\}} \\
	{\Cube^{\: \!1}}
	\arrow[hook, from=1-1, to=2-1]
	\arrow[two heads, from=1-1, to=1-2]
\end{tikzcd}\]
(in $\Diff^r$) is called the \Emph{$\varepsilon$-squishy interval} and is denoted by $\Cube_{\; \! \varepsilon}^{\: \! 1}$. For any $n \in \mathbf{N}$ the $n$-fold product of $\Cube_{\; \! \varepsilon}^{\: \! 1}$ is called the \Emph{$\varepsilon$-squishy $n$-cube}, and is denoted by $\Cube_{\; \! \varepsilon}^{\; \! n}$. \qede
\end{definition} 
  
 \begin{proposition} 
 The $\varepsilon$-squishy $n$-cube $\Cube_{\; \! \varepsilon}^{\; \! n}$ is $0$-truncated for all $n \in \mathbf{N}$. 
 \end{proposition} 
 
 \begin{proof}
 This is an immediate consequence of Lemma \ref{pushout}. 
 \end{proof}
 
 By Proposition \ref{universal cube} we obtain a cocubical diagram 

$$	\begin{array}{llll} 
	\Cube_{\; \! \varepsilon}^{\; \! \bullet}:	&	\Cube		&	\to		&	\Diff^r_{\leq 0}					\\
	{}							&	\Cube^{\; \! n}	&	\mapsto	&	\Cube_{\; \! \varepsilon}^{\; \! n}.
	\end{array}
$$

\begin{notation}
We write 
$$	\begin{array}{rcll}
	\partial \: \! \Cube_{\; \! \varepsilon}^{\; \! n}	&	\defeq	&	(\Cube_{\; \! \varepsilon}^{\; \! \bullet})_! \partial \: \!  \Cube^{\; \! n},	&	n \geq 0						\\
	\CHorn_{\; \! k, \xi, \varepsilon}^{\; \! n}		&	\defeq	&	(\Cube_{\; \! \varepsilon}^{\; \! \bullet})_! \CHorn_{\; \! k, \xi}^{\; \! n},	&	n \geq 1, \; n \geq k \geq 0, \xi = 0,1.
	\end{array}
$$ \qede
\end{notation}

 \begin{proposition} 
 The $\varepsilon$-squishy cubes generate $\Diff^r$ under colimits. 
 \end{proposition} 
 
 \begin{proof} 
 For each $d \geq 0$ and for the map $\coprod_{x \in \mathbf{R}^d} \Cube_{\; \! \varepsilon}^{\; \! d} \xrightarrow{\big((\lambda_{\varepsilon'}^\varepsilon)^d + x\big)_{x \in \mathbf{R}^d}} \mathbf{R}^d$ is an effective epimorphism, as it is surjective and admits local sections, so the proposition follows from \cite[Prop.~20.4.5.1]{jL2017}. 
 \end{proof} 
 
\begin{lemma} 
The differentiable sheaf $\Cube_{\: \! \varepsilon}^{\: \! 1}$ is $\Cube_{\: \! \varepsilon}^{\: \! 1}$-contractible. 
\end{lemma}

\begin{proof}
Set $\alpha = \varepsilon$, fix any $\alpha < \beta < \frac{1}{2}$, and write $\lambda \defeq \lambda_\alpha^\beta$. Also, define 
$$	\begin{array}{rrcl}
	\mu:	&	\left[\varepsilon, \frac{1}{2} \right]	&	\to		&	\left[ \varepsilon, \frac{1}{2} \right]																			\\
	{}	&	s							&	 \mapsto	&	\left ( \frac{1}{2} - \varepsilon \right ) \cdot \lambda\left(\frac{1}{\frac{1}{2} - \varepsilon} \left(s - \varepsilon \right) \right) + \varepsilon,
	\end{array}
$$
 and 
 $$	\begin{array}{rrcl}
 	\nu:	&	\left[ \frac{1}{2}, 1 - \varepsilon \right ]	&	\to		&	\left[ \frac{1}{2}, 1 - \varepsilon \right ]																		\\
	{}	&	s								&	 \mapsto	& \left ( \frac{1}{2} - \varepsilon \right ) \cdot \lambda \left ( \frac{1}{\frac{1}{2} - \varepsilon}  \left (s - \frac{1}{2} \right) \right ) + \frac{1}{2}.
	\end{array}
$$ 
Consider the map 
 $$	\begin{array}{rrcl}
 	H:	&	[0,1]	\times [0,1]	&	\to		&	[0,1]																							\\
	{}	&	(s,t)				&	\mapsto	&	\left	\{	\begin{array}{lll}	
											t																								&	\text{if}	&	0 \leq s \leq \varepsilon			\\
											\frac{1}{\frac{1}{2} - \varepsilon} \big ( \left (\lambda(t) - t \right )  \cdot \mu(s) + \frac{1}{2} t - \lambda(t) \cdot \varepsilon  \big )  	&	\text{if}	&	\varepsilon \leq s \leq \frac{1}{2}	\\
											\frac{1 - \varepsilon - \nu(s)}{\frac{1}{2} - \varepsilon} \cdot \lambda(t)   											&	\text{if}	&	\frac{1}{2} \leq s \leq 	1 -\varepsilon	\\
											0																								&	\text{if}	&	1 - \varepsilon \leq s \leq 1.
											\end{array}					\right	.
	\end{array}
$$
(Qualitatively, $H|_{\left [ 0, \frac{1}{2} \right] \times [0,1]}$ interpolates between $t \mapsto t$ and $\lambda$, and $H|_{\left [\frac{1}{2}, 0 \right] \times [0,1]}$ interpolates between $\lambda$ and $t \mapsto 0$.) 

Writing $ \Cube_{\; \! \varepsilon}^2$ as a colimit of 
$$(\{0\} \cup \{1\} \twoheadleftarrow [0,\varepsilon] \cup [1 - \varepsilon,1] \hookrightarrow [0,1] ) \times (\{0\} \cup \{1\} \twoheadleftarrow [0,\varepsilon] \cup [1 - \varepsilon,1] \hookrightarrow [0,1] )$$
we see that we need to check that the induced map $[0,1] \times [0,1] \to \Cube_{\; \! \varepsilon}^{\; \! 1}$ factors as $\Cube_{\; \! \varepsilon}^{\; \! 1} \times [0,1] \to \Cube_{\; \! \varepsilon}^{\; \! 1}$ and $[0,1] \times \Cube_{\; \! \varepsilon}^{\; \! 1} \to \Cube_{\; \! \varepsilon}^{\; \! 1}$, and that moreover $([0,\varepsilon] \cup [1 - \varepsilon,1]) \times ([0,\varepsilon] \cup [1 - \varepsilon,1]) \to \Cube_{\; \! \varepsilon}^{\; \! 1}$ factors through $(\{0\} \cup \{1\}) \times (\{0\} \cup \{1\}) \to \Cube_{\; \! \varepsilon}^{\; \! 1}$. 

The last point is clear, as well as the fact that $H$ factors through
$\Cube_{\; \! \varepsilon}^{\; \! 1} \times [0,1] \to [0,1]$. Thus, we are left with showing the second property. 

Observe that 
$H(s,t) = \lambda(t)$ for $s \in \left ( \frac{1}{2} - \delta,  \frac{1}{2} + \delta \right)$ and some sufficiently small $\delta > 0$. We will check separately that 
$H|_{\left [0, \frac{1}{2} + \delta \right ) \times [0,1]}$ 
and 
$H|_{\left ( \frac{1}{2} - \delta, 1 \right ] \times [0,1]}$ 
factor through 
$\left [0, \frac{1}{2} + \delta \right ) \times \Cube_{\; \! \varepsilon}^{\; \! 1} \to \Cube_{\; \! \varepsilon}^{\; \! 1}$ 
and 
$\left ( \frac{1}{2} - \delta, 1 \right ]   \times \Cube_{\; \! \varepsilon}^{\; \! 1} \to [0,1]$, 
respectively. In the first case, 
$H(s,t) \in [0, \varepsilon) \cup (1 - \varepsilon, 1]$ 
for all values 
$t \in [0, \varepsilon) \cup (1 - \varepsilon, 1]$, 
so that 
$H|_{\left [0, \frac{1}{2} + \delta \right ) \times [0,\varepsilon)}$ 
and
$H|_{\left [0, \frac{1}{2} + \delta \right ) \times (1- \varepsilon, 1]}$ 
composed with 
$[0,1] \to \Cube_{\; \! \varepsilon}^{\; \! 1}$ are constant. In the second case, as 
$\lambda|_{[0,\varepsilon)} \equiv 0$ 
and
$\lambda|_{(1-\varepsilon, 1]} \equiv 1$, 
the map 
$H|_{\left ( \frac{1}{2} - \delta, 1 \right ] \times [0,1]}$ 
is independent of $t$  
for all 
$t \in [0, \varepsilon) \cup (1 - \varepsilon, 1]$.
\end{proof}

Just as for the diagrams discussed in \S \ref{comparing nerves}, the diagram $\Cube \to \Diff^r_{\leq 0}$ induced from $\Cube_{\; \! \varepsilon}^{\; \! 1}$ via Proposition \ref{universal cube} satisfies the assumptions of Proposition \ref{infinity transfer} and Theorem \ref{transfer}, yielding the following proposition: 

\begin{proposition}	\label{escm}
The pullback functors along the diagrams $\Cube \to \Diff^r_{\leq 0}$ produces right transferred model structures on  $\Diff^r$ and $\Diff^r_{\leq 0}$ in which the weak equivalences are the shape equivalences.  \qed
\end{proposition} 
 
\paragraph{Squishy intervals and cubes} 

\begin{definition}	\label{sqi}
The pro-differentiable sheaf 
$$\Fube^{\: \! 1} \defeq ``\lim_{\varepsilon > 0}\!" \,  \Cube_{\: \! \varepsilon}^{\: \! 1}$$
is called the \Emph{squishy interval}. For any $n \in \mathbf{N}$ the $n$-fold product of $\Fube^{\: \! 1}$ is called the \Emph{squishy $n$-cube}, and is denoted by $\Fube^{\; \! n}$. The resulting cocubical pro-object is denoted as follows: 
$$	\begin{array}{rrcl}
	\Fube^{\: \! \bullet}:	&	\Cube^{\;\! \phantom{n}}	&	\to		&	\Pro(\Diff^r)													\\
	{}				&	\Cube^{\; \! n}			&	\mapsto	&	\Fube^{\; \! n}
	\end{array}
$$ \qede
\end{definition} 

By Proposition \ref{pro cocomplete} the functor $\Fube^{\; \! \bullet}:	\Cube	\to	\Pro(\Diff^r)$ may thus be extended to a colimit preserving functor $\Fube_{\; \! !}^{\; \! \bullet}: [\Cube^{\;\! \op}, \mathcal{S}] \to \Pro(\Diff^r)$. 

\begin{notation}
We write 
$$	\begin{array}{rcll}
	\partial \: \! \Fube^{\; \! n}	&	\defeq	&	\Fube_{\; \! !}^{\; \! \bullet} \partial \: \!  \Cube^{\; \! n},	&	n \geq 0						\\
	\Forn_{\; \! k, \xi}^{\; \! n}	&	\defeq	&	\Fube_{\; \! !}^{\; \! \bullet} \CHorn_{\; \! k, \xi}^{\; \! n},	&	n \geq 1, \; n \geq k \geq 0, \xi = 0,1.
	\end{array}
$$ \qede
\end{notation} 

%needed for applying transferred bla
\begin{proposition} 
There is a canonical isomorphism
$$		\Fube^{\; \! n}	\simeq	``\lim_{\varepsilon > 0}" \, \Cube_{\; \! \varepsilon}^{\; \! n} \quad	n \geq 0.	$$
\end{proposition} 

\begin{proof}
There is an isomorphism $\Fube^{\; \! n}	\simeq	``\lim_{(\varepsilon_1 > 0) \times \cdots \times (\varepsilon_n > 0)}" \Cube_{\; \! \varepsilon_1}^{\; \! 1} \times \cdots \times  \Cube_{\;\! \varepsilon_n}^{\; \! 1}$ by the proof of Proposition \ref{proproducts}. As the ordered set $\left ( 0, \frac{1}{2} \right )$ admits products it is sifted, and the diagonal map $\left ( 0, \frac{1}{2} \right ) \to \left ( 0, \frac{1}{2} \right ) \times \cdots \times \left ( 0, \frac{1}{2} \right )$  is initial so that the induced map $``\lim_{\varepsilon > 0}" \Cube_{\; \! \varepsilon}^{\; \! 1} \times \cdots \times  \Cube_{\; \! \varepsilon}^{\; \! 1} \to ``\lim_{(\varepsilon_1 > 0) \times \cdots \times (\varepsilon_n > 0)}" \Cube_{\; \! \varepsilon_1}^{\; \! 1} \times \cdots \times  \Cube_{\;\! \varepsilon_n}^{\; \! 1}$ is an isomorphism. 
\end{proof}

\paragraph{Squishy fibrations} 

\begin{definition} 
A morphism $X \to Y$ in $\Diff^r$ is called a \Emph{squishy fibration} if the morphism of simplicial homotopy types $(\Fube^{\; \! \bullet})^* X \to (\Fube^{\; \! \bullet})^* Y$ is shape fibration. \qede
\end{definition} 

\begin{proposition}	\label{sfs}
Any squishy fibration is sharp. 
\end{proposition} 

\begin{proof} 
We will show that the functor satisfies the conditions of Proposition \ref{sd}. 

The inclusion $\Pro(\Diff^r) \hookleftarrow \Diff^r$ preserves finite limits by \cite[Prop.~5.3.5.14]{jL2009}, and $[\Delta^{\op}, \mathcal{S}] \leftarrow \Pro(\Diff^r): (\Fube^{\; \! \bullet})^*$ preserves all limits, as it is a right adjoint. 

We conclude with the following two steps, which follow from Proposition \ref{escm} and the fact that shape equivalences are closed under colimits: 

Let $X \to Y$ be a shape equivalence in $\Diff^r$ then we have 
		$$
		{\arraycolsep=1.5pt\def\arraystretch{1.9}
		\begin{array}{>{\displaystyle}r>{\displaystyle}l}
		{}	&	\underline{\Diff}^r(\Fube^{\; \! \bullet},X) \to \underline{\Diff}^r(\Fube^{\; \! \bullet},Y)															\\
		=	&	\underline{\Diff}^r(\colim_{\varepsilon > 0}\Cube_{\; \! \varepsilon}^{\; \! \bullet},X) \to \underline{\Diff}^r(\colim_{\varepsilon > 0}\Cube_{\; \! \varepsilon}^{\; \! \bullet},Y)	\\
		=	&	\colim_{\varepsilon > 0}\underline{\Diff}^r(\Cube_{\; \! \varepsilon}^{\; \! \bullet},X) \to \colim_{\varepsilon > 0} \underline{\Diff}^r(\Cube_{\; \! \varepsilon}^{\; \! \bullet},Y)	\\
		=	&	\colim_{\varepsilon > 0}\big(\underline{\Diff}^r\big(\Cube_{\; \! \varepsilon}^{\; \! \bullet},X) \to  \underline{\Diff}^r(\Cube_{\; \! \varepsilon}^{\; \! \bullet},Y)\big).			
		\end{array}
		}
		$$
Finally, the base change map $\colim \circ (\Fube^{\; \! \bullet})^* \to \pi_!$ is a natural isomorphism, as for each differentiable sheaf $X$ we have 
		$$
		{\arraycolsep=1.5pt\def\arraystretch{1.9}
		\begin{array}{>{\displaystyle} r >{\displaystyle} c >{\displaystyle}l}
		\colim \circ (\Fube^{\; \! \bullet})^* X	&	= 	&	\underline{\Diff}^r(\Fube^{\; \! \bullet},X) \to \underline{\Diff}^r(\Fube^{\; \! \bullet},Y)		\\
		{}						&	=	&	\underline{\Diff}^r(\colim_{\varepsilon > 0}\Cube_{\; \! \varepsilon}^{\; \! \bullet},X) 		\\
		{}						&	=	&	\colim_{\varepsilon > 0}\underline{\Diff}^r(\Cube_{\; \! \varepsilon}^{\; \! \bullet},X)		\\
		{}						&	=	&	\colim_{\varepsilon > 0} \pi_!X										\\
		{}						&	=	&	\pi_!X.								
		\end{array}
		}
		$$

\end{proof} 

\begin{remark}	\label{squishy remark} 
It is possible to show that the squishy fibrations together with the shape equivalences form a fibration structure, yielding a different proof that fibrations are sharp.  \qede
\end{remark} 

\begin{proposition}	\label{squishy products} 
The squishy fibrations are closed under arbitrary products. 
\end{proposition} 

\begin{proof} 
This follows from the fact that the squishy fibrations may be characterised as those morphisms which lift against all the horn inclusions $\Forn_{\; \! k, \xi}^{\; \! n} \hookrightarrow \Fube^{\; \! n}$. 
\end{proof}

\begin{lemma}	\label{rl}
Let $0 < \varepsilon' < \varepsilon < \frac{1}{2}$, then the triangle 
% https://q.uiver.app/?q=WzAsMyxbMCwwLCJcXEN1YmVeMSJdLFsyLDAsIlxcQ3ViZV4xIl0sWzEsMSwiIFxcRnViZV9cXHZhcmVwc2lsb25eMSJdLFswLDEsIlxcbGFtYmRhX3tcXHZhcmVwc2lsb24nfV5cXHZhcmVwc2lsb24iXSxbMCwyLCIiLDIseyJzdHlsZSI6eyJoZWFkIjp7Im5hbWUiOiJlcGkifX19XSxbMSwyLCIiLDAseyJzdHlsZSI6eyJoZWFkIjp7Im5hbWUiOiJlcGkifX19XV0=
\[\begin{tikzcd}
	{\Cube^{\; \! 1}} && {\Cube^{\; \! 1}} \\
	& { \Cube_{\; \! \varepsilon}^{\; \! 1}}
	\arrow["{\lambda_{\varepsilon'}^\varepsilon}", from=1-1, to=1-3]
	\arrow[two heads, from=1-1, to=2-2]
	\arrow[two heads, from=1-3, to=2-2]
\end{tikzcd}\]
commutes. 
\end{lemma}

\begin{proof}
It is enough to show that composing $[\varepsilon', 1 - \varepsilon'] \to \Cube^{\; \! 1} \to \Cube_{\; \! \varepsilon}^{\; \! 1}$ yields an epimorphism, then the statement follows from the observation that the triangle 
% https://q.uiver.app/?q=WzAsMyxbMSwwLCJbXFx2YXJlcHNpbG9uJywgMSAtIFxcdmFyZXBzaWxvbiddIl0sWzAsMSwiXFxGdWJlX1xcdmFyZXBzaWxvbl4xIl0sWzIsMSwiXFxGdWJlX1xcdmFyZXBzaWxvbl4xIl0sWzAsMV0sWzEsMiwiW1xcbGFtYmRhX3tcXHZhcmVwc2lsb24nfV5cXHZhcmVwc2lsb25dIl0sWzAsMl1d
\[\begin{tikzcd}
	& {[\varepsilon', 1 - \varepsilon']} \\
	{\Cube_{\; \! \varepsilon}^{\; \! 1}} && {\Cube_{\; \! \varepsilon}^{\; \! 1}}
	\arrow[from=1-2, to=2-1]
	\arrow["{[\lambda_{\varepsilon'}^\varepsilon]}", from=2-1, to=2-3]
	\arrow[from=1-2, to=2-3]
\end{tikzcd}\]
commutes. To see this, let $X$ be any differentiable space, then any map $f:\Cube^{\; \! 1} \to X$, which descends to a map $\Cube_{\; \! \varepsilon}^{\; \! 1} \to X$, may be obtained by glueing $f|_{(\varepsilon', 1 - \varepsilon')}: (\varepsilon', 1 - \varepsilon') \to X$ with $\left[0,\frac{1}{2} (\varepsilon' + \varepsilon) \right) \to 1 \xrightarrow{f(\varepsilon')} X$ and $\left(1 - \frac{1}{2}(\varepsilon' + \varepsilon), 1 \right]\to 1 \xrightarrow{f(1 - \varepsilon')} X$ along their common intersection. 
\end{proof}

\begin{theorem}	\label{squishy fibration} 
Let $X$ be a differentiable sheaf, then 
$$X^{\Delta^n} \to X^{\partial \Delta^n}$$
is a squishy fibration for any $n \geq 0$. 
\end{theorem} 

\begin{proof}
In this proof we use the following notation $\left(0 < \varepsilon < \frac{1}{2}\right)$:
$$	{\arraycolsep= .5pt \def\arraystretch{1.6}
	\begin{array}{lllll}
	\Cube^{\; \! k} \star_{i,\xi} \Delta^n				&	\enskip	\defeq	&	\enskip	\left(\CHorn_{\; \! i,\xi}^{\; \! k} \times \Delta^n \right) & \sqcup_{\scriptsize \, \CHorn_{\; \! i,\xi}^{\; \! k} \times \partial \Delta^n} & \left( \Cube^{\; \! k} \times \partial \Delta^n \right)	\\
	\Fube^{\; \! k} \star_{i,\xi} \Delta^n				&	\enskip	\defeq	&	\enskip	\left(\Forn_{\; \! i,\xi}^{\; \! k} \times \Delta^n \right) & \sqcup_{\scriptsize  \, \Forn_{\; \! i,\xi}^{\; \! k} \times \partial \Delta^n} & \left( \Fube^{\; \! k} \times \partial \Delta^n \right)		\\
	\Cube_{\; \! \varepsilon}^{\; \! k} \star_{i,\xi} \Delta^n	&	\enskip	\defeq	&	\enskip	\left(\CHorn_{\; \! i,\xi, \varepsilon}^{\; \! k} \times \Delta^n \right) & \sqcup_{\scriptsize  \, \CHorn_{\; \! i,\xi, \varepsilon}^{\; \! k} \times \partial \Delta^n} & \left( \Cube_{\; \! \varepsilon}^{\; \! k} \times \partial \Delta^n \right)
	\end{array}
	}
$$
We must show that for every $n \geq 1$, $n \geq k \geq 0$ and $\xi = 0,1$ 
\begin{equation}	\label{squishy lift}  
% https://q.uiver.app/?q=WzAsMyxbMCwwLCJcXEZ1YmVeayBcXHN0YXJfe2ksXFxkZWx0YX0gXFxEZWx0YV5uIl0sWzEsMCwiWCJdLFswLDEsIlxcRnViZV5uIFxcdGltZXMgXFxEZWx0YV5uIl0sWzAsMV0sWzAsMiwiIiwyLHsic3R5bGUiOnsidGFpbCI6eyJuYW1lIjoiaG9vayIsInNpZGUiOiJ0b3AifX19XV0=
\begin{tikzcd}
	{\Fube^{\; \! k} \star_{i,\xi} \Delta^n} & X \\
	{\Fube^{\; \! n} \times \Delta^n}
	\arrow[from=1-1, to=1-2]
	\arrow[hook, from=1-1, to=2-1]
\end{tikzcd}
\end{equation}
admits a lift. The horizontal map is represented by a map 
$$	\Cube^{\; \! k} \star_{i,\xi} \Delta^n \to X 	$$
which factors through $\Cube_{\; \! \varepsilon}^{\; \! k} \star_{i,\xi} \Delta^n$ for some $0 < \varepsilon < \frac{1}{2}$. Fix $0 < \varepsilon' < \varepsilon$, and write $\lambda \defeq \lambda_{\varepsilon'}$. To prove the statement we define maps $\mu, \nu: \Cube^{\; \! k} \times \Delta^n \to \Cube^{\; \! k} \times \Delta^n$ such that the digram 
% https://q.uiver.app/?q=WzAsNyxbMCwwLCJcXEN1YmVeayBcXHN0YXJfe2ksXFxkZWx0YX0gXFxEZWx0YV5uIl0sWzAsMSwiXFxDdWJlXmsgXFx0aW1lcyBcXERlbHRhXm4iXSxbMSwxLCJcXEN1YmVeayBcXHRpbWVzIFxcRGVsdGFebiJdLFsyLDEsIlxcQ3ViZV5rIFxcdGltZXMgXFxEZWx0YV5uIl0sWzMsMSwiXFxDdWJlXmsgXFx0aW1lcyBcXERlbHRhXm4iXSxbNCwxLCJcXEN1YmVeayBcXHRpbWVzIFxcRGVsdGFebiJdLFs0LDAsIlxcQ3ViZV5rIFxcc3Rhcl97aSxcXGRlbHRhfSBcXERlbHRhXm4iXSxbMCwxLCIiLDAseyJzdHlsZSI6eyJ0YWlsIjp7Im5hbWUiOiJob29rIiwic2lkZSI6InRvcCJ9fX1dLFsxLDIsIlxcbGFtYmRhXmsgXFx0aW1lcyBcXGlkX3tcXERlbHRhXm59Il0sWzIsMywiXFxtdSJdLFszLDQsIlxcbnUiXSxbNCw1LCJcXGxhbWJkYV5rIFxcdGltZXMgXFxpZF97XFxEZWx0YV5ufSJdLFswLDYsIlxcbGVmdChcXGxhbWJkYV5rIFxcdGltZXMgXFxpZF97XFxEZWx0YV5ufXxfe1xcQ3ViZV5rIFxcc3Rhcl97aSxcXGRlbHRhfSBcXERlbHRhXm59XFxyaWdodCleMiJdLFs2LDUsIiIsMix7InN0eWxlIjp7InRhaWwiOnsibmFtZSI6Imhvb2siLCJzaWRlIjoidG9wIn19fV1d
\[\begin{tikzcd}[column sep = 4em]
	{\Cube^{\; \! k} \star_{i,\xi} \Delta^n} &&&& {\Cube^{\; \! k} \star_{i,\xi} \Delta^n} \\
	{\Cube^{\; \! k} \times \Delta^n} & {\Cube^{\; \! k} \times \Delta^n} & {\Cube^{\; \! k} \times \Delta^n} & {\Cube^{\; \! k} \times \Delta^n} & {\Cube^{\; \! k} \times \Delta^n}
	\arrow[hook, from=1-1, to=2-1]
	\arrow["{\lambda^{\; \! k} \times \id_{\Delta^n}}", from=2-1, to=2-2]
	\arrow["\mu", from=2-2, to=2-3]
	\arrow["\nu", from=2-3, to=2-4]
	\arrow["{\lambda^{\; \! k} \times \id_{\Delta^n}}", from=2-4, to=2-5]
	\arrow["{\left(\lambda^{\; \! k} \times \id_{\Delta^n}|_{\Cube^{\; \! k} \star_{i,\xi} \Delta^n}\right)^2}", from=1-1, to=1-5]
	\arrow[hook, from=1-5, to=2-5]
\end{tikzcd}\]
commutes and admits a diagonal lift. (Qualitatively, the first instance of $\lambda^{\; \! k} \times \id_{\Delta^n}$ ensures that the resulting lift factors through $\Cube_{\; \! \varepsilon'}^{\; \! k} \times \Delta^n$, $\mu$ is a first approximation to the desired retract, next $\nu$ completes the retraction in the ``$\Delta^n$-direction'', and, finally, the second instance of $\lambda^{\; \! k} \times \id_{\Delta^n}$ completes the retract in the ``$\Cube^{\; \! k}$-direction''.)   Recall, that by Lemma \ref{rl} the map $\lambda^{\; \! k} \times \id_{\Delta^n}: \Cube^{\; \! k} \times \Delta^n \to \Cube^{\; \! k} \times \Delta^n$ descends to the identity map $\id: \Cube_{\; \! \varepsilon}^{\; \! k} \times \Delta^n \to \Cube_{\; \! \varepsilon}^{\; \! k} \times \Delta^n$, so that the diagram
% https://q.uiver.app/#q=WzAsNCxbMCwwLCJcXEN1YmVeayBcXHN0YXJfe2ksXFx4aX0gXFxEZWx0YV5uIl0sWzEsMCwiXFxDdWJlXmsgXFxzdGFyX3tpLFxceGl9IFxcRGVsdGFebiJdLFswLDEsIlxcQ3ViZV5rIFxcdGltZXMgXFxEZWx0YV5uIl0sWzIsMCwiXFxGdWJlX3tcXHZhcmVwc2lsb259XmsgXFxzdGFyX3tpLFxceGl9IFxcRGVsdGFebiJdLFswLDFdLFswLDJdLFsyLDFdLFsxLDNdXQ==
\[\begin{tikzcd}
	{\Cube^{\; \! k} \star_{i,\xi} \Delta^n} & {\Cube^{\; \! k} \star_{i,\xi} \Delta^n} & {\Cube_{\; \! \varepsilon}^{\; \! k} \star_{i,\xi} \Delta^n} \\
	{\Cube^{\; \! k} \times \Delta^n}
	\arrow[from=1-1, to=1-2]
	\arrow[from=1-1, to=2-1]
	\arrow[from=2-1, to=1-2]
	\arrow[from=1-2, to=1-3]
\end{tikzcd}\]
induces a commutative diagram  
% https://q.uiver.app/#q=WzAsMyxbMCwwLCJcXEZ1YmVfe1xcdmFyZXBzaWxvbid9XmsgXFxzdGFyX3tpLFxceGl9IFxcRGVsdGFebiJdLFswLDEsIlxcRnViZV97XFx2YXJlcHNpbG9uJ31eayBcXHRpbWVzIFxcRGVsdGFebiJdLFsxLDAsIlxcRnViZV97XFx2YXJlcHNpbG9ufV5rIFxcc3Rhcl97aSxcXHhpfSBcXERlbHRhXm4iXSxbMCwxXSxbMCwyXSxbMSwyXV0=
\[\begin{tikzcd}
	{\Cube_{\; \! \varepsilon'}^{\; \! k} \star_{i,\xi} \Delta^n} & {\Cube_{\; \! \varepsilon}^{\; \! k} \star_{i,\xi} \Delta^n} \\
	{\Cube_{\; \! \varepsilon'}^{\; \! k} \times \Delta^n}
	\arrow[from=1-1, to=2-1]
	\arrow[from=1-1, to=1-2]
	\arrow[from=2-1, to=1-2]
\end{tikzcd}\] 
and thus a commutative diagram
% https://q.uiver.app/#q=WzAsMyxbMCwwLCJcXEZ1YmVfe1xcdmFyZXBzaWxvbid9XmsgXFxzdGFyX3tpLFxceGl9IFxcRGVsdGFebiJdLFswLDEsIlxcRnViZV97XFx2YXJlcHNpbG9uJ31eayBcXHRpbWVzIFxcRGVsdGFebiJdLFsxLDAsIlgiXSxbMCwxXSxbMCwyXSxbMSwyXV0=
\[\begin{tikzcd}
	{\Cube_{\; \! \varepsilon'}^{\; \! k} \star_{i,\xi} \Delta^n} & X \\
	{\Cube_{\; \! \varepsilon'}^{\; \! k} \times \Delta^n}
	\arrow[from=1-1, to=2-1]
	\arrow[from=1-1, to=1-2]
	\arrow[from=2-1, to=1-2]
\end{tikzcd}\]
which descends to a lift of (\ref{squishy lift}).  \\

\noindent\underline{Construction of $\mu$ and $\nu$:}  In order to ease the notational burden we will only define $\mu$ and $\nu$ for $i = k$ and $\xi = 1$. \\
To define $\mu$, I require an auxiliary smooth function $\rho: \Cube^{\;\! k-1} \times \Delta^n \to \Cube^{\; \! 1}$, such that
\begin{enumerate}[label = (\alph*)]
\item	$\rho(t_1, \ldots, t_k, s_0, \ldots, s_n) = 1$ if $t_1, \ldots, t_k > \frac{2}{3} \cdot \varepsilon'$ or $s_0 + \cdots + s_n > \frac{2}{3}$;
\item	$\rho(t_1, \ldots, t_k, s_0, \ldots, s_n) = 0$ if $t_1, \ldots, t_k < \frac{1}{3} \cdot \varepsilon'$ and $s_0 + \cdots + s_n < \frac{1}{3}$. 
\end{enumerate} 
Then, we define
$$	\begin{array}{rrcl}
	\mu:	&	\Cube^{\; \! k} \times \Delta^n	&	\to		&	\Cube^{\; \! k} \times \Delta^n	\\
	{}	&	((t_1, \ldots, t_k), s)		&	\mapsto	&	((t_1, \ldots, t_{k-1}, \rho(t_1, \ldots, t_{k-1}, s) \cdot t_k), s).
	\end{array}
$$
Using partition of unity one can patch together the retractions $\Delta^n \to \Lambda_{k_2}^n, \, 1 \leq {k_2} \leq n$ to obtain a retract $\sigma: \setbuilder{(s_0, \ldots, s_n) \in \Delta^n}	{s_0 + \cdots + s_n > \frac{1}{3}} \to \partial \Delta^n$. Now, let $\tau: \Cube^{\; \! 1} \to \Cube^{\; \! 1}$ be a smooth map such that   
\begin{enumerate}[label = (\alph*)]
\item	$\tau(t) = 1$ for $t > \frac{2}{3} \cdot \varepsilon'$, and 
\item	$\tau(t) = 0$ for $t < \frac{1}{3} \cdot \varepsilon'$. 
\end{enumerate} 
Then, we define
$$	\begin{array}{rrcl}
	\nu:	&	\Cube^{\; \! k} \times \Delta^n	&	\to		&	\Cube^{\; \! k} \times \Delta^n	\\
	{}	&	((t_1, \ldots, t_k), s)		&	\mapsto	&	(\, (t_1, \ldots, t_k), \, \id_{\Delta^n} + (\,\id_{\Delta^n} + \tau(t_k) \cdot (\sigma - \id_{\Delta^n})\,)(s) \, ).
	\end{array}
$$ \ \\
\underline{Proof of smoothness of lift:} By construction, it is clear that the lift is smooth at any point which gets mapped to $\Cube^{\; \! k} \times \Delta^n \setminus \big(\Cube^{\;\! k-1} \times \{0\}\big) \times \partial \Delta^n$. Points which get mapped to $\big(\Cube^{\;\! k-1} \times \{0\}\big) \times \partial \Delta^n$ admit a neighbourhood which gets mapped to $\big(\Cube^{\;\! k-1} \times \{0\}\big) \times \Delta^n$, which concludes the proof. 
\end{proof}

\begin{remark} 
The proof of Theorem \ref{squishy fibration} does not imply that the maps $\Fube^k \star_{i,\xi} \Delta^n \hookrightarrow \Fube^{\; \! n} \times \Delta^n$ admits a retract; only that they lift against all objects in $\Diff^r$. \qede
\end{remark} 

\subsubsection{Closure properties Oka cofibrant objects}	\label{closure}

Oka cofibrant objects are closed under various operations. 

\begin{proposition}	\label{fcc}
The subcategory $\Diff^r$ of Oka cofibrant objects is closed under arbitrary coproducts. 
\end{proposition} 

\begin{proof}
For any collection $\{A_i\}_{i \in I}$ of Oka cofibrant objects and any differentiable sheaf $X$ we have 
$$	\begin{array}{rcl}
	\pi_! \underline{\Diff}^r\left(\coprod_{i \in I}A_i, X \right)	&	=	&	\pi_! \prod_{i \in I} \underline{\Diff}^r\left(A_i, X \right)		\\
	{}											&	=	&	\prod_{i \in I} \pi_!  \, \underline{\Diff}^r\left(A_i, X \right)	\\
	{}											&	=	&	\prod_{i \in I} \mathcal{S}(\pi_!A_i, \pi_!X)				\\
	{}											&	=	&	\mathcal{S}\left(\coprod_{i \in I} \pi_!A_i, \pi_!X\right)		\\
	{}											&	=	&	\mathcal{S}\left(\pi_! \coprod_{i \in I} A_i, \pi_!X\right)		\\
	\end{array}
$$
where the second isomorphism follows from Corollary \ref{all products}. 
\end{proof}

\begin{proposition}	\label{filtform}
Let $A: \mathbf{N} \to \Diff^r$ be a diagram such that each object $A_i$ is Oka cofibrant, and such that $A_i \to A_{i + 1}$ is a cofibration in the Kihara model structure for all $i \in \mathbf{N}$, then $\colim A$ is Oka cofibrant. 
\end{proposition} 

\begin{proof}
Let $X$ be any differentiable sheaf, then 
	$$	\begin{array}{rcl}
		\pi_! \underline{\Diff}^r(\colim A, X)	&	=	&	\pi_! \lim\underline{\Diff}^r(A, X)						\\
		{}							&	=	&	\colim \,  (\Fube^\bullet)^* ( \lim \underline{\Diff}^r(A, X))	\\
		{}							&	=	&	\colim \lim\, (\Fube^\bullet)^* ( \underline{\Diff}^r(A, X))	\\
		{}							&	=	&	\lim \colim \, (\Fube^\bullet)^* ( \underline{\Diff}^r(A, X))	\\
		{}							&	=	&	\lim \pi_! \underline{\Diff}^r(A, X)	\\
		{}							&	=	&	\lim \mathcal{S}(\pi_! A, \pi_!X)		\\	
		{}							&	=	&	\mathcal{S}(\colim \pi_! A, \pi_!X)	\\	
		{}							&	=	&	\mathcal{S}(\pi_! \colim  A, \pi_!X),		
		\end{array}
	$$	
where the fourth isomorphism follows from \cite[Prop.~1.23]{aMG2021}.  
\end{proof}

\begin{proposition}	\label{Oka cofibrant products} 
The subcategory $\Diff^r$ of Oka cofibrant stacks is closed under finite products.  
\end{proposition} 

\begin{proof}
Let $A,B$ be Oka cofibrant stacks, and let $X$ be any differentiable sheaf, then one obtains the following series of canonical equivalences: 
	$$	\begin{array}{rcl}
		\pi_! \underline{\Diff}^r(A \times B, X)	&	=	&	\pi_! \underline{\Diff}^r(A, \underline{\Diff}^r( B, X))	\\
		{}							&	=	&	\mathcal{S}(\pi_! A, \pi_! \underline{\Diff}^r( B, X))	\\
		{}							&	=	&	\mathcal{S}(\pi_! A, \mathcal{S}(\pi_! B, \pi_! X))		\\
		{}							&	=	&	\mathcal{S}(\pi_! A \times \pi_! B, \pi_! X)				\\
		{}							&	=	&	\mathcal{S}(\pi_! (A \times  B), \pi_! X).
		\end{array}
	$$
\end{proof}

\begin{lemma}	\label{hl} 
The map $X \to \underline{\Diff}^r(\Delta^1,X)$ is a $\Delta^1$-homotopy equivalence for every object $X$ in $\Diff^r$. 
\end{lemma} 

\begin{proof} 
The $\Delta^1$-homotopy inverse is constructed using the inclusion $\Delta^{\{0\}} \hookrightarrow \Delta^1$. The composition of $\underline{\Hom}(\Delta^{\{0\}}, X) \to \underline{\Hom}(\Delta^1, X)  \to \underline{\Hom}(\Delta^{\{0\}}, X)$ composes to the identity, so we are left with showing that the composition of  $\underline{\Hom}(\Delta^1, X) \to \underline{\Hom}(\Delta^{\{0\}}, X) \to \underline{\Hom}(\Delta^1, X)$ is $\Delta^1$-homotopic to the identity. 

Such a homotopy is given by the transpose of the diagram  
% https://q.uiver.app/#q=WzAsNixbMCwwLCJcXERlbHRhXntcXHswXFx9fVxcdGltZXMgXFxEZWx0YV4xIFxcdGltZXMgWF57XFxEZWx0YV4xfSJdLFsxLDEsIlxcRGVsdGFeMSBcXHRpbWVzIFxcRGVsdGFeMSBcXHRpbWVzIFhee1xcRGVsdGFeMX0iXSxbMCwyLCJcXERlbHRhXntcXHsxXFx9fVxcdGltZXMgXFxEZWx0YV4xIFxcdGltZXMgWF57XFxEZWx0YV4xfSJdLFsxLDAsIlxcRGVsdGFeMSBcXHRpbWVzIFhee1xcRGVsdGFee1xcezBcXH19fSJdLFsyLDEsIlxcRGVsdGFeMSBcXHRpbWVzIFhee1xcRGVsdGFeMX0iXSxbMywxLCJYIl0sWzAsM10sWzAsMV0sWzIsMV0sWzEsNCwiXFx6ZXRhIFxcdGltZXMgXFxpZCJdLFs0LDUsIlxcbWF0aHJte2V2fSJdLFsyLDUsIlxcbWF0aHJte2V2fSIsMix7ImN1cnZlIjozfV0sWzMsNSwiIiwyLHsiY3VydmUiOi0yfV1d
\[\begin{tikzcd}
	{\Delta^{\{0\}}\times \Delta^1 \times X^{\Delta^1}} & {\Delta^1 \times X^{\Delta^{\{0\}}}} \\
	& {\Delta^1 \times \Delta^1 \times X^{\Delta^1}} & {\Delta^1 \times X^{\Delta^1}} & X \\
	{\Delta^{\{1\}}\times \Delta^1 \times X^{\Delta^1}}
	\arrow[from=1-1, to=1-2]
	\arrow[from=1-1, to=2-2]
	\arrow[from=3-1, to=2-2]
	\arrow["{\zeta \times \id}", from=2-2, to=2-3]
	\arrow["{\mathrm{ev}}", from=2-3, to=2-4]
	\arrow["{\mathrm{ev}}"', curve={height=18pt}, from=3-1, to=2-4]
	\arrow[curve={height=-12pt}, from=1-2, to=2-4]
\end{tikzcd}\]
where $\zeta: \Delta^1 \times \Delta^1 \to \Delta^1$ is given by $(s,t) \mapsto s \cdot t$ (here we identify $\Delta^1$ with $[0,1]$), so we are left with showing that it commutes. The bottom part commutes because the morphisms $\Delta^{\{1\}}\times \Delta^1 \times X^{\Delta^1} \to \Delta^1 \times \Delta^1 \times X^{\Delta^1} \xrightarrow{\zeta \times \id} \Delta^1 \times X^{\Delta^1}$ compose to the identity. The composition of $\Delta^{\{0\}} \times \Delta^1 \times X^{\Delta^1} \to \Delta^1 \times \Delta^1 \times X^{\Delta^1} \xrightarrow{\zeta \times \id} \Delta^1 \times X^{\Delta^1}$ is equivalent to the composition of $\Delta^1 \times X^{\Delta^1} \to \Delta^1 \times X^{\Delta^{\{ 0 \}}} \to \Delta^1 \times X^{\Delta^1}$, so that we are left with proving that the square obtained after composing with the evaluation in 
% https://q.uiver.app/#q=WzAsNSxbMCwxLCJcXERlbHRhXjEgXFx0aW1lcyBYXntcXERlbHRhXjF9Il0sWzEsMCwiXFxEZWx0YV4xIFxcdGltZXMgWF57XFxEZWx0YV57XFx7MFxcfX19Il0sWzIsMSwiXFxEZWx0YV4xIFxcdGltZXMgWF57XFxEZWx0YV4xfSJdLFsxLDIsIlxcRGVsdGFee1xcezBcXH19IFxcdGltZXMgWF57XFxEZWx0YV4xfSJdLFszLDEsIlgiXSxbMiw0LCJcXG1hdGhybXtldn0iXSxbMCwxXSxbMCwzXSxbMSwyXSxbMywyXV0=
\[\begin{tikzcd}
	& {\Delta^1 \times X^{\Delta^{\{0\}}}} \\
	{\Delta^1 \times X^{\Delta^1}} && {\Delta^1 \times X^{\Delta^1}} & X \\
	& {\Delta^{\{0\}} \times X^{\Delta^1}}
	\arrow["{\mathrm{ev}}", from=2-3, to=2-4]
	\arrow[from=2-1, to=1-2]
	\arrow[from=2-1, to=3-2]
	\arrow[from=1-2, to=2-3]
	\arrow[from=3-2, to=2-3]
\end{tikzcd}\]
commutes, but this follows from the more general observation that for any map $A \to B$ the square 
% https://q.uiver.app/#q=WzAsNCxbMCwxLCJBIFxcdGltZXMgWF5CIl0sWzEsMCwiQiBcXHRpbWVzIFheQiJdLFsyLDEsIlgiXSxbMSwyLCJBIFxcdGltZXMgWF5BIl0sWzAsMV0sWzEsMl0sWzAsM10sWzMsMl1d
\[\begin{tikzcd}
	& {B \times X^B} \\
	{A \times X^B} && X \\
	& {A \times X^A}
	\arrow[from=2-1, to=1-2]
	\arrow[from=1-2, to=2-3]
	\arrow[from=2-1, to=3-2]
	\arrow[from=3-2, to=2-3]
\end{tikzcd}\]
commutes, which is true because both the top and the bottom composition transpose to the map $X^B \to X^A$; the top one from the general formula of obtaining a transpose using the counit, and the bottom composition by naturality. 
\end{proof} 

\begin{proposition}	\label{hfc}
The $\infty$-category of Oka cofibrant objects is closed under $\Delta^1$-homotopy equivalence.  
\end{proposition} 

\begin{proof} 
Let $A$ be Oka cofibrant, and $A \to B$, a $\Delta^1$-homotopy equivalence, then we obtain a commutative diagram 
% https://q.uiver.app/#q=WzAsNCxbMCwwLCJcXHBpXyFcXHVuZGVybGluZXtcXERpZmZ9XnIoQSxYKSJdLFswLDEsIlxccGlfIVxcdW5kZXJsaW5le1xcRGlmZn1ecihCLFgpIl0sWzEsMCwiXFxtYXRoY2Fse1N9KFxccGlfIUEsXFxwaV8hWCkiXSxbMSwxLCJcXG1hdGhjYWx7U30oXFxwaV8hQixcXHBpXyFYKSJdLFsxLDBdLFszLDJdLFswLDJdLFsxLDNdXQ==
\[\begin{tikzcd}
	{\pi_!\underline{\Diff}^r(A,X)} & {\mathcal{S}(\pi_!A,\pi_!X)} \\
	{\pi_!\underline{\Diff}^r(B,X)} & {\mathcal{S}(\pi_!B,\pi_!X)}
	\arrow[from=2-1, to=1-1]
	\arrow[from=2-2, to=1-2]
	\arrow[from=1-1, to=1-2]
	\arrow[from=2-1, to=2-2]
\end{tikzcd}\]
in which we must show that the vertical arrows are isomorphisms, which in turn follows from showing that the functors $\mathcal{S}(\pi_!\emptyinput,\pi_!X)$ and $\pi_!\underline{\Diff}^r(\emptyinput,X)$ send $\Delta^1$-homotopic maps to equivalent maps. For $\mathcal{S}(\pi_!\emptyinput,\pi_!X)$ this is clear, as $\pi_!$ preserves products and $\pi_! \Delta^1 = \mathbf{1}$, for $\pi_!\underline{\Diff}^r(\emptyinput,X)$ this follows from Lemma \ref{hl}.
\end{proof}

\begin{remark} 
Proposition \ref{hfc} remains true for other intervals than $\Delta^1$, such as $\mathbf{R}$ or $\Cube_{\; \! \varepsilon}^{\; \! 1}$ for $0 < \varepsilon < \frac{1}{2}$. \qede
\end{remark}

\subsubsection{Proof of the differentiable Oka principle}	\label{Proof of the differentiable Oka principle}

Throughout \S \ref{Proof of the differentiable Oka principle} we fix $r = \infty$, as we cite \cite[Prop.~9.5~\&~Th.~11.20]{hK2020} (see Lemma \ref{numerable lemma} and Theorem \ref{manfib}) which are both stated in the smooth setting. (We are confident that \cite[Prop.~9.5]{hK2020} also holds for $r < \infty$, but the classes of manifolds in \cite[Th.~11.20]{hK2020} would probably need to be modified, as smoothness plays an important role in the theory of infinite dimensional manifolds). 

Let $X$ be a diffeological space, and let $U = \{U_\alpha\}_{\alpha \in A}$ be a cover of $X$ then there exists a $\Diff_{\concr}^\infty$-enriched category $X_U$ with 
$$	\begin{array}{rcl}	
	\Obj X_U	&	=	&	\coprod_\sigma U_\sigma	\\
	\Mor X_U	&	=	&	\coprod_{\sigma \supseteq \tau} U_\sigma
	\end{array}
$$
where $\sigma, \tau$ denote non-empty finite subsets of $A$ such that $U_\sigma \defeq \bigcap_{\alpha \in \sigma} U_\alpha \neq \varnothing$. The topological realisation of (the nerve of) $X_U$ is denoted by $BX_U$. The space $BX_U$ may be constructed in  stages using the pushouts 
% https://q.uiver.app/?q=WzAsNCxbMCwwLCJcXGRpc3BsYXlzdHlsZSAgXFxjb3Byb2Rfe1xcc2lnbWFfbiBcXHN1cHNldG5lcSBcXGNkb3RzIFxcc3Vwc2V0bmVxIFxcc2lnbWFfMH0gVV97XFxzaWdtYV9ufSBcXHRpbWVzIFxccGFydGlhbCBcXERlbHRhXm4iXSxbMCwxLCJcXGRpc3BsYXlzdHlsZSBcXGNvcHJvZF97XFxzaWdtYV9uIFxcc3Vwc2V0bmVxIFxcY2RvdHMgXFxzdXBzZXRuZXEgXFxzaWdtYV8wfSBVX3tcXHNpZ21hX259IFxcdGltZXMgXFxEZWx0YV5uIl0sWzEsMCwiQlhfVV57KG4tMSl9Il0sWzEsMSwiQlhfVV57KG4pfSJdLFswLDEsIiIsMCx7InN0eWxlIjp7InRhaWwiOnsibmFtZSI6Imhvb2siLCJzaWRlIjoidG9wIn19fV0sWzAsMiwiIiwyLHsic3R5bGUiOnsidGFpbCI6eyJuYW1lIjoiaG9vayIsInNpZGUiOiJ0b3AifX19XSxbMiwzLCIiLDIseyJzdHlsZSI6eyJ0YWlsIjp7Im5hbWUiOiJob29rIiwic2lkZSI6InRvcCJ9fX1dLFsxLDMsIiIsMCx7InN0eWxlIjp7InRhaWwiOnsibmFtZSI6Imhvb2siLCJzaWRlIjoidG9wIn19fV1d
\begin{equation}	\label{BXU pushout}
\begin{tikzcd}
	{\displaystyle  \coprod_{\sigma_n \supsetneq \cdots \supsetneq \sigma_0} U_{\sigma_n} \times \partial \Delta^n} & {BX_U^{(n-1)}} \\
	{\displaystyle \coprod_{\sigma_n \supsetneq \cdots \supsetneq \sigma_0} U_{\sigma_n} \times \Delta^n} & {BX_U^{(n)}}
	\arrow[hook, from=1-1, to=2-1]
	\arrow[hook, from=1-1, to=1-2]
	\arrow[hook, from=1-2, to=2-2]
	\arrow[hook, from=2-1, to=2-2]
\end{tikzcd}
\end{equation} 
At each stage one can construct inductively an obvious commutative square obtained by replacing $BX_U^{(n)}$ by $X$ in (\ref{BXU pushout}), thus producing a canonical map $BX_U \to X$. As the pushouts at each step satisfy the conditions Proposition \ref{concrete pushout},  each stage $BX_U^{(n)}$ is a diffeological space; the object $BX_U$ is then a diffeological space by Proposition \ref{concrete filtered}, as it is a filtered colimit of diffeological spaces along monomorphisms.

\begin{definition} 
A covering on a diffeological space is called \Emph{numerable} if it admits a subordinate partition of unity.  \qede
\end{definition} 

The original formulation of the following lemma in the setting of topological spaces is due to Segal \cite[\S 4]{gS1968} and tom Dieck \cite[Th.~4]{ttD1971}. Translating these results into the smooth setting is very technical, and is carried out by Kihara in \cite[\S 9]{hK2020}. 

\begin{lemma}[{\cite[Prop.~9.5]{hK2020}}]	\label{numerable lemma}
Let $X$ be a diffeological space, and let $U$ be a numerable cover of $X$, then the canonical map $BX_U \to X$ is a $\Delta^1$-homotopy equivalence.  \qed
\end{lemma}

\begin{theorem}	\label{WD}
Let $X$ be a diffeological space, and let $U$ be a numerable cover of $X$. If each member of $U$ is Oka cofibrant, then so is $X$. 
\end{theorem} 

\begin{proof}
By Lemma \ref{numerable lemma} and Proposition \ref{hfc} the space $X$ is Oka cofibrant iff $BX_U$ is. We will show that each stage $BX_U^{(n)}$ is Oka cofibrant, and then conclude that $BX_U$ is Oka cofibrant by Proposition \ref{filtform}. The diffeological space $BX_U^{(0)}$ is Oka cofibrant by Proposition \ref{fcc}. Applying $\underline{\Diff}^\infty(\emptyinput, X)$ to the square (\ref{BXU pushout}) yields the pullback
% https://q.uiver.app/?q=WzAsNCxbMCwwLCJcXHVuZGVybGluZXtcXERpZmZ9XlxcaW5mdHkoQlhfVV57KG4pfSwgWCkiXSxbMCwxLCJcXHVuZGVybGluZXtcXERpZmZ9XlxcaW5mdHkoQlhfVV57KG4tMSl9LCBYKSJdLFsxLDAsIlxcZGlzcGxheXN0eWxlIFxccHJvZF97XFxzaWdtYV9uIFxcc3Vwc2V0bmVxIFxcY2RvdHMgXFxzdXBzZXRuZXEgXFxzaWdtYV8wfVxcdW5kZXJsaW5le1xcRGlmZn1eXFxpbmZ0eShVX3tcXHNpZ21hX259LCBYKV57XFxEZWx0YV5ufSJdLFsxLDEsIlxcZGlzcGxheXN0eWxlIFxccHJvZF97XFxzaWdtYV9uIFxcc3Vwc2V0bmVxIFxcY2RvdHMgXFxzdXBzZXRuZXEgXFxzaWdtYV8wfVxcdW5kZXJsaW5le1xcRGlmZn1eXFxpbmZ0eShVX3tcXHNpZ21hX259LCBYKV57XFxwYXJ0aWFsIFxcRGVsdGFebn0iXSxbMCwxXSxbMCwyXSxbMiwzXSxbMSwzXV0=
\[\begin{tikzcd}
	{\underline{\Diff}^\infty(BX_U^{(n)}, X)} & {\displaystyle \prod_{\sigma_n \supsetneq \cdots \supsetneq \sigma_0}\underline{\Diff}^\infty(U_{\sigma_n}, X)^{\Delta^n}} \\
	{\underline{\Diff}^\infty(BX_U^{(n-1)}, X)} & {\displaystyle \prod_{\sigma_n \supsetneq \cdots \supsetneq \sigma_0}\underline{\Diff}^\infty(U_{\sigma_n}, X)^{\partial \Delta^n}}
	\arrow[from=1-1, to=2-1]
	\arrow[from=1-1, to=1-2]
	\arrow[from=1-2, to=2-2]
	\arrow[from=2-1, to=2-2]
\end{tikzcd}\]
in which the vertical morphism to the right is sharp as it is a squishy fibration by Theorem \ref{squishy fibration} and Proposition \ref{squishy products}. 
\end{proof} 

For finite dimensional Hausdorff 1st countable manifolds the following theorem was first proved in \cite{dBEpBDBdP2019}. 

\begin{theorem}	\label{manfib}
Any paracompact Hausdorff $C^\infty$-manifold locally modelled on Hilbert spaces, nuclear Fr\'{e}chet spaces, or nuclear Silva spaces satisfies the differentiable Oka principle. %! define!
\qed
\end{theorem} 

\begin{proof} 
The content of \cite[Th.~11.20]{hK2020} is precisely that the manifolds considered in the statement of the theorem are diffeological spaces satisfying the condition in Theorem \ref{WD}. 
\end{proof} 

The infinite dimensional manifolds considered  in Theorem \ref{manfib} include many interesting examples, such as the $\underline{\Diff}^\infty(M,N)$ or the manifold of submanifolds of $N$ diffeomorphic to $M$, for $M,N$ smooth finite dimensional paracompact Hausdorff manifolds without corners and $M$ compact. 

\subsubsection{Counterexamples}	\label{Counterexamples}

There are many directions in which it is not possible to extend Corollary \ref{manfib}.

\begin{example}
$B\mathbf{Z} = \pi_! \underline{\Diff}^r(1,S^1) = \pi_! \underline{\Diff}^r(\pi^*B\mathbf{Z},S^1) \neq \mathcal{S}(B\mathbf{Z}, B\mathbf{Z}) = \mathbf{Z}.$	\qede
\end{example}

One must be careful when dropping the Hausdorfness requirement: 

\begin{example} 
Denote by $\mathbf{R}_{\bullet\bullet}$ the real line with two origins, then 
$$	\begin{array}{rcl}	B\mathbf{Z}	&	=	&	\pi_! \underline{\Diff}^r(\mathbf{R},S^1)					\\
					{}			&	=	&	\pi_! \underline{\Diff}^r(\mathbf{R}_{\bullet\bullet},S^1)		\\
					{}			&	\neq	&	\mathcal{S}(\pi_!\mathbf{R}_{\bullet\bullet} , \pi_! S^1)	\\ 
					{}			&	=	&	\mathcal{S}(B\mathbf{Z}, B\mathbf{Z}) 				\\
					{}			&	=	&	\mathbf{Z}.	
	\end{array}																				$$
\ \qede
\end{example} 

\begin{example}
Denote by $\mathbf{R}_{||}$ the space obtained by glueing two copies of $\mathbf{R}$ along the subspace $(-\infty,-1) \cup (1,\infty)$, then $\mathbf{R}_{||}$ is $\mathbf{A}^1$-homotopy equivalent to $S^1$, so that it is Oka cofibrant. In particular,   
$$\pi_! \underline{\Diff}^r(\mathbf{R}_{||},S^1)	= \pi_! \underline{\Diff}^r(S^1,S^1) = \mathcal{S}(\pi_!S^1, \pi_! S^1) = \mathcal{S}(\pi_!\mathbf{R}_{||} , \pi_! S^1).$$	\qede
\end{example}

Non-paracompact manifolds may not be Oka cofibrant:

\begin{example}
Let $\mathbf{L}$ denote the long line. It has trivial shape but is not contractible. Thus $\mathcal{S}(\pi_!\mathbf{L} , \pi_! \mathbf{L}) = \mathcal{S}(1,1) =1$, while $\underline{\Diff}^r(\mathbf{L}, \mathbf{L})$ has at least two path components. \qede
\end{example} 

\part*{Appendix}	\label{Appendix}
\addcontentsline{toc}{part}{\nameref{Appendix}}
\appendix

\section{The cube category}	\label{cubical}

Here we discuss some background material on the cube category. 

\begin{definition} 
The \Emph{cube category $\Cube$} is the subcategory of $\Set$ whose objects are given by $\{0,1\}^n$ ($n \geq 0$), and whose morphisms are generated by the maps 	
$$	\begin{array}{rrcl}
		\delta_i^\xi:	&	\Cube^{\;\! n-1}		&	\to		&	\Cube^{\; \! n}	\\
		{}			&	(x_1, \ldots, x_{n-1})	&	\mapsto	&	(x_1, \ldots, x_{i-1}, \xi, x_i, \ldots, x_{n-1})
		\end{array}		
	$$
for $n \geq i \geq 1$ and $\varepsilon = 0,1$, and
	$$	\begin{array}{rrcl}
		\sigma_i:	&	\Cube^{\;\! n+1}		&	\to		&	\Cube^{\; \! n}	\\
		{}		&	(x_1, \ldots, x_{n+1})	&	\mapsto	&	(x_1, \ldots, x_{i-1}, x_{i+1}, \ldots, x_{n+1})
		\end{array}		
	$$
for $n \geq 0$ and $n \geq i \geq 1$. The category of \Emph{cubical sets} is the category $\widehat{\Cube}$ of presheaves on $\Cube$. \qede
\end{definition} 

The cube category $\Cube$ admits a (strict) monoidal structure given by $(\Cube^{\; \! m}, \Cube^{\; \! n}) \mapsto \Cube^{\;\! m+n}$ which extends to cubical sets via Day convolution. This monoidal structure is denoted by $\otimes$. \par
We denote by $\Cube^{\;\! \leq 1}$ the full subcategory of $\Cube$ spanned by $\Cube^{\; \! 0}, \Cube^{\; \! 1}$. 

\begin{proposition}[{\cite[Prop.~8.4.6]{dcC2006}}]	\label{universal cube} 
Let $M$ be a monoidal category, then the restriction functor 
$$	[\; \! \Cube, M] \to [\; \! \Cube^{\;\! \leq 1}, M]$$
induces an equivalence of categories between the full subcategory of $[\; \! \Cube, M]$ spanned by monoidal functors, and the full subcategory of $[\; \! \Cube^{\;\! \leq 1}, M]$ spanned by functors sending $\Cube^{\; \! 0}$ to the monoidal unit of $M$. \qed
\end{proposition} 

\begin{definition} 
For every $n \geq 0$ the \Emph{boundary of $\Cube^{\; \! n}$} is the subobject $\partial \; \! \Cube^{\; \! n} \defeq \cup_{(j,\zeta)} \varIm{\delta_j^\zeta} \subset \Cube^{\; \! n}$, and for every $n \geq i \geq 1$ and $\xi = 0,1$ the \Emph{$(i,\xi)$-th horn of $\Cube^{\; \! n}$} is the subobject $\CHorn_{\; \! i, \xi}^{\; \! n} \defeq \cup_{(j, \zeta) \neq (i, \xi)} \varIm{\delta_j^\zeta} \subset \Cube^{\; \! n}$.	\qede
\end{definition} 

\begin{proposition}[{\cite[Lm.~8.4.36]{dcC2006}}]	\label{pp cubes}
For $m \geq 1$, $n \geq k \geq 1$ and $\varepsilon = 0,1$ the universal morphisms determined by the pushouts of the spans contained in the commutative squares 
% https://q.uiver.app/?q=WzAsOCxbMCwwLCJcXENIb3JuX3tpLFxcdmFyZXBzaWxvbn1ebiBcXG90aW1lcyBcXHBhcnRpYWwgXFxDdWJlXm0iXSxbMCwxLCJcXENIb3JuX3tpLFxcdmFyZXBzaWxvbn1ebiBcXG90aW1lcyBcXEN1YmVebSJdLFsxLDAsIlxcQ3ViZV5uIFxcb3RpbWVzIFxccGFydGlhbFxcQ3ViZV5tIl0sWzEsMSwiXFxDdWJlXm4gXFxvdGltZXMgXFxDdWJlXm0iXSxbMiwwLCIgIFxccGFydGlhbCBcXEN1YmVebSBcXG90aW1lcyBcXENIb3JuX3tpLFxcdmFyZXBzaWxvbn1ebiJdLFsyLDEsIiBcXEN1YmVebSBcXG90aW1lcyBcXENIb3JuX3tpLFxcdmFyZXBzaWxvbn1ebiAiXSxbMywwLCJcXHBhcnRpYWwgXFxDdWJlXm4gXFxvdGltZXMgXFxDdWJlXm0iXSxbMywxLCJcXEN1YmVebiBcXG90aW1lcyBcXEN1YmVebSJdLFswLDFdLFswLDJdLFsyLDNdLFsxLDNdLFs0LDVdLFs0LDZdLFs2LDddLFs1LDddXQ==
\[\begin{tikzcd}
	{\CHorn_{\; \! i,\varepsilon}^{\; \! n} \otimes \partial \: \! \Cube^{\; \! m}} & {\Cube^{\; \! n} \otimes \partial \: \! \Cube^{\; \! m}} & {  \partial \: \! \Cube^{\; \! m} \otimes \CHorn_{\; \! i,\varepsilon}^{\; \! n}} & {\partial \: \! \Cube^{\; \! n} \otimes \Cube^{\; \! m}} \\
	{\CHorn_{\; \! i,\varepsilon}^{\; \! n} \otimes \Cube^{\; \! m}} & {\Cube^{\; \! n} \otimes \Cube^{\; \! m}} & { \Cube^{\; \! m} \otimes \CHorn_{\; \! i,\varepsilon}^{\; \! n} } & {\Cube^{\; \! n} \otimes \Cube^{\; \! m}}
	\arrow[from=1-1, to=2-1]
	\arrow[from=1-1, to=1-2]
	\arrow[from=1-2, to=2-2]
	\arrow[from=2-1, to=2-2]
	\arrow[from=1-3, to=2-3]
	\arrow[from=1-3, to=1-4]
	\arrow[from=1-4, to=2-4]
	\arrow[from=2-3, to=2-4]
\end{tikzcd}\]
recover the canonical inclusions $\CHorn_{\; \! i,\varepsilon}^{\; \! n+m} \hookrightarrow \Cube^{\;\! n+m}$ and $\CHorn_{\; \!  i + m,\varepsilon}^{\; \! n+m} \hookrightarrow \Cube^{\;\! n+m}$ and the universal morphism determined by the pushout of the span contained in the commutative square
% https://q.uiver.app/?q=WzAsNCxbMCwwLCJcXHBhcnRpYWxcXEN1YmVebSBcXG90aW1lcyBcXHBhcnRpYWwgXFxDdWJlXm4iXSxbMCwxLCJcXEN1YmVebSBcXG90aW1lcyBcXHBhcnRpYWwgXFxDdWJlXm4iXSxbMSwwLCJcXHBhcnRpYWxcXEN1YmVebSBcXG90aW1lcyBcXEN1YmVebiJdLFsxLDEsIlxcQ3ViZV5tIFxcb3RpbWVzIFxcQ3ViZV5uIl0sWzAsMV0sWzAsMl0sWzEsM10sWzIsM11d
\[\begin{tikzcd}
	{\partial \: \! \Cube^{\; \! m} \otimes \partial \: \! \Cube^{\; \! n}} & {\partial \: \! \Cube^{\; \! m} \otimes \Cube^{\; \! n}} \\
	{\Cube^{\; \! m} \otimes \partial \: \! \Cube^{\; \! n}} & {\Cube^{\; \! m} \otimes \Cube^{\; \! n}}
	\arrow[from=1-1, to=2-1]
	\arrow[from=1-1, to=1-2]
	\arrow[from=2-1, to=2-2]
	\arrow[from=1-2, to=2-2]
\end{tikzcd}\]
recovers the  inclusion $\partial \: \! \Cube^{\;\! m+n} \hookrightarrow  \Cube^{\;\! m+n}$.	\qed
\end{proposition} 

\begin{theorem}[{\cite[Cor.~8.4.13~or~Prop.~8.4.27]{dcC2006}}]
The cube category $\Cube$ is a test category. 	\qed
\end{theorem} 

\begin{theorem}[{\cite[Th.~8.4.38]{dcC2006}}]
	The maps 
	\begin{enumerate}[label = (\roman*)]
	\item	$\partial \: \! \Cube^{\; \! n} \hookrightarrow \Cube^{\; \! n}$ ($n \geq 0$), and
	\item $\CHorn_{\; \! i,\varepsilon}^{\; \! n} \hookrightarrow \Cube^{\; \! n}$ ($n \geq i \geq 1$, $\varepsilon = 0,1$)
	\end{enumerate}
generate, respectively, the cofibrations and acyclic cofibrations of the test model structure on $\widehat{\Cube}$. \qed
\end{theorem} 

%\begin{theorem}[{\cite[Th.~8.4.38]{dcC2006}}]
%The test model structure together with the monoidal structure on $\widehat{\Cube}$ form a monoidal model structure.	\qed
%\end{theorem} 

\section{Model structures on $\infty$-categories}	\label{model structure} 

Here we collect some basic definitions and facts about model $\infty$-categories. The proofs are the same as in the ordinary categorical case. As we are not aware of any references in which more involved diagrammatic arguments are carried out with full coherence in the $\infty$-categorical setting (i.e., not in homotopy categories), we have spelled these out in complete detail in Proposition \ref{cmc}, to serve as an illustration of such arguments. 

\begin{definition} 
Let $C$ be an $\infty$-category, and let $f: a \to b$, $g: x \to y$ be morphisms in $C$, then $f$ has the \Emph{left lifting property} w.r.t.\ $g$, and $g$ has the \Emph{right lifting property} w.r.t.\ $f$ if every commutative square 
% https://q.uiver.app/#q=WzAsNCxbMCwwLCJhIl0sWzAsMSwiYiJdLFsxLDAsIngiXSxbMSwxLCJ5Il0sWzAsMSwiZiIsMl0sWzAsMl0sWzIsMywiZyIsMl0sWzEsM11d
\[\begin{tikzcd}
	a & x \\
	b & y
	\arrow["f"', from=1-1, to=2-1]
	\arrow[from=1-1, to=1-2]
	\arrow["g"', from=1-2, to=2-2]
	\arrow[from=2-1, to=2-2]
\end{tikzcd}\] 
may be extended to a $3$-simplex 
% https://q.uiver.app/#q=WzAsNCxbMCwwLCJhIl0sWzAsMSwiYiJdLFsxLDAsIngiXSxbMSwxLCJ5Il0sWzAsMSwiZiIsMl0sWzAsMl0sWzIsMywiZyIsMl0sWzEsM10sWzEsMiwiIiwxLHsic3R5bGUiOnsiYm9keSI6eyJuYW1lIjoiZGFzaGVkIn19fV1d
\[\begin{tikzcd}
	a & x \\
	b & y.
	\arrow["f"', from=1-1, to=2-1]
	\arrow[from=1-1, to=1-2]
	\arrow["g"', from=1-2, to=2-2]
	\arrow[from=2-1, to=2-2]
	\arrow[dashed, from=2-1, to=1-2]
\end{tikzcd}\]
In this case we write $f \boxslash g$. More generally, if $L,F$ are two collections of morphisms in $C$, we write $L \boxslash R$ if $f \boxslash g$ for all $f \in L$ and $g \in R$. Finally, we write $L^{\boxslash}$ for the collection of morphism $g$ such that $f \boxslash g$ for all $f \in L$, and $^{\boxslash}\!R$ for the collections of morphisms $g$ such that $f \boxslash g$ for all $g \in R$.  \qede
\end{definition} 

\begin{definition}	\label{wfs}
Let $C$ be an $\infty$-category, then a pair $(L,R)$ of collections of morphisms in $C$ form a \Emph{weak factorisation system} if 
	\begin{enumerate}[label = (\alph*)]
	\item					$L = {}^{\boxslash}\!R$,
	\item					$L^{\boxslash} = R$, and
	\item	\label{factorisation}	any morphism $a \to x$ may be factored as 
% https://q.uiver.app/#q=WzAsMyxbMCwxLCJhIl0sWzEsMCwicSJdLFsyLDEsIngiXSxbMCwxXSxbMSwyXSxbMCwyXV0=
\[\begin{tikzcd}
	& q \\
	a && x
	\arrow[from=2-1, to=1-2]
	\arrow[from=1-2, to=2-3]
	\arrow[from=2-1, to=2-3]
\end{tikzcd}\]
with $a \to q \in L$ and $q \to x \in R$. 
	\end{enumerate} 
\qede
\end{definition} 

\begin{remark} 
Note, that we do not require the factorisation of $a \to x$ in Definition \ref{wfs} to be either unique nor functorial.	\qede
\end{remark} 

\begin{remark} 
Observe that both classes of a weak factorisation system contain all isomorphisms and are closed under composition (and therefore also under homotopy). \qede
\end{remark} 

\begin{proposition}[{\cite[Prop.~1.4.9]{DAGX}}]	\label{rtcp}
Let $C$ be an $\infty$-category, and $(L,R)$ a weak factorisation system, then both classes are closed under retracts, and the left class is closed under transfinite compositions of pushouts of morphisms in $L$. \qed
\end{proposition} 

\begin{proposition}	\label{wfsa}
Let $C$, $D$ be $\infty$-categories equipped with the weak factorisation systems $(L,R)$ and $(L',R')$, respectively, then for any adjunction $\cadjunction{f:C}{D:g}$ we have
$$ fL \subseteq L' \Longleftrightarrow R \supseteq gR'.	$$
\end{proposition} 

\begin{proof} 
Let $a \to b$ and $x \to y$ morphisms in $C$ and $D$, respectively, then we we want to show that the transpose of the lift in any square 
\begin{equation}	\label{wfs1}
% https://q.uiver.app/#q=WzAsNCxbMCwwLCJhIl0sWzAsMSwiYiJdLFsxLDAsImd4Il0sWzEsMSwiZ3kiXSxbMCwxXSxbMiwzXSxbMCwyXSxbMSwzXSxbMSwyLCIiLDEseyJzdHlsZSI6eyJib2R5Ijp7Im5hbWUiOiJkYXNoZWQifX19XV0=
\begin{tikzcd}
	a & gx \\
	b & gy
	\arrow[from=1-1, to=2-1]
	\arrow[from=1-2, to=2-2]
	\arrow[from=1-1, to=1-2]
	\arrow[from=2-1, to=2-2]
	\arrow[dashed, from=2-1, to=1-2]
\end{tikzcd}
\end{equation} 
gives a lift 
\begin{equation}	\label{wfs2}
% https://q.uiver.app/#q=WzAsNCxbMCwwLCJmYSJdLFswLDEsImZiIl0sWzEsMCwieCJdLFsxLDEsInkiXSxbMCwxXSxbMCwyXSxbMiwzXSxbMSwzXSxbMSwyLCIiLDEseyJzdHlsZSI6eyJib2R5Ijp7Im5hbWUiOiJkYXNoZWQifX19XV0=
\begin{tikzcd}
	fa & x \\
	fb & y
	\arrow[from=1-1, to=2-1]
	\arrow[from=1-1, to=1-2]
	\arrow[from=1-2, to=2-2]
	\arrow[from=2-1, to=2-2]
	\arrow[dashed, from=2-1, to=1-2]
\end{tikzcd}
\end{equation} 
and vice vera. In the $\infty$-categorical setting this requires a little bit of care, because lifts of squares correspond to extensions 
% https://q.uiver.app/#q=WzAsMyxbMCwwLCJcXERlbHRhXntcXHswLDEsM1xcfVxcY3VwIFxcezAsMiwzXFx9fSJdLFswLDEsIlxcRGVsdGFeMyJdLFsxLDAsIkMsRCJdLFswLDEsIiIsMCx7InN0eWxlIjp7InRhaWwiOnsibmFtZSI6Imhvb2siLCJzaWRlIjoidG9wIn19fV0sWzAsMl0sWzEsMiwiIiwwLHsic3R5bGUiOnsiYm9keSI6eyJuYW1lIjoiZGFzaGVkIn19fV1d
\[\begin{tikzcd}
	{\Delta^{\{0,1,3\}\cup \{0,2,3\}}} & {C,D} \\
	{\Delta^3}
	\arrow[hook, from=1-1, to=2-1]
	\arrow[from=1-1, to=1-2]
	\arrow[dashed, from=2-1, to=1-2]
\end{tikzcd}\]
which must be transported back and forth between $C$ and $D$. We use some elementary facts about the Joyal model structure and the calculus of simplices to accomplish this.  Recall that the datum of exhibiting $f$ and $g$ as adjoint is given by a weak equivalence $A \xrightarrow{\sim} B$ in the Joyal model structure 
% https://q.uiver.app/#q=WzAsNyxbMiwxLCJDIFxcdGltZXMgRCJdLFsxLDAsIkEiXSxbMCwwLCJDXntcXERlbHRhXjF9Il0sWzAsMSwiQ157XFxwYXJ0aWFsIFxcRGVsdGFeMX0iXSxbMywwLCJCIl0sWzQsMCwiRF57XFxEZWx0YV4xfSJdLFs0LDEsIkRee1xccGFydGlhbCBcXERlbHRhXjF9Il0sWzEsMCwiIiwwLHsic3R5bGUiOnsiaGVhZCI6eyJuYW1lIjoiZXBpIn19fV0sWzIsMywiIiwwLHsic3R5bGUiOnsiaGVhZCI6eyJuYW1lIjoiZXBpIn19fV0sWzAsMywiKFxcaWQsIGcpICIsMl0sWzEsMl0sWzEsMywiIiwwLHsic3R5bGUiOnsibmFtZSI6ImNvcm5lciJ9fV0sWzQsMCwiIiwyLHsic3R5bGUiOnsiaGVhZCI6eyJuYW1lIjoiZXBpIn19fV0sWzUsNl0sWzAsNiwiKGYsIFxcaWQpIl0sWzQsNV0sWzQsNiwiIiwwLHsic3R5bGUiOnsibmFtZSI6ImNvcm5lciJ9fV0sWzEsNCwiXFxzaW0iXV0=
\[\begin{tikzcd}
	{C^{\Delta^1}} & A && B & {D^{\Delta^1}} \\
	{C^{\partial \Delta^1}} && {C \times D} && {D^{\partial \Delta^1}}
	\arrow[two heads, from=1-2, to=2-3]
	\arrow[two heads, from=1-1, to=2-1]
	\arrow["{(\id, g) }"', from=2-3, to=2-1]
	\arrow[from=1-2, to=1-1]
	\arrow["\lrcorner"{anchor=center, pos=0.125, rotate=-90}, draw=none, from=1-2, to=2-1]
	\arrow[two heads, from=1-4, to=2-3]
	\arrow[from=1-5, to=2-5]
	\arrow["{(f, \id)}", from=2-3, to=2-5]
	\arrow[from=1-4, to=1-5]
	\arrow["\lrcorner"{anchor=center, pos=0.125}, draw=none, from=1-4, to=2-5]
	\arrow["\sim", from=1-2, to=1-4]
\end{tikzcd}\]
Recalling that $\Delta^{\{0,1,3\}}\cup \Delta^{\{0,2,3\}} \simeq \Delta^1 \times \Delta^1$, exponentiating the above diagram by $\Delta^1$ yields the lower half of the following diagram: 
% https://q.uiver.app/#q=WzAsMTEsWzIsMiwiQ157XFxEZWx0YV57XFx7MCwxXFx9fX1cXHRpbWVzIERee1xcRGVsdGFee1xcezIsM1xcfX19Il0sWzEsMSwiQV57XFxEZWx0YV4xfSJdLFswLDEsIkNee1xcRGVsdGFee1xcezAsMSwzXFx9fVxcY3VwIFxcRGVsdGFee1xcezAsMiwzXFx9fX0iXSxbMCwyLCJDXntcXERlbHRhXntcXHswLDFcXH19XFxjdXAgXFxEZWx0YV57XFx7MiwzXFx9fX0iXSxbMywxLCJCXntcXERlbHRhXjF9Il0sWzQsMSwiRF57XFxEZWx0YV57XFx7MCwxLDNcXH19XFxjdXAgXFxEZWx0YV57XFx7MCwyLDNcXH19fSJdLFs0LDIsIkRee1xcRGVsdGFee1xcezAsMVxcfX1cXGN1cCBcXERlbHRhXntcXHsyLDNcXH19fSJdLFswLDAsIkNee1xcRGVsdGFeM30iXSxbMSwwLCJBJyJdLFszLDAsIkInIl0sWzQsMCwiRF57XFxEZWx0YV4zfSJdLFsxLDAsIiIsMCx7InN0eWxlIjp7ImhlYWQiOnsibmFtZSI6ImVwaSJ9fX1dLFsyLDMsIiIsMCx7InN0eWxlIjp7ImhlYWQiOnsibmFtZSI6ImVwaSJ9fX1dLFswLDMsIihcXGlkLCBnKSAiLDJdLFsxLDJdLFsxLDMsIiIsMCx7InN0eWxlIjp7Im5hbWUiOiJjb3JuZXIifX1dLFs0LDAsIiIsMix7InN0eWxlIjp7ImhlYWQiOnsibmFtZSI6ImVwaSJ9fX1dLFs1LDZdLFswLDYsIihmLCBcXGlkKSJdLFs0LDVdLFs0LDYsIiIsMCx7InN0eWxlIjp7Im5hbWUiOiJjb3JuZXIifX1dLFsxLDQsIlxcc2ltIl0sWzcsMiwiIiwwLHsic3R5bGUiOnsiaGVhZCI6eyJuYW1lIjoiZXBpIn19fV0sWzgsN10sWzgsMSwiIiwyLHsic3R5bGUiOnsiaGVhZCI6eyJuYW1lIjoiZXBpIn19fV0sWzgsMiwiIiwwLHsic3R5bGUiOnsibmFtZSI6ImNvcm5lciJ9fV0sWzksNCwiIiwwLHsic3R5bGUiOnsiaGVhZCI6eyJuYW1lIjoiZXBpIn19fV0sWzEwLDUsIiIsMCx7InN0eWxlIjp7ImhlYWQiOnsibmFtZSI6ImVwaSJ9fX1dLFs5LDEwXSxbOCw5LCJcXHNpbSJdLFs5LDUsIiIsMSx7InN0eWxlIjp7Im5hbWUiOiJjb3JuZXIifX1dXQ==
\[\begin{tikzcd}
	{C^{\Delta^3}} & {A'} && {B'} & {D^{\Delta^3}} \\
	{C^{\Delta^{\{0,1,3\}}\cup \Delta^{\{0,2,3\}}}} & {A^{\Delta^1}} && {B^{\Delta^1}} & {D^{\Delta^{\{0,1,3\}}\cup \Delta^{\{0,2,3\}}}} \\
	{C^{\Delta^{\{0,1\}}\cup \Delta^{\{2,3\}}}} && {C^{\Delta^{\{0,1\}}}\times D^{\Delta^{\{2,3\}}}} && {D^{\Delta^{\{0,1\}}\cup \Delta^{\{2,3\}}}}
	\arrow[two heads, from=2-2, to=3-3]
	\arrow[two heads, from=2-1, to=3-1]
	\arrow["{(\id, g) }"', from=3-3, to=3-1]
	\arrow[from=2-2, to=2-1]
	\arrow["\lrcorner"{anchor=center, pos=0.125, rotate=-90}, draw=none, from=2-2, to=3-1]
	\arrow[two heads, from=2-4, to=3-3]
	\arrow[from=2-5, to=3-5]
	\arrow["{(f, \id)}", from=3-3, to=3-5]
	\arrow[from=2-4, to=2-5]
	\arrow["\lrcorner"{anchor=center, pos=0.125}, draw=none, from=2-4, to=3-5]
	\arrow["\sim", from=2-2, to=2-4]
	\arrow[two heads, from=1-1, to=2-1]
	\arrow[from=1-2, to=1-1]
	\arrow[two heads, from=1-2, to=2-2]
	\arrow["\lrcorner"{anchor=center, pos=0.125, rotate=-90}, draw=none, from=1-2, to=2-1]
	\arrow[two heads, from=1-4, to=2-4]
	\arrow[two heads, from=1-5, to=2-5]
	\arrow[from=1-4, to=1-5]
	\arrow["\sim", from=1-2, to=1-4]
	\arrow["\lrcorner"{anchor=center, pos=0.125}, draw=none, from=1-4, to=2-5]
\end{tikzcd}\]
On fibres over $((a \to b), (x \to y))$ the weak equivalence $A^{\Delta^1} \xrightarrow{\sim} B^{\Delta^1}$ yields an equivalence between the spaces of squares of the form (\ref{wfs1}) and (\ref{wfs2}). On fibres over equivalent squares in the spaces $A^{\Delta^1}|_{((a \to b), (x \to y))} \sim B^{\Delta^1}|_{((a \to b), (x \to y))}$ the weak equivalence $A' \xrightarrow{\sim} B'$ yields an equivalence between the respective spaces of lifts.  
\end{proof} 

\begin{lemma}[Retract argument]	\label{retract argument}
Let $C$ be an $\infty$-category with two sets of maps $L,R$ such that $L^{\boxslash} \supseteq R$ ($L \subseteq {}^{\boxslash}R$). Assume that every morphisms in $C$ factors as a morphism in $L$ followed by a morphism in $R$, and that $R$ ($L$) is closed under retracts, then $L^{\boxslash} = R$ ($L = {}^{\boxslash}R$). 
\end{lemma} 

\begin{proof} 
We will prove that if $R$ is closed under retracts, then $L^{\boxslash} = R$. The other statement is dual. Assume that $x \to y$ is in $L^{\boxslash}$, then we may factor it into a morphism $x \to z$ in $L$ followed by a morphism $z \to y$ and consider the diagram 
% https://q.uiver.app/#q=WzAsNCxbMSwwLCJ4Il0sWzEsMSwieSJdLFswLDAsIngiXSxbMCwxLCJ6Il0sWzAsMV0sWzIsM10sWzIsMCwiXFxpZCJdLFszLDFdXQ==
\[\begin{tikzcd}
	x & x \\
	z & y
	\arrow[from=1-2, to=2-2]
	\arrow[from=1-1, to=2-1]
	\arrow["\id", from=1-1, to=1-2]
	\arrow[from=2-1, to=2-2]
\end{tikzcd}\]
which admits a lift by assumption, yielding the retract 
% https://q.uiver.app/#q=WzAsNixbMCwwLCJ4Il0sWzAsMSwieSJdLFsyLDAsIngiXSxbMiwxLCJ5Il0sWzEsMCwieiJdLFsxLDEsInkiXSxbMCwxXSxbMiwzXSxbNCwyXSxbMCw0XSxbMSw1LCJcXGlkIl0sWzUsMywiXFxpZCJdLFs0LDVdXQ==
\[\begin{tikzcd}
	x & z & x \\
	y & y & y
	\arrow[from=1-1, to=2-1]
	\arrow[from=1-3, to=2-3]
	\arrow[from=1-2, to=1-3]
	\arrow[from=1-1, to=1-2]
	\arrow["\id", from=2-1, to=2-2]
	\arrow["\id", from=2-2, to=2-3]
	\arrow[from=1-2, to=2-2]
\end{tikzcd}\]
\end{proof} 

\begin{definition}	\label{msa}
Let $(M,W)$ be a relative $\infty$-category with finite limits and colimits, in which $W$ satisfies the 2-out-of-3 property. A \Emph{model structure} on $M$ is a pair $(C,F)$ of collections of morphisms in $M$ such that $(C \cap W, F)$ and $(C, F \cap W)$ form weak factorisation systems. A relative $\infty$-category with finite limits and colimits equipped with a model structure is called a \Emph{model $\infty$-category}. The morphisms in $C$ (resp.\ $C \cap W$) are called \Emph{(trivial) cofibrations}, and the morphisms in $F$ (resp.\ $F \cap W$) are called \Emph{(trivial) fibrations}. \qede
\end{definition} 

We give two equivalent characterisations of model structures. 

\begin{proposition}	\label{cmc} 
Let $(M,W)$ be a relative $\infty$-category with finite limits and colimits, in which $W$ satisfies the 2-out-of-3 property, then a pair $(C,F)$ of collections of morphisms in $M$ is a model structure iff 
	\begin{enumerate}[label = {\normalfont (\alph*)}]
	\item \label{cm1}	$W,C,F$ are closed under retracts, 
	\item	\label{cm2}	$(W \cap C) \boxslash F$, $C \boxslash (W \cap F)$, and 
	\item	\label{cm3}	any morphisms in $M$ factors both as a morphism in $W \cap C$ followed by a morphism in $F$, as well as a morphism in $C$ followed by a morphism in $W \cap F$. 
	\end{enumerate} 
\end{proposition} 

\begin{proof} 
The proof works exactly the same as in the ordinary categorical case, except that we need to keep track of composition data.  

We first prove that if the pair $(C,F)$ satisfies the axioms of Definition \ref{msa}, then it satisfies properties \ref{cm1} - \ref{cm3}. The classes $C$ and $F$ are closed under retracts by Proposition \ref{rtcp}, and \ref{cm2} \& \ref{cm3} follow by definition. Thus, we are left with showing that $W$ is closed under retracts. To this end we will compile the proof of \cite[Prop.~7.8]{aJmT2007} in the $\infty$-categorical setting (we advise the reader to consult \cite[Prop.~7.8]{aJmT2007} if they are not already familiar with the argument). The gap map of the pushout product of $\partial \Delta^1 \hookrightarrow \Delta^1$ and $\Lambda_1^2 \hookrightarrow \Delta^2$ is inner anodyne (see \cite[Cor.~3.2.4]{dcC2019}), so that the datum of a retract is determined up to contractible choice by a diagram 
\begin{equation}	\label{simplified retract}
% https://q.uiver.app/#q=WzAsNixbMSwwLCJcXGJ1bGxldCJdLFswLDAsIlxcYnVsbGV0Il0sWzIsMCwiXFxidWxsZXQiXSxbMCwxLCJcXGJ1bGxldCJdLFsxLDEsIngiXSxbMiwxLCJ5Il0sWzEsMCwiXFxlbGwiXSxbMCwyXSxbMSwzLCJmIiwyXSxbMyw0XSxbMCw0LCJ3IiwyXSxbNCw1LCJyIiwyXSxbMiw1LCJmIl0sWzEsNF0sWzAsNV0sWzEsMiwiXFxpZCIsMCx7ImN1cnZlIjotM31dLFszLDUsIlxcaWQiLDIseyJjdXJ2ZSI6M31dXQ==
\begin{tikzcd}
	\bullet & \bullet & \bullet \\
	\bullet & x & y.
	\arrow["\ell", from=1-1, to=1-2]
	\arrow[from=1-2, to=1-3]
	\arrow["f"', from=1-1, to=2-1]
	\arrow[from=2-1, to=2-2]
	\arrow["w"', from=1-2, to=2-2]
	\arrow["r"', from=2-2, to=2-3]
	\arrow["f", from=1-3, to=2-3]
	\arrow[from=1-1, to=2-2]
	\arrow[from=1-2, to=2-3]
	\arrow["\id", curve={height=-18pt}, from=1-1, to=1-3]
	\arrow["\id"', curve={height=18pt}, from=2-1, to=2-3]
\end{tikzcd}
\end{equation} 
We assume that $w$ is in $W$, and want to show that $f$ is likewise in $W$. At first, we also assume that $f$ is a fibration, and then we deduce the general case from this special case. We begin by factoring $w$ into a trivial cofibration $u$ followed by a fibration $v$. This corresponds to glueing the $2$-simplex 
% https://q.uiver.app/#q=WzAsMyxbMCwwLCJcXGJ1bGxldCJdLFswLDIsIngiXSxbMSwxLCJcXGJ1bGxldCJdLFswLDEsInciLDJdLFswLDIsInUiXSxbMiwxLCJ2Il1d
\[\begin{tikzcd}
	\bullet \\
	& \bullet \\
	x
	\arrow["w"', from=1-1, to=3-1]
	\arrow["u", from=1-1, to=2-2]
	\arrow["v", from=2-2, to=3-1]
\end{tikzcd}\]
to the above diagram, i.e., we obtain the diagram 
% https://q.uiver.app/#q=WzAsNyxbMCwwLCJcXGJ1bGxldCJdLFszLDAsIlxcYnVsbGV0Il0sWzYsMCwiXFxidWxsZXQiXSxbNiwyLCJ5Il0sWzQsMSwiXFxidWxsZXQiXSxbMywyLCJ4Il0sWzAsMiwiXFxidWxsZXQiXSxbMCwxLCJcXGVsbCJdLFsxLDJdLFswLDIsIlxcaWQiLDAseyJjdXJ2ZSI6LTN9XSxbMSwzXSxbMSw0LCJ1IiwyXSxbNCw1LCJ2IiwyXSxbMSw1LCJ3IiwyXSxbNiw1XSxbMiwzLCJmIl0sWzUsMywiciIsMl0sWzYsMywiXFxpZCIsMix7ImN1cnZlIjozfV0sWzAsNiwiZiIsMl0sWzAsNV1d
\[\begin{tikzcd}
	\bullet &&& \bullet &&& \bullet \\
	&&&& \bullet \\
	\bullet &&& x &&& y.
	\arrow["\ell", from=1-1, to=1-4]
	\arrow[from=1-4, to=1-7]
	\arrow["\id", curve={height=-18pt}, from=1-1, to=1-7]
	\arrow[from=1-4, to=3-7]
	\arrow["u"', from=1-4, to=2-5]
	\arrow["v"', from=2-5, to=3-4]
	\arrow["w"', from=1-4, to=3-4]
	\arrow[from=3-1, to=3-4]
	\arrow["f", from=1-7, to=3-7]
	\arrow["r"', from=3-4, to=3-7]
	\arrow["\id"', curve={height=18pt}, from=3-1, to=3-7]
	\arrow["f"', from=1-1, to=3-1]
	\arrow[from=1-1, to=3-4]
\end{tikzcd}\]

As the inclusions $\Delta^{\{0,1,3\} \cup \{1,2,3\}} \hookrightarrow \Delta^3$ and $\Delta^{\{0,1,2\} \cup \{0,2,3 \}} \hookrightarrow \Delta^3$ are inner anodyne, we may extend the above diagram to an equivalent one, containing two new $3$-simplices as indicated, 

% https://q.uiver.app/#q=WzAsNyxbMCwwLCJcXGJ1bGxldCJdLFszLDAsIlxcYnVsbGV0Il0sWzYsMCwiXFxidWxsZXQiXSxbNiwyLCJ5Il0sWzQsMSwiXFxidWxsZXQiXSxbMywyLCJ4Il0sWzAsMiwiXFxidWxsZXQiXSxbMSwyXSxbMCwyLCJcXGlkIiwwLHsiY3VydmUiOi0zfV0sWzEsM10sWzEsNCwidSIsMl0sWzQsNSwidiIsMl0sWzEsNSwidyIsMl0sWzYsNV0sWzIsMywiZiJdLFs1LDMsInIiLDJdLFs2LDMsIlxcaWQiLDIseyJjdXJ2ZSI6M31dLFswLDYsImYiLDJdLFswLDVdLFswLDEsIlxcZWxsIl0sWzAsNF0sWzQsM11d
\[\begin{tikzcd}
	\bullet &&& \bullet &&& \bullet \\
	&&&& \bullet \\
	\bullet &&& x &&& y
	\arrow[from=1-4, to=1-7]
	\arrow["\id", curve={height=-18pt}, from=1-1, to=1-7]
	\arrow[from=1-4, to=3-7]
	\arrow["u"', from=1-4, to=2-5]
	\arrow["v"', from=2-5, to=3-4]
	\arrow["w"', from=1-4, to=3-4]
	\arrow[from=3-1, to=3-4]
	\arrow["f", from=1-7, to=3-7]
	\arrow["r"', from=3-4, to=3-7]
	\arrow["\id"', curve={height=18pt}, from=3-1, to=3-7]
	\arrow["f"', from=1-1, to=3-1]
	\arrow[from=1-1, to=3-4]
	\arrow["\ell", from=1-1, to=1-4]
	\arrow[from=1-1, to=2-5, crossing over]
	\arrow[from=2-5, to=3-7]
\end{tikzcd}\]
from which we remove $w$ to obtain 
\begin{equation}	\label{factor uv}
% https://q.uiver.app/#q=WzAsNyxbMSwwLCJcXGJ1bGxldCJdLFsxLDEsIlxcYnVsbGV0Il0sWzEsMiwieCJdLFswLDAsIlxcYnVsbGV0Il0sWzAsMiwiXFxidWxsZXQiXSxbMiwwLCJcXGJ1bGxldCJdLFsyLDIsInkiXSxbMCwxLCJ1IiwyXSxbMSwyLCJ2Il0sWzMsNCwiZiIsMl0sWzMsMCwiXFxlbGwiXSxbNCwyXSxbMywxXSxbMywyXSxbMCw1XSxbNSw2LCJmIl0sWzAsNl0sWzEsNl0sWzIsNiwiciIsMl0sWzMsNSwiXFxpZCIsMCx7ImN1cnZlIjotM31dLFs0LDYsIlxcaWQiLDIseyJjdXJ2ZSI6M31dXQ==
\begin{tikzcd}
	\bullet & \bullet & \bullet \\
	& \bullet \\
	\bullet & x & y
	\arrow["u"', from=1-2, to=2-2]
	\arrow["v", from=2-2, to=3-2]
	\arrow["f"', from=1-1, to=3-1]
	\arrow["\ell", from=1-1, to=1-2]
	\arrow[from=3-1, to=3-2]
	\arrow[from=1-1, to=2-2]
	\arrow[from=1-1, to=3-2]
	\arrow[from=1-2, to=1-3]
	\arrow["f", from=1-3, to=3-3]
	\arrow[from=1-2, to=3-3]
	\arrow[from=2-2, to=3-3]
	\arrow["r"', from=3-2, to=3-3]
	\arrow["\id", curve={height=-18pt}, from=1-1, to=1-3]
	\arrow["\id"', curve={height=18pt}, from=3-1, to=3-3]
\end{tikzcd}\end{equation} 
(which is again equivalent to the previous one, by an argument involving inner anodyne extensions). The commutative square with sides $u$ and $f$ admits a lift, giving rise to the diagram 
\begin{equation}	\label{pentagram} 
% https://q.uiver.app/#q=WzAsNyxbMSwwLCJcXGJ1bGxldCJdLFsxLDEsIlxcYnVsbGV0Il0sWzEsMiwieCJdLFswLDAsIlxcYnVsbGV0Il0sWzAsMiwiXFxidWxsZXQiXSxbMiwwLCJcXGJ1bGxldCJdLFsyLDIsInkiXSxbMCwxLCJ1IiwyXSxbMSwyLCJ2Il0sWzMsNCwiZiIsMl0sWzMsMCwiXFxlbGwiXSxbNCwyXSxbMywxXSxbMywyXSxbMCw1XSxbNSw2LCJmIl0sWzEsNl0sWzIsNiwiciIsMl0sWzMsNSwiXFxpZCIsMCx7ImN1cnZlIjotM31dLFs0LDYsIlxcaWQiLDIseyJjdXJ2ZSI6M31dLFsxLDUsInMiLDJdXQ==
\begin{tikzcd}
	\bullet & \bullet & \bullet \\
	& \bullet \\
	\bullet & x & y
	\arrow["u"', from=1-2, to=2-2]
	\arrow["v", from=2-2, to=3-2]
	\arrow["f"', from=1-1, to=3-1]
	\arrow["\ell", from=1-1, to=1-2]
	\arrow[from=3-1, to=3-2]
	\arrow[from=1-1, to=2-2]
	\arrow[from=1-1, to=3-2]
	\arrow[from=1-2, to=1-3]
	\arrow["f", from=1-3, to=3-3]
	\arrow[from=2-2, to=3-3]
	\arrow["r"', from=3-2, to=3-3]
	\arrow["\id", curve={height=-18pt}, from=1-1, to=1-3]
	\arrow["\id"', curve={height=18pt}, from=3-1, to=3-3]
	\arrow["s"', from=2-2, to=1-3]
\end{tikzcd}
\end{equation} 
(Again thinking about inner anodyne inclusions, we see that we do not have to include the $3$-simplex exhibiting the composition of $u, s, f$.) Mapping $\Lambda_1^3$ to 
% https://q.uiver.app/#q=WzAsNCxbMSwwLCJcXGJ1bGxldCJdLFsxLDEsIlxcYnVsbGV0Il0sWzAsMCwiXFxidWxsZXQiXSxbMiwwLCJcXGJ1bGxldCJdLFswLDEsInUiLDJdLFsyLDAsIlxcZWxsIl0sWzIsMV0sWzAsM10sWzIsMywiXFxpZCIsMCx7ImN1cnZlIjotM31dLFsxLDMsInMiLDJdXQ==
\[\begin{tikzcd}
	\bullet & \bullet & \bullet \\
	& \bullet
	\arrow["u"', from=1-2, to=2-2]
	\arrow["\ell", from=1-1, to=1-2]
	\arrow[from=1-1, to=2-2]
	\arrow[from=1-2, to=1-3]
	\arrow["\id", curve={height=-18pt}, from=1-1, to=1-3]
	\arrow["s"', from=2-2, to=1-3]
\end{tikzcd}\]
we may extend (\ref{pentagram}) by a $3$-simplex via $\Lambda_1^3 \hookrightarrow \Delta^3$, from which we may remove $\Delta^{\{1\}}$ to obtain the diagram 
% https://q.uiver.app/#q=WzAsNixbMSwwLCJcXGJ1bGxldCJdLFswLDAsIlxcYnVsbGV0Il0sWzIsMCwiXFxidWxsZXQiXSxbMCwxLCJcXGJ1bGxldCJdLFsxLDEsIngiXSxbMiwxLCJ5Il0sWzEsMF0sWzAsMiwicyJdLFsxLDMsImYiLDJdLFszLDRdLFswLDQsInYiLDJdLFs0LDUsInIiLDJdLFsyLDUsImYiXSxbMSw0XSxbMCw1XSxbMSwyLCJcXGlkIiwwLHsiY3VydmUiOi0zfV0sWzMsNSwiXFxpZCIsMix7ImN1cnZlIjozfV1d
\[\begin{tikzcd}
	\bullet & \bullet & \bullet \\
	\bullet & x & y
	\arrow[from=1-1, to=1-2]
	\arrow["s", from=1-2, to=1-3]
	\arrow["f"', from=1-1, to=2-1]
	\arrow[from=2-1, to=2-2]
	\arrow["v"', from=1-2, to=2-2]
	\arrow["r"', from=2-2, to=2-3]
	\arrow["f", from=1-3, to=2-3]
	\arrow[from=1-1, to=2-2]
	\arrow[from=1-2, to=2-3]
	\arrow["\id", curve={height=-18pt}, from=1-1, to=1-3]
	\arrow["\id"', curve={height=18pt}, from=2-1, to=2-3]
\end{tikzcd}\]
exhibiting $f$ as the retract of a trivial fibration, so that $f$ is a trivial fibration, and thus a weak equivalence. (Observe that only in this last step did we throw away information, that cannot be recovered using anodyne extensions.)  

We now show that $f$ in the diagram (\ref{simplified retract}) is a weak equivalence without the assumption that it is a fibration.  Consider the outer square of (\ref{simplified retract}): 
% https://q.uiver.app/#q=WzAsNCxbMCwwLCJcXGJ1bGxldCJdLFsxLDAsIlxcYnVsbGV0Il0sWzAsMSwiXFxidWxsZXQiXSxbMSwxLCJ5Il0sWzAsMSwiXFxpZCJdLFswLDIsImYiLDJdLFsyLDMsIlxcaWQiXSxbMSwzLCJmIl0sWzAsM11d
\[\begin{tikzcd}
	\bullet & \bullet \\
	\bullet & y
	\arrow["\id", from=1-1, to=1-2]
	\arrow["f"', from=1-1, to=2-1]
	\arrow["\id", from=2-1, to=2-2]
	\arrow["f", from=1-2, to=2-2]
	\arrow[from=1-1, to=2-2]
\end{tikzcd}\]
Factoring $f$ into a trivial cofibration $g$ followed by a fibration $h$, yields the diagram 
% https://q.uiver.app/#q=WzAsNixbMCwwLCJcXGJ1bGxldCJdLFswLDEsIlxcYnVsbGV0Il0sWzAsMiwiXFxidWxsZXQiXSxbMSwwLCJcXGJ1bGxldCJdLFsxLDEsInoiXSxbMSwyLCJ5Il0sWzAsMSwiZyIsMl0sWzEsMiwiaCIsMl0sWzAsMywiXFxpZCJdLFszLDQsImciXSxbNCw1LCJoIl0sWzIsNSwiXFxpZCJdLFsxLDQsIlxcaWQiXSxbMCw0XSxbMSw1XSxbMCwyLCJmIiwyLHsiY3VydmUiOjN9XSxbMyw1LCJmIiwwLHsiY3VydmUiOi0zfV1d
\[\begin{tikzcd}
	\bullet & \bullet \\
	\bullet & z \\
	\bullet & y
	\arrow["g"', from=1-1, to=2-1]
	\arrow["h"', from=2-1, to=3-1]
	\arrow["\id", from=1-1, to=1-2]
	\arrow["g", from=1-2, to=2-2]
	\arrow["h", from=2-2, to=3-2]
	\arrow["\id", from=3-1, to=3-2]
	\arrow["\id", from=2-1, to=2-2]
	\arrow[from=1-1, to=2-2]
	\arrow[from=2-1, to=3-2]
	\arrow["f"', curve={height=18pt}, from=1-1, to=3-1]
	\arrow["f", curve={height=-18pt}, from=1-2, to=3-2]
\end{tikzcd}\]
Completing the diagram (\ref{simplified retract}) and the above diagram to diagrams indexed by $\Delta^1 \times \Delta^2$ we can glue them along the face $\Delta^1 \times \Delta^{\{0,2\}}$ yielding a diagram with the following shape:
\begin{figure}[H]
\centering
\includegraphics[scale = .3, angle = -90]{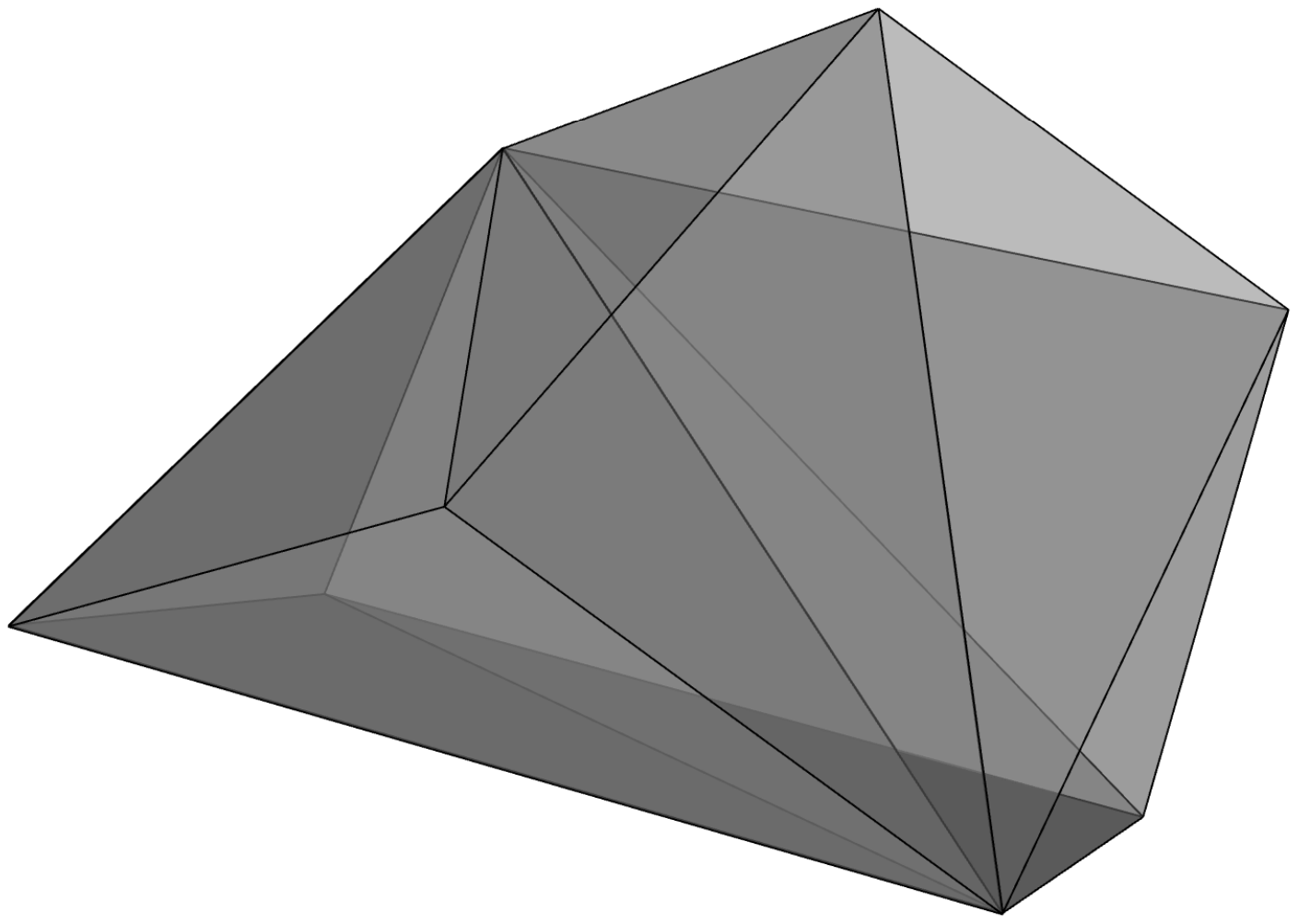}
\end{figure} 
 We wish to exhibit the objects $x,y,z$ as apices of cocones on $\bullet \xleftarrow{g} \bullet \xrightarrow{\ell} \bullet$ together with cocone morphisms $x \to y \leftarrow z$. In the bottom left we have two $3$-simplices exhibiting, respectively, the compositions of $\bullet \xrightarrow{f} \bullet \xrightarrow{\ell} x \xrightarrow{r} y$ and $\bullet \xrightarrow{g} \bullet \xrightarrow{h} \bullet \xrightarrow{\ell} x$, which are glued together along the faces $\Delta^{\{0,1,3\}}$ and $\Delta^{\{0,2,3\}}$, yielding an inner anodyne inclusion $\Delta^{\{0,2,3,4\}} \cup \Delta^{\{0,1,2,4\}} \subseteq \Delta^4$. Restricting along $\Delta^{\{0,1,3,4\}}$ produces a $3$-simplex whitnessing the composition of $\bullet \xrightarrow{f} \bullet \xrightarrow{\ell g} x \xrightarrow{r} y$. Performing the same procedure yields a $3$-simplex exhibiting the composition of $\bullet \xrightarrow{\ell} \bullet \xrightarrow{gr} z \xrightarrow{h} y$. These two $3$-simplices together with the $3$-simplices exhibiting the compositions of $\bullet \xrightarrow{\ell} \bullet \xrightarrow{w} x \xrightarrow{r} y$ and $\bullet \xrightarrow{g} \bullet \xrightarrow{\id} y \xrightarrow{h} z$ glued along their common faces produce the desired cocone morphisms $x \to y \leftarrow z$. Denote by $c$ the (apex of the) colimit of $\bullet \xleftarrow{g} \bullet \xrightarrow{\ell} \bullet$, then we obtain a square of cocones whose apices yield the lower right square in the diagram 
% https://q.uiver.app/#q=WzAsOSxbMCwwLCJcXGJ1bGxldCJdLFsxLDAsIlxcYnVsbGV0Il0sWzIsMCwiXFxidWxsZXQiXSxbMSwxLCJjIl0sWzAsMSwiXFxidWxsZXQiXSxbMiwxLCJ6Il0sWzIsMiwieSJdLFsxLDIsIngiXSxbMCwyLCJcXGJ1bGxldCJdLFswLDFdLFsxLDJdLFsxLDNdLFswLDQsImciLDJdLFs0LDNdLFszLDVdLFsyLDUsImciXSxbNSw2LCJoIl0sWzMsN10sWzcsNl0sWzQsOCwiaCIsMl0sWzgsN10sWzAsM10sWzEsNV0sWzQsN10sWzMsNl0sWzAsMiwiXFxpZCIsMCx7ImN1cnZlIjotM31dLFs4LDYsIlxcaWQiLDIseyJjdXJ2ZSI6M31dXQ==
\[\begin{tikzcd}
	\bullet & \bullet & \bullet \\
	\bullet & c & z \\
	\bullet & x & y
	\arrow[from=1-1, to=1-2]
	\arrow[from=1-2, to=1-3]
	\arrow[from=1-2, to=2-2]
	\arrow["g"', from=1-1, to=2-1]
	\arrow[from=2-1, to=2-2]
	\arrow[from=2-2, to=2-3]
	\arrow["g", from=1-3, to=2-3]
	\arrow["h", from=2-3, to=3-3]
	\arrow[from=2-2, to=3-2]
	\arrow[from=3-2, to=3-3]
	\arrow["h"', from=2-1, to=3-1]
	\arrow[from=3-1, to=3-2]
	\arrow[from=1-1, to=2-2]
	\arrow[from=1-2, to=2-3]
	\arrow[from=2-1, to=3-2]
	\arrow[from=2-2, to=3-3]
	\arrow["\id", curve={height=-18pt}, from=1-1, to=1-3]
	\arrow["\id"', curve={height=18pt}, from=3-1, to=3-3]
\end{tikzcd}\]
together with a $2$-simplex exhibiting the composition of $\bullet \to c \to z$ to $\id$, which we have not indicated. The lower left and upper right squares are obtained from glueing $2$-simplices coming from the diagram of cocones constructed as well as the two $4$-simplices constructed above. The morphism $g$ is a weak equivalence by assumption, as it is a trivial cofibration, and $h$ is a weak equivalence, as it is a fibration, so that we may apply the argument from the beginning of the proof.  

Conversely, assume that $(C,F)$ satisfies \ref{cm1} - \ref{cm3}, then we only need to show that $((C \cap W), F)$ and $(C, (F \cap W))$ form weak factorisation systems, which follows from the retract argument (Lemma \ref{retract argument}). 
\end{proof} 

\begin{proposition}	\label{wmsa}
Let $(M,W)$ be a relative $\infty$-category with finite limits and colimits, in which $W$ satisfies the 2-out-of-3 property, then a pair $(C,F)$ of collections of morphisms in $M$ forms a model structure iff 
	\begin{enumerate}[label = {\normalfont (\alph*)}]
	\item	the pair $((C \cap W), F)$ (resp.\ $(C, F \cap W)$) forms a weak factorisation sytem, 
	\item	\label{acm2}	any morphisms in $M$ factors as a morphism in $C$ (resp. ${}^{\boxslash}F$) followed by a morphism in $C^{\boxslash}$ (resp. $F$), and
	\item	$C^{\boxslash} \subseteq F \cap W$ (resp. ${}^{\boxslash}F \subseteq C \cap W$). 
	\end{enumerate} 
\end{proposition} 

\begin{proof} 
Consider a morphism $x \to y$ in $F \cap W$, then we must show that it lies in $C^{\boxslash}$. First, factor $x \to y$ as a morphism $x \to y'$ in $C$, followed by a morphism $y' \to y$ in $C^{\boxslash}$. By assumption $y' \to y$ is in $W$, so that by the 2-out-of-3 property $x \to y'$ is in $W$, and the lifting problem 
% https://q.uiver.app/#q=WzAsNCxbMSwwLCJ4Il0sWzEsMSwieSJdLFswLDAsIngiXSxbMCwxLCJ5JyJdLFswLDFdLFsyLDNdLFsyLDBdLFszLDFdXQ==
\[\begin{tikzcd}
	x & x \\
	{y'} & y
	\arrow[from=1-2, to=2-2]
	\arrow[from=1-1, to=2-1]
	\arrow[from=1-1, to=1-2]
	\arrow[from=2-1, to=2-2]
\end{tikzcd}\]
admits a solution $y' \to x$, as $x \to y$ is in $F$. Then $y' \to x$ may be used to construct a retract 
% https://q.uiver.app/#q=WzAsNixbMSwwLCJ5JyJdLFsxLDEsInkiXSxbMCwwLCJ4Il0sWzAsMSwieSJdLFsyLDAsIngiXSxbMiwxLCJ5Il0sWzAsMV0sWzIsM10sWzQsNV0sWzIsMF0sWzAsNF0sWzMsMV0sWzEsNV1d
\[\begin{tikzcd}
	x & {y'} & x \\
	y & y & y,
	\arrow[from=1-2, to=2-2]
	\arrow[from=1-1, to=2-1]
	\arrow[from=1-3, to=2-3]
	\arrow[from=1-1, to=1-2]
	\arrow[from=1-2, to=1-3]
	\arrow[from=2-1, to=2-2]
	\arrow[from=2-2, to=2-3]
\end{tikzcd}\]
so that $x \to y$ is contained in $C^{\boxslash}$ by Proposition \ref{rtcp}. 
\end{proof} 

\section{Pro-objects in $\infty$-categories}	\label{pro}

Here we collect some useful properties of pro-objects in $\infty$-categories used in \S \ref{squishy squish}. 

\begin{proposition}	\label{proproducts}
Let $C$ be an $\infty$-category admitting finite products, then $\Pro(C)$ admits finite products. 
\end{proposition} 

\begin{proof}
Let $x_0, \ldots, x_n$ be objects in $\Pro(C)$ then for each $0 \leq i \leq n$ there exists a filtered small ordinary category $A_i$ and a functor $x_{i \bullet}: A_i \to C$ such that $x_i \simeq {\displaystyle ``\lim_{\alpha \in A_i}" }x_{i \alpha}$ (see \cite[Prop.~5.3.1.16]{jL2009}). The category $A_0 \times \cdots \times A_n$ is filtered, and we claim that ${\displaystyle``\lim_{\alpha \in A_0 \times \cdots \times A_n}" } x_{0 \bullet} \times \cdots \times x_{n \bullet} $ pro-represents the product of $x_0, \ldots, x_n$. To see this, let $y$ be any objects of $C$, then the isomorphisms 
$$	\begin{array}{rcl}
	\Pro(C)({\displaystyle``\lim_{A_0 \times \cdots \times A_n}" }x_{0 \bullet} \times \cdots \times x_{n \bullet}, y)	&	\simeq	&	\lim_{A_0 \times \cdots \times A_n}C(x_{0 \bullet} \times \cdots \times x_{n \bullet}, y)					\\
	{}																					&	\simeq	&	\lim_{A_0 \times \cdots \times A_n}C(x_{0 \bullet}, y) \times \cdots \times C(x_{n \bullet}, y)				\\
	{}																					&	\simeq	&	\lim_{A_0}  \cdots  \lim_{A_n}C(x_{0 \bullet}, y) \times \cdots \times C(x_{n \bullet}, y)					\\
	{}																					&	\simeq	&	\lim_{A_0}  \cdots  \lim_{A_{n-1}}C(x_{0 \bullet}, y) \times \cdots \times C(x_{n-1 \bullet}, y) \times C(x_n, y)	\\
	{}																					&	\cdots	&	{}	\\
	{}																					&	\simeq	&	\lim_{A_0}  C(x_{0 \bullet}, y) \times C(x_1, y) \times \cdots \times C(x_n, y)								\\
	{}																					&	\simeq	&	C(x_0, y)	\times \cdots \times C(x_n, y)	\\
	\end{array}	$$
are natural in $y$. 
\end{proof}

\begin{lemma}	\label{pie}
Let $I$ be a set, and for each element $i \in I$ consider a small filtered category $A_i$ and a functor $X_i: A_i \to \mathcal{S}$, then the canonical morphism 
\begin{equation} \label{IPC}
	\colim_{(\alpha_i) \in \prod A_i} \prod_{i \in I} X_{i, \alpha_i} 	\to	\prod_{i \in I} \colim_{\alpha_i \in A_i} X_{i, \alpha_i} 	
\end{equation} 
is an equivalence. 
\end{lemma} 

\begin{proof}
By \cite[Prop.~3.1.11.ii]{mKpS2006} the statement is true in $\Set$. Then, by \cite[Cor.~7.9.9]{dcC2019} we may lift the functors $X_i: A_i \to \mathcal{S}$ to functors $A_i \to \widehat{\Delta}$, which we may then compose with the $\Ex^\infty$ functor to obtain functors valued in Kan complexes. The morphism in $\widehat{\Delta}$ corresponding to (\ref{IPC}) is then an isomorphism, and the statement follows from the fact that Kan complexes as well as weak equivalences are closed under filtered colimits (see \cite[Lm.~3.1.24~\&~Cor.~4.1.17]{dcC2019}). 
\end{proof}

\begin{proposition}	\label{pro cocomplete}
Let $C$ be an accessible $\infty$-category admitting finite limits and coproducts, then the $\infty$-category $\Pro(C)$ is cocomplete.
\end{proposition} 

%generalise to accessible admitting coproducts. 

\begin{proof}
We show that $\Pro(C)$ admits pushouts and small coproducts.  \\

\noindent\underline{$\Pro(C)$ admits pushouts:}		Recall that $\Pro(C)$ may be identified with the full subcategory of $[C,\mathcal{S}]^{\op}$ spanned by the left exact functors $f: C \to \mathcal{S}$ such that $C_{/f}$ is accessible by \cite[Prop.~3.1.6]{DAGXIII}. Consider a pullback square  
% https://q.uiver.app/?q=WzAsNCxbMSwwLCJmIl0sWzEsMSwiaCJdLFswLDEsImciXSxbMCwwLCJwIl0sWzAsMV0sWzIsMV0sWzMsMl0sWzMsMF0sWzMsMSwiIiwxLHsic3R5bGUiOnsibmFtZSI6ImNvcm5lciJ9fV1d
\[\begin{tikzcd}
	p & f \\
	g & h
	\arrow[from=1-2, to=2-2]
	\arrow[from=2-1, to=2-2]
	\arrow[from=1-1, to=2-1]
	\arrow[from=1-1, to=1-2]
	\arrow["\lrcorner"{anchor=center, pos=0.125}, draw=none, from=1-1, to=2-2]
\end{tikzcd}\]
of functors in $[C,\mathcal{S}]$ with $f,g,h$ in $\Pro(C)$.   As limits of functors are computed pointwise, $p: C \to \mathcal{S}$ commutes with finite limits. Moreover, the above diagram induces a homotopy pullback diagram 
% https://q.uiver.app/?q=WzAsNCxbMCwwLCJDX3svcH1ee1xcb3B9Il0sWzAsMSwiQ197L2d9XntcXG9wfSJdLFsxLDAsIkNfey9mfV57XFxvcH0iXSxbMSwxLCJDX3svaH1ee1xcb3B9Il0sWzAsMV0sWzAsMl0sWzIsM10sWzEsM11d
\[\begin{tikzcd}
	{C_{/p}^{\op}} & {C_{/f}^{\op}} \\
	{C_{/g}^{\op}} & {C_{/h}^{\op}}
	\arrow[from=1-1, to=2-1]
	\arrow[from=1-1, to=1-2]
	\arrow[from=1-2, to=2-2]
	\arrow[from=2-1, to=2-2]
\end{tikzcd}\]
in $\widehat{\Delta}$ w.r.t.\ the Joyal model structure. The morphisms $C_{/f}^{\op} \to C_{/g}^{\op}$ and $C_{/h}^{\op} \to C_{/g}^{\op}$ are colimit preserving, so that $C_{/p}$ is accessible by \cite[Prop.~5.4.6.6]{jL2009}.	\\

\noindent\underline{$\Pro(C)$ admits small coproducts:}	Let $I$ be a small set, and consider a family of objects $x_\bullet: I \to \Pro(C)$, then for each $i$ there exists a filtered small ordinary %explain
 category $A_i$ and a functor $x_{i \bullet}: A_i \to C$ such that $x_i \simeq {\displaystyle ``\lim_{\alpha \in A_i}" }x_{i \alpha}$ (see \cite[Prop.~5.3.1.16]{jL2009}). By Lemma \ref{pie} we obtain the canonical isomorphisms 
 $$	\coprod_{i \in I} x_i \simeq \coprod_{i \in I} ``\lim_{\alpha_i \in A_i}" x_{i \alpha_i}	\simeq	``\lim_{(\alpha_i) \in \prod A_i}"  \coprod_{i \in I}  x_{i \alpha_i},	$$
in $\Pro(C)$, as limits and colimits in presheaf categories are computed pointwise. 
\end{proof} 

\section*{Conventions and notation}	\label{Conventions}
\addcontentsline{toc}{section}{\nameref{Conventions}}

\paragraph{Linguistic conventions}

In order to facilitate readability we use the following contractions: 

\begin{itemize}
\item We write ``iff'' instead of ``if and only if''. 
\item We write ``w.l.o.g.'' instead of ``without loss of generality''. 
\item We write ``w.r.t.'' instead of ``with respect to''. 
\end{itemize}

\paragraph{Editorial conventions}

\begin{itemize}
\item Propositions stated without proof are marked with the symbol ``$\Box$''.  
\end{itemize}

\paragraph{Mathematical conventions}

\begin{itemize}
\item The term \emph{$\infty$-category} means \emph{quasi-category}. 
\item	We identify ordinary categories with their nerves, and consequently do not notationally distinguish between ordinary categories and their nerves. 
\item	$[\emptyinput, \emptyinput]$ denotes the internal hom in $\widehat{\Delta}$, the category of simplicial sets. 
\item	Let $C,D$ be $\infty$-categories, and $W \subseteq C$, a subcategory, then $[C,D]_W$ denotes the subcategory of $[C,D]$ spanned by those functors sending every morphism in $W$ to an isomorphism. 
\item	Let $X$ be a simplicial set, then $X_\simeq$ denotes the classifying space of $X$, given e.g.\ by $\mathrm{Ex}^\infty A$. 
\item	$\infty$-categories (including ordinary categories) are denoted by $C$, $D$, \ldots
\item Let $C$ be an $\infty$-category and let $x,y \in {C}$ be two objects, then the homotopy type of morphisms from $x$ to $y$ is denoted by $C(x,y)$. 
\item A final object in an $\infty$-category $C$ is denoted by $\mathbf{1}_C$, or simply by $\mathbf{1}$, when $C$ is clear from context. 
\item	For any Cartesian closed $\infty$-category $C$ and any two objects $x,y$ in $C$ the internal hom object in $C$ is denoted by $\underline{C}(x,y)$ or sometimes $y^x$.
\item	For any $\infty$-category $C$ we denote its subcategory of $n$-truncated objects by $C_{\leq n}$. 
\item	For any $\infty$-category $C$ with finite products and any group object $G$ in $C$, we denote $C_G$ the category of $G$-objects in $C$.  
\item For $A$ any small ordinary category $\widehat{A}$ denotes the category of (set-valued) presheaves on $A$. 
\item	For any two categories $C, D$, an arrow $C \hookrightarrow D$ denotes a fully faithful functor.
\item We use the following notation for various $\infty$-categories:	
	\begin{itemize}
	\item	$\Delta$ denotes the category of simplices. Its objects are denoted by $\Delta^n$ or $[n]$, depending on context. 
	\item	$\Cube$ denotes the category of cubes. 
	\item	$\mathcal{S}$ denotes the $\infty$-categories of homotopy types. 
%	\item	$\mathcal{C}$ denotes the $\infty$-category of ordinary categories. 
	\item	$\Cat$ denotes the $\infty$-category of $\infty$-categories. 
	\item	$\Cat_{(1,1)}$ denotes the $(2,1)$-category of ordinary categories. 
	\item	$\Cat_{(1,1)}'$ denotes the relative ordinary category of ordinary categories, with weak equivalences given by equivalences of ordinary categories. 
	\item	$\Cat$ denotes the $\infty$-category of $\infty$-categories. 
	\item	$\Top$ denotes the $\infty$-category of $\infty$-toposes. 
	\item	Let $X$ be a topological space, then $\Open_X$ denotes the locale of open subsets of $X$. 
	\item $\Set$ denotes the category of sets. 
	\item $\TSpc$ denotes the category of topological spaces.
	\item	$\Delta\TSpc$ is the full subcategory of $\TSpc$ spanned by the $\Delta$-generated topological spaces. 
	\item $\Mfd^r$ denotes the category of $r$-times differentiable smooth manifolds and smooth maps.
	\item	$\Cart^r$ denotes the full subcategory of $\Mfd^r$ spanned by the spaces of $\mathbf{R}^n$ ($0 \leq n < \infty$).
	\item	$\Diff^r$ denotes, equivalently, the $\infty$-category of sheaves on $\Mfd^r$ or $\Cart^r$.
	\end{itemize}
\item	We denote $\infty$-toposes by $\mathcal{E}, \mathcal{F}, \ldots$, when they are thought of as ambient settings in which to do geometry, and by $\mathcal{X}, \mathcal{Y}, \ldots$, when they are thought of as geometric objects in their own right. 
\end{itemize}

\newcommand{\etalchar}[1]{$^{#1}$}
\providecommand{\noopsort}[1]{}\def\cprime{$'$}

%\section*{References}	\label{References}
%\addcontentsline{toc}{section}{\nameref{References}}

%\bibliographystyle{/usr/local/texlive/2023/texmf-dist/bibtex/bst/base/alpha2} 
%\bibliography{/Users/adrianclough/Mathematics/My_Mathematics/Ancillary_files/Documents/Adrians_bibliography}

\end{document}